%% file: main.tex
\documentclass[10pt]{article}
\usepackage{graphicx} 
\usepackage[english]{babel}
\usepackage[utf8]{inputenc}
\usepackage[letterpaper, margin=2.5cm]{geometry}
\usepackage{authblk}
\usepackage{csquotes}

\usepackage{hyperref}

\usepackage{fancyhdr}

\fancypagestyle{theme}{
\fancyhead[L]{}
\fancyhead[R]{}
\fancyfoot[C]{\thepage}
\fancyfoot[R]{S. Houdaigoui and K. Kawarabayashi}
\fancyfoot[L]{National Institute of Informatics}

}
\usepackage{amsmath}
\usepackage{amsfonts}
\usepackage{amssymb}
\usepackage{amsthm}
\usepackage[only,llbracket,rrbracket]{stmaryrd}

\usepackage{tikzit}
\input{images/samples.tikzstyles}
\usepackage{subcaption}

\makeatletter
\newtheorem*{rep@theorem}{\rep@title}
\newcommand{\newreptheorem}[2]{%
\newenvironment{rep#1}[1]{%
 \def\rep@title{#2 \ref{##1}}%
 \begin{rep@theorem}}%
 {\end{rep@theorem}}}
\makeatother

\theoremstyle{plain}
\newtheorem{theo}{Theorem}[section]
\newreptheorem{theo}{Theorem}
\newtheorem{lem}[theo]{Lemma}
\newtheorem{sublem}{Lemma}[theo]

\newtheorem{prop}[theo]{Proposition}
\newreptheorem{prop}{Proposition}
\newtheorem{cor}[theo]{Corollary}
\newreptheorem{cor}{Corollary}
\newtheorem*{prob}{Open problem}
\newtheorem{claim}{Claim}[theo]
\newtheorem{subclaim}{Claim}[sublem]
\theoremstyle{definition}
\newtheorem{defi}{Definition}[section]

\newenvironment{proof_claim}{\noindent\textit{Proof of the claim.}}{\hfill$\blacksquare$}

\definecolor{Tolpurple}{HTML}{AA3377}

\title{A quasi-polynomial bound for the minimal excluded minors for a surface}
\author{Sarah Houdaigoui\\
\href{mailto:shoudaigoui@nii.ac.jp}{shoudaigoui@nii.ac.jp}\\ 
National Institute of Informatics, Tokyo, Japan\\
Graduate university for advanced studies (SOKENDAI), Hayama, Japan
\and Ken-ichi Kawarabayashi\footnote{Research supported by JSPS Kakenhi 22H05001 and by JST ASPIRE JPMJAP2302.}\\
\href{mailto:k_keniti@nii.ac.jp}{k\_keniti@nii.ac.jp}\\
National Institute of Informatics, Tokyo, Japan\\
The University of Tokyo, Japan
}
\date{}

\begin{document}

\maketitle

\begin{abstract}
As part of their graph minor project, Robertson and Seymour showed in 1990 that the class of graphs that can be embedded in a given surface can be characterized by a finite set of minimal excluded minors \cite{GM8}. However, their proof, because existential, does not provide any information on these excluded minors. Seymour proved in 1993 the first and, until now, only known upper bound on the order of the minimal excluded minors for a given surface \cite{seymour}. This bound is double exponential in the Euler genus $g$ of the surface and, therefore, very far from the $\Omega(g)$ lower bound on the maximal order of minimal excluded minors for a surface and most likely far from the best possible bound. More than thirty years later, this paper finally makes progress in lowering this bound to a quasi-polynomial in the Euler genus of the surface. 

The main catalyzer to reach a quasi-polynomial bound is a breakthrough on the characteristic size of a forbidden structure for a minimal excluded minor $G$ for a surface of Euler genus $g$: although it is not hard to show that $G$ does not contain $O(g)$ disjoint cycles that are contractible and nested in some embedding of $G$ as demonstrated by Seymour in \cite{seymour}, this bound can be lowered to $O(\log g)$ which is essential to obtain the quasi-polynomial bound in this paper.
Moreover, we find an upper bound on the maximum degree of $G$ and the maximum size of a face in an embedding of $G$ in a surface of Euler genus $g+1$ or $g+2$, which is, to our understanding, the first such bound. Finally, we develop a new method to bound the height of the tree in a tree decomposition of $G$.

As subsidiary results, we also improve the current bound on the treewidth of a minimal excluded minor $G$ for a surface by improving the first and, until now, only known bound provided by Seymour in \cite{seymour}. Moreover, we show a better upper bound on the order of a grid minor in $G$, improving the result by Thomassen \cite{Thomassen} from 1997.
\end{abstract}

\newpage
\section{Introduction}

The graph minor theorem \cite{GM20}, proved by Robertson and Seymour in 2004, states that every class of graphs closed under minors can be defined by a finite list of minimal excluded minors  (which we will denote by \textit{excluded minors} from now on). Understanding the nature of the excluded minors for various classes of graphs closed under minor operations is a very active area of research. Especially, exploring the link between graph parameters that are minor-monotone and the presence of specific graphs as minor has been a very fruitful research area.
The most iconic example of such work is the grid minor theorem \cite{GM5} which states that there exists a function $f$ such that, for every grid $H$, every graph with treewidth at least $f(|V(H)|)$ contains $H$ as a minor. There has been extensive work on lower and upper bounds for the function $f$, see \cite{Chekuri_chuzhoy, Chuzhoy_Tan} for the current state of the art. There are also results analogous to the grid minor theorem for pathwidth \cite{Kinnersley, TUK} and treedepth \cite{Kawarabayashi_Rossman}.
Another minor monotone graph parameter that has always aroused major interest is the genus. It was studied well before the graph minor project, but this project sparked a renewal in this field of study and a shift in the angle of approach.

In 1930, Kuratowski \cite{Kuratowski1930} showed that the planar graphs can be described as graphs that do not contain a subdivision of $K_5$ nor $K_{3,3}$ as a subgraph. Subsequently, in 1937, Wagner showed that the planar graphs can be described as graphs that do not contain $K_5$ nor $K_{3,3}$ as a minor. These two results can easily be shown to be equivalent.
Finding an analog to Kuratowski's or Wagner's theorems for higher genus surfaces was for a long time a central open problem: in 1978, Glover and Huneke showed that there is a finite number of minimal subgraphs excluded as subdivision (thereafter minimal excluded subgraphs) for the projective plane \cite{Glover_Huneke}; then Glover, Huneke and Wang presented in 1979 a list of 103 minimal excluded subgraphs \cite{Glover_Huneke_Wang}; finally, Archdeacon proved in 1980-81 that this list is complete \cite{Archdeacon_1980, Archdeacon_1981}. 
It is easy to show that a class of graphs is characterized by a finite set of excluded minors if and only if it is characterized by a finite set of minimal excluded subgraphs (see, e.g., \cite[Proposition 6.1.1]{graphs_on_surfaces}).
In 1989, Archdeacon and Huneke extended the result from the projective plane to every non orientable surface by showing that every non orientable surface admits only a finite number of minimal excluded subgraphs. Robertson and Seymour finally settled the question by showing that the number of excluded minors for a surface is always finite in a preliminary result of the celebrated graph minor theorem \cite{GM8}. 
Simpler proofs of the graph minor theorem for surfaces have been then found by Mohar \cite{Mohar_2001} and Thomassen \cite{Thomassen}.

However, the graph minor theorem's proof is purely existential and hence gives no clue on how to find the excluded minors for a surface. Therefore, except for the plane and the projective plane, the exact lists of excluded minors for each surface are unknown. Considerable efforts toward such a result for the torus have been made, but even though several classes of minimal excluded subgraphs for the torus are known (see \cite{Decker_1978, Juvan_1995, Duke_Haggard, Hlavacek_1997, Bodendiek_Wagner}), no exhaustive list was eventually produced. Among all of these results, there are at least 2200 excluded minors for the torus. It is, therefore, unrealistic in the foreseeable future to hope that a characterization of excluded minors will be explicitly found for surfaces of the Euler genus higher than the torus'. 

Consequently, the research shifted from finding an explicit list to bounding parameters (order, degree, treewidth...) of these excluded minors for a surface, usually by a function of the Euler genus of the surface.

In 1993, Seymour gave the first and, until now, the only known upper bound on the order of an excluded minor for a surface \cite{seymour}.

\begin{theo}[\cite{seymour}]
    Let $S$ be a given surface of Euler genus $g$. Every excluded minor for $S$ has at most $2^{2^k}$ vertices where $k = (3g+9)^9$. 
\end{theo}

This bound is a double-exponential in the Euler genus $g$ of the surface. Moreover, the lower bound on the maximal order of excluded minors for a surface is easily shown to be $\Omega(g)$ (by considering $g+1$ copies of $K_{3,3}$).
Therefore, the gap between known lower and upper bounds is huge, and Seymour's upper bound is most likely far from the best possible bound.

In the first improvement in thirty years, we drastically lower the bound on the order of an excluded minor for a surface to a quasi-polynomial in the Euler genus $g$ of the surface.
The main result of this paper is indeed the following:

\begin{theo}
    \label{main_result}
    Let $S$ be a given surface of Euler genus $g$. Every excluded minor for $S$ is of order at most $U(g) = g^{O(\log^3 g)}$.
\end{theo}

\vspace{0.5cm}

Let us highlight the technical aspects of our result. 
Let $S, S'$ be surfaces with $S'$ of Euler genus $g$ and $S$ of Euler genus $g+1$ or $g+2$. Let $G$ be an excluded minor for the surface $S'$ and suppose that $G$ can be embedded in surface $S$ with embedding $\Pi$.

The main catalyzer to reach a quasi-polynomial bound is a breakthrough on the characteristic size of a forbidden structure for $G$: it is not hard to show that $G$ does not contain $O(g)$ disjoint cycles that are $\Pi$-contractible and nested in $\Pi$, as demonstrated by Seymour in \cite[(2.2)]{seymour}. We show that this bound can be lowered to $O(\log g)$ (Section \ref{sec:structural_tools}, Proposition \ref{good_square}). A related result was obtained in \cite{KS19}. Moreover, we extend this result to $\Pi$-noncontractible homotopic cycles of $G$. 
These structural results are utterly essential to obtain the quasi-polynomial bound in this paper.

Another important ingredient of our proof is a $g^{O(\log^2 g)}$ upper bound on the maximum degree of $G$ and on the maximum size of a face in $(G, \Pi)$ (Subsection \ref{subsec:max_degree} Theorem \ref{max_degree}). It is, to our understanding, the first such bound. Finally, we develop a new method to bound the height of the tree in a tree decomposition of $G$ and we obtain a bound that is drastically lower than the one found by Seymour \cite[claim (6) in (4.1)]{seymour}. Indeed, we lower it from a single exponential in $g$ to a quasi-polynomial $g^{O(\log^3 g)}$ in $g$ (Subsection \ref{height_tree_decomposition}, Theorem \ref{no_long_planar_path_T_extended}).


\vspace{0.3cm}

Another very important result from Seymour's paper \cite{seymour} is the first and, until now, only known bound on the treewidth of an excluded minor for a surface:

\begin{theo}[{\cite[(3.3)]{seymour}}]

    \label{seymour_treewidth}
    The treewidth of $G$ is bounded by a polynomial in $g$:
    \[ tw(G) \leq T_{S}(g)\]

    with $T_{S}(g) = 3(g+3)^2(3g+16)-3 = O(g^3)$.
\end{theo}

As a corollary of the structural results we obtain in this paper, we improve the bound on the treewidth:

\begin{theo}
    \label{treewidth}
    The treewidth of $G$ is bounded by the following function of $g$:
    \[ tw(G) \leq T(g)\]

    with $T(g) = 264(g+2)(m+1)-1 = O(g \log g)$, where $m = 2(\lfloor\log_{q}(3g+4)\rfloor + 2)$ and $q = \frac{1153}{1152}$.
\end{theo}

Moreover, as a subsidiary result, we obtain an improvement on the following important result of Thomassen \cite{Thomassen} (see the end of Subsection \ref{homotopic_nested_squares}):

\begin{theo}[\cite{Thomassen}]
    \label{thomassen_cor}
    Let $H$ be a $2$-connected graph such that $g(H) = g$ and $g(H -e) < g$ for every edge $e$ in $H$. Then $H$ contains no subdivision of the $4k \times 2k$ grid, with $k = \lceil 800 g^{3/2} \rceil$.
\end{theo}

We improve the bound on the size of the forbidden grid:

\begin{theo}
    \label{thomassen_improvement}
    Let $G$ be a $2$-connected excluded minor for a surface of Euler genus $g$. Then $G$ contains no subdivision of the $4k \times 2k$ grid, with $k = O(\sqrt{g} \log g)$.
\end{theo}

In an appendix, we give a tight lower bound on the order of the excluded minors for each surface, together with a quick proof:

\begin{theo}
    \label{lower_bound}
    Let $S$ be a surface of Euler genus $g$ and let $L(g)  = \left\lfloor \frac{7 + \sqrt{1+24g}}{2} \right \rfloor +1$. There exists a excluded minor for $S$ of order 

    \[ |V(G)| = L(g) \]

    and no excluded minor for $S$ is of order $< L(g)$.
\end{theo}

\vspace{0.3cm}

Let us highlight thereafter the algorithmic implications of our results.
As a consequence of the graph minor theorem, Robertson and Seymour \cite{GM13} showed in 1995 that, for every minor-closed class of graphs $\mathcal{G}$, there exists a cubic time algorithm that decides whether a graph belongs to $\mathcal{G}$ (membership test). The running time was later improved to quadratic by Kawarabayashi et al. \cite{KKR} in 2012 and to quasi-linear by Korhonen et al. \cite{KPS} in 2024.
However, although the algorithm in itself is explicit and therefore constructive, it relies on the (finite) set of excluded minors that characterizes the class of graphs $\mathcal{G}$ which is in general unknown. 
Furthermore, Fellows and Langston \cite{Fellows_Langston} observed in 1989 that there is no algorithm that, given a Turing machine that is a membership test for a minor-closed class of graphs $\mathcal{C}$, computes the excluded minors that characterize $\mathcal{C}$. Nevertheless, Fellows and Langston showed in the same paper that an explicit upper bound on the treewidth of the excluded minors for a minor-closed class of graphs leads to an algorithm to find them. Adler et al. \cite{AGK} also gave such an algorithm. Therefore, our improvement on the upper bound on the treewidth of the excluded minors for a given surface induces a speedup for these algorithms.


Moreover, we take two examples to illustrate the importance of a bound for the order of the excluded minors for a surface. First, Kawarabayashi et al. \cite{KMR} gave in 2008 a linear time FPT-algorithm for embedding graphs into an arbitrary surface $S$, with parameter the genus of $S$, which relies on the bound on the order of the excluded minors for $S$. Second, Grohe et al. \cite{GKR} gave in 2013 a quadratic time algorithm for computing graph minor decompositions, which also makes use of the bound on the order of the excluded minors for a surface.

Therefore, by improving in this paper the bounds for various parameters (especially order and treewidth, but also degree, height, and degree of the tree of a tree decomposition) of the excluded minors for a surface, we contribute significantly to the effort to solve various problems more efficiently.

\section{Overview of the proof and the paper} \label{proof_strategy}

Let $S, S'$ be surfaces with $S'$ of Euler genus $g$ and $S$ of Euler genus $g+1$ or $g+2$. Let $G$ be an excluded minor for the surface $S'$ and suppose that $G$ can be embedded in surface $S$ with embedding $\Pi$. We define a \textit{piece} as a vertex or a face of $(G, \Pi)$.

To prove that $G$ is of order bounded by a quasi-polynomial in $g$, we proceed in broad outline as follows:

\begin{itemize}
    \item \textbf{We first prove that the degree of $G$ and the size of faces in $(G, \Pi)$ are bounded by a quasi-polynomial in $g$ (Subsection \ref{subsec:max_degree}).} 
    
    The proof proceeds by contradiction. We make the assumption that there is a piece $p$ that is either a vertex of degree $> \Delta(g)$ or a face of size $> \Delta(g)$ for some function $\Delta$ (defined in Theorem \ref{max_degree}), and deduce that there is a $\Pi$-contractible fan on $p$ (a fan $F$ on a piece $p$ is roughly a subgraph consisting of internally disjoint paths from $p$ to a path, see Definition \ref{fan}) whose size is smaller than $\Delta(g)$ by a polynomial factor in $g$. Then, we show that this fan extends to a $\Pi$-contractible subgraph of $G$ that contains a forbidden structure consisting of $O(\log g)$ nested cycles that intersect only in $p$ and that all are $\Pi$-contractible (see Subsection \ref{subsec:nested_squares} for the proper statement).
    
    \item \textbf{We then show that the height of the tree of a tree decomposition of $G$ of width $w$ is bounded by a quasi-polynomial in $g$ and $w$ (Subsection \ref{height_tree_decomposition}).} 
    
    We assume that there is a path of length $> P(g,w)$ (for some function $P$ defined in Proposition \ref{no_long_planar_path}) in the tree of a tree decomposition of $G$ of width $w$ that induces a $\Pi$-contractible subgraph $G'$ of $G$. On one hand, we show that it implies that $|V(G')| > P(g,w)$. On the other hand, we take advantage of the natural separations given by the tree decomposition to show that $G'$ is separated from the rest of $G$ by a small (polynomial in $g$) separation, and then to deduce that $|V(G')|$ must be bounded by $P(g,w)$. This leads to a contradiction. Finally, there is no such path in the tree of a tree decomposition of $G$ and we deduce a bound on the height of the tree of a tree decomposition of $G$.
    
    \item \textbf{We finally put everything together to show our main result (Subsections \ref{previously_known_results} and \ref{size_quasi_polynomial}).} 
    
    Seymour (Theorems \ref{seymour_treewidth} and \ref{seymour_tree_degree}, {\cite[(3.3) and claim (5) in (4.1)]{seymour}}) shows that both the treewidth of $G$ and the maximum degree of the tree in an optimal tree decomposition of $G$ (with additional conditions) are bounded by $O(g^3)$. As subsidiary results of the structural results, we improve both of these results to $O(g \log g)$. These two bounds, together with the bound on the height of the tree, give a first bound, single exponential in $g$, on the order of $G$. Then, by looking at the pathwidth and a path decomposition of $G$, and by taking advantage of the bound on the height of the tree of a tree decomposition of $G$ (that still holds when applied to a path decomposition), we finally reach a $g^{O(\log^3 g)}$ bound on the order of $G$.
\end{itemize}

In this paper, we prove two main structural results that describe forbidden structures in $\Pi$-contractible subgraphs of $(G,\Pi)$ and that will be used repetitively to find a contradiction in the main proof:

\begin{itemize}
    \item There are at most $O(g)$ internally disjoint paths from one piece to another piece (Section \ref{sec:structural_tools}, Proposition \ref{isolated_paths}).
    \item There are at most $O(\log g)$ disjoint cycles that are either $\Pi$-contractible and nested in $\Pi$ or $\Pi$-noncontractible homotopic (Section \ref{sec:structural_tools}, Propositions \ref{good_square} and \ref{good_square_variant}). These two results are improvements on a result of Seymour (Proposition \ref{seymour_nested_cycles} in Section \ref{sec:structural_tools}, {\cite[(2.2)]{seymour}}) and are central to obtain the main result of this paper.
\end{itemize}

\vspace{0.5cm}

This paper is organized as follows. Section \ref{sec:definitions} contains the basic definitions and notations for graphs and surfaces, which are used throughout the paper. In Section \ref{sec:preliminary_results}, we introduce basic results for graphs on surfaces and preliminary results on the excluded minors for a surface. Section \ref{sec:structural_tools} contains the two main structural results mentioned above. In Section \ref{sec:main_proof}, we detail the proof of our main result. This section contains three subsections whose content has been explained above. Section \ref{sec:conclusion} explores the bottlenecks of our proof and provides suggestions to improve the main result of this paper further. Finally, we provide a quick proof for a tight lower bound on the order of the excluded minors for each surface in the appendix.

\section{Definitions and Notations} \label{sec:definitions}

\subsubsection*{Basics}

We consider simple graphs. Let $G$ be a graph.
A walk in $G$ is a finite sequence of vertices so that two consecutive vertices are adjacent. A path is a walk whose vertices are distinct, except potentially its two extremities. The interior of a path $P$ consists of $P$ from which we remove its two extremities. A circuit (or closed walk) is a walk with equal extremities. A cycle is a circuit whose vertices are all distinct except its two extremities.

The degree of a vertex $v$ of $G$ is the number of edges adjacent to $v$; we denote it by $d_G(v)$ (or $d(v)$ if clear in the context). The maximum degree of $G$ is the maximum degree of a vertex of $G$ and is denoted $\Delta(G)$.

A \textit{minor} $H$ of $G$ is a (multi)graph obtained from $G$ by successive operations of suppression of a vertex, suppression of an edge, and contraction of an edge. A \textit{proper minor} of $G$ is a minor of $G$ different from $G$.
We denote by $G - v$ (resp. $G-e$) the graph obtained from $G$ after the suppression of a vertex $v$ (resp. an edge $e$), and we denote $G - X$ the graph obtained from $G$ after the suppression of a set $X$ of vertices or edges. We denote by $G / e$ the multigraph obtained from $G$ after the contraction of an edge $e$. We denote $G / X$ the multigraph obtained from $G$ after the contraction of a set $X$ of edges (remark that the order in which the edges in $X$ are contracted is not important).

\subsubsection*{Trees and tree decomposition}

A \textit{tree} $T$ is a connected graph without any cycle. We can distinguish a vertex $v$ of $T$, called the \textit{root}, then $T$ is said to be \textit{rooted} in $v$. We define the height of a tree as the order of its longest path.
A \textit{spanning tree} $T$ of $G$ is a subgraph of $G$ that is a tree and contains all the vertices of $G$. 
A \textit{tree decomposition} of $G$ is a pair $(T, (V_t)_{t \in V(T)})$ with $T$ a tree and, for every $t \in V(T)$, $V_t \subseteq V(G)$ with the following properties:
\begin{itemize}
    \item $\bigcup_{t \in V(T)} V_t = V(G)$,
    \item for every $e = uv \in E(G)$, there exists $t \in V(T)$ so that $u,v \in V_t$,
    \item for $t,t',t'' \in V(T)$ so that $t'$ is on the path between $t$ and $t''$ in $T$, $V_t \cap V_{t''} \subseteq V_{t'}$.
\end{itemize}
The \textit{width} of a tree decomposition $(T, (V_t)_{t \in V(T)})$ of $G$ is $\max_{t \in V(T)} |V_t| -1$ and the \textit{treewidth} of $G$ is the minimal width of its tree decompositions.

We say that a tree decomposition $(T, (V_t)_{t \in V(T)})$ of $G$ (possibly with some additional properties) is \textit{minimal}, if $|V(T)|$ is chosen minimum (and the tree decomposition has the additional properties).

\subsubsection*{Connectivity}

A graph $G$ is \textit{connected} if, for every pair of vertices $u,v \in V(G)$, there is a path in $G$ between $u$ and $v$. Let $k \in \mathbb{N}^*$, $G$ is \textit{$k$-connected} if, for every set $V \subseteq V(G)$ of size $k-1$, $G - V$ is still connected. Let $k \geq 1$, a \textit{$k$-separator} of $G$ is a set $V \subseteq V(G)$ of $k$ vertices so that $G - V$ is not connected. If $G$ contains a $1$-separator $\{v\}$, $v$ is called a \textit{cutvertex} of $G$. Remark that, for $k \geq 1$, a graph contains no $k-1$-separator if and only if it is $k$-connected.
A \textit{separation} of $G$ is a pair $(A, B)$ of the subgraph of $G$ such that $A \cup B = G$ and $A \cap B$ contains no edge. Remark that $V(A \cap B)$ is a separator of $G$. For $k \in \mathbb{N}$, we say that $(A,B)$ is a \textit{$k$-separation} of $G$ if $(A,B)$ is a separation of $G$ and $|V(A \cap B)| = k$.
A \textit{($2$-connected) block} of $G$ is the subgraph of $G$ induced by an equivalence class of the following equivalence relation on $E(G)$: $e_1 \sim e_2$ if $e_1 = e_2$ or there is in $G$ a cycle that contains $e_1$ and $e_2$. Remark that $H$ is a block of $G$ if and only if $H$ is a maximal $2$-connected subgraph of $G$.

\subsubsection*{Bridges}

Let $H$ be a graph and $H_0$ be a subgraph of $H$. A \textit{bridge} $B$ of $H$ on $H_0$ is either an edge with both ends in $H_0$ (and its ends do not belong to the bridge) or a connected subgraph of $H - V(H_0)$ together with all the edges which have one end in this component and the other end in $H_0$ (the ends in $H_0$ does not belong to the bridge).
Remark that the bridges of $H$ on $H_0$ partition $E(H) - E(H_0)$.
We say that $w \in H_0$ is an \textit{attach} of a bridge $B$ on $H_0$ if an edge $e = vw \in E(B)$ exists.

\vspace{0.5cm}

To describe graphs on surfaces, we use almost identical definitions and notations (the ones that differ will be indicated) as in the book \textit{Graphs on Surfaces} from Mohar and Thomassen \cite{graphs_on_surfaces}. Thereafter, we give a condensed summary of the definitions and notation used in this paper.

\subsubsection*{Surface and disk}

We define a \textit{surface} as a connected compact Hausdorff topological space $S$ so that each point of $S$ has an open neighborhood homeomorphic to the open unit disk in $\mathbb{R}^2$. A \textit{simple closed curve} in a surface $S$ is a continuous 1-1 function $f : [0,1] \mapsto S$ with $f(0) = f(1)$. A \textit{disk} on $S$ is a simple closed curve that can be continuously deformed into a single point.

\subsubsection*{Embedding and facial walk}

Let $G$ be a graph and $S$ a surface, we define an \textit{embedding} $\Pi$ of $G$ in $S$ to be a pair $\Pi = (\pi, \lambda)$ where $\pi = \{ \pi_v | v \in V(G) \}$ associates to each vertex $v \in V(G)$ a cyclic permutation of the edges incident to $v$ and $\lambda : E(G) \mapsto \{-1,1\}$ associates to each edge $e \in E(G)$ a sign (either $-1$ or $1$) called its signature. 

A \textit{local change} of an embedding $\Pi = (\pi, \lambda)$ of $G$ changes the clockwise ordering to anticlockwise at some vertex $v \in V(G)$, i.e. $\pi_v$ is replaced by its inverse $\pi_v^{-1}$, and $\lambda(e)$ is replaced by $-\lambda(e)$ for every edge $e$ that is incident with $v$. Two embeddings of $G$ are \textit{equivalent} if one can be obtained from the other by a sequence of local changes.

The \textit{face traversal procedure} is the following: We start with a vertex $v \in V(G)$ and an edge $e = vw \in E(G)$. Let's traverse the edge from $v$ to $w$, if the signature of $e$ is $1$, then we continue the walk along the edge $e' = \pi_w(e)$, otherwise ($\lambda(e) = -1$), we change $\pi_u$ by $\pi_u^{-1}$ for every $u \in V(G)$ and we continue the walk along the edge $e' = \pi_w(e)$. And so forth. The walk is completed when the initial edge $e$ is encountered in the same direction (from $v$ to $w$) and with the same orientation ($\pi$ is the same as before we began the walk). A \textit{$\Pi$-facial walk} (or, if clear in the context, \textit{face}) is a closed walk in $G$ determined by the face traversal procedure. We define the \textit{size} of a face to be the number of edges that the walk contains. We define the \textit{maximum face degree} of $(G, \Pi)$ to be $\Delta_F(G, \Pi) = \max \{ |f|, f \text{ is a face in } (G, \Pi) \}$. 

Let $C$ be a cycle of $G$, we define the \textit{signature} of $C$ in $\Pi$ to be $\lambda(C) = \prod_{e \in C} \lambda(e)$. We say that $C$ is \textit{two-sided} if $\lambda(C) = 1$ and \textit{one-sided} otherwise. We say that an embedding $\Pi$ of $G$ is \textit{orientable} if every cycle $C$ of $G$ is two-sided, otherwise we say that it is \textit{non orientable}.

\subsubsection*{Genus and embeddability}

In this paper, the \textit{genus} of a surface always refers to the Euler genus, that is to say twice its orientable genus if the surface is orientable and its non orientable genus if the surface is non orientable. The sphere is the only surface of genus $0$, the projective plane is the only surface of genus $1$ and is a non orientable surface.
Let $F(G, \Pi)$ be the $\Pi$-facial walks of $(G, \Pi)$. We define the Euler characteristic of $\Pi$ to be $\chi(\Pi) = |V(G)| - |E(G)| + |F(G, \Pi)|$.
For an embedding $\Pi$ and a graph $G$, $g(\Pi) = 2 - \chi(\Pi)$ denotes the \textit{Euler genus} of $\Pi$, and $g(G)$ denotes the Euler genus of an embedding of $G$ with minimal Euler genus. The formula \[ \chi(\Pi) = 2 - g(\Pi)\] is called the \textit{Euler formula}.
We say that $G$ is \textit{embeddable} in a surface $S$, if there exists an embedding $\Pi$ of $G$ so that $g(\Pi) = g(S)$ and $\Pi$ is \textit{orientable} if and only if $S$ is orientable.
Let $G'$ be a minor of $G$; it is easy to prove that $g(G') \leq g(G)$. We say that $G$ is a \textit{(minimal) excluded minor} for a surface $S$ if $G$ is not embeddable in $S$ but every proper minor of $G$ is embeddable in $S$.

\subsubsection*{Separating cycles}

Let $C = v_0 e_1 v_1 e_2 ... v_{l-1}e_lv_0$ be a $\Pi$-twosided cycle of a $\Pi$-embedded graph $G$. Suppose that the signature of $C$ is positive in $\Pi$. We define the \textit{left graph} and the \textit{right graph} of $C$ as follows: for $1 \leq i \leq l$, if $e_{i+1} = \pi_{v_i}^{k_i}(e_i)$, then all the edges $\pi_{v_i}(e_i), \pi_{v_i}^2(e_i), ..., \pi_{v_i}^{k_i-1}(e_i)$ are said to be on the left side of $C$ and the left graph of $C$, denoted by $G_l(C, \Pi)$, is defined as the union of all bridges on $C$ that attach to $C$ by at least one edge on the left side of $C$. The right graph $G_r(C, \Pi)$ is defined analogously.
A cycle $C$ of a $\Pi$-embedded graph $G$ is \textit{$\Pi$-separating} if $C$ is two-sided and $G_l(C, \Pi)$ and $G_r(C, \Pi)$ have no edges in common. If $\Pi(G_l(C, \Pi))$ or $\Pi(G_r(C, \Pi))$ is an induced embedding of genus $0$, then we say that $C$ is \textit{$\Pi$-contractible}. Suppose without loss of generality that $\Pi(G_l(C, \Pi))$ is an induced embedding of genus $0$, we define $\text{int}(C, \Pi_H)$ (resp. $\text{ext}(C, \Pi_H)$) to be the bridges onto $C$ in $G_l(C, \Pi)$ (resp. $G_r(C, \Pi)$). Moreover, we define $\text{Int}(C, \Pi_H) = \text{int}(C, \Pi_H) \cup C$ and $\text{Ext}(C, \Pi_H) = \text{ext}(C, \Pi_H) \cup C$. If it is clear in the context, the mention of the embedding is removed, and we write $\text{Int}(C)$, $\text{int}(C)$, $\text{Ext}(C)$, $\text{ext}(C)$.

\subsubsection*{Planar graphs}

A \textit{planar graph} is a graph embeddable in the sphere (surface of genus $0$). Let $G$ be a planar graph. A planar embedding $\Pi$ of $G$ is an embedding of $G$ in the sphere with a distinguished face called the \textit{outer face}. Let $C$ be a cycle of $G$, then $C$ is two-sided and $\Pi$-separating. Let $G_l(C, \Pi)$ and $G_r(C, \Pi)$ be the left and right graph of $G$ and suppose without loss of generality that the outer face of $\Pi$ is in $G_r(C, \Pi)$. Then, we define $\text{Int}(C, \Pi)$, $\text{int}(C, \Pi)$, $\text{Ext}(C, \Pi)$, $\text{ext}(C, \Pi)$ as above, by making sure that the outer face in the exterior of $C$.

\subsubsection*{Cutting along a cycle}

Let $G$ be a $\Pi$-embedded graph. Let $C$ be a $\Pi$-separating cycle, then cutting along $C$ gives rise to two graphs $G_l(C, \Pi) \cup C$ and $G_r(C, \Pi) \cup C$ and their induced embedding $\Pi(G_l(C, \Pi))$ and $\Pi(G_r(C,\Pi))$. Let $C$ be a two-sided $\Pi$-nonseparating cycle. Let $\overline{G}$ be the graph obtained from $G$ by replacing $C$ with two copies of $C$ such that all the edges of $C$ on the left side of $C$ are incident one copy of $C$ and all the edges on the right side of $C$ are incident with the other copy of $C$. We say that $\overline{G}$ is the graph obtained by \textit{cutting along $C$}, and we call the induced embedding $\overline{\Pi}$. Let $C= v_0 e_1 v_1 e_2 ... v_{l-1}e_lv_0$ be a one-sided $\Pi$-nonseparating cycle. Let's first define edges on the left side of $G$ and edges on the right side of $G$ at each vertex of $C$ in a similar way as in the case of two-sided cycles: suppose first that the signature $\lambda$ of $\Pi$ satisfies $\lambda(e_i) = 1$ for $1 \leq i < l$ and $\lambda(e_l) = -1$. Then, we use the pairs of consecutive edges $e_i, e_{i+1}$ on $C$ ($1 \leq i < l$)  to define edges on the left side of $C$ incident with the vertex $v_i$. Then we construct $\overline{G}$ by replacing $C$ in $G$ by the cycle $\overline{C} = v_0 e_1 ... e_l \overline{v}_0 \overline{e}_1 ... \overline{e}_l v_0$. The edges on the left side of $C$ are adjacent to $v_0, ..., v_{l-1}$ and the edges on the right side of $C$ are incident to $\overline{v}_0, ..., \overline{v}_{l-1}$. We extend $\lambda$ by putting $\lambda(\overline{e}_0) = ... = \lambda(\overline{e}_{l-1}) = 1$ and $\lambda(\overline{e}_l) = -1$ and hence obtain an embedding $\overline{\Pi}$ of $\overline{G}$.
We say that $\overline{G}$ is the graph obtained by \textit{cutting along $C$}. 

\subsubsection*{Homotopic cycles}

Let $C$ and $C'$ be two-sided cycles of a $\Pi$-embedded graph $G$. Suppose that $C$ and $C'$ are either disjoint or share a path (that might be reduced to a vertex). We say that $C$ and $C'$ are \textit{$\Pi$-homotopic} if cutting along $C$ and $C'$ results in a graph which has a component $D$ which contains precisely one copy of $C$ and one copy of $C'$ and the induced genus of the embedding $\Pi(D)$ is $0$. We then write $D = \text{Int}(C \cup C', \Pi)$. 

\subsubsection*{Cycles bounding a disk or cylinder}

Let $G$ be a graph $\Pi$-embedded in a surface $S$. Let $C$ be a $\Pi$-contractible cycle in $G$. Then, we say that $C$ is the \textit{boundary} of $\text{Int}(C, \Pi)$. Moreover, we say that $C$ \textit{bounds a disk} if $\text{int}(C, \Pi)$ is empty. Let $C$ and $C'$ be two cycles that are disjoint and $\Pi$-homotopic. We say that $C$ and $C'$ are the \textit{boundaries} of $\text{Int}(C \cup C', \Pi)$. Moreover, we say that $C$ and $C'$ \textit{bound an cylinder} if $\text{int}(C \cup C', \Pi)$ is empty.

\section{Preliminary results} \label{sec:preliminary_results}

\subsection{Basic results on graphs on surfaces}

\begin{defi}[Flipping]
    Let $H$ be a 2-connected planar graph with embedding $\Pi_H$ in the plane. Let $C$ be a cycle of $H$ such that only two vertices $v$ and $w$ of $C$ have incident edges in $\text{ext}(C, \Pi_H)$. Then, we define a \textit{flipping} of $H$ with respect to $C$ as a reembedding of $H$ such that the embedding in $\text{ext}(C, \Pi_H)$ is unchanged. The embedding of $H' = H \cap \text{Int}(C, \Pi_H)$ is changed so that the new embedding of $H'$ is equivalent to the original one, but the clockwise orientations of all the facial cycles are reversed.
\end{defi}

\begin{prop}[Whitney's Theorem {\cite[Theorem 2.6.8]{graphs_on_surfaces}}]

    \label{Whitney_planar_graph_flipping}
    Let $H$ be a 2-connected plane graph and $\Pi_H$ be an embedding of $H$ in the plane. Then, any embedding of $H$ in the plane can be obtained from $\Pi_H$ by a sequence of flippings.
\end{prop} 

\begin{prop}[{\cite[Lemma 4.2.4]{graphs_on_surfaces}}]
    \label{4.2.4}

    Suppose that $C$ is a $\Pi_H$-nonseparating cycle of a $\Pi_H$-embedded graph $H$ in a surface. Let $\overline{H}$ be the graph obtained by cutting along $C$ and let $\overline{\Pi}_H$ be its resulting embedding. Then all $\Pi_H$-facial walks are $\overline{\Pi}_H$-facial walks in $\overline{H}$, where edges of $C$ are replaced by their copies in $\overline{H}$. If $C$ is $\Pi_H$-twosided, then $g(\overline{\Pi}_H) = g(\Pi_H)-2$ and the two copies of $C$ are the new $\Pi_H$-facial cycles. If $C$ is $\Pi_H$-onesided, then $g(\overline{\Pi}_H) = g(\Pi_H)-1$ and $\overline{C}$ is a facial cycle in $\overline{H}$.
\end{prop}

\begin{prop}[{\cite[Lemma 4.2.5]{graphs_on_surfaces}}]
    \label{3_homotopic_cycles}
    
    Suppose that $C_1, C_2, C_3$ are pairwise disjoint cycles of a $\Pi_H$-embedded graph $H$ and that $C_1, C_2$ and $C_2, C_3$ are $\Pi_H$-homotopic (respectively). Then also $C_1$ and $C_3$ are $\Pi_H$-homotopic If one of the cycles is $\Pi$-noncontractible, then so are the other two and there is a pair $\{r,s\} \subset \{1,2,3\}$ such that for $1 \leq i < j \leq 3$

    \[ \text{Int}(C_i \cup C_j, \Pi_H) \subseteq \text{Int}(C_r \cup C_s, \Pi_H)\]
\end{prop}

\begin{prop}[{\cite[Proposition 4.2.7]{graphs_on_surfaces}}]
    \label{homotopic_cycles}
    Let $H$ be a $\Pi_H$-embedded graph and $a$, $b$ vertices of $H$ (possibly $a = b$). If $P_0, ..., P_k$ are pairwise internally disjoint paths (or cycles) from $a$ to $b$ such that no two of them are $\Pi$-homotopic, then

    \[ k \leq
    \left\{
    \begin{array}{lr}
        g(\Pi_H) &  \text{ if } g(\Pi_H) \leq 1\\
        3 g(\Pi_H) - 3 & \text{ if } g(\Pi_H) \geq 2 \\ 
    \end{array}\right.\]
\end{prop}

\begin{lem}[\cite{graphs_on_surfaces}]
    \label{genus_additivity}
    
\end{lem}

\begin{prop}[{\cite[Theorem 4.4.2]{graphs_on_surfaces}}]
    \label{4.4.2}
    Let $H$ be a connected graph and $H_1, ..., H_p$ ($p \geq 1$) be its 2-connected blocks. Then, \[ g(H) = g(H_1) + ... + g(H_p)\]
\end{prop}

\begin{defi}[Almost disjoint cycles]
    Let $k \in \mathbb{N}$. Let $C_0,..., C_k$ be cycles of a graph $H$. We say that $C_0, ..., C_k$ are \textit{almost disjoint cycles} if each cycle $C_i$ ($0 \leq i \leq k$) has at most one vertex in common with the subgraph $\cup_{j \neq i} C_j$.

    See Figure \ref{fig:almost_disjoint_cycles} for an example of almost disjoint cycles.
\end{defi}

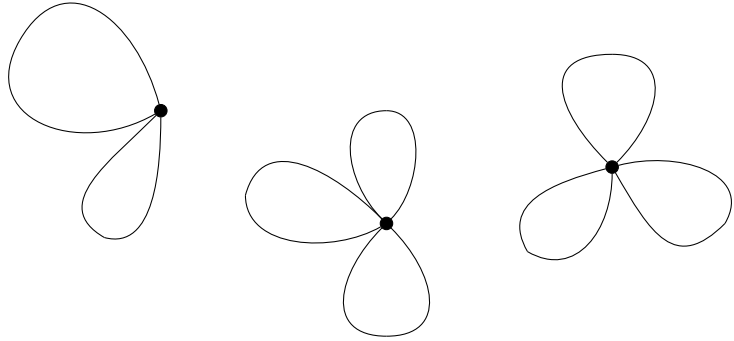
\begin{figure}[h!]
    \centering
    \input{images/almost_disjoint_cycles.tikz}
    \caption{Almost disjoint cycles. The almost disjoint cycles are depicted in solid black lines. The black dots are vertices shared by several almost disjoint cycles.}
    \label{fig:almost_disjoint_cycles}
\end{figure}

\begin{prop}[Variant of Proposition \ref{homotopic_cycles}]
    \label{homotopic_cycles_variant1}
    Let $H$ be a $\Pi_H$-embedded connected graph. If $C_1, ..., C_k$ are cycles of $H$ that are almost disjoint, $\Pi_H$-noncontractible and such that no two of them are $\Pi_H$-homotopic, then

    \[ k \leq
    \left\{
    \begin{array}{lr}
        g(\Pi_H) &  \text{ if } g(\Pi_H) \leq 1\\
        3 g(\Pi_H) - 3 & \text{ if } g(\Pi_H) \geq 2 \\ 
    \end{array}\right.\]
\end{prop}

\begin{proof}
    If $g(\Pi_H) = 0$, then $H$ is embedded in the sphere, and there are indeed no $\Pi_H$-noncontractible cycles.

    If $g(\Pi_H) = 1$, then $H$ is embedded in the projective plane. If there is a $\Pi_H$-noncontractible cycle $C$, let's cut along it, getting a new embedding $\Pi_H'$. Then, by Proposition \ref{4.2.4}, $\Pi_H'$ is an embedding in the plane. Hence, every $\Pi_H$-noncontractible cycle almost disjoint from $C$ is $\Pi_H$-homotopic to $C$ and $k \leq 1$.

    Let's suppose $g(\Pi_H) \geq 2$. For $1 \leq i \leq k$, let $e_i$ be an edge of $C_i$. Let $T$ be a spanning tree of $H$ containing $C_i-e_i$ for $1 \leq i \leq k$ (such a spanning tree exists because $\cup_{1 \leq i \leq k} (C_i-e_i)$ is an acyclic subgraph of $H$). After contracting $E(T)$, the resulting multigraph $H / E(T)$ is embedded in the same surface with embedding $\Tilde{\Pi}_H$. It contains only one vertex, the loops $e_1, ..., e_k$, and some additional loops. Let $H', \Pi_H'$ be the graph and embedding obtained from $H / E(T)$ by removing every loop distinct from $e_1, ..., e_k$ that is contained in two $\Tilde{\Pi}_H$-facial walks. Then, $\Pi_H'$ is an embedding in the same surface.

    Let $W$ be a $\Pi_H'$-facial walk. If $W$ contains an edge $e$ distinct from $e_1, ..., e_k$, then $e$ occurs twice in $W$ (as otherwise it would have been removed previously), and hence the length of $W$ is at least $3$. If $W$ consists of loops from $\{e_1,...,e_k\}$ only, then $W$ is not a single loop $e_i$ since $C_i$ would then be $\Pi_H$-contractible. Similarly, if $W = e_i e_j$, then $C_i$ and $C_j$ are $\Pi_H$-homotopic. Therefore, the length of $W$ is at least $3$. It follows that $2 |E(H')| \geq 3f$ with $f$ the number of $\Pi_H'$-facial walks. Now, by Euler's formula:

    \begin{align*} 
    |V(H')| - |E(H')| + f & =  2 - g(\Pi_H) \\
    |E(H')| - f & = g(\Pi_H) - 1 \\
    3|E(H')| - 3f & = 3g(\Pi_H) - 3 \\
    \end{align*}

    Finally, as $2 |E(H')| \geq 3f$, it follows that
    \[ k \leq |E(H')| \leq  3|E(H')| - 3f = 3 g(\Pi_H) - 3\]

    This concludes the proof.
\end{proof}

\begin{defi}[Cycles on a spanning tree] \label{def:free_cycles}
     Let $H$ be a graph and $a$ a vertex of $H$. Let $k \geq 1$ and let $C_0, ..., C_k$ be subgraphs of $H$.
     We say that $C_0, ..., C_k$ are \textit{cycles on a spanning tree} rooted in $a$ if:
     \begin{itemize}
         \item there exists a spanning tree $T$ of $C_0 \cup ... \cup C_k$ rooted in $a$ such that $(C_0 \cup ... \cup C_k) - T$ consists of $k+1$ edges $\{e_0, ..., e_k\}$ and, for $0 \leq i \leq k$, $e_i$ belongs solely to $C_i$;
         \item let $C'_i$ be the unique cycle induced by $T$ and $e_i$, then $C_i$ is the subgraph consisting of $C'_i$ together with the unique path in $T$ from $a$ to $C'_i$ (if $C'_i$ contains $a$ then $C_i = C'_i$).
     \end{itemize}

     Figure \ref{fig:free_cycles} shows an example of cycles on a spanning tree.
\end{defi}

\begin{figure}[h!]
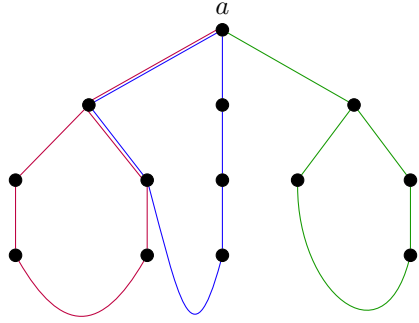

    \centering
    \ctikzfig{images/free_cycles}
    \caption{Cycles on a spanning tree rooted in $a$. The three colored subgraphs are cycles on a spanning tree rooted in $a$.}
    \label{fig:free_cycles}
\end{figure}

\begin{prop}[Variant of Proposition \ref{homotopic_cycles}]
    \label{homotopic_cycles_variant2}
    Let $H$ be a $\Pi_H$-embedded graph and $a$ a vertex of $H$. Let $C_0, ..., C_k$ be free cycles on $a$ that are $\Pi_H$-noncontractible and pairwise $\Pi_H$-nonhomotopic.
    Then,
    \[ k \leq
    \left\{
    \begin{array}{lr}
        g(\Pi_H) &  \text{ if } g(\Pi_H) \leq 1\\
        3 g(\Pi_H) - 3 & \text{ if } g(\Pi_H) \geq 2 \\ 
    \end{array}\right.\]
\end{prop}

\begin{proof}
    Let $T$ be a spanning tree of the free cycles $C_0, ..., C_k$ with the same properties as in Definition \ref{def:free_cycles} and let $a$ be its root.

    Let $H'$ be the multigraph obtained from $H$ by contracting $T$ into the vertex $a$ and let $\Pi_H'$ be the embedding of $H'$ induced by $\Pi_H$. Then, $e_0, ..., e_k$ are loops on $a$ in $H'$. Moreover, remark that, for every $0 \leq i \leq k$, $e_i$ in $\Pi_H'$ is in the same homotopy class as $C_i$ in $\Pi_H$. Hence, by Proposition \ref{homotopic_cycles}, 
    \[ k \leq
    \left\{
    \begin{array}{lr}
        g(\Pi_H) &  \text{ if } g(\Pi_H) \leq 1\\
        3 g(\Pi_H) - 3 & \text{ if } g(\Pi_H) \geq 2 \\ 
    \end{array}\right.\]

    \noindent This concludes the proof.
\end{proof}

\begin{lem}
    \label{2_connected_faces_cycles}
    Let $H$ be a $2$-connected $\Pi_H$-embedded graph. Let $C$ be a $\Pi_H$-contractible cycle. Then, every face in $\text{Int}(C, \Pi_H)$ is a cycle and, therefore, every edge in $\text{int}(C, \Pi_H)$ is contained in two distinct faces.
\end{lem}

\begin{proof}
    Remark that if a face $f$ of $(H, \Pi_H)$ in $\text{Int}(C, \Pi_H)$ is not a cycle, then $H$ must contain a cutvertex. However, this contradicts the fact that $H$ is $2$-connected. Hence, $f$ is a cycle.
    
    Moreover, as every face in $\text{Int}(C, \Pi_H)$ is a cycle, the two sides of an edge $e$ in $\text{int}(C, \Pi_H)$ cannot be in the same face. Therefore, every edge in $\text{int}(C, \Pi_H)$ is contained in two distinct faces.
\end{proof}

\begin{defi}[Same relative orientation]
    Let $H$ be a $\Pi_H$-embedded graph in a surface $S_H$, and let $H'$ be a subgraph of $H$ that is $\Pi_{H'}$-embedded in a surface $S_{H'}$. Let $C$ and $C'$ be two almost disjoint cycles in $H' \subseteq H$ that are $\Pi_H$-contractible, both embedded in some disk or cylinder $S_0$ on $S_H$, and $\Pi_{H'}$-noncontractible homotopic.

    We fix the clockwise orientation, in both $S_0$ and the embedding $\overline{\Pi}_{H'}$ in the sphere obtained from $\Pi_{H'}$ by cutting along $C$ and $C'$, to be the orientation of some traversal of $C$.

    We say that $C$ and $C'$ \textit{have the same relative orientation} in $\Pi_H$ and $\Pi_{H'}$ if the clockwise traversal of $C'$ in $S_0$ is the same as its clockwise traversal in the embedding in $\overline{\Pi}_{H'}$.
\end{defi}

\begin{lem}
    \label{cylinder_reembedding}
    Let $H$ be a $\Pi_H$-embedded graph in surface a $S_H$, and let $H'$ be a subgraph of $H$ that is $\Pi_{H'}$-embedded in a surface $S_{H'}$. Let $C$ and $C'$ be two almost disjoint cycles in $H'$ that are $\Pi_{H'}$-noncontractible homotopic. If
    \begin{itemize}
        \item $V(C \cup C') = V(A \cap B)$ for a separation $(A,B)$ of $H$,
        \item $\text{Int}(C \cup C', \Pi_{H'}) = A \cap H'$,
        \item $B \subseteq H$,
        \item $C \cup C' \cup A$ is a $\Pi_H$-contractible subgraph of $H$ embedded in some disk or cylinder on $S_H$, and
        \item $C$ and $C'$ have the same relative orientation in $\Pi_H$ and in $\Pi_{H'}$,
    \end{itemize}    
    then there exists an embedding $\Pi'_H$ of $H$ in $S_{H'}$ so that $\Pi_H(A)$ and $\Pi'_H(A)$ are equivalent.
\end{lem}

\begin{proof}
    Remark that $\text{int}(C \cup C', \Pi_{H'}(H'-A))$ is empty. Then, as $C$ and $C'$ have the same relative orientation in $\Pi_H$ and in $\Pi_{H'}$, it is possible to embed $A$ inside the cylinder $\text{int}(C \cup C', \Pi_{H'}(H'-A))$ so that the embedding of $A$ obtained is equivalent to $\Pi_H(A)$. 
    
    Then, by combining $\Pi_{H'}(H'-A)$ and this embedding of $A$ inside $\text{int}(C \cup C', \Pi_{H'}(H'-A))$, we get an embedding of $H$ in $S_{H'}$ so that $\Pi_H(A)$ and $\Pi'_H(A)$ are equivalent.
\end{proof}

\subsection{Basic results on the structure of an excluded minor for a surface}

Let's define $G, S, S'$ for the rest of the paper as follows: Let $S, S'$ be surfaces with $S'$ of Euler genus $g$ and $S$ of Euler genus $g+1$ or $g+2$. Let $G$ be a minimal excluded minor for $S'$ and suppose that $G$ can be embedded in $S$ with embedding $\Pi$. Remark that $G$ is of Euler genus either $g+1$ or $g+2$: for any edge $e \in E(G)$, $G-e$ is embeddable in $S'$, and it is easy to show that any extension of an embedding $\Pi_e$ of $G-e$ in $S'$ to an embedding of $G$ is of Euler genus at most $g+2$. Therefore $g < g(G) \leq g+2$.

\subsubsection{Separators in \texorpdfstring{$G$}{G}}

We look at the connectivity of $G$. We show that $G$ does not have a $2$-separation between a $\Pi$-contractible subgraph $B$ and the rest of $G$, unless $B$ is reduced to an edge. Then, we prove that it is enough to consider an excluded minor $G$ for $S'$ that is $2$-connected, as the other cases can be reduced to this one.

\begin{lem}
    \label{2_separated_subgraph_is_edge}
    Let $(A,B)$ be a $2$-separation of $G$ and let $a,b \in V(G)$ so that $A \cap B = \{a,b\}$. Then, either $B$ is an edge or is not contained in a disk in $\Pi$. 
\end{lem}

\begin{proof}
    If $B$ is an edge then there is nothing to prove.
    
    Suppose, by contradiction, that $B$ is not an edge but is contained in a disk in $\Pi$. See Figure \ref{fig:2_separation} for an illustration.

    Then, let $G'$ be the graph obtained by contracting $B - b$ into a single vertex $a'$. 
    As $G'$ is a proper minor of $G$, $G'$ can be embedded on surface $S'$ by embedding $\Pi'$.

    Let $e = a'b \in E(G')$. In $\Pi'(G'-e)$, there exists a disk with $a'$ and $b$ on its boundary ($e$ is embedded in $\Pi'$ inside of it). Then, it is possible to embed $B$ in this disk (because $B$ can be embedded in a disk, as it was in $\Pi$). But, then, this embedding of $B$ together with $\Pi'(G'-e)$ constitutes an embedding of $G$ onto the surface $S'$, a contradiction.
\end{proof}

\begin{figure}[h!]
    \centering
    \input{images/2_separation.tikz}
    \caption{The subgraph $B$ is 2-separated from the rest of the graph by $\{a, b\}$ and is contained in a disk in $\Pi$.}
    \label{fig:2_separation}
\end{figure}
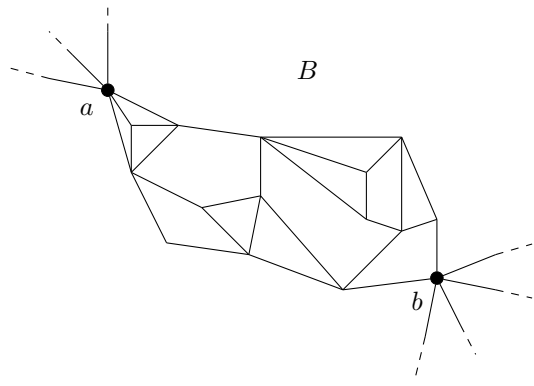

\begin{lem}
    \label{G_blocks_excluded_minors}
    Let $G_1, ..., G_p$ ($p \geq 1$) be the $2$-connected blocks of $G$. Then, for $1 \leq i \leq p$, $G_i$ is an excluded minor for some surface $S_i$.
\end{lem}

\begin{proof}
    Let $1 \leq i \leq p$. It is sufficient to show that for any edge $e \in G_i$, $G_i - e $ and $G_i / e$ is of Euler genus strictly smaller than $g(G_i)$.

    Let $G' = G-e$ and $G_i' = G_i - e$ or $G' = G/e$ and $G_i' = G_i /e$. As $G'$ is a proper minor of $G$, its Euler genus is strictly smaller than $g(G)$. Moreover, $G_1, ..., G_{i-1}, G_{i+1}, ..., G_p$ are still blocks of $G'$ and $G_i'$ is either a block or contains several blocks of $G'$. By Proposition \ref{4.4.2}, 

    \begin{align*}
        g(G) &= g(G_1) + ... + g(G_p) \\
        & > g(G') = g(G_1) + ... + g(G_{i-1}) + g(G_i') + g(G_{i+1}) + ... + g(G_p)
    \end{align*}

    Hence, $g(G_i') < g(G_i)$. This concludes the proof.
\end{proof}

\begin{cor}
    \label{G_2_connected}
    Suppose that, for any graph $H$ that is an excluded minor for some surface $S_H$ and that is $2$-connected, $|V(H)| \leq N(g(S_H))$ with $N$ an increasing function.

    Then, $|V(G)| \leq (g+2) \times N(g)$.
\end{cor}

\begin{proof}
    If $G$ is $2$-connected, this is trivial.

    Suppose not: let $G_1, ..., G_p$ ($p \geq 2$) be the $2$-connected blocks of $G$. Then, by Lemma \ref{G_blocks_excluded_minors}, for $1 \leq i \leq p$, $G_i$ is an excluded minor for some surface $S_i$. 
    Remark then that, using Lemma \ref{4.4.2}, $p \leq g(G_1)+...+g(G_p) = g(G) \leq g+2$. Then, by Lemma \ref{4.4.2},
    \begin{align*}
        |V(G)| &\leq \sum_{i=1}^p |V(G_i)| \\
        & \leq \sum_{i=1}^p N(g(S_i)) \\
        & \leq \sum_{i=1}^p N(g(G_i)-1) \\
        & \leq p N\left(\sum_{i=1}^p (g(G_i)-1)\right) \\
        & \leq p N(g(G)-2) \\
        & \leq (g+2) N(g)
    \end{align*} 
\end{proof}

Corollary \ref{G_2_connected} shows that it is sufficient to prove the main result for excluded minors that are $2$-connected, up to multiplying by a factor linear in $g$. Therefore, we will consider hereafter that $G$ is $2$-connected.

\subsubsection{Structure of \texorpdfstring{$G-e$}{G-e}}

For $e \in E(G)$, we define $G_e = G - e$. As $G_e$ is a proper minor of $G$, it can be embedded in $S'$. We define $\Pi_e$ as an embedding of $G_e$ in $S'$. Let's prove the following result, which compares the embeddings $(G,\Pi)$ and $(G_e, \Pi_e)$.

\begin{lem}
    \label{C_e_non_contractible}
    Let $C$ be a $\Pi$-contractible cycle and let $e \in \text{int}(C,\Pi)$. Let $f,f'$ be the faces in $\Pi$ that contain $e$ and let $C_e = f \cup f' -e$. Then, $C_e$ is a contractible cycle in $(G,\Pi)$ such that $\text{Int}(C_e,\Pi) = C_e + e$. 
    Moreover, $C_e$ is $\Pi_e$-nonseparating (and thus $\Pi_e$-noncontractible).
\end{lem}

\begin{proof}
    As $G$ is $2$-connected and $f$ and $f'$ are faces in $\text{Int}(C, \Pi)$, by Lemma \ref{2_connected_faces_cycles}, $f$ and $f'$ are distinct faces of $(G, \Pi)$ and they are cycles.
    
    To show that $C_e$ is a cycle, it is enough to show that $f$ and $f'$ intersect only in $e$. If $f \cap f'$ is disconnected, there are two vertices $u$ and $v$ so that $u,v \in f \cap f'$ and $u$ and $v$ are not consecutive in $f$ nor $f'$. Then $\{u,v\}$ is a 2-separator in $G$ that separates a $\Pi$-contractible subgraph $H$ of $G$ which contains two paths from $u$ to $v$ (respectively subpaths of the facial walks $f$ and $f'$) from the rest of the graph. However, as $H$ is a $\Pi$-contractible subgraph and is not an edge, this contradicts Lemma \ref{2_separated_subgraph_is_edge}.
    
    We can, therefore, suppose that $f \cap f'$ is connected. Then suppose $f \cap f'$ is a path $P$ of length at least $2$. However, a 2-separator in $G$ separates $P$ from the rest of $G$. Lemma \ref{2_separated_subgraph_is_edge} leads to a contradiction. Finally, $f \cap f'$ is reduced to $e$, and $C_e$ is indeed a cycle.

    As $f$ and $f'$ are $\Pi$-contractible cycles and intersects on $e$, $C_e$ is $\Pi$-contractible and $\text{Int}(C_e,\Pi) = C_e + e$.

    \vspace{0.3cm}
    
    For the second part of the proof, suppose, by contradiction, that $C_e$ is $\Pi_e$-separating.
    
    Let $\mathcal{B}$ be the sets of the nontrivial bridges on $C_e$ in $G_e$. Let $B_C$ be the bridge in $\mathcal{B}$ that contains $C$. Remark that every bridge $B \in \mathcal{B}-B_C$ is so that $B \subseteq \text{Int}(C,\Pi)$. Without loss of generality, $B_C$ is on the left of $C_e$ in $\Pi_e$. Then, let's reembed $G_e$ so that $C_e \cup B_C$ is embedded as in $\Pi_e$ and $\mathcal{B}-B_C$ is embedded as in $\Pi$ on the left of $C_e$ (which is possible because, for each bridge $B \in \mathcal{B} - B_C$, there is a disc in $G_l(C_e, \Pi_e)$ with the attaches of $B$ in their order on $C_e$ on its boundary). This new embedding $\tilde{\Pi}$ of $G_e$ is of Euler genus at most $g$. Moreover, $G_r(C_e,\Tilde{\Pi})$ is empty.
    Then, by embedding $e$ in $G_r(C_e,\tilde{\Pi})$, we obtain an embedding of $G$ in surface $S'$. Contradiction.
\end{proof}

\subsubsection{Bounds for tree decompositions of \texorpdfstring{$G$}{G}} \label{previously_known_results}

Previous studies of the excluded minors for a surface lead to bounds regarding the tree decomposition of such excluded minors. Seymour \cite{seymour} proved a $O(g^3)$ bound on the treewidth of $G$ and on the maximum degree of the tree $T$ in a tree decomposition of $G$ (for a tree decomposition with particular properties). We improve both of these bounds to $O(g^2 \log g)$.

As stated in the introduction, the first and, until now, only known bound for the treewidth of $G$ is the following:

\begin{reptheo}{seymour_treewidth}[{\cite[(3.3)]{seymour}}]
    The treewidth of $G$ is bounded by a polynomial in $g$:
    \[ tw(G) \leq T_{S}(g)\]

    with $T_{S}(g) = 3(g+3)^2(3g+16)-3 = O(g^3)$.
\end{reptheo}

In this paper, we improve this bound from $O(g^3)$ to $O(g \log g)$:

\begin{reptheo}{treewidth}
    The treewidth of $G$ is bounded by the following function of $g$:
    \[ tw(G) \leq T(g)\]

    with $T(g) = 264(g+2)(m+1)-1 = O(g \log g)$, where $m = 2(\lfloor\log_{q}(3g+4)\rfloor + 2)$ and $q = \frac{1153}{1152}$.
\end{reptheo}

The proof of this result can be found at the end of Subsection \ref{subsec:contractible_nested_squares}.

\vspace{0.5cm}

Seymour proved a bound on the maximum degree of the tree in a tree decomposition of $G$ of width $<w$ parametrized by $g$ and $w$, for a tree decomposition with particular properties.

\begin{defi}[Linked tree decomposition]
    We say that a tree decomposition $(T,(V_t)_{t \in T})$ of $H$ is \textit{linked} if, for any $t_1,t_2 \in V(T)$ and any $k \geq 1$, either there are $k$ disjoint paths between $V_{t_1}$ and $V_{t_2}$, or there is $t \in V(T)$ on the path between $t_1$ and $t_2$ in $T$ such that $|V_t| < k$.
\end{defi}

A result from Thomas \cite{thomas} showed that, for any graph $H$ there is a linked tree decomposition of $H$ of width $tw(H)$. Therefore, the following result from Seymour leads to a $O(g^3)$ bound when combined with Theorem \ref{seymour_treewidth}.

\begin{theo}[{\cite[claim (5) in (4.1)]{seymour}}]
    \label{seymour_tree_degree}

    Let $(T,(V_t)_{t \in T})$ be a minimal linked tree decomposition of $G$ of width $<w$. Then, the maximum degree of $T$ is bounded by a polynomial in $g$ and $w$:

    \[\Delta(T) \leq \Delta_T(g,w)\]

    with $\Delta_T(g,w) = 2g + 2w$.
\end{theo}

As a consequence of the improved bound for the treewidth of $G$, we also improve this bound from $O(g^3)$ to $O(g \log g)$:

\begin{cor}
    \label{seymour_tree_degree_corollary}
    Let $(T,(V_t)_{t \in T})$ be a minimal linked tree decomposition of $G$ of width $tw(G)$. Then, the degree of $T$ is bounded by a polynomial in $g$:

    \[ \Delta(T) \leq \Delta_T(g)\]

    with $\Delta_T(g) = \Delta_T(g,T(g)+1) = 2g + 2(T(g)+1) = O(g \log g)$.
\end{cor}

\begin{proof}
    This result is obtained by combining Theorems \ref{treewidth} and \ref{seymour_tree_degree}.
\end{proof}

Another bound for the maximum degree of the tree in a minimal tree decomposition of $G$, polynomial but higher, was also provided by Mohar {\cite[Theorem 4.5]{Mohar_2001}} {\cite[in the proof of (7.2.2)]{graphs_on_surfaces}}.









\section{Structural tools for the main proof} \label{sec:structural_tools}

In this section, we prove three structural results on $G$ (Propositions \ref{isolated_paths}, \ref{good_square} and \ref{good_square_variant}). These results describe forbidden structures in planar embeddings induced by subgraphs of $(G,\Pi)$ and will be used repetitively to find contradictions in the main proof. 

\subsection{Isolated paths}

The main result of this subsection is Proposition \ref{isolated_paths}, which shows that a subgraph of $G$, which is $\Pi$-contractible, contains at most $O(g)$ internally disjoint paths from a piece to another piece.

\begin{defi}[Degree of a piece]
    Recall that a \textit{piece} of $(G,\Pi)$ is either a vertex or a face of $(G,\Pi)$. Let $p$ be a piece of $(G, \Pi)$, we define $d_{(G,\Pi)}(p)$ (or $d(p)$ if it is clear in the context) to be $d_G(p)$ if $p$ is a vertex and $|p|$ if $p$ is a face. 
\end{defi}

\begin{defi}[Isolated paths]
    Let $p$ and $p'$ be two disjoint pieces of $(G,\Pi)$. We say that $G$ contains \textit{$m$ isolated paths from $p$ to $p'$ in $\Pi$} if:

    \begin{itemize}
        \item $G$ contains $m$ internally disjoint paths $P_0, ..., P_{m-1}$ from $p$ to $p'$.
        \item For $0 \leq i \leq m-1$, let $a_i$ and $a'_i$ be respectively the endpoints of $P_i$ on $p$ and $p'$. If $p$ (resp. $p'$) is a face of $(G, \Pi)$, then $a_0, ..., a_{m-1}$ (resp. $a'_0, ..., a'_{m-1}$) are pairwise disjoint and listed in this order on $p$ (resp. $p'$).
        \item For $0 \leq i \leq m-1$, the cycle induced by $P_i \cup P_{i+1} \cup p \cup p'$ is $\Pi$-contractible.
    \end{itemize}

    \vspace{0.3cm}

    We say that the isolated paths are:
    \begin{itemize}
        \item \textit{disjoint} if $p$ and $p'$ are both faces.
        \item \textit{almost disjoint} if $p$ (resp. $p'$) is a vertex and $p'$ (resp. $p$) is a face.
        \item \textit{joint} if $p$ and $p'$ are both vertices.
    \end{itemize}

    An example of the three types of isolated paths can be found in Figure \ref{fig:isolated_paths}.
\end{defi}

\begin{figure}[h!]
    \centering
    \begin{subfloat}[Disjoint isolated paths]{\input{images/isolated_paths_disjoint.tikz}}
    \end{subfloat}
    \hspace{0.5cm}
    \begin{subfloat}[Almost disjoint isolated paths]{\input{images/isolated_paths_semi_disjoint.tikz}}
    \end{subfloat}
    \begin{subfloat}[Joint isolated paths]{\input{images/isolated_paths_joint.tikz}}
    \end{subfloat} 

    \caption{Isolated paths. The solid lines indicate paths, whereas the dotted lines show the boundaries of the faces which the isolated paths use.}
    \label{fig:isolated_paths}
\end{figure}
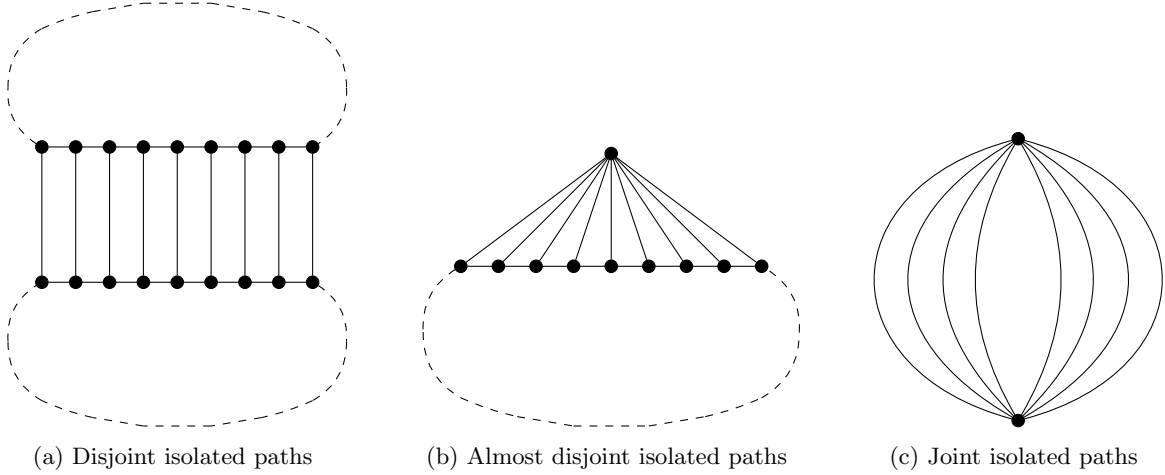

\begin{prop}
    \label{isolated_paths}
    $G$ contains at most $4 \times (6g(\Pi)-5) \leq 4 \times (6g+7)$ isolated paths in $\Pi$ from a piece $p$ to a piece $p'$. 
\end{prop}

\begin{proof}
    Suppose, by contradiction, that there exist in $\Pi$ at least $4 \times (6g+7)+1$ isolated paths from a piece $p$ to a piece $p'$.
    Let $P_0, ..., P_m$ with $m \geq 4 \times (6g+7)$ be these paths.

    For $0 \leq i \leq m-1$, let $C_i$ be the unique cycle induced by $P_i \cup P_{i+1} \cup p \cup p'$ that does not contain $p \cup p'$ in its interior.
    Let 
    \begin{align*}
        \mathcal{C}_1 & = \{ C_0, C_2, ..., C_{2 \times (6g+7) - 2}\} \\
        \mathcal{C}_2 & = \{C_{2 \times  (6g+7) + 1}, C_{2 \times  (6g+7) + 3}, ..., C_{4 \times (6g+7) - 1}\}
    \end{align*}
    
    Let $C_e$ be the unique cycle induced by $P_{2 \times (6g+7) - 1} \cup P_{2 \times (6g+7) +1} \cup p \cup p'$ in $\Pi$ that does not contain $p \cup p'$ in its interior. Then, the interior of $C_e$ is not empty as it contains $P_{2 \times (6g+7)}$. Let $e$ be the first edge of path $P_{2 \times (6g+7)}$.

    Now, let's modify the embedding $\Pi_e$ of $G_e$ into another embedding $\Pi_e'$ of $G_e$ on surface $S'$.
    Let's first show that both $\mathcal{C}_1$ and $\mathcal{C}_2$ contain a $\Pi_e$-contractible cycle:

    \begin{itemize}
        \item Suppose first that the isolated paths are disjoint or almost disjoint. 
        
        Remark then that $\mathcal{C}_1 \cup \mathcal{C}_2$ contains almost disjoint cycles of $G$ and $|\mathcal{C}_1| = |\mathcal{C}_2| = 6g+7$. By Proposition \ref{homotopic_cycles_variant1}, as $|\mathcal{C}_1| \geq 6g+7$ and $g(\Pi) \leq g+2$, either $\mathcal{C}_1$ contains a $\Pi_e$-contractible cycle (in which case, we are done) or, by pigeon hole principle, there are three cycles $\Tilde{C}_1, \Tilde{C}_2, \Tilde{C}_3 \in \mathcal{C}_1$ that are $\Pi_e$-noncontractible homotopic. Suppose that we are in the second case. There exist $0 \leq i < j < k \leq 6g+6$ such that $\Tilde{C}_1 = C_{2i}$, $\Tilde{C}_2 = C_{2j}$ and $\Tilde{C}_3 = C_{2k}$. Remark that at least two of the cycles $\Tilde{C}_1,  \Tilde{C}_2, \Tilde{C}_3$ have the same relative orientation in $\Pi$ and in $\Pi_e$, suppose without loss of generality that it is $\Tilde{C}_1$ and $\Tilde{C}_2$.
        
        By Lemma \ref{cylinder_reembedding}, it is then possible to reembed $\text{int}(C_{2i} \cup C_{2j}, \Pi_e)$ so that its new embedding is equivalent to $\Pi(\text{int}(C_{2i} \cup C_{2j}, \Pi_e))$. Moreover, it contains at least a cycle (for instance, $C_{2i+1}$) that becomes $\Pi_e$-contractible with the reembedding. The same reasoning applies to $\mathcal{C}_2$.

        \item Suppose now that the isolated paths are joint.
        Let 
        \begin{align*}
            \mathcal{P}_1 & = \{ P_0, P_2, ..., P_{2 \times (6g+7) - 2}\} \\
            \mathcal{P}_2 & = \{P_{2 \times  (6g+7) + 1}, P_{2 \times  (6g+7) + 3}, ..., P_{4 \times (6g+7)-1}\}
        \end{align*}
        
        Then, $\mathcal{P}_1 \cup \mathcal{P}_2$ contains internally disjoint paths from $p$ to $p'$ and $|\mathcal{P}_1| = |\mathcal{P}_2| = 6g+7$. By Proposition \ref{homotopic_cycles}, as $|\mathcal{P}_1| \geq 6g+7$ and $g(\Pi) \leq g+2$, there exist two paths $\Tilde{P}_1, \Tilde{P}_2 \in \mathcal{P}_1$ that are $\Pi_e$-homotopic. Then there exist $0 \leq i < j \leq 3g+3$ such that $\Tilde{P}_1 = P_{2i}$ and $\Tilde{P}_2 = P_{2j}$ are $\Pi_e$-homotopic. However, it is then possible to reembed $\text{int}(P_{2i}, P_{2j},\Pi_e)$ so that it is contained in a disk and it contains at least a cycle (for instance $C_{2i}$). The same reasoning applies to $\mathcal{P}_2$.
    \end{itemize}

    Finally, we know that there are two cycles $\Tilde{C}_1 = C_{2i} \in \mathcal{C}_1$ and $\Tilde{C}_2 = C_{2 \times (6g+7) + 1 + 2j} \in \mathcal{C}_2$ (for some $0 \leq i,j \leq 3g+3$) that are $\Pi_e$-contractible. Let $C$ be the cycle induced by $P_{2i+1} \cup P_{2 \times (6g+7) + 1 + 2j} \cup p \cup p'$ that does not contain $p \cup p'$ in its interior.

    First, let's reembed the bridges in $\mathcal{B}_1 = \text{int}(\Tilde{C}_1,\Pi)$ and $\mathcal{B}_2 = \text{int}(\Tilde{C}_2,\Pi)$ on the interior of the disks delimited by $\Tilde{C}_1$ and $\Tilde{C}_2$ in $\Pi_e$: the bridges respectively in $\text{ext}(\Tilde{C}_1,\Pi)$ and $\text{ext}(\Tilde{C}_2,\Pi)$ could not have been embedded in them, hence, in $\Pi_e(G_e - (\mathcal{B}_1 \cup \mathcal{B}_2))$, these disks are empty. It is then possible to embed $\mathcal{B}_1$ and $\mathcal{B}_2$ in their respective disks (as in $\Pi$, for example).

    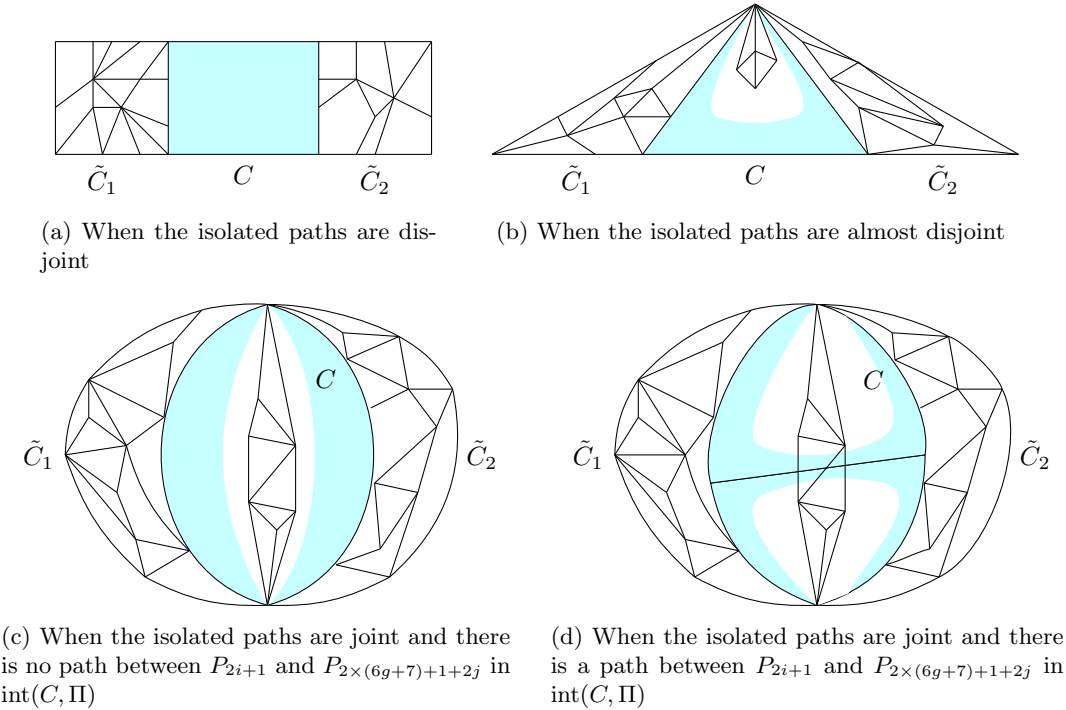
\begin{figure}[h!]
    \centering
    \begin{subfloat}[When the isolated paths are disjoint]{\input{images/isolated_paths_3.tikz}}
    \end{subfloat}
    \hspace{0.3cm}
    \begin{subfloat}[When the isolated paths are almost disjoint]{\input{images/isolated_paths_4.tikz}}
    \end{subfloat}
    \vspace{0.3cm}

    \begin{subfloat}[When the isolated paths are joint and there is no path between $P_{2i+1}$ and $P_{2 \times (6g+7) + 1 +2j}$ in $\text{int}(C, \Pi)$]{\input{images/isolated_paths_5.tikz}}
    \end{subfloat}
    \hspace{0.3cm}
    \begin{subfloat}[When the isolated paths are joint and there is a path between $P_{2i+1}$ and $P_{2 \times (6g+7) + 1 +2j}$ in $\text{int}(C, \Pi)$]{\input{images/isolated_paths_6.tikz}}
    \end{subfloat}
    \vspace{0.3cm}
    

    \caption{Illustration for the second part of the proof of Proposition \ref{isolated_paths}. The disk in which $\text{int}(C, \Pi)$ can be reembedded is depicted in blue in every one of the cases.}
    \label{fig:isolated_paths_2}
\end{figure}

    Let's now try to reembed $\text{int}(C, \Pi)$ in a disk in $\Pi_e(G_e - \text{int}(C,\Pi))$, the different cases are depicted in Figure \ref{fig:isolated_paths_2}.
    
    \begin{itemize}
        \item Suppose first that the isolated paths are disjoint or almost disjoint. 
        
        Then, it is possible to find in embedding $\Pi_e(G_e - \text{int}(C,\Pi))$ a disk with $C$ on its boundary. As $\text{int}(C,\Pi)$ is contained in a disk in $S$, then we can create a new embedding $\Pi_e'$ on surface $S'$ with $G_e - \text{int}(C,\Pi)$ embedded as in $\Pi_e$ and $\text{int}(C,\Pi)$ embedded in this disk as in $\Pi$, a contradiction.
        
        \item Suppose now that the isolated paths are joint.

        First, modify the embedding $\Pi_e$ so that the edges on $P_{2i+1} \cup P_{2 \times (6g+7) + 1 +2j} \cup p \cup p'$ all have positive signature, which is possible because $\Tilde{C}_1$ and $\Tilde{C}_2$ are $\Pi_e$-contractible.
        
        Now, if there is no path from the interior of $P_{2i+1}$ to the interior of $P_{2 \times (6g+7) + 1 +2j}$ in $\text{int}(C, \Pi)$ then the bridges on $C$ in $\text{Int}(C,\Pi)$ can be partitioned into two sets $\mathcal{B}_1$ and $\mathcal{B}_2$ having attaches respectively only onto $P_{2i+1}$ and only onto $P_{2 \times (6g+7) + 1 +2j}$. Let $H_1 = \cup_{B \in \mathcal{B}_1} B$ and $H_2 = \cup_{B \in \mathcal{B}_2} B$. It is possible to find in embedding $\Pi_e(G_e - \text{int}(C,\Pi))$ disjoint disks with respectively $P_{2i+1}$ and $P_{2 \times (6g+7) + 1 +2j}$ on their boundaries. Then, we can create a new embedding $\Pi_e'$ on surface $S'$ with $G_e - \text{int}(C,\Pi)$ embedded as in $\Pi_e$ and $H_1$ and $H_2$ embedded in these disks as in $\Pi$, a contradiction.
        
        Suppose that there is a path $\Tilde{P}$ from the interior of $P_{2i+1}$ to the interior of $P_{2 \times (6g+7) + 1 +2j}$ in $\text{int}(C, \Pi)$. If $\Tilde{P}$ has a negative signature, modify the embedding $\Pi_e$ so that all the edges on $H = \Tilde{P} \cup P_{2i+1} \cup (P_{2 \times (6g+7) + 1 +2j} - \{p,p'\})$ have a positive signature, which is possible because $H$ contains no cycle.
        
        Then, the tree $T = P_{2i+1} \cup \Tilde{P} \cup (P_{2 \times (6g+7) + 1 +2j} - \{p,p'\})$ is contained in $\Pi_e$ in a disk that has $C$ on its boundary and that is empty in $\Pi_e(G_e - \text{int}(C,\Pi))$. Then, we obtain a contradiction in the same way as above.
    \end{itemize} 
\end{proof}

\subsection{Nested squares} \label{subsec:nested_squares}

The main results of this subsection are Propositions \ref{good_square} and \ref{good_square_variant}. They are improvements over a result of Seymour \cite[(2.2)]{seymour} and are central to obtain the main result of this paper.

\begin{prop}[{\cite[(2.2)]{seymour}}]
    \label{seymour_nested_cycles}
    Let $e \in V(G)$ and suppose that there are $g+2$ disjoint cycles of $G-e$, all bounding disks in $S'$ which include $e$. If $G - e$ can be drawn in $S$, then $G$ can be drawn in $S$ as well.
\end{prop}

We improve this bound from linear to logarithmic in $g$ and extend this result not only to cycles that are $\Pi$-contractible but also to $\Pi$-noncontractible homotopic ones.

The proofs of Lemmas \ref{bad_square}, \ref{good_square}, \ref{bad_square_variant} and \ref{good_square_variant} bear similarity with the proofs of the following results from \cite{Mohar_1995} and \cite{Seymour_Thomas}. Proposition \ref{good_square} is related to Proposition \ref{KS_result}, which is a result that was obtained in \cite{KS19}.





\vspace{0.5cm}

Let $C$ be a $\Pi$-contractible cycle of $G$. Then, let $\mathcal{F}(C, \Pi)$ be the set of all faces inside of $C$ in $\Pi$. Moreover, for a subgraph $G_0$ of $G$ that is $\Pi$-contractible, let $\mathcal{F}(G_0, \Pi) = \bigcup_{C \subseteq G_0} \mathcal{F}(C, \Pi)$.

\subsubsection{Contractible nested squares} \label{subsec:contractible_nested_squares}

\begin{defi}[Nested cycles]
    Let $H$ be a graph with embedding $\Pi_H$. Let $C$ and $C'$ be two $\Pi_H$-contractible cycles of $H$ such that $C' \subseteq \text{Int}(C, \Pi_H)$. We say that $C'$ is \textit{$\Pi_H$-nested} (or \textit{nested} if it is clear in the context) in $C$.
    
    Let $k \geq 1$ and let $C_0, ..., C_k$ be $\Pi_H$-contractible cycles of $H$.
    We say that $C_0, ..., C_k$ are \textit{$\Pi_H$-nested} if for $0 \leq i \leq k-1$, $C_i$ is $\Pi_H$-nested in $C_{i+1}$.

    For instance, in Figure \ref{fig:contractible_square}, the cycles $C, C', C''$ are nested in the reversed order in all three figures.
\end{defi}

For two $\Pi$-contractible nested cycles $C$ and $C'$, we denote by $\text{Int}(C \cup C', \Pi)$ the subgraph of $G$ inside the cylinder bounded by $C$ and $C'$ in $\Pi$, together with $C$ and $C'$.

\begin{defi}[Well nested]
    Let $H$ be a graph with embedding $\Pi_H$. Let $C$ and $C'$ be two $\Pi_H$-contractible cycles of $H$ such that $C'$ is $\Pi_H$-nested in $C$. We say that $C'$ is \textit{$\Pi_H$-well-nested} (or \textit{well nested}, if it is clear in the context) in $C$ if one of the following condition holds:
    \begin{itemize}
        \item $C$ and $C'$ are disjoint (\textit{fully well nested}), or 
        \item $C$ and $C'$ intersect at a vertex $v$ (\textit{well nested pinched on $v$}), or 
        \item $C$ and $C'$ both intersect the same face $f$ of $(G, \Pi)$ and if $C$ intersects $f$ in a subpath $P$ (of length at least $3$) then $C'$ intersect $f$ in a subpath that is in the interior of $P$ (\textit{well nested pinched on $f$}).
    \end{itemize}

    In both of the last two cases, we say that $C$ and $C'$ are \textit{well nested pinched on a piece $p$}.

    Let $k \geq 1$ and let $C_0, ..., C_k$ be $\Pi_H$-contractible nested cycles of $H$. We say that $C_0, ..., C_k$ are \textit{$\Pi_H$-well-nested} if:
    \begin{itemize}
        \item for every $0 \leq i \leq k-1$, $C_i$ is $\Pi_H$-fully-well-nested in $C_{i+1}$, or 
        \item for every $0 \leq i \leq k-1$, $C_i$ is $\Pi_H$-well-nested in $C_{i+1}$ pinched on a piece.
    \end{itemize}

    The three figures of Figure \ref{fig:contractible_square} show cycles $C, C', C''$ that are respectively fully-well-nested, well-nested pinched on the vertex, and well-nested pinched on a face, in the reversed order.
\end{defi}

\begin{defi}[Closest]
    Let $H$ be a graph with embedding $\Pi_H$. Let $C$ and $C'$ be two $\Pi_H$-contractible cycles with $C'$ nested in $C$ of $H$. Let $\mathcal{P}$ be a property on the cycles of $(H,\Pi_H)$. 
    
    Then, we say that $C'$ is the cycle \textit{closest} to $C$ that has property $\mathcal{P}$ if, for every cycle $C''$ that is $\Pi_H$-contractible nested in $C$ and has property $\mathcal{P}$, we have $C' \subseteq \text{Int}(C \cup C'', \Pi_H)$.
\end{defi}

\begin{defi}[Cycles in this order]
    Let $H$ be a graph with embedding $\Pi_H$. 
    
    Let $k \geq 1$ and let $C_0, ..., C_k$ be $\Pi_H$-noncontractible homotopic cycles of $H$.
    We say that $C_0, ..., C_k$ are \textit{in this order} in $\Pi_H$ if for $0 \leq i \leq k-1$ and $0 \leq j \leq k$, $C_j \cap \text{int}(C_i \cup C_{i+1}, \Pi_H) = \emptyset$.

    For instance, in Figure \ref{fig:at_most_2_homotopic_cycles}(b), the cycles $C_1, C_2, C_3$ are homotopic in this order.
\end{defi}

\begin{defi}[Contractible square]
    \label{def:contractible_square}
    Let $C, C', C''$ be $\Pi$-contractible cycles in $G$, well nested in the reverse order. Let $\mathcal{B}(C)$ be the set of faces in $\mathcal{F}(C \cup C', \Pi)$ that touch $C$ (we call this set the boundary of $C$). 
    
    We say that $(C, C', C'')$ is a \textit{contractible square} with respect to $C$ if $C'$ is the cycle closest to $C$ so that $\mathcal{B}(C) \subseteq \text{Int}(C \cup C',\Pi)$.

    We say that a contractible square $(C, C', C'')$ is \textit{full} (resp. \textit{pinched on $p$}) if $C, C', C''$ are $\Pi$-fully-well-nested (resp. $\Pi$-well-nested pinched on $p$).

    The three figures of Figure \ref{fig:contractible_square} depict contractible squares $(C, C', C'')$ that are respectively full, pinched on a vertex, and pinched on a face.
\end{defi}

\begin{figure}[h!]
    \centering
    \begin{subfloat}[Full contractible square]{\input{images/contractible_square.tikz}}
    \end{subfloat}
    \hspace{0.5cm}
    \begin{subfloat}[Contractible square pinched on a vertex]{\input{images/contractible_square_pinched_on_vertex.tikz}}
    \end{subfloat}
    
    \vspace{0.5cm}
    \begin{subfloat}[Contractible square pinched on a face]{\input{images/contractible_square_pinched_on_face.tikz}}
    \end{subfloat}
    
    \caption{Contractible squares. The solid lines indicate paths, whereas the dotted line shows the face's boundary on which the contractible square in Figure (c) is pinched. The zone represented in blue is $\mathcal{B}(C)$ (defined, e.g., in Definition \ref{def:contractible_square}), and the zone represented in pink is $\mathcal{I}(C)$ (defined in Definition \ref{def:good_bad_contractible_square}).}
    \label{fig:contractible_square}
\end{figure}


\begin{defi}[Good/bad contractible square]
    \label{def:good_bad_contractible_square}
    Let $(C, C', C'')$ be a square with respect to $C$, let $\mathcal{I}(C)$ be the set of $\Pi$-faces in $\text{Int}(C'', \Pi)$ and let $\mathcal{I}_{\rm N}(C)$ be a maximal size subset of $\Pi$-faces of $\mathcal{I}(C)$ that are almost-disjoint and that are not $\Pi_e$-faces. Let $\mathcal{B}(C)$ be the set of $\Pi$-faces in $\text{Int}(C \cup C', \Pi)$ that touch $C$.
    We say that $(C, C', C'')$ is a \textit{bad square} if there exists an edge $e \in \text{int}(C'', \Pi)$ so that $\mathcal{I}_{\rm N}(C)$ contains more than $4 \times (6 |\mathcal{B}_{\rm N}(C)| - 3)$ $\Pi$-faces with $\mathcal{B}_{\rm N}(C)$ being the set of $\Pi$-faces in $\mathcal{B}(C)$ that are not $\Pi_e$-faces. 
    
    Otherwise, we say that this square is \textit{good}.
\end{defi}

\begin{defi}[Edge-sharing components]
    Let $H$ be a nonplanar graph and let $\Pi_H$ an embedding of $H$ in a surface. Let $\mathcal{F}$ be a set of faces of $(H, \Pi_H)$. We say that two subgraphs of $H$ are edge-sharing if they share at least one edge. Let the edge-sharing components of $\mathcal{F}$ be its edge-sharing equivalence classes. 

    Let $\mathcal{C}$ be an edge-sharing component of $\mathcal{F}$. Suppose it induces a $\Pi_H$-contractible subgraph $H'$ of $H$. Then, remark that, by Proposition \ref{Whitney_planar_graph_flipping}, the $\Pi_H$-embedding of $H'$ is the only planar embedding of $H'$ up to equivalence. In particular, if $H'$ is $\Pi_H'$-contractible for some other embedding $\Pi_H'$ of $H$ in a surface, the boundary of $H'$ in $\Pi_H'$ is the same as in $\Pi_H$.
\end{defi}

\begin{lem}
    \label{bad_square}
    $G$ contains no bad contractible square.
\end{lem}

\begin{proof}
    Suppose by contradiction that $G$ contains a bad contractible square $(C, C', C'')$. Let $e \in \text{int}(C'', \Pi)$. Let $\mathcal{I}(C)$ be the set of faces in $\text{Int}(C'', \Pi)$ in $\Pi$ and let $\mathcal{B}(C)$ be the set of faces in $\text{Int}(C \cup C', \Pi)$ that touch $C$. Let $\mathcal{B}_\text{Con}(C)$ and $\mathcal{B}_{\rm N}(C)$ be the set of the faces in $\mathcal{B}(C)$ that are respectively $\Pi_e$-contractible and $\Pi_e$-noncontractible. Moreover, let $\mathcal{I}_{\rm N}(C)$ be a maximal size subset of almost disjoint $\Pi$-faces from $\mathcal{I}(C)$ that are not $\Pi_e$-faces.
    We suppose that $|\mathcal{I}_{\rm N}(C)| \geq 4 \times (6 |\mathcal{B}_{\rm N}(C)| - 3)$.
    
    First, if there exists a bad contractible proper square with $|\mathcal{B}_{\rm N}(C)| = 0$, then the graph $B$ induced by $\mathcal{B}(C)$ is 2-connected and it is, by Proposition \ref{Whitney_planar_graph_flipping}, embedded in $\Pi_e$ in the same way as in $\Pi$. Then, let $\Tilde{G}$ be the graph obtained from $G_e$ by removing $H = \text{int}(C', \Pi)$. $\Tilde{G}$ is of Euler genus at most $g$ (because it is a subgraph of $G_e$). First, we reembed, in $\Pi_e(\Tilde{G})$, the bridges on $B$ in $\text{Int}(C \cup C', \Pi) - B$ which have attaches only on vertices of $B$ as in $\Pi$ (which is possible because $B$ is embedded as in $\Pi$). Then, in $\Pi_e(\Tilde{G})$, there is a disk whose boundary is $C'$. It is thus possible to create an embedding of $G$ in a surface of Euler genus $g$ by embedding $H$ planarly (as in $\Pi$ for example) in the designated disk in $\Pi_e(\Tilde{G})$, a contradiction.
    
    We now suppose that $|\mathcal{B}_{\rm N}(C)| > 0$.

    \begin{claim}
        \label{at_most_2_homotopic_cycles}
        Suppose that embedding $\Pi_e$ was chosen to minimize the size of $\mathcal{I}_{\rm N}(C)$. Let $F_1, ..., F_m \in \mathcal{I}_{\rm N}(C)$ be a maximal size set of almost disjoint $\Pi$-faces that are $\Pi_e$-noncontractible homotopic, then $m \leq 4$.
    \end{claim}

    \begin{proof_claim}
        See Figure \ref{fig:at_most_2_homotopic_cycles} for illustrations for the proof.
    
        Suppose that there are at least $5$ almost disjoint $\Pi$-faces that are $\Pi_e$-noncontractible homotopic. Let $C_1, C_2, C_3, C_4, C_5$ be these faces (that are cycles) in $\mathcal{I}_{\rm N}(C)$. Without loss of generality, we have that $C_1, C_2, C_3, C_4, C_5$ are $\Pi_e$-homotopic in this order. 
        Remark that there are at least two of the cycles $C_1, C_3, C_5$ that have the same relative orientation in $\Pi$ and $\Pi_e$, suppose without loss of generality that it is $C_1$ and $C_3$.
        Let's show that it is possible to modify $\Pi_e$ so that $C_2$ becomes contractible. 

        First, remark that:
        \begin{itemize}
            \item \underline{If the square is full:} no vertex of $C$ is in $\text{Int}(C_1 \cup C_3, \Pi_e)$, otherwise the whole subgraph $\text{Ext}(C,\Pi)$ would be in $\text{Int}(C_1 \cup C_3, \Pi_e)$ which is impossible because $\text{Ext}(C,\Pi)$ cannot not be $\Pi_e$-contractible. Therefore, $\text{Int}(C_1 \cup C_3, \Pi_e) \subseteq \text{Int}(C,\Pi)$.
            \item \underline{If the square is pinched on a piece $p$:} the same reason for $C - p$ shows that no vertex of $C - p$ is in $\text{Int}(C_1 \cup C_3, \Pi_e)$. If $p$ is a face, it implies that $\text{Int}(C_1 \cup C_3, \Pi_e) \subseteq \text{Int}(C,\Pi)$ as there is no edge in $\text{Ext}(C, \Pi)$ with an endpoint on $p$. If $p$ is a vertex, then a bridge from $\text{ext}(C, \Pi)$ in $\text{Int}(C_1 \cup C_3, \Pi_e)$, if it exists, would be attached solely to $v$, which contradicts the fact that $G$ is $2$-connected and therefore has no cutvertex.
        \end{itemize} 

        Therefore, $\text{Int}(C_1 \cup C_3, \Pi_e) \subseteq \text{Int}(C,\Pi)$ and $\Pi(\text{Int}(C_1 \cup C_3, \Pi_e))$ is an embedding in a disk in $S$ in which $C_1$, $C_2$ and $C_3$ are $\Pi$-facial walks. 

        Finally, by Lemma \ref{cylinder_reembedding}, we can modify the embedding $\Pi_e$ so that $C_2$ is now a $\Pi_e$-contractible cycle.

        This concludes the proof of the claim.
    \end{proof_claim}

    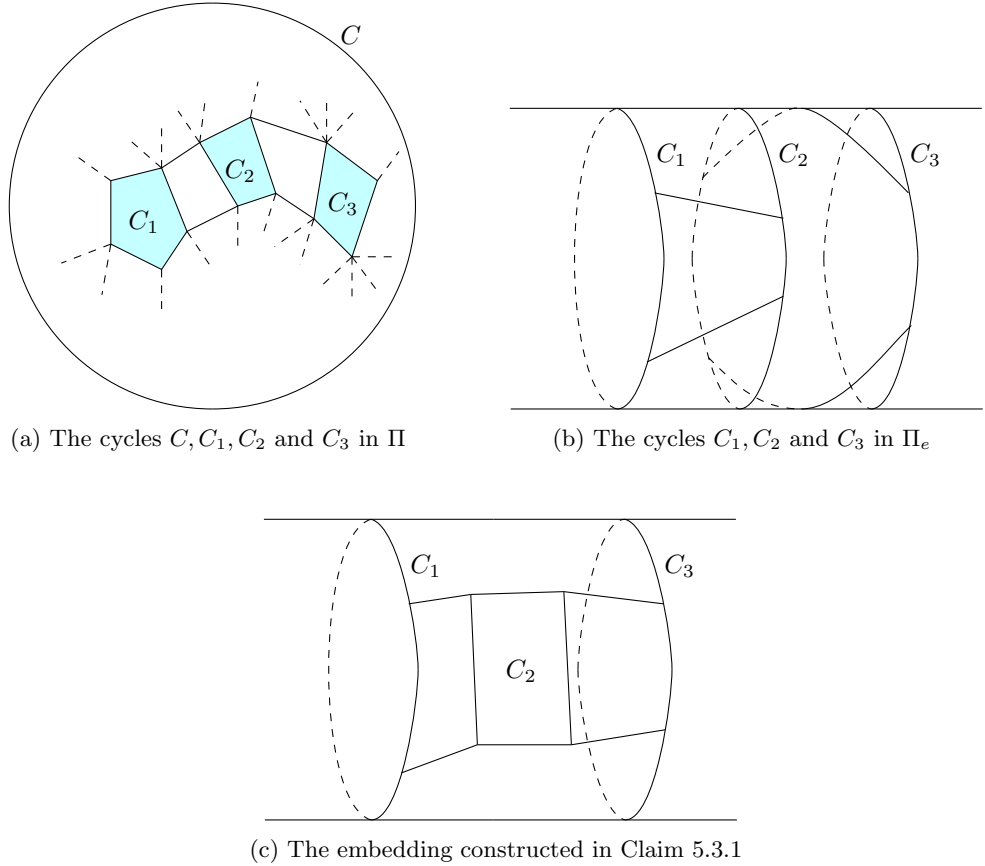
\begin{figure}[h!]
    \centering
    \begin{subfloat}[The cycles $C, C_1, C_2$ and $C_3$ in $\Pi$]{\input{images/contractible_nested_square_claim1_1.tikz}}
    \end{subfloat}
    \hspace{0.75cm}
    \begin{subfloat}[The cycles $C_1, C_2$ and $C_3$ in $\Pi_e$]{\input{images/contractible_nested_square_claim1_2.tikz}}
    \end{subfloat}
    \vspace{0.8cm}
    
    \begin{subfloat}[The embedding constructed in Claim \ref{at_most_2_homotopic_cycles}]{\input{images/contractible_nested_square_claim1_3.tikz}}
    \end{subfloat} 

    \caption{Illustration for the proof of Claim \ref{at_most_2_homotopic_cycles}.}
    \label{fig:at_most_2_homotopic_cycles}
    \end{figure}

    \vspace{0.3cm}
    Remove $H = \text{int}(C', \Pi)$ from $G_e$ to obtain a graph $\Tilde{G}$. 

    \begin{claim}
        The embedding of $\Tilde{G}$ induced by $\Pi_e$ is of Euler genus at most $g - 2|\mathcal{B}_{\rm N}(C)|$.
    \end{claim}

    \begin{proof_claim}
        Remark that $\mathcal{I}_{\rm N}(C) \subseteq \mathcal{I}(C)$ and therefore $\Pi_e(\mathcal{I}(C))$ contains $4 \times (6 |\mathcal{B}_{\rm N}(C)| - 3)$ almost disjoint $\Pi_e$-noncontractible cycles. Moreover, by Claim \ref{at_most_2_homotopic_cycles}, there are at least $6 |\mathcal{B}_{\rm N}(C)| - 3$ of them that are pairwise $\Pi_e$-nonhomotopic. Finally, by Proposition \ref{homotopic_cycles_variant1}, $\Pi_e(\mathcal{I}(C))$ is of Euler genus at least $2|\mathcal{B}_{\rm N}(C)|$.

        \begin{itemize}
            \item \underline{If the square is full:} Remark that $\Tilde{G}$ and $\mathcal{I}(C)$ are disjoint subgraphs of $G$. Hence, $g(\Pi_e(G)) \geq g(\Pi_e(\Tilde{G})) + g(\Pi_e(\mathcal{I}(C)))$.
            
            \item \underline{If the square is pinched on a vertex $v$:} Remark that $\Tilde{G}$ and $\mathcal{I}(C)$ intersect only in $v$. It is easy to verify that $g(\Pi_e(G)) \geq g(\Pi_e(\Tilde{G})) + g(\Pi_e(\mathcal{I}(C)))$.

            \item \underline{If the square is pinched on a subpath P of a face $f$:} Remark that $\Tilde{G}$ and $\mathcal{I}(C)$ intersect only with $P$ and, in each of $\Tilde{G}$, every vertex of $P$ is of degree $2$. It is easy to verify that $g(\Pi_e(G)) \geq g(\Pi_e(\Tilde{G})) + g(\Pi_e(\mathcal{I}(C)))$.
        \end{itemize}
        
        In every case, we get $g(\Pi_e(\Tilde{G})) \leq g(\Pi_e(G)) - g(\Pi_e(\mathcal{I}(C)))$. As $\Pi_e(\mathcal{I}(C))$ is of Euler genus at least $2|\mathcal{B}_{\rm N}(C)|$, we finally conclude that $\Pi_e(\Tilde{G})$ is of Euler genus at most $g - 2|\mathcal{B}_{\rm N}(C)|$. 
    \end{proof_claim}

    \vspace{0.3cm}

    Now, we will construct a new embedding of $G$ from the embedding $\Pi_e(\Tilde{G})$. See Figure \ref{fig:contractible_nested_square_end} for an illustration.

    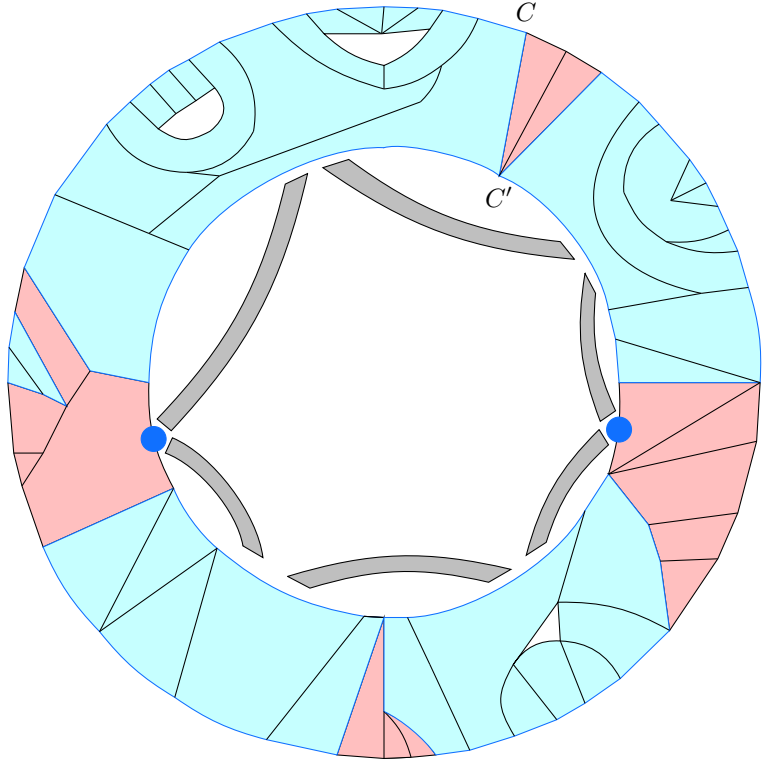
\begin{figure}[h!]
    \centering
    \input{images/contractible_nested_square_end.tikz}

    \caption{Illustration for the end of the proof of Lemma \ref{bad_square}. The vertices in $\tilde{V}(C')$ are depicted in blue. The edge-sharing components of $\mathcal{B}_{\rm Con}(C)$ are in light blue and have a blue outline. The bridges attached to only one edge-sharing component are in the white zones in the blue edge-sharing components. The faces from $\mathcal{B}(C)$ that are $\Pi_e$-noncontractible are depicted in pink. Finally, the added handles are depicted in grey. They are between consecutive edge-sharing components/vertices and can be handles or twisted handles.}
    \label{fig:contractible_nested_square_end}
    \end{figure}

    Let $B = \mathcal{B}_{\rm Con}(C) \cup C'$ and let $\tilde{V}(C') = B - \mathcal{B}_{\rm Con}(C)$. 

    We reembed, in $\Pi_e(\Tilde{G})$, the bridges on the graph induced by $B$ that are attached to only one of its edge-sharing components as in $\Pi$ (which is possible because each edge-sharing component is embedded as in $\Pi$).
    
    The edge-sharing components of $\mathcal{B}_{\rm Con}(C)$ that touches $C'$ and the vertices in $\tilde{V}(C')$ can be ordered with respect to $C'$.
    Then, we add a handle or twisted handle between each consecutive edge-sharing component/vertex $x$ and $y$ of $B$ on $C'$ so that, for an edge-sharing component, the part of its boundary that contains endpoints of edges from $H$ is reachable from the handle and has the expected orientation. 
    
    Remark that we add at most $|\mathcal{B}_{\rm N}(C)|$ handles and twisted handles. It creates a cylinder with $C'$ on one of its boundaries. Then, $H$ can be embedded planarly (for example, as in $\Pi$) in this cylinder. The graph $G$ can then be embedded in a surface of Euler genus $g$, a contradiction.
\end{proof}

\begin{lem}
    \label{almost_disjoint_faces_between_two_cycles}
    Let $C$ and $C'$ be two $\Pi$-contractible nested cycles or two $\Pi$-homotopic cycles of $G$. Let $\mathcal{F}$ be a subset of the faces in $\text{Int}(C \cup C', \Pi)$ that intersect both $C$ and $C'$.

    There exists a subset of faces from $\mathcal{F}$ of size at least $\frac{|\mathcal{F}|}{6}$ that are almost disjoint.

\end{lem}

\begin{proof}
    Let $\mathcal{B}(\mathcal{F})$ be the partition of $\mathcal{F}$ into its edge-sharing components.

    We will then partition each edge-sharing component $B \in \mathcal{B}(\mathcal{F})$ into units $\mathcal{U}(B)$.
    Let $B \in \mathcal{B}(\mathcal{F})$. Remark that the faces in $B$ are naturally (circularly) ordered inside $B$ so that two faces are consecutive in this order if they are edge-sharing. If the faces in $B$ are linearly ordered, take the first face $f_0$ in this order or any face $f_0$ of $B$ if the faces in $B$ are circularly ordered. Let $P_0$ be the path from $C$ to $C'$ in $f_0$ that is the first we go through when following the order of the faces of $B$ from $f_0$ (if the order is circular, then $P_0$ is chosen so that we go through $P_0$ and then the other path from $C$ to $C'$ in $f_0$ consecutively in our traversal). 
    We will define the units of $B$ from $f_0$ using the following method: let $f_1$ be the first face vertex-disjoint from $P_0$ in the order. Then, the first unit of $B$ contains $f_0$ and all the faces between $f_0$ (included) and $f_1$ (excluded) in the order. Iterate this method with $f_1$ to find the next unit and so on until every face of $B$ belongs to a unit of $B$. See Figure \ref{fig:square_lemma_between} for an illustration.

    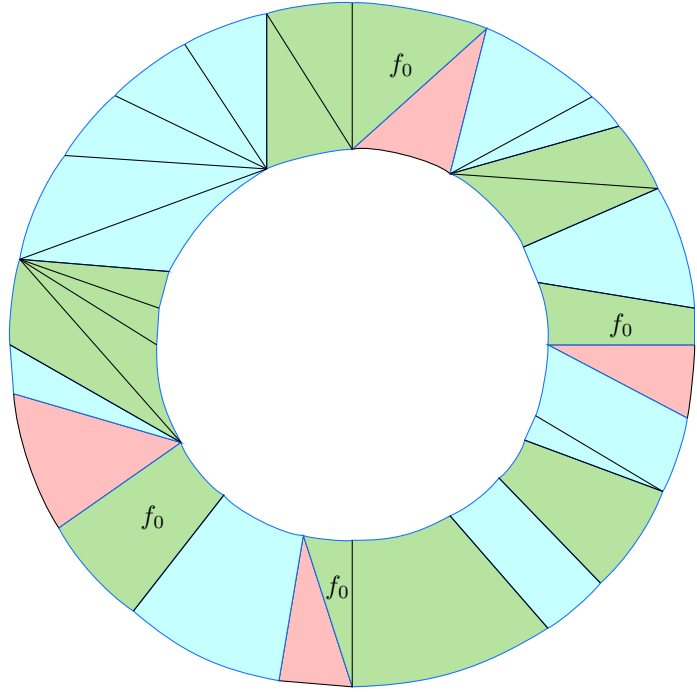
\begin{figure}[h!]
    \centering
    \input{images/square_lemma_between.tikz}

    \caption{Illustration for the edge-sharing components and units in Lemma \ref{almost_disjoint_faces_between_two_cycles}. The subgraphs in pink correspond to faces that are not in $\mathcal{F}$. The edge-sharing components in $\mathcal{B}(\mathcal{F})$ have a dark blue outline. The units of each edge-sharing component are in alternating blue and green. We also indicate the face $f_0$ selected for each edge-sharing component.}
    \label{fig:square_lemma_between}
    \end{figure}

    Remark that the units of an edge-sharing component are naturally ordered with respect to $C$ and $C'$ and that the faces from two units that are not (circularly) consecutive are vertex-disjoint.

    Let $\mathcal{U}(\mathcal{F}) = \bigcup_{B \in \mathcal{B}(\mathcal{F})} \mathcal{U}(B)$.
    Let $\Gamma_G$ be an auxiliary graph whose vertices are the units in $\mathcal{U}(\mathcal{F})$, and there is an edge between two units $U, U' \in \mathcal{U}(\mathcal{F})$ if they are not vertex-disjoint. By construction, $\Gamma_G$ is a subgraph of a Hamiltonian cycle with chords between vertices of distance two on the cycle pairwise nonintersecting. Moreover, this graph is 3-colorable (color the Hamiltonian cycle by alternating the three colors $c_1, c_2, c_3, c_1, c_2, c_3...$).

    Therefore, by taking a color class of $\Gamma_G$ which corresponds to a subset of $\mathcal{F}$ that has size at least $\frac{|\mathcal{F}|}{3}$ (this is possible because the color classes partition $\mathcal{F}$ into three sets), we get a subset of units that are vertex-disjoint and contain at least $\frac{|\mathcal{F}|}{3}$ of the faces in $\mathcal{F}$.

    Moreover, remark that the faces in a unit $U \in \mathcal{U}(\mathcal{F})$ are naturally ordered with respect to $C$ and $C'$ and that by taking every other face in $U$, we get a subset of almost disjoint faces of $U$ of size at least $\frac{|U|}{2}$.

    By combining the two observations, we conclude that there exists a subset of $\mathcal{F}$ of almost disjoint faces of size at least $\frac{|\mathcal{F}|}{6}$.
\end{proof}

\begin{prop}
    \label{good_square}
    Let $q = \frac{1153}{1152}$ and $m = 2(\lfloor\log_{q}(3g+4)\rfloor + 2)$. The graph $G$ contains at most $m$ cycles that are $\Pi$-well-nested. 
\end{prop} 

\begin{proof}
    Let $m' \in \mathbb{N}$. Let $C_1, ..., C_{2m'}$ be $\Pi$-contractible cycles of $G$ that are $\Pi$-well-nested in this order.
    
    By Lemma \ref{bad_square}, each cycle $C_i$ ($4 \leq i \leq 2m'$) induces good contractible proper squares.
    Let $e \in E(C_1)$.
    For $2 \leq i \leq 2m'-2$, let $\mathcal{I}(C_{i+2})$ be the set of faces in $\text{Int}(C_i, \Pi)$ in $\Pi$ and let $\mathcal{B}(C_{i+2})$ be the set of faces in $\text{Int}(C_{i+2} \cup C_{i+1}, \Pi)$ that touch $C_{i+2}$. Let $\mathcal{B}_{\rm N}(C_{i+2})$ be the set of the faces in $\mathcal{B}(C_{i+2})$ that are $\Pi_e$-noncontractible. Moreover, let $\mathcal{I}_{\rm N}(C_{i+2})$ be a maximal size subset of almost disjoint cycles from $\mathcal{I}(C_{i+2})$ that are $\Pi_e$-noncontractible. Then, we define $\Tilde{C}_{i+1}$ to be the cycle closest to $C_{i+2}$ so that $\mathcal{B}(C_{i+2}) \subseteq \text{Int}(C_{i+2} \cup \Tilde{C}_{i+1}, \Pi)$. Finally, by Lemma \ref{bad_square}, $(C_{i+2}, \Tilde{C}_{i+1}, C_i)$ is a good contractible proper square with respect to $C_{i+2}$. 
    
    By Lemma \ref{C_e_non_contractible}, as $\text{int}(C_2, \Pi)$ contains $e$, it is $\Pi_e$-noncontractible.
    Hence, $\mathcal{I}_{\rm N}(C_4)$ is not empty.

    Let's first show that, for every $2 \leq i \leq 2m'$ there exists a subset of almost disjoint faces of $\mathcal{B}_{\rm N}(C_i)$ of size at least $\frac{| \mathcal{B}_{\rm N}(C_i)|}{24}$.

    \setcounter{claim}{0}

    \begin{claim}
        \label{fraction_of_B_N_almost_disjoint}
        For $2 \leq i \leq 2m'$, there exists a subset of almost disjoint faces of $\mathcal{B}_{\rm N}(C_i)$ of size at least $\frac{|\mathcal{B}_{\rm N}(C_i)|}{24}$.
    \end{claim}

    \begin{proof_claim}
        First, let's partition the faces in $\mathcal{B}(C_i)$ into sets $(\mathcal{F}_j)_{j \in \mathbb{N}^*}$. We define the sets $(\mathcal{F}_j)_{j \in \mathbb{N}^*}$ inductively as follows:

        For $j = 1$, $\mathcal{F}_1$ is the set of faces from $\mathcal{B}(C_i)$ that intersect $\Tilde{C}_{i-1}$. Moreover, we define $C_1'$ to be the cycle closest to $C_i$ so that $\mathcal{B}(C_i) - \mathcal{F}_1 \subseteq \text{Int}(C_i \cup C_1', \Pi)$.

        Suppose that for $j \geq 1$, $\mathcal{F}_j$ and $C'_j$ have been defined. Then, we define $\mathcal{F}_{j+1}$ as the set of faces from $\mathcal{B}(C_i) - \cup_{k=1}^j \mathcal{F}_k$ that intersect $C'_j$. Moreover, we define $C_{j+1}'$ to be the cycle closest to $C_i$ so that $\mathcal{B}(C_i) - \bigcup_{k=1}^{j+1} \mathcal{F}_k \subseteq \text{Int}(C_i \cup C_{j+1}', \Pi)$.

        Finally, $(\mathcal{F}_i)_{i \in \mathbb{N}^*}$ is defined. See Figure \ref{fig:good_square_the_F_i} for an illustration. As $\mathcal{B}(C_i)$ is finite, there exists $M \in \mathbb{N}^*$ such that $\mathcal{F}_j = \emptyset$ for all $j > M$.

		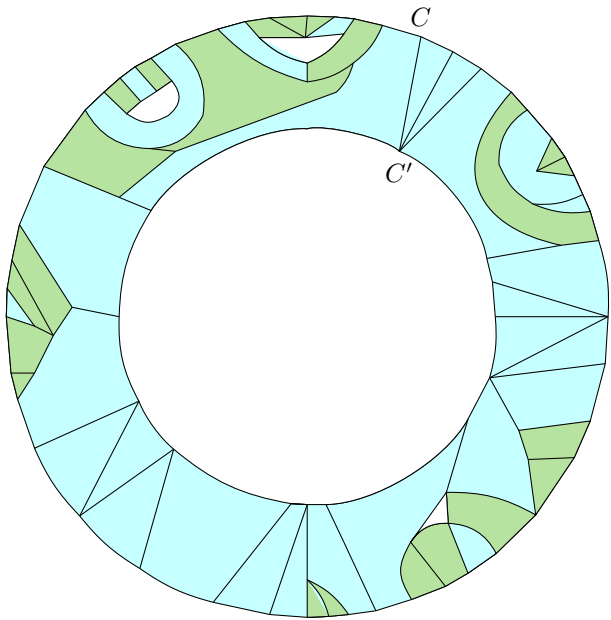
\begin{figure}[h!]
	    \centering
    	\input{images/good_square_the_F_i.tikz}

    	\caption{Illustration for the sets $(\mathcal{F}_i)_{i \in \mathbb{N}^*}$ in Claim \ref{fraction_of_B_N_almost_disjoint}. The subgraphs in white correspond to faces that are not in $\mathcal{B}(C)$ because they don't touch $C$. The sets $\bigcup_{j \geq 1} \mathcal{F}_{2j}$ and $\bigcup_{j \geq 0} \mathcal{F}_{2j+1}$ are depicted respectively in blue and green. Remark that, in this example, for $i > 5$, $\mathcal{F}_i$ is empty.}
    	\label{fig:good_square_the_F_i}
    	\end{figure}

        Remark that by construction, for $j \geq 1$, $\mathcal{F}_j$ and $\mathcal{F}_{j+2}$ contain vertex-disjoint faces. Moreover, remark that, for $j \geq 1$, the faces in $\mathcal{F}_j \cap \mathcal{B}_{\rm N}(C_i)$ touch the two cycles $C_i$ and $C_{j-1}'$ (or $\Tilde{C}_{i-1}$ if $j=1$). Hence, by Lemma \ref{almost_disjoint_faces_between_two_cycles}, there is a subset of faces from $\mathcal{F}_j \cap \mathcal{B}_{\rm N}(C_i)$ that are almost disjoint and this subset has size at least $\frac{|\mathcal{F}_j \cap \mathcal{B}_{\rm N}(C_i)|}{6} \leq \frac{|\mathcal{F}_j \cap \mathcal{B}_{\rm N}(C_i)|}{12}$.

        To conclude, by taking the biggest subset in the partition of $\mathcal{B}_{\rm N}(C_i)$ into $\bigcup_{j \geq 1} \mathcal{F}_{2j} \cap \mathcal{B}_{\rm N}(C_i)$ and $\bigcup_{j \geq 0} \mathcal{F}_{2j+1} \cap \mathcal{B}_{\rm N}(C_i)$ and then applying Lemma \ref{almost_disjoint_faces_between_two_cycles}, we get a subset of almost disjoint faces from $\mathcal{B}_{\rm N}(C_i)$ of size at least $\frac{|\mathcal{B}_{\rm N}(C_i)|}{24}$.
    \end{proof_claim}

    \vspace{0.3cm}

    As $(C_{2i}, \tilde{C}_{2i-1}, C_{2i-2})$ is a good contractible square, for $2 \leq i \leq m'-1$, $\mathcal{B}_{\rm N}(C_{2i})$ has size at least $\frac{|\mathcal{I}_{\rm N}(C_{2i})|+12}{24} \leq \frac{|\mathcal{I}_{\rm N}(C_{2i})|+12}{48}$. 
    Hence, by Claim \ref{fraction_of_B_N_almost_disjoint}, $\mathcal{I}_{\rm N}(C_{2i+2})$ has size at least 
    \begin{align*}
        \mathcal{I}_{\rm N}(C_{2i})| + \frac{|\mathcal{B}_{\rm N}(C_{2i})|}{24}  &\leq |\mathcal{I}_{\rm N}(C_{2i})| + \frac{|\mathcal{I}_{\rm N}(C_{2i})|+12}{1152} \\
        & = \frac{1153|\mathcal{I}_{\rm N}(C_{2i})|+12}{1152} \\
        & \geq \frac{1153}{1152} |\mathcal{I}_{\rm N}(C_{2i})| = q |\mathcal{I}_{\rm N}(C_{2i})| \\
    \end{align*}

    By recurrence, \[ |\mathcal{I}_{\rm N}(C_{2m'})| \geq q^{m'-2} |\mathcal{I}_{\rm N}(C_4)| \geq q^{m'-2}\]

    If $m'-2 \geq \log_{q}(3g+4)$, then $q^{m'-2} \geq 3g+4$. However, by Proposition \ref{homotopic_cycles_variant1}, as $g(\Pi) \leq g+2$, $|\mathcal{I}_{\rm N}(C_{2m'})| \leq 3g+3$. Then, this implies that $m' \leq \lfloor\log_{q}(3g+4)\rfloor + 2$. Finally, there are at most $m = 2m' = 2(\lfloor\log_{q}(3g+4)\rfloor + 2)$ disjoint cycles that are nested in $\Pi$.
\end{proof}

Proposition \ref{good_square} makes it possible to find a better bound on the treewidth of $G$ than the previously best known one from Seymour \cite{seymour}, stated in Theorem \ref{seymour_treewidth}, by combining it with the following result of Kawarabayashi and Sidiropoulos \cite{KS19}:

\begin{prop}
    \label{KS_result}
    Let $H$ be a graph embedded with embedding $\Pi_H$ in a surface of Euler genus $g_H$, and $tw(H) \geq 264g_H\ell$. Then $H$ contains $\ell$ $\Pi_H$-well-nested cycles.
\end{prop}

\begin{reptheo}{treewidth}
    The treewidth of $G$ is bounded by the following function of $g$:
    \[ tw(G) \leq T(g)\]

    with $T(g) = 264(g+2)(m+1)-1 = O(g \log g)$, where $m = 2(\lfloor\log_{q}(3g+4)\rfloor + 2)$ and $q = \frac{1153}{1152}$.
\end{reptheo}

\begin{proof}
     This result is obtained by combining Propositions \ref{KS_result} and \ref{good_square}. 
\end{proof}

Moreover, we can remark that the following results by Thomassen \cite{Thomassen} can be improved by adapting the original proof and then by using Proposition \ref{good_square}, in the case that $G$ is 2-connected:

\begin{defi}[Hexagonal grid]
    We define $J_k$ and $C_{J_k}$ by induction on $k \geq 1$.
    Let $J_1$ be the cycle of length $6$ $C_6$, and let $C_{J_1}$ be this cycle. For each $k \geq 2$, let $J_k$ be the union of $J_{k-1}$ and all those $6$-cycles in the hexagonal grid tiling which intersect $J_{k-1}$ and let $C_{J_k}$ be the cycle that bounds $J_k$ in the hexagonal tiling.

    Figure \ref{fig:hexagonal_grid} represents the hexagonal grid $J_3$.
\end{defi}

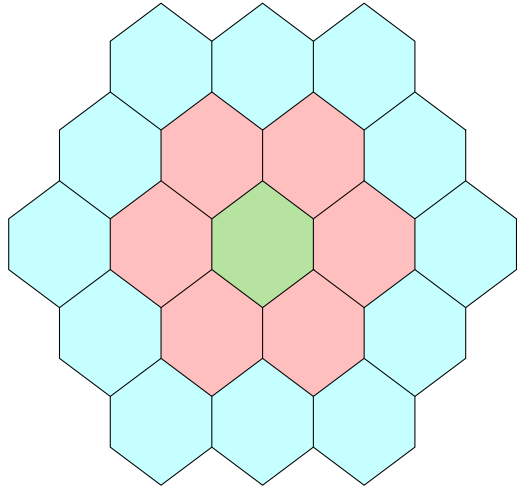
\begin{figure}[h!]
    \centering
    \input{images/hexagonal_grid.tikz}
    \caption{The hexagonal grid $J_3$. The colors show the construction of the grid by induction: $J_1$ is represented in green; then, by adding to $J_1$ the pink part, we obtain $J_2$; finally, by adding to $J_2$ the blue part, we get $J_3$.}
    \label{fig:hexagonal_grid}
\end{figure}

\begin{prop}[\cite{Thomassen}, {\cite[Theorem 7.3.3]{graphs_on_surfaces}}]
    \label{thomassen}
    Let $H$ be a $2$-edge-connected graph such that $g(H) = g$ and $g(H_e) < g$ for every edge $e$ in $H$. Then $H$ contains no subdivision of the grid $J_k$, $k = \lceil 800 g^{3/2} \rceil$.
\end{prop}

\begin{repcor}{thomassen_cor}
    Let $H$ be a $2$-edge-connected graph such that $g(H) = g$ and $g(H_e) < g$ for every edge $e$ in $H$. Then $H$ contains no subdivision of the $4k \times 2k$ grid, $k = \lceil 800 g^{3/2} \rceil$.
\end{repcor}

\begin{proof}
    The result follows from the easily verifiable fact that the $4k \times 2k$ grid contains $J_k$ for all $k \geq 1$.
\end{proof}

Indeed, the proof of Proposition \ref{thomassen} relies on the two following lemmas:

Let $H$ be a subgraph of $G$ such that $H$ is a subdivision of a hexagonal grid. Let $C$ be the outer cycle of $H$. We say that $H$ is \textit{good} in $G$ if the union of $H$ and all the $H$-bridges which have an attach in $H - C$ is planar.

\begin{lem}[\cite{Thomassen}, {\cite[Proposition 7.3.1]{graphs_on_surfaces}}]
    \label{thomassen_lemma1}
    Let $H'$ be a subgraph of $G$ such that $H'$ is a subdivision of the hexagonal grid $J_n$. Let $k \geq 1$, if $n > 100k \sqrt{g}$, then $H'$ contains a subdivision $H$ of the hexagonal grid $J_k$ such that $H$ is good in $G$.
\end{lem}

\begin{lem}[\cite{Thomassen}, {\cite[Proposition 7.3.2]{graphs_on_surfaces}}]
    \label{thomassen_lemma2}
    Let $H$ be a subgraph of $G$ such that $H$ is a subdivision of the hexagonal grid $J_k$ and such that $H$ is good in $G$. If $k \geq 4g+14$ then the $\Pi$-genus of the subdivision of $J_{k-4g-12}$ in $J_k$ is zero. In other words, $\Pi$ induces a planar embedding of $J_{k-4g-12}$.
\end{lem}

By using Lemma \ref{thomassen_lemma1} and Proposition \ref{good_square} (instead of Lemma \ref{thomassen_lemma2}), we get a significant improvement on Proposition \ref{thomassen}:

\begin{reptheo}{thomassen_improvement}
    $G$ contains no subdivision of the grid $J_k$ and no subdivision of the $4k \times 2k$ grid, $k = 100\sqrt{g} \times m = O(\sqrt{g} \log g)$.
\end{reptheo}

\subsubsection{Homotopic nested squares} \label{homotopic_nested_squares}

\begin{defi}[Well-homotopic]
    Let $H$ be a graph with embedding $\Pi_H$. Let $C$ and $C'$ be two $\Pi_H$-noncontractible homotopic cycles of $H$. We say that $C$ is \textit{$\Pi_H$-well-homotopic} (or \textit{well homotopic}, if it is clear in the context) in $C'$ if one of the following condition holds:
    \begin{itemize}
        \item $C$ and $C'$ are disjoint (\textit{fully well homotopic}), or 
        \item $C$ and $C'$ intersect in a vertex $v$ (\textit{well homotopic pinched on $v$}), or 
        \item $C$ and $C'$ both intersect the same face $f$ of $(G, \Pi)$ and if $C$ intersect $f$ in a path $P$ (of length at least three) then $C'$ intersect $f$ in a subpath that is in the interior of $P$ (\textit{well homotopic pinched on $f$}).
    \end{itemize}

    In both of the last two cases, we say that $C$ and $C'$ are \textit{well homotopic pinched on a piece $p$}.

    Let $k \geq 1$ and let $C_0, ..., C_k$ be $\Pi_H$-noncontractible homotopic cycles of $H$ in this order in $\Pi_H$. We say that $C_0, ..., C_k$ are \textit{$\Pi_H$-well-homotopic} if:
    \begin{itemize}
        \item for every $0 \leq i \leq k-1$, $C_i$ is $\Pi_H$-fully-well-homotopic in $C_{i+1}$, or 
        \item for every $0 \leq i \leq k-1$, $C_i$ is $\Pi_H$-well-homotopic in $C_{i+1}$ pinched on a piece.
    \end{itemize}

    The Figure \ref{fig:homotopic_square} show cycles $C^1, C^{1'}, C^{1''}, C^{2''}, C^{2'}, C^2$ that are fully-well-homotopic in this order.
\end{defi}

We extend the definition of closest cycles to non-contractible homotopic cycles:

\begin{defi}[Closest]
    Let $H$ be a graph with embedding $\Pi_H$. Let $C$ and $C'$ be two $\Pi_H$-noncontractible homotopic cycles in this order of $H$. Let $\mathcal{P}$ be a property on the cycles of $(H,\Pi_H)$. 
    
    Then, we say that $C'$ is the cycle \textit{closest} to $C$ that has property $\mathcal{P}$ if, for every cycle $C''$ that is $\Pi_H$-homotopic to $C$ in the order $C, C''$ and has property $\mathcal{P}$, we have $C' \subseteq \text{Int}(C \cup C'', \Pi_H)$.
\end{defi}

Let $C^1$ and $C^2$ be two $\Pi$-well-homotopic cycles in $G$. Then, let $\mathcal{F}(C^1 \cup C^2, \Pi)$ be the set of all faces inside of $\text{Int}(C^1 \cup C^2, \Pi)$ in $\Pi$.

\begin{defi}[Homotopic square]
    \label{def:homotopic_square}
    Let $C^1, C^2, C^{1'}, C^{2'}, C^{1''}, C^{2''}$ be $\Pi$-noncontractible homotopic cycles in $G$, well-homotopic in the following order: $C^1, C^{1'}, C^{1''}, C^{2''}, C^{2'}, C^2$. Let $C^{12} = C^1 \cup C^2$. Let $\mathcal{B}(C^{12})$ be the union of the sets of faces in $\mathcal{F}(C^1 \cup C^{1'}, \Pi)$ that touch $C^1$ and of faces in $\mathcal{F}(C^2 \cup C^{2'}, \Pi)$ that touch $C^2$ (we call this set the boundary of $C$). 
    
    We say that $(C^1, C^2, C^{1'}, C^{2'}, C^{1''}, C^{2''})$ is a \textit{homotopic square} with respect to $C^1$ and $C^2$ if $C^{1'}$ and $C^{2'}$ are so that, for every other cycles $\Tilde{C}^1$ and $\Tilde{C}^2$ with $\mathcal{B}(C^{12}) \subseteq \text{Int}(C^1 \cup \Tilde{C}^1,\Pi) \cup \text{Int}(C^2 \cup \Tilde{C}^2,\Pi)$, we have $C^{1'} \subseteq \text{Int}(C^1 \cup \Tilde{C}^1,\Pi)$ and $C^{2'} \subseteq \text{Int}(C^2 \cup \Tilde{C}^2,\Pi)$. 

    We say that a homotopic square $(C^1, C^2, C^{1'}, C^{2'}, C^{1''}, C^{2''})$ is \textit{full} (resp. \textit{pinched on $p$}) if $C^1, C^2, C^{1'}, C^{2'},$ $C^{1''}, C^{2''}$ are $\Pi$-fully-well-homotopic (resp. $\Pi$-well-homotopic pinched on $p$) in this order in $\Pi$.

    Figure \ref{fig:homotopic_square} depicts the full homotopic square $(C^1, C^2, C^{1'}, C^{2'}, C^{1''}, C^{2''})$.
\end{defi}

\begin{figure}[h!]
    \centering
    \input{images/homotopic_square.tikz}
    \caption{Homotopic square. The zone represented in blue is $\mathcal{B}(C^{12})$ (defined e.g. in Definition \ref{def:homotopic_square}) and the zone represented in pink is $\mathcal{I}(C^{12})$ (defined in Definition \ref{def:good_bad_homotopic_square}) where $C^{12} = C^1 \cup C^2$.}
    \label{fig:homotopic_square}
\end{figure}
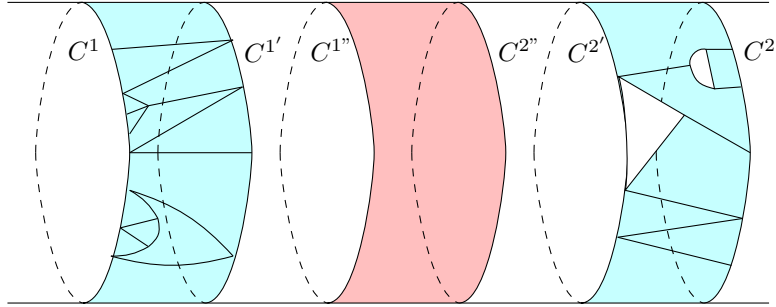

\begin{defi}[Good/bad homotopic square]
    \label{def:good_bad_homotopic_square}
    Let $(C^1, C^2, C^{1'}, C^{2'}, C^{1''}, C^{2''})$ be a homotopic square with respect to $C^1$ and $C^2$. Let $C^{12} = C^1 \cup C^2$. Let $\mathcal{I}(C^{12})$ be the set of $\Pi$-faces in $\text{Int}(C^{1''} \cup C^{2''}, \Pi)$ and let $\mathcal{I}_{\rm N}(C^{12})$ be a maximal size subset of $\Pi$-faces from $\mathcal{I}(C^{12})$ that are almost-disjoint and that are not $\Pi_e$-faces. Let $\mathcal{B}(C^{12})$ be the union of the sets of faces in $\mathcal{F}(C^1 \cup C^{1'}, \Pi)$ that touch $C^1$ and of faces in $\mathcal{F}(C^2 \cup C^{2'}, \Pi)$ that touch $C^2$.
    We say that $(C^1, C^2, C^{1'}, C^{2'}, C^{1''}, C^{2''})$ is a \textit{bad square} if there exists an edge $e \in \text{int}(C^{1''} \cup C^{2''},\Pi)$ so that $\mathcal{I}_{\rm N}(C^{12})$ contains more than $48 |\mathcal{B}_{\rm N}(C^{12})|$ $\Pi$-faces with $\mathcal{B}_{\rm N}(C^{12})$ being the set of the faces in $\mathcal{B}(C^{12})$ that are not $\Pi_e$-faces. 
    
    Otherwise, we say that this square is \textit{good}.
\end{defi}

\begin{lem}
    \label{bad_square_variant}
    $G$ contains no bad homotopic square.
\end{lem}

\begin{proof}
    Suppose by contradiction that $G$ contains a bad homotopic square $(C^1, C^2, C^{1'}, C^{2'}, C^{1''}, C^{2''})$. Let $C^{12} = C^1 \cup C^2$.  Let $e \in \mathcal{I}(C^{12})$. Let $\mathcal{I}(C^{12})$ be the set of faces in $\text{Int}(C^{1''} \cup C^{2''}, \Pi)$ and let $\mathcal{B}(C^{12})$ be the union of the sets of faces in $\mathcal{F}(C^1 \cup C^{1'}, \Pi)$ that touch $C^1$ and of faces in $\mathcal{F}(C^2 \cup C^{2'}, \Pi)$ that touch $C^2$. Let $\mathcal{B}_\text{Con}(C^{12})$ and $\mathcal{B}_{\rm N}(C^{12})$ be the set of the faces in $\mathcal{B}(C^{12})$ that are respectively $\Pi_e$-contractible and $\Pi_e$-noncontractible. Moreover, let $\mathcal{I}_{\rm N}(C^{12})$ be a maximal size subset of $\Pi$-faces of $\mathcal{I}(C^{12})$ that are almost disjoint and that are not $\Pi_e$-faces.
    We suppose $|\mathcal{I}_{\rm N}(C^{12})| \geq 48 |\mathcal{B}_{\rm N}(C^{12})|$.
    
    First, if there exists a bad homotopic proper square with $|\mathcal{B}_{\rm N}(C^{12})| = 0$, then as the connected components of the graph $B$ induced by $\mathcal{B}(C^{12})$ are 2-connected, by Proposition \ref{Whitney_planar_graph_flipping}, they are embedded in $\Pi_e$ in the same way as in $\Pi$. 
    Then, let $\Tilde{G}$ be the graph obtained from $G_e$ by removing $H = \text{Int}(C^{1'} \cup C^{2'}, \Pi)$. $\Tilde{G}$ is of Euler genus at most $g$ (because it is a subgraph of $G_e$). First, we reembed, in $\Pi_e(\Tilde{G})$, the bridges on $B$ in $\left(\text{Int}(C^{1} \cup C^{1'}, \Pi) \cup \text{Int}(C^2 \cup C^{2'}, \Pi)\right) - B$ which have attaches only on vertices of $B$ as in $\Pi$ (which is possible because the connected components of $B$ are embedded as in $\Pi$). Then, in $\Pi_e(\Tilde{G})$, there is an empty surface whose boundary is $C^{1'} \cup C^{2'}$. It is thus possible to create an embedding of $G$ in a surface of Euler genus $g$ by embedding $H$ planarly (as in $\Pi$ for example) in the designated cylinder in $\Pi_e(\Tilde{G})$, a contradiction. 
    
    We now suppose $|\mathcal{B}_{\rm N}(C^{12})| > 0$.

    \setcounter{claim}{0}

    \begin{claim}
        \label{at_most_2_homotopic_cycles_variant}
        Suppose that embedding $\Pi_e$ was chosen so that it minimizes the size of $\mathcal{I}_{\rm N}(C^{12})$. Let $F_1, ..., F_m \in \mathcal{I}_{\rm N}(C^{12})$ be a maximal size set of almost disjoint $\Pi$-faces that are $\Pi_e$-noncontractible homotopic, then $m \leq 8$.
    \end{claim}

    \begin{proof_claim}
        Suppose that there are at least $9$ almost disjoint $\Pi$-faces that are $\Pi_e$-noncontractible homotopic. Let $F_1, ..., F_9$ be the $\Pi$-faces in $\mathcal{I}_{\rm N}(C^{12})$. Contrary to Claim \ref{at_most_2_homotopic_cycles}, these faces may not be cycles (but two cycles that intersect in a vertex or an edge). Remark moreover that two faces that are not cycles share no vertex, as otherwise there would be a cutvertex. Let $C_i$ be a cycle contained in the face $F_i$ ($1 \leq i \leq 5$) (there is at least one).
        
        Without loss of generality, we have that $C_1, ..., C_9$ are $\Pi_e$-homotopic in this order. There are at least two of the cycles $C_1, C_5, C_9$ that have the same relative orientation in $\Pi$ and $\Pi_e$, suppose without loss of generality that it is $C_1$ and $C_5$. Let's show that it is possible to modify $\Pi_e$ so that one of $F_2$, $F_3$ or $F_4$ becomes a $\Pi_e$-face. 

        First, remark that:
        \begin{itemize}
            \item \underline{If the square is full:} no vertex of $C^1$ (resp. $C^2$) is in $\text{Int}(C_1 \cup C_5, \Pi_e)$, as otherwise, the bridges of $\text{Ext}(C^{12},\Pi)$ attached to $C^1$ (resp. $C^2$) would be in $\text{Int}(C_1 \cup C_5, \Pi_e)$ which is impossible because the subgraph induced by these bridges cannot be $\Pi_e$-contractible (as otherwise $C^1$ (resp. $C^2$) would be $\Pi_e$-contractible, a contradiction).
            \item \underline{If the square is pinched in a piece $p$:} the same reasoning on $C^1 - p$ (resp. $C^2 - p$) shows that no vertex of $C-p$ is in $\text{Int}(C_1 \cup C_5, \Pi_e)$. If $p$ is a face, it implies that $\text{Int}(C_1 \cup C_5, \Pi_e) \subseteq \text{Int}(C^{12}, \Pi)$ as there is no edge in $\text{Ext}(C^{12}, \Pi)$ with an endpoint on $p$. If $p$ is a vertex, then a bridge from $\text{ext}(C^{12}, \Pi)$ in $\text{Int}(C_1 \cup C_5, \Pi_e)$, if it exists, would be attached solely to $v$, which contradicts the fact that $G$ is $2$-connected and therefore has no cutvertex.
        \end{itemize}

        Therefore, $\text{Int}(C_1 \cup C_5, \Pi_e) \subseteq \text{Int}(C^{12},\Pi)$ and $\Pi(\text{Int}(C_1 \cup C_5, \Pi_e))$ is an embedding in a disk in $\Pi$ in which $C_1$ and $C_5$ are $\Pi$-facial walks. 

        Let's show that one face $F$ of $F_2, F_3, F_4$ is contained in $\Pi(\text{Int}(C_1 \cup C_5, \Pi_e))$ and therefore $F$ is a $\Pi$-facial walk in this embedding. First, if one face $F$ of $F_2, F_3, F_4$ is a cycle, then by the hypothesis on the order of the cycles $C_i$ ($1 \leq i \leq 9$), this face $F$ is contained in $\Pi(\text{Int}(C_1 \cup C_5, \Pi_e))$. Now, as the faces are not cycles share no vertex, the cycles $C_2, C_3, C_4$ are disjoint. Therefore, the cycle $C_3$ is contained in $\text{int}(C_1 \cup C_5, \Pi_e)$ and, by connectivity, $\text{Int}(C_1 \cup C_5, \Pi_e)$ contains $F_3$. We indeed find that one of $F_2, F_3, F_4$ is contained in $\Pi(\text{Int}(C_1 \cup C_5, \Pi_e))$ and therefore is a $\Pi$-facial walk in this embedding.

        Finally, by Lemma \ref{cylinder_reembedding}, we can modify the embedding $\Pi_e$ so that one of $F_2, F_3, F_4$ is now a $\Pi_e$-face.

        This concludes the proof of the claim.
    \end{proof_claim}

    \vspace{0.3cm}

    Remove $H = \text{Int}(C^{1'} \cup C^{2'}, \Pi)$ from $G_e$ to obtain a graph $\Tilde{G}$.

    \begin{claim}
        The resulting embedding of $\Tilde{G}$ is of Euler genus at most $g - (2|\mathcal{B}_{\rm N}(C^{12})|+1)$.
    \end{claim}

    \begin{proof_claim}
        Remark that $\mathcal{I}_{\rm N}(C^{12}) \subseteq \mathcal{I}(C^{12})$ and therefore $\Pi_e(\mathcal{I}(C^{12}))$ contains at least $24 |\mathcal{B}_{\rm N}(C^{12})|$ almost disjoint $\Pi$-faces that are $\Pi_e$-noncontractible and therefore contains at least $48 |\mathcal{B}_{\rm N}(C^{12})|$ almost disjoint $\Pi_e$-noncontractible cycles. Moreover, by Claim \ref{at_most_2_homotopic_cycles_variant}, there are at least $6 |\mathcal{B}_{\rm N}(C^{12})|$ of them that are pairwise $\Pi_e$-nonhomotopic. Finally, by Proposition \ref{homotopic_cycles_variant1}, $\Pi_e(\mathcal{I}(C^{12}))$ is of Euler genus at least $2|\mathcal{B}_{\rm N}(C^{12})|+1$. 

        \begin{itemize}
            \item \underline{If the square is full:} Remark that $\Tilde{G}$ and $\mathcal{I}(C^{12})$ are disjoint subgraphs of $G$. Hence, $g(\Pi_e(G)) \geq g(\Pi_e(\Tilde{G})) + g(\Pi_e(\mathcal{I}(C^{12})))$.
            
            \item \underline{If the square is pinched on a vertex $v$:} Remark that $\Tilde{G}$ and $\mathcal{I}(C^{12})$ intersect only in $v$. It is easy to verify that $g(\Pi_e(G)) \geq g(\Pi_e(\Tilde{G})) + g(\Pi_e(\mathcal{I}(C^{12})))$.

            \item \underline{If the square is pinched on a subpath $P$ of a face:} Remark that $\Tilde{G}$ and $\mathcal{I}(C^{12})$ intersect only with $P$ and, in $\Tilde{G}$, every vertex of $P$ is of degree $2$. It is easy to verify that $g(\Pi_e(G)) \geq g(\Pi_e(\Tilde{G})) + g(\Pi_e(\mathcal{I}(C^{12})))$.
        \end{itemize}
        
        In every case, we get $g(\Pi_e(\Tilde{G})) \leq g(\Pi_e(G)) - g(\Pi_e(\mathcal{I}(C^{12})))$. As $\Pi_e(\mathcal{I}(C^{12}))$ is of Euler genus at least $2|\mathcal{B}_{\rm N}(C^{12})|$, we finally conclude that $\Pi_e(\Tilde{G})$ is of Euler genus at most $g - (2|\mathcal{B}_{\rm N}(C^{12})|+1)$. 
    \end{proof_claim}

    \vspace{0.3cm}

    Now, we will construct a new embedding of $G$ from the embedding $\Pi_e(\Tilde{G})$. 

    Let $C^{12'} = C^{1'} \cup C^{2'}$, let $B = \mathcal{B}_{\rm Con}(C^{12}) \cup C^{12'}$ and let $\tilde{V}(C^{12'}) = B - \mathcal{B}_{\rm Con}(C^{12})$.

    We reembed, in $\Pi_e(\Tilde{G})$, the bridges on the graph induced by $B$ that are attached to only one edge-sharing component as in $\Pi$ (which is possible because each edge-sharing component is embedded as in $\Pi$).

    Let $\mathcal{B}_{\rm Con}(C^1)$ (resp. $\mathcal{B}_{\rm Con}(C^2)$) be the faces in $\mathcal{B}_(C^{12})$ that are in $\text{Int}(C^1 \cup C^{1'}, \Pi)$ (resp. $\text{Int}(C^2 \cup C^{2'}, \Pi)$) and let $\tilde{V}(C^{1'}) = \tilde{V}(C^{12'}) \cap C_{1'}$, $\tilde{V}(C^{2'}) = \tilde{V}(C^{12'}) \cap C_{2'}$.
    
    The edge-sharing components of $\mathcal{B}_{\rm Con}(C^1)$ and the vertices of $\tilde{V}(C^{1'})$ (resp. $\mathcal{B}_{\rm Con}(C^2)$ and $\tilde{V}(C^{2'})$) can be ordered with respect to $C^{1'}$ (resp. $C^{2'}$). Then, we add a handle or twisted handle between each consecutive edge-sharing component/vertex $x$ and $y$ so that, for an edge-sharing component, the part of its boundary that contains endpoints of edges from $H$ is reachable from the handle and has the expected orientation.

    \begin{itemize}
        \item If the square is pinched on a vertex $v$, we ensure that the handles on $v$ or the edge-sharing components containing $v$ are well ordered.
        \item If the square is full or pinched on subpath $P$ of a face, let $Q$ be a path from one edge-sharing component of $\mathcal{B}_{\rm Con}(C^1)$ or a vertex of $\tilde{V}(C^{1'})$ to one edge-sharing component on $\mathcal{B}_{\rm Con}(C^2)$ or a vertex of $\tilde{V}(C^{2'})$. If this path has positive signature, we are done. Otherwise, we add a twisted handle to get a path of positive signature between the two cycles.
    \end{itemize}
    
    Remark that we add at most $|\mathcal{B}_{\rm N}(C^{12})|$ handles or twisted handles and $1$ twisted handle. Hence, the Euler genus goes up from $g(\Tilde{G}) = g - (2|\mathcal{B}_{\rm N}(C^{12})|+1)$ by at most $2|\mathcal{B}_{\rm N}(C^{12})|+1$. Moreover, we created an empty connected surface with $C^{1'} \cup C^{2'}$ on its boundary, embedded as in $\Pi$ in this surface. Then, $H$ can be embedded planarly (for example, as in $\Pi$) in this surface. $G$ can then be embedded in a surface of genus $g$, a contradiction.
\end{proof}

\begin{lem}
    \label{almost_disjoint_faces_between_two_cycles_variant}
    Let $C_1,C'_1,C'_2,C_2$ be $\Pi$-well-homotopic cycles of $G$ in this order. Let $\mathcal{F}_1$ (resp. $\mathcal{F}_2$) be a subset of the faces in $\text{Int}(C_1 \cup C'_1, \Pi)$ (resp. $\text{Int}(C_2 \cup C'_2, \Pi)$) that intersect both $C_1$ and $C'_1$ (resp. $C_2$ and $C_2'$).

    There exists a subset of faces from $\mathcal{F}_1 \cup \mathcal{F}_2$ of size at least $\frac{|\mathcal{F}_1| + |\mathcal{F}_2|}{12}$ that are almost disjoint.
\end{lem}

\begin{proof}

    \begin{itemize}
        \item \underline{If $C_1' \cap C_2' = \emptyset$:} Remark that $\mathcal{F}_1$ and $\mathcal{F}_2$ are disjoint, then, using Lemma \ref{almost_disjoint_faces_between_two_cycles} for both sets, we find that there exists a subset of faces from $\mathcal{F}_1 \cup \mathcal{F}_2$ of size at least $\frac{|\mathcal{F}_1| + |\mathcal{F}_2|}{6} \leq \frac{|\mathcal{F}_1| + |\mathcal{F}_2|}{12}$ that are almost disjoint.

        \item \underline{If $C_1' \cap C_2'$ intersect in a piece $p$:} Then faces from $\mathcal{F}_1$ and $\mathcal{F}_2$ can only intersect in $p$. Let $\mathcal{U}_1,\Gamma_1$ and $\mathcal{U}_2,\Gamma_2$ be the units and the auxiliary graphs obtained respectively from $\mathcal{F}_1$ and $\mathcal{F}_2$ by following the process in Lemma \ref{almost_disjoint_faces_between_two_cycles}. By Lemma \ref{almost_disjoint_faces_between_two_cycles}, $\Gamma_1$ and $\Gamma_2$ are $3$-colorable. Let $\Gamma$ be the auxiliary graph obtained from $\mathcal{F}_1 \cup \mathcal{F}_2$ following the process in Lemma \ref{almost_disjoint_faces_between_two_cycles}, then, as $\mathcal{U}_1$ and $\mathcal{U}_2$ intersect only in $p$, it is easy to see that $\Gamma$ is $6$-colorable.
        Finally, we conclude, by the same reasoning as Lemma \ref{almost_disjoint_faces_between_two_cycles}, that there exists a subset of faces from $\mathcal{F}_1 \cup \mathcal{F}_2$ of size at least $\frac{|\mathcal{F}_1 \cup \mathcal{F}_2|}{12} = \frac{|\mathcal{F}_1| + |\mathcal{F}_2|}{12}$ that are almost disjoint.
    \end{itemize}
\end{proof}

\begin{prop}
    \label{good_square_variant}
    Let $q = \frac{1153}{1152}$ and $m = 2(\lfloor\log_{q}(3g+4)\rfloor + 2)$. $G$ contains at most $2m$ $\Pi$-well-homotopic cycles. 
\end{prop} 

\begin{proof}
    Let $m' \in \mathbb{N}$. Let $C_1, ..., C_{4m'}$ be $\Pi$-well-homotopic cycles of $G$. Suppose that $C_1, ..., C_{4m'}$ are in this order in $\Pi$.
    
    By Lemma \ref{bad_square_variant}, each pair of cycles $D_i=(C_{2m'-i+1}, C_{2m'+i})$ ($4 \leq i \leq 4m'$) induces a good homotopic proper square.
    Let $e \in int(C_{2m'}\cup C_{2m'+1})$.
    For $2 \leq i \leq 2m'$, let $\mathcal{I}(D_{i+2})$ be the set of faces in $\text{Int}(C_{2m'-i+1} \cup C_{2m'+i}, \Pi)$ and let $\mathcal{B}(D_{i+2})$ be the union of the sets of faces in $\text{Int}((C_{2m'-(i+2)+1} \cup C_{2m'-(i+1)+1}, \Pi)$ that touch $C_{2m'-(i+1)+1}$ and of faces in $\text{Int}((C_{2m'+(i+2)} \cup C_{2m'+(i+1)}, \Pi)$ that touch $C_{2m'+(i+2)}$. Let $\mathcal{B}_{\rm N}(D_{i+2})$ be the set of the faces in $\mathcal{B}(D_{i+2})$ that are $\Pi_e$-noncontractible. Moreover, let $\mathcal{I}_{\rm N}(D_{i+2})$ be a maximal size subset of almost disjoint cycles from $\mathcal{I}(D_{i+2})$ that are $\Pi_e$-noncontractible. Then, we define $\Tilde{C}_{2m'-(i+1)+1}$ and $\Tilde{C}_{2m'+(i+1)}$ to be the cycles closest respectively to $C_{2m'-(i+2)+1}$ and $C_{2m'+(i+2)}$ so that $\mathcal{B}(D_{i+2}) \subseteq \text{Int}(C_{2m'-(i+2)+1} \cup C_{2m'-(i+1)+1}, \Pi) \cup \text{Int}(C_{2m'+(i+2)} \cup C_{2m'+(i+1)}, \Pi)$. By Lemma \ref{bad_square_variant}, for $2 \leq i \leq 2m'-2$, \[(C_{2m'-(i+2)+1}, \Tilde{C}_{2m'-(i+1)+1}, C_{2m'-i+1}, C_{2m'+i}, \Tilde{C}_{2m'+(i+1)}, C_{2m'+(i+2)})\] is a good homotopic proper square with respect to $D_{i+2}$. 
    
    By Lemma \ref{C_e_non_contractible}, as $\text{int}(D_2, \Pi)$ contains $e$, it is $\Pi_e$-noncontractible.
    Hence, $\mathcal{I}_{\rm N}(D_4)$ is not empty.

    Let's first show that, for every $2 \leq i \leq 2m'$ there exists a subset of almost disjoint faces of $\mathcal{B}_{\rm N}(D_i)$ of size at least $\frac{| \mathcal{B}_{\rm N}(D_i)|}{24}$.

    \setcounter{claim}{0}

    \begin{claim}
        \label{fraction_of_B_N_almost_disjoint_variant}
        For $2 \leq i \leq 2m'$, there  exists a subset of almost disjoint faces of $\mathcal{B}_{\rm N}(D_i)$ of size at least $\frac{|\mathcal{B}_{\rm N}(D_i)|}{24}$.
    \end{claim}

    \begin{proof_claim}
        First, let's partition the faces in $\mathcal{B}(D_i)$ into sets $(\mathcal{F}_j)_{j \in \mathbb{N}^*}$. We will define the sets $(\mathcal{F}_j)_{j \in \mathbb{N}^*}$ inductively as follows:

        For $j = 1$, $\mathcal{F}_1$ is the set of faces from $\mathcal{B}(D_i)$ that intersect $\Tilde{C}_{2m'-(i-1)+1} \cup \Tilde{C}_{2m'+(i-1)}$. Moreover, we define $C_1'$ and $C_1''$ to be the cycles closest to respectively $C_{2m'-i+1}$ and $C_{2m'+i}$ so that $\mathcal{B}(D_i) - \mathcal{F}_1 \subseteq \text{Int}(C_{2m'-i+1} \cup C_1', \Pi) \cup \text{Int}(C_{2m'+i} \cup C_1'', \Pi)$.

        Suppose that for $j \geq 1$, $\mathcal{F}_j$, $C'_j$ and $C_j''$ have been defined. Then, we define $\mathcal{F}_{j+1}$ as the set of faces from $\mathcal{B}(D_i) - \bigcup_{k=1}^{j} \mathcal{F}_k$ that intersect $C'_j \cup C_j''$. Moreover, we define $C_{j+1}'$ and $C_{j+1}''$ to be the cycles closest to respectively $C_{2m'-i+1}$ and $C_{2m'+i}$ so that $\mathcal{B}(D_i) - \bigcup_{k=1}^{j+1} \mathcal{F}_k \subseteq \text{Int}(C_{2m'-i+1} \cup C_{j+1}', \Pi) \cup \text{Int}(C_{2m'+i} \cup C_{j+1}'', \Pi)$.

        Finally, $(\mathcal{F}_i)_{i \in \mathbb{N}^*}$ is defined.

        Remark that by construction, for $j \geq 1$, $\mathcal{F}_j$ and $\mathcal{F}_{j+2}$ contain vertex-disjoint faces. Moreover, remark that, for $j \geq 1$, the faces in $\mathcal{F}_j \cap \mathcal{B}_{\rm N}(D_i)$ touch either the two cycles $C_{2m'-i+1}$ and $C_{j-1}'$ (or $\Tilde{C}_{2m'-(i-1)+1}$ if $j=1$) or the two cycles $C_{2m'+i}$ and $C_{j-1}''$ (or $\Tilde{C}_{2m'+(i-1)}$ if $j=1$). Hence, by Lemma \ref{almost_disjoint_faces_between_two_cycles_variant}, there is a subset of faces from $\mathcal{F}_j \cap \mathcal{B}_{\rm N}(D_i)$ that are almost disjoint and this subset has size at least $\frac{|\mathcal{F}_j \cap \mathcal{B}_{\rm N}(D_i)|}{12}$.

        To conclude, by taking the biggest subset in the partition of $\mathcal{B}_{\rm N}(D_i)$ into $\bigcup_{j \geq 1} \mathcal{F}_{2j} \cap \mathcal{B}_{\rm N}(D_i)$ and $\bigcup_{j \geq 0} \mathcal{F}_{2j+1} \cap \mathcal{B}_{\rm N}(D_i)$ and then applying Lemma \ref{almost_disjoint_faces_between_two_cycles_variant}, we get a subset of faces from $\mathcal{B}_{\rm N}(D_i)$ of size at least $\frac{|\mathcal{B}_{\rm N}(D_i)|}{24}$.
    \end{proof_claim}

    \vspace{0.3cm}

    As the homotopic square with respect to $D_{2i}$ is good for $2 \leq i \leq m'-1$, $\mathcal{B}_{\rm N}(D_{2i})$ has size at least $\frac{|\mathcal{I}_{\rm N}(D_{2i})|}{48}$. Hence, by Claim \ref{fraction_of_B_N_almost_disjoint_variant}, $\mathcal{I}_{\rm N}(D_{2i+2})$ has size at least 
    \begin{align*}
        |\mathcal{I}_{\rm N}(D_{2i})| + \frac{|\mathcal{B}_{\rm N}(D_{2i})|}{24} & \leq |\mathcal{I}_{\rm N}(D_{2i})| + \frac{|\mathcal{I}_{\rm N}(D_{2i})|}{1152} \\
        & = \frac{1153}{1152} |\mathcal{I}_{\rm N}(D_{2i})| = q |\mathcal{I}_{\rm N}(D_{2i})| \\
    \end{align*}

    By recurrence, \[|\mathcal{I}_{\rm N}(D_{2m'})| \geq q^{m'-2} |\mathcal{I}_{\rm N}(D_4)| \geq q^{m'-2}\].

    If $m'-2 \geq \log_{q}(3g+4)$, $q^{m'-2} \geq 3g+4$. However, by Proposition \ref{homotopic_cycles_variant1}, as $g(\Pi) \leq g+2$, $|\mathcal{I}_{\rm N}(D_{2m'})| \leq 3g+3$. Then, this implies that $m' \leq \lfloor\log_{q}(3g+4)\rfloor + 2$. Finally, there are at most $2m = 4m' = 2\times (2\lfloor(\log_{q}(3g+4)\rfloor + 2))$ disjoint cycles that are nested in $\Pi$.
\end{proof}

A simple consequence of Proposition \ref{good_square_variant} is the following:

\begin{cor}
    \label{nb_non_contractible_cycles}
    $G$ contains at most $2m \times (3g+3)$ disjoint $\Pi$-noncontractible cycles.
\end{cor}

\begin{proof}
    By Proposition \ref{homotopic_cycles_variant1}, as $g(\Pi) \leq g+2$, there are at most $3g+3$ disjoint $\Pi$-noncontractible cycles that are not pairwise $\Pi$-homotopic in $G$. Moreover, by Proposition \ref{good_square_variant}, there are at most $2m$ disjoint $\Pi$-noncontractible homotopic cycles in $G$. By combining these two results, we get that $G$ contains at most $2m \times (3g+3)$ disjoint $\Pi$-noncontractible cycles.
\end{proof}

\section{Main proof} \label{sec:main_proof}

The main proof consists of three parts:
\begin{itemize}
    \item First, we prove that both the maximum degree of a vertex of $G$ and the maximum size of a face in $(G,\Pi)$ are bounded by $\Delta(g) = g^{O(\log^2 g)}$.
    \item Then, we show that the height of the tree $T$ in a minimal tree decomposition of $G$ is bounded by $g^{O(\log^3 g)}$.
    \item Finally, we assemble the already known results on $G$ (Theorems \ref{treewidth} and \ref{seymour_tree_degree}) and the two results above and use a switch to the pathwidth to get that $G$ has its order bounded by $S(g) = g^{O(\log^3 g)}$.
\end{itemize}

\subsection{Maximum degree of a vertex and maximum size of a face of \texorpdfstring{$(G,\Pi)$}{G,pi}} \label{subsec:max_degree}

We prove that the maximum degree of a vertex of $G$ and the maximum size of a face in $(G,\Pi)$ are bounded by $\Delta(g) = g^{O(\log^2 g)}$.
To prove this, we proceed by contradiction and reach a contradiction by proving that a high degree of a vertex (resp. big size of a face) leads to a large number of $\Pi$-contractible nested cycles, which contradicts Proposition \ref{good_square}. The proof makes repetitively use of Propositions \ref{homotopic_cycles_variant2}, \ref{isolated_paths}, \ref{good_square} and \ref{good_square_variant}.

\begin{defi}[Independent subgraphs]
    Let $H$ be a graph, let $k \geq 1$ and let $\mathcal{H} = \{H_0, ..., H_k\}$ be a set of subgraphs of $H$. We say that $\mathcal{H}$ is a set of \textit{independent subgraphs} of $H$ if it is possible to associate to each subgraph $H_i$ an edge $e_i$ so that $e_i \notin \cup_{j \neq i} H_j$.
\end{defi}

\begin{defi}[Fan (with an arch)]
    \label{fan}
    Let $p$ be a piece of $(G, \Pi)$. A \textit{fan} $H$ from $p$ in $G$ is a subgraph of $G$ such that:
    \begin{itemize}
        \item $H$ contain a path $P$ (\textit{horizontal path} of $H$), disjoint from $p$;
        \item There are $M$ paths $Q_1, ..., Q_M$ from $p$ to $P$ for some $M \geq 1$, such that, for every $1 \leq j \neq k \leq M$, if $p$ is a vertex, $Q_j \cap Q_k = p$ and, if $p$ is a face, $Q_j \cap Q_k = \emptyset$, $Q_1, ..., Q_M$ intersect $p$ in this order (\textit{vertical paths} of $H$);
        \item For every $1 \leq k \leq M$, $Q_k \cap p$ and $Q_k \cap P$ is a vertex;
        \item $\Pi(H)$ induces a planar embedding.
    \end{itemize}

    The size of $H$ is $M$.

    An illustration of a fan from a vertex and a face can be found in Figure \ref{fig:fan_vertex_face}.

    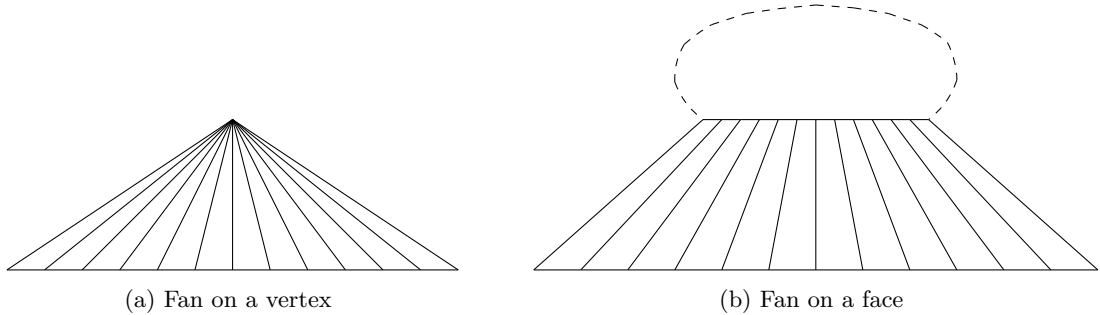
\begin{figure}[h!]
    \centering
    \begin{subfloat}[Fan on a vertex]{\input{images/fan_vertex.tikz}}
    \end{subfloat}
    \hspace{0.5cm}
    \begin{subfloat}[Fan on a face]{\input{images/fan_face.tikz}}
    \end{subfloat}

    \caption{Fans from a vertex and a face. The solid lines indicate paths, whereas the dotted line shows the boundary of the face from which the fan in Figure (b) is.}
    \label{fig:fan_vertex_face}
    \end{figure}

    \vspace{0.3cm}

    Let $p$ be a piece of $(G, \Pi)$. A \textit{fan with an arch} $H$ from $p$ in $G$ is a subgraph of $G$ such that:
    \begin{itemize}
        \item $H$ contains a fan $H'$ from $p$;
        \item Let $Q_1, ..., Q_M$ be the vertical paths of $H'$, $H$ contains a path $A$  (\textit{arch path} of $H$) whose endpoints are on $p$, which is disjoint from $H'$ except on its endpoints and so that, if $p$ is a face, $A$ does not intersect $Q_i$ for $1 \leq i \leq m$. Moreover, $H' \cup A$ is $\Pi$-contractible;
        \item There is at most one vertex $v$ in the interior of $A$ (\textit{intersection point} of $A$) such that there is no edge in $\text{Int}(H,\Pi)$ which has an endpoint in $\text{int}(A) - v$. 
    \end{itemize}

    An illustration of a fan with an arch from a vertex and a face can be found in Figure \ref{fig:fan_with_an_arch_vertex_face}.

    \begin{figure}[h!]
    \centering
    \begin{subfloat}[Fan with an arch from a vertex]{\input{images/fan_with_an_arch_vertex.tikz}}
    \end{subfloat} \hspace{0.5cm}
    \begin{subfloat}[Fan with an arch from a face]{\input{images/fan_with_an_arch_face.tikz}}
    \end{subfloat}

    \caption{Fans with an arch from a vertex and a face. The solid lines indicate paths, whereas the dotted line shows the boundary of the face from which the fan in Figure (b) is.}
    \label{fig:fan_with_an_arch_vertex_face}
    \end{figure}
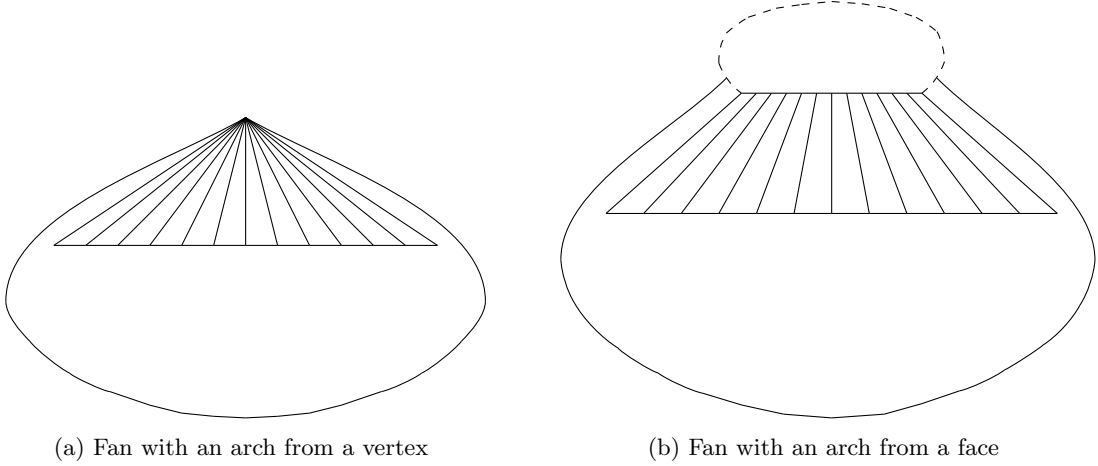
\end{defi}

\begin{theo}
    \label{max_degree}
    Let $\tilde{g} = 4 (6g+7)$, $q = \frac{1153}{1152}$ and $m = 2(\lfloor\log_{q}(3g+4)\rfloor + 2)$.
    
    \[\Delta(G) \leq \Delta(g) \text{\quad and \quad} \Delta_F(G, \Pi) \leq \Delta(g)\]
    
    with $\Delta(g) =2m(\tilde{g}+1)^{4} \left(4 m(\tilde{g}+1)^{2}\right)^{m^2}$.
\end{theo}

\begin{proof}
    We first define several objects that will play an essential role throughout the proof.

    Let $H$ be a fan (without an arch) from $p$ in $G$. Let $P$ be its horizontal path and $Q_1, ..., Q_M$ be its vertical paths with $M \geq 2$.
    We define a \textit{column} of $H$ to be the subpath of $P$ between the two endpoints of two consecutive vertical paths from $p$: let $1 \leq k \leq M-1$, then the $k$-th column of $H$ is the subpath of $P$ between the endpoints $Q_k \cap P$ (included) and $Q_{k+1} \cap P$ (excluded).

    Then, we define $\mathcal{F}(H)$ to be the set of faces that touches the interior of the horizontal path $P$ of $H$ in $\text{Ext}(H,\Pi)$. Assume now that $H$ is a fan with an arch, let $A$ be its arch path. Let $C$ be the cycle induced by $A \cup p$ that does not contain $p$ in $\Pi$ and let $C'$ be the cycle bounding $P \cup \bigcup_{1 \leq i \leq M} Q_i$ in $\Pi$, we define $C_A$ to be the circuit that bounds $\text{Int}(C \cup C', \Pi)$. Then, we define $\mathcal{F}(H)$ to be the set of faces in $\text{Int}(C_A,\Pi)$ that touches the horizontal path $P$ of $H$ and does not intersect $A$. We also define $\mathcal{F}_{arch}(H)$ to be the set of faces in $\text{Int}(C_A, \Pi)$ that touches the horizontal path $P$ of $H$ and intersect $A$. See Figure \ref{fig:fan_faces} for illustrations.

    \begin{figure}[h!]
    \centering
    \begin{subfloat}[The faces in $\mathcal{F}(H)$ of the fan $H$ are represented in blue.]{\input{images/fan_faces.tikz}}
    \end{subfloat} \hspace{0.5cm}
    \begin{subfloat}[The intersection point of the fan with an arch $H$ is $a$. The faces in $\mathcal{F}(H)$ of $H$ are represented in blue, and the faces in $\mathcal{F}_{arch}(H)$ are represented in pink.]{\input{images/fan_with_an_arch_faces.tikz}}
    \end{subfloat}

    \caption{The set $\mathcal{F}(H)$ for a fan, and $\mathcal{F}(H)$ and $\mathcal{F}_{arch}(H)$ for a fan with an arch.}
    \label{fig:fan_faces}
    \end{figure}
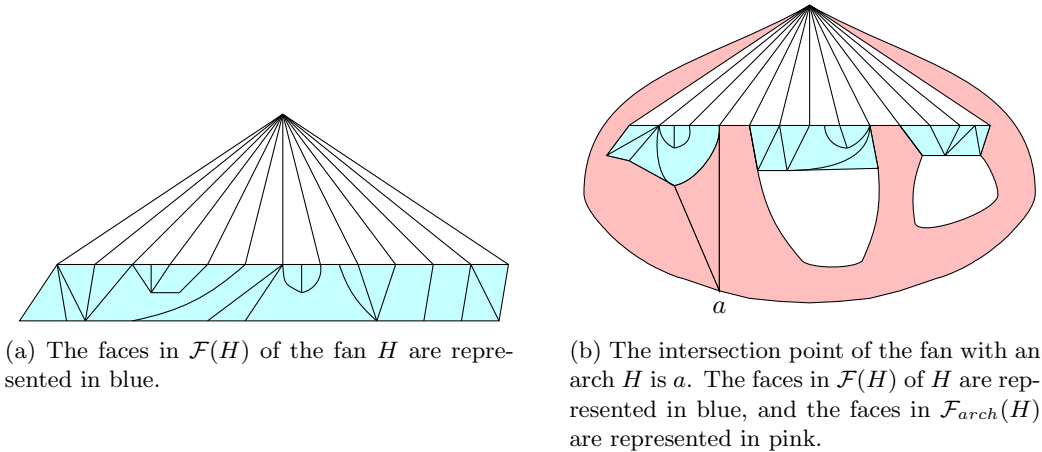

    Let $H$ be a fan (with an arch) from $p$ with the horizontal path $P$. Suppose the signature on the edges of $H$ in $\Pi$ is positive. Let $C_H$ be the cycle of $H$ so that $H \subseteq \text{Int}(C_H,\Pi)$. Then, a face $f \in \mathcal{F}(H)$ is an \textit{arched face} if $f$ does not intersect $C_H$ in a single segment (vertex or path). Let $f \in \mathcal{F}(H)$ be a face; we define the \textit{arches} of $f$ to be the maximal sub walks of $f$ whose interior vertices are not in $C_H$. Remark that the arches of any face $f \in \mathcal{F}(H)$ are paths with their ends on $C_H$. For an arch $A$ of a face $f \in \mathcal{F}(H)$, we define $C_A$ to be the $\Pi$-contractible cycle induced by $A \cup C_H$ that does not contain $H$. Moreover, remark that if an arch $A$ of a face $f \in \mathcal{F}(H)$ has positive signature, then either $\text{int}(C_A, \Pi)$ or $\text{ext}(C_A, \Pi)$ contain no edge with an endpoint on the interior of $A$ (because $A$ is a subwalk of a face in $\Pi$). 
    We say that an arch $A$ of a face $f \in \mathcal{F}(H)$ is \textit{$\Pi$-noncontractible} if $A \cup H$ induces a $\Pi$-noncontractible cycle. See Figure \ref{fig:arches} for an illustration.

    Let $f, f' \in \mathcal{F}(H)$ be two faces intersecting in $G-H$. Remark that by minimality of $G$, $f$ and $f'$ intersect in a vertex or an edge. Let $v_1$ and $v_2$ be the first and last intersection points of $f$ and $f'$ in $G$ (potentially $v_1 = v_2$) and let $P_{1,f}, P_{2,f}, P_{1,f'}, P_{2,f'}$ be the subpaths of $f$ and $f'$ from $H$ to $v_1$ and $v_2$. Then, we define the \textit{arches} of $f$ and $f'$ to be $A_1 = P_{1,f} \cup P_{1,f'}$ and $A_2 = P_{2,f} \cup P_{2,f'}$. Remark that $A_1$ and $A_2$ are paths with both ends on $H$. Moreover, remark that if an arch $A$ of two faces $f,f' \in \mathcal{F}(H)$ has a positive signature, then either $\text{int}(C_A, \Pi)$ or $\text{ext}(C_A, \Pi)$ contain no edge with an endpoint on the interior of $A$ except on one sole vertex that we call the \textit{intersection point} of $A$. 
    We say that an arch $A$ of two faces $f,f' \in \mathcal{F}(H)$ is \textit{$\Pi$-noncontractible} if $A \cup H$ induces a $\Pi$-noncontractible cycle. See Figure \ref{fig:arches} for an illustration.

    We define $\mathcal{A}(H)$ to be the set of all arches induced by one or two faces from $\mathcal{F}(H)$ whose endpoints are not both on $p$.
    Suppose that every arch $A \in \mathcal{A}(H)$ induces together with $H$ a $\Pi$-contractible subgraph of $G$. As $G' = H \cup \mathcal{A}(H)$ is $\Pi$-contractible, we can suppose that the signature on every edge of $G'$ is positive. We say that $A$ is \textit{empty above} (resp. \textit{empty below}) if $\text{ext}(C_A, \Pi)$ (resp. $\text{int}(C_A, \Pi)$) contains no edge with an endpoint on the interior of $A$ (except maybe on one sole vertex). Remark that, in that case, every arch $A \in \mathcal{A}(H)$ is either empty above or empty below. See Figure \ref{fig:arches} for an illustration.

    Let $A \in \mathcal{A}(H)$ and let $P_A = C_A \cap P$. We define the \textit{size} of $A$ to be the number of columns that $P_A$ intersects. 
    
    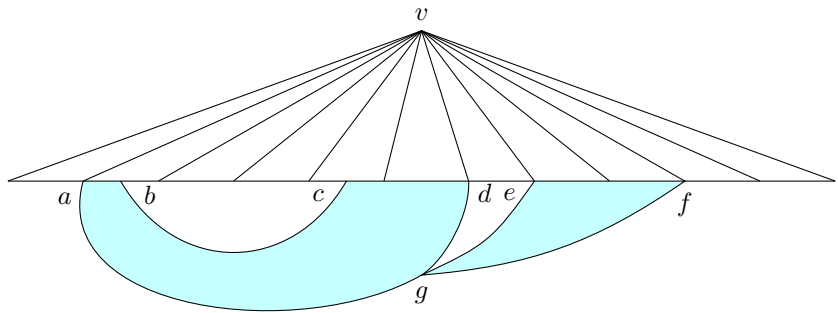
\begin{figure}[h!]
    \centering
    \input{images/arches.tikz}

    \caption{Illustration for arched faces and arches. The fan $H$ is represented, and two faces from $\mathcal{F}(H)$ are represented in blue. The face $F$ on the left is an arched face, but the face $F'$ on the right is not. The paths from $a$ to $d$ and from $b$ to $c$ are the two arches of $F$, and the path from $e$ to $f$ is the only arch of $F'$. The path from $d$ to $e$ through $g$ and the one from $a$ to $f$ through $g$ are both arches of $F$ and $F'$. In this representation, the arches from $a$ to $d$ and from $a$ to $f$ are empty below, and the arches from $b$ to $c$ and from $d$ to $e$ are empty above.}
    \label{fig:arches}
    \end{figure}
    
    Suppose that every arch $A \in \mathcal{A}(H)$ induces together with $H$ a $\Pi$-contractible subgraph of $G$. We define $\mathcal{A}_{max}(H)$ to be the set of arches in $\mathcal{A}(H)$ so that $A \in \mathcal{A}_{max}(H)$ if, for every $A' \in \mathcal{A}(H)$ with $A \neq A'$, $A \nsubseteq \text{Int}(C_A', \Pi)$. We call the arches in $\mathcal{A}_{max}(H)$ the \textit{maximal arches} of $H$. Remark that, by construction, every arch in $\mathcal{A}_{max}(H)$ is empty below. Remark, moreover, that they are naturally ordered along $P$. Let's modify $\mathcal{A}_{max}$ so that no maximal arches of size at most $2$ are adjacent by merging together some maximal arches. Then, remark that two maximal arches are consecutive in this order if and only if they intersect. See Figure \ref{fig:maximal_arches} for an illustration.

    \begin{figure}[h!]
    \centering
    \input{images/maximal_arches.tikz}

    \caption{Illustration for maximal arches. The edges inside a maximal arch are represented as dotted. The fan $H$ represented has $3$ maximal arches, represented in alternating blue and pink.}
    \label{fig:maximal_arches}
    \end{figure}
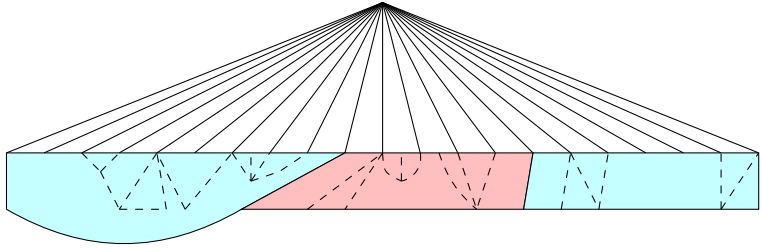

    For the purpose of initialization, we define all the notions defined for a fan (with an arch) for a piece $p$: We define the \textit{horizontal path} and the \textit{vertical paths} to be the piece $p$ and therefore $\mathcal{F}(H)$ to be the set of faces that touches $p$. We define the notions of \textit{arched faces}, \textit{arch on a face}, \textit{arch on two faces}, etc., as for the general case. 
    We define a \textit{column} of $H$ to be an edge adjacent to $p$ if $p$ is a vertex and a vertex of $p$ if $p$ is a face.

    \setcounter{claim}{0}

    \begin{sublem}
        \label{face_no_multiple_columns}
         Let $p$ be a piece of $(G, \Pi)$ and let $H$ be a fan (with an arch) from $p$ in $(G,\Pi)$ or $H = p$. Let $P$ be the horizontal path of $H$. Let $f \in \mathcal{F}(H) \cup \mathcal{F}_{arch}(H)$ ($\mathcal{F}_{arch}(H) = \emptyset$ if $H$ is a fan without an arch or $H = p$) so that $f \cup H$ is contained either in a disk or in a cylinder in $\Pi$. Then, $f$ touches at most $\tilde{g}$ columns of $H$.
    \end{sublem}

    \begin{proof}
        If $H = p$, then a face $f$ so that $f \cup H$ is contained either in a disk or in a cylinder in $\Pi$ cannot intersect several columns of $p$, as otherwise, $p$ would either be a cutvertex of $G$ which is in contradiction with the fact that $G$ is $2$-connected or have a $2$-separator that separates a $\Pi$-contractible subgraph from the rest of the graph which is in contradiction with Lemma \ref{2_separated_subgraph_is_edge}.
        
        Suppose that $H$ is a fan (with an arch). Suppose that there is a face $f \in \mathcal{F}(H) \cup \mathcal{F}_{arch}(H)$ that touches at least $\tilde{g} +1$ columns of $H$. Then, $f$ together with $H$ forms $\tilde{g}+1$ isolated paths on $p$. However, this is a contradiction according to Proposition \ref{isolated_paths}.
    \end{proof}

    \begin{sublem}
        \label{layer_non_contractible_cycles}
        Let $H$ be either a piece $p$ or a fan from a piece $p$ in $(G,\Pi)$ of size at least $M$. Let $P$ be the horizontal path of $H$. Suppose every arch in $\mathcal{A}(H)$ intersects $H$ in $P \cup p$. Then, there exists at least $\frac{M}{4m(\tilde{g}+1)^4}$ consecutive columns that induce a fan $H'$ from $v$ so that the arches in $\mathcal{A}(H')$ are all $\Pi$-contractible. 
    \end{sublem}

    \begin{proof}
        Let's show that there exists at most $(3g+3) \times 2 \times 4m(\tilde{g}^2+\tilde{g}+1)$ columns so that the $\Pi$-noncontractible arches of $H$ all touch one of these columns.
    
        First, remark that it is enough only to consider independent $\Pi$-noncontractible arches of $H$: indeed if $A$ is a $\Pi$-noncontractible arch of $H$ that is not independent of a set $\mathcal{A}$ of $\Pi$-noncontractible arches of $H$, then $A$ shares its first and last edge with arches in $\mathcal{A}$. Hence, $\mathcal{A}$ and $\mathcal{A} \cup \{A\}$ intersect the same columns. 
        
        Let $Q_1, ..., Q_M$ be the vertical paths of $H$ with $M \geq 2$.
    
        Let $\mathcal{A}$ be a maximal set of independent $\Pi$-noncontractible arches of $H$. For $A \in \mathcal{A}$, let $\mathcal{I}_A \subseteq \llbracket 1,M \rrbracket$ be the indexes of the columns of $H$ that $A$ intersects ($|\mathcal{I}_A| \leq 2$) and let $C_A$ be the $\Pi$-noncontractible cycle on $p$ that contains $A$ and does not contain the interior of $p$ if $p$ is a face induced by $A \cup P \cup p \cup \bigcup_{k \in \mathcal{I}} Q_k$. We then define the set $\mathcal{C} = \{ C_A | A \in \mathcal{A}\}$ of cycles on $p$ that correspond to the arches in $\mathcal{A}$.
        
        By Proposition \ref{homotopic_cycles_variant2}, as $g(\Pi) \leq g+2$, there are at most $3g + 3$ pairwise $\Pi$-noncontractible nonhomotopic cycles in $\mathcal{C}$. Thus, it remains to prove that, for a given class of homotopy in $\Pi$, there exists a set of at most $8m(\tilde{g}^2+\tilde{g}+1)$ columns so that all the arches on $p$ whose associated cycles all belong to this class of homotopy touch one of these columns.

        Let $\Tilde{\mathcal{A}}$ be a set of arches from $\mathcal{A}$ that are $\Pi$-noncontractible and whose associated cycles are $\Pi$-homotopic and let $\Tilde{\mathcal{C}}$ be the set of associated cycles. Let's show that there exists a set of at most $8m(\tilde{g}^2+\tilde{g}+1)$ columns so that every arch in $\Tilde{\mathcal{A}}$ touches one of these columns.

        First, we modify the subgraph $H \cup \bigcup_{A \in \Tilde{\mathcal{A}}} A$ so that it has a positive signature in $\Pi$. This is possible because, for each cycle $C \in \Tilde{\mathcal{C}}$, $C$ is two-sided and hence has a positive signature. 
        
        Let $A_1$ and $A_1'$ be the two arches from $\Tilde{\mathcal{A}}$ whose associated cycles $C_1$ and $C_1'$ in $\Tilde{\mathcal{C}}$ are so that, for every cycle $C \in \Tilde{\mathcal{C}}$, $C \subseteq \text{Int}(C_1 \cup C_1', \Pi)$. 
        Let $\Tilde{G} = \text{Int}(C_1 \cup C_1', \Pi)$. The embedding $\Pi(\Tilde{G})$ is an embedding in a cylinder bounded by $C_1$ and $C_1'$; hence it is a planar embedding. For $A \in \Tilde{\mathcal{A}}$, we say that $A$ is empty above (resp. empty below) if $A$ is empty above (resp. empty below) in the planar embedding $\Pi(\Tilde{G})$ with outer face $C_1$. Remark that every arch $A \in \Tilde{\mathcal{A}}$ is either empty above or empty below.
        
        We then partition $\Tilde{\mathcal{A}}$ into two sets $\overline{\mathcal{A}}$ and $\underline{\mathcal{A}}$ which correspond to the set of arches of $\Tilde{\mathcal{A}}$ which are empty respectively above and below. Let $\overline{\mathcal{C}}$ and $\underline{\mathcal{C}}$ be the sets of the associated cycles.

        Let's show that there exists a set of at most $4m(\tilde{g}^2+\tilde{g}+1)$ columns so that every arch in $\overline{A}$ touches one of these columns. The same result for $\underline{A}$ can be proved using the same method.

        Let's first find a maximal subset of internally disjoint arches from $\overline{\mathcal{A}}$. First, take one of the two arches from $\overline{\mathcal{A}}$ whose associated cycles $C_1$ and $C_1'$ in $\overline{\mathcal{C}}$ are so that, for every cycle $C \in \overline{\mathcal{C}}$, $C \subseteq \text{Int}(C_1 \cup C_1', \Pi)$. Let $A_1$ be such an arch. Then, select the arch $A_2$ from $\overline{\mathcal{A}}$ that is internally disjoint from $A_1$ and so the cycle associated with $A_2$ is the closest to the cycle associated with $A_1$. 
        Iterate this process to get a maximal set $\overline{\mathcal{A}}_0 = \{A_1, ..., A_p\}$ of pairwise internally disjoint arches (for some $p \geq 1$). Let $\overline{\mathcal{C}}_0 = \{C_1, ..., C_p\}$ be the set of associated cycles. Let $C_{p+1}$ be the cycle so that, for every cycle $C \in \overline{\mathcal{C}}$, $C \subseteq \text{Int}(C_1 \cup C_{p+1}, \Pi)$ and $A_{p+1}$ the associated arch. See Figure \ref{fig:construction_A_i} for an illustration.

        \begin{figure}[h!]
        \centering
        \input{images/construction_A_i.tikz}

        \caption{Illustration for the sets $\overline{\mathcal{A}}_0$ and $(\overline{\mathcal{A}}_i)_{1 \leq i \leq p}$. The internally disjoint arches in $\overline{\mathcal{A}}_0$ are in this example $A_1, A_2, A_3, A_4, A_5, A_6$. We have $p=6$ and $A_7 = A_6$. Moreover, the disks containing the arches in $\overline{\mathcal{A}}_i$ are depicted in alternating blue and green.}
        \label{fig:construction_A_i}
        \end{figure}
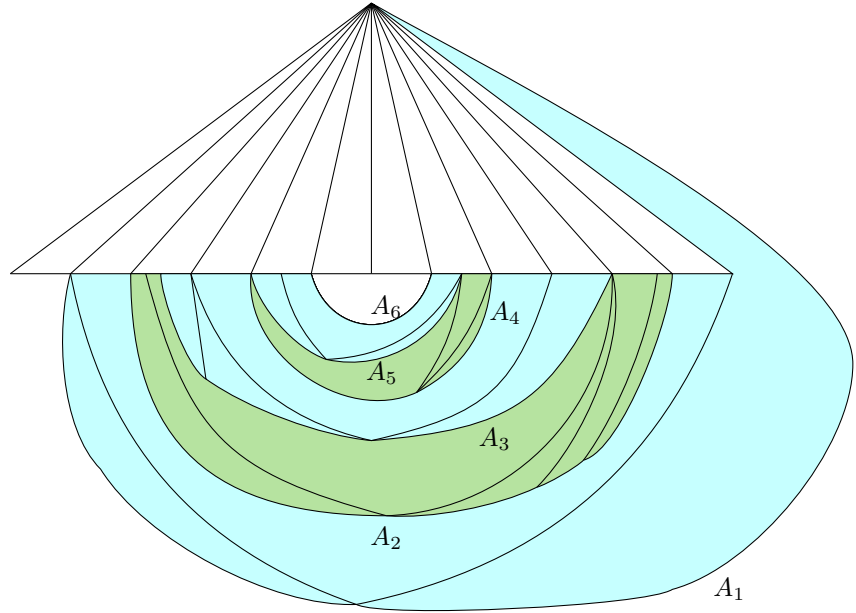

        Remark that if an arch of $\overline{\mathcal{A}}$ intersects $\text{int}(C_i \cup C_{i+1}, \Pi)$ ($1 \leq i \leq p$) then it is contained in $\text{Int}(C_i \cup C_{i+1}, \Pi)$.  
        Let $\overline{\mathcal{A}}_i$ be the set of arches in $\overline{\mathcal{A}} - \overline{\mathcal{A}}_0$ that lie in $\text{Int}(C_i \cup C_{i+1}, \Pi)$ ($1 \leq i \leq p$). Then, the arches in $\overline{\mathcal{A}}$ can be partitioned into the sets $\overline{\mathcal{A}}_i$ for $0 \leq i \leq p$. See Figure \ref{fig:construction_A_i} for an illustration.

        \begin{subclaim}
            \label{columns_A_0}
            There exists a set of at most $4m$ columns of $H$ so that all arches in $\overline{\mathcal{A}}_0$ touch at least one of these columns.
        \end{subclaim}
        
        \begin{proof_claim}
            Let's create an auxiliary graph $\Gamma$ whose vertices are the columns of $H$ and $p$ intersected by the arches in $\overline{\mathcal{A}}_0$. There is an edge between two columns if an arch in $\overline{\mathcal{A}}_0$ touches both columns. Remark that, by Proposition \ref{good_square_variant}, there are at most $2m$ cycles from $\overline{\mathcal{C}}$ that correspond to a matching in $\Gamma$. Hence, a maximal matching of $\Gamma$ has a size at most $2m$. Moreover, the set of the vertices in a maximal matching of $\Gamma$ is a dominating set of $\Gamma$. Therefore, a set of at most $4m$ columns of $H$ exists so that all arches in $\overline{\mathcal{A}}_0$ touch at least one of these columns.
        \end{proof_claim}
        
        \vspace{0.3cm}
        
        \begin{subclaim}
            \label{not_covered_A_i}
            Suppose that $p > 4m$. Let $\mathcal{B}$ be the set of at most $4m$ columns of $H$ obtained in Claim \ref{columns_A_0}. Then, there exists a set $\mathcal{I} \subseteq \llbracket 1,p \rrbracket$ of size at most $4m$ so that all arches in $\overline{\mathcal{A}}_0$ and in $\cup_{i \notin \mathcal{I}} \overline{\mathcal{A}}_i$ touch at least one of these columns.
        \end{subclaim}

        \begin{proof_claim}

        We will prove the property by induction on $p' = p-4m$.

        For $p' = 1$, we have $p = 4m+1$. By the pigeon hole principle, there are $2$ arches $A_i$ and $A_j$ of $\overline{\mathcal{A}}_0$ ($A \leq i \leq j \leq p$, potentially $i=j$) that touch the same column of $\mathcal{B}$.
            
        Assume $i \neq j$. Remark then that every arch in $\overline{\mathcal{A}}_k$ for $i \leq k < j$ touches this same column. Finally, by taking $\mathcal{I} = \llbracket 1 , p\rrbracket -\llbracket i,j\llbracket$ which has size at most $4m$, we indeed have that all arches in $\overline{\mathcal{A}}_0$ and in  $\cup_{i \notin \mathcal{I}} \overline{\mathcal{A}}_i$ touch at least one of the $4m$ columns $\mathcal{B}$.

        Assume $i = j$. Remark that every arch in $\overline{\mathcal{A}}_k$ for either $k \geq i$ or $k \leq i$ touches this same column. Finally, by taking $\mathcal{I} = \llbracket 1 , p\rrbracket -\llbracket i,p\rrbracket$ or $\mathcal{I} = \llbracket 1, p\rrbracket - \llbracket 1,i\rrbracket$ which has size at most $4m$, we indeed have that all arches in $\overline{\mathcal{A}}_0$ and in $\cup_{i \notin \mathcal{I}} \overline{\mathcal{A}}_i$ touch at least one of the $4m$ columns $\mathcal{B}$.

        Now, let $p' = p -4m \geq 2$ and suppose that the property is true for every $p'' < p'$.

        By the pigeon hole principle, there are arches $A_i$ and $A_j$ of $\overline{\mathcal{A}}_0$ ($1 \leq i \leq j \leq p$, potentially $i=j$) that touch the same column of $\mathcal{B}$. 
            
        Assume $i \neq j$. Remark then that every arch in $\overline{\mathcal{A}}_k$ for $i \leq k < j$ touches this same column. Let $\mathcal{I} = \llbracket 1 , p\rrbracket - \llbracket i,j\llbracket$ which has size at most $p-1$. We have that that all arches in $\overline{\mathcal{A}}_0$ and in $\cup_{i \notin \mathcal{I}} \overline{\mathcal{A}}_i$ touch at least one of the $4m$ columns $\mathcal{B}$.

        Assume $i = j$. Remark that every arch in $\overline{\mathcal{A}}_k$ for either $k \geq i$ or $k \leq i$ touches this same column. Let $\mathcal{I} = \llbracket 1 , p\rrbracket -\llbracket i,p\rrbracket$ or $\mathcal{I} = \llbracket 1, p\rrbracket - \llbracket1,i\rrbracket$ which has size at most $p-1$. We have that all arches in $\overline{\mathcal{A}}_0$ and in $\cup_{i \notin \mathcal{I}} \overline{\mathcal{A}}_i$ touch at least one of the $4m$ columns $\mathcal{B}$.

        In both cases, we proved that there is $\mathcal{I} \subseteq \llbracket 1,p\rrbracket$ of size at most $p-1$ so that all arches in $\overline{\mathcal{A}}_0$ and in $\cup_{i \notin \mathcal{I}} \overline{\mathcal{A}}_i$ touch at least one of the $4m$ columns $\mathcal{B}$. Let's apply induction on $\cup_{i \in \mathcal{I}} A_i$ and $\cup_{i \in \mathcal{I}} \overline{\mathcal{A}}_i$ with $p'' = |\mathcal{I}|-4m < p'$. By induction hypothesis, there is a set $\mathcal{I}' \subseteq \mathcal{I}$ of size at most $4m$ so that all arches in $\cup_{i \in \mathcal{I}} A_i$ and $\cup_{i \in \mathcal{I} - \mathcal{I}'} \overline{\mathcal{A}}_i$ touch at least one of the $4m$ columns $\mathcal{B}$.

        Finally, all arches in $\overline{\mathcal{A}}_0$ and in $\cup_{i \notin \mathcal{I}'} \overline{\mathcal{A}}_i$ touch at least one of the $4m$ columns $\mathcal{B}$.
        \end{proof_claim}
		
		\vspace{0.3cm}
		
        \begin{subclaim}
            \label{columns_A_i}
            For $1 \leq i \leq p$, suppose that $A_i$ intersect the column $c$ of $H$. Then, there exists a set of at most $\tilde{g}^2 + \tilde{g}$ columns of $H$ so that every arch in $\overline{\mathcal{A}}_i$ touches either $c$ or one of these columns.
        \end{subclaim}
		
        \begin{proof_claim}
        \begin{itemize}
        
            \item \underline{If $A_i$ has no intersection point:}

            Remark that, by construction, the arches in $\overline{\mathcal{A}}_i$ always intersect the interior of $A_i$ either in a vertex or in a subpath of $A_i$. Remark moreover that, by construction, as the arches in $\overline{\mathcal{A}}_i$ are empty above, their paths in $\text{int}(C_i \cup C_{i+1}, \Pi)$ are pairwise internally disjoint.

            For each arch $A \in \overline{\mathcal{A}}_i$, let $v_A$ be its intersection point with $A_i$ (if $A$ shares a subpath with $A_i$, the intersection point is the vertex on which $A$ and $A_i$ diverge), we call the partition of $A$ in the two subpaths of $A$ with $v_A$ as an endpoint its \textit{half-arches}.
            Let $\mathcal{B}_1$ and $\mathcal{B}_2$ be the partition of the half-arches in $\overline{\mathcal{A}}_i$ so that, for two half-arches $B, B' \in \mathcal{B}_1$ (resp. $\mathcal{B}_2$), the paths from $p$ to $A_i$ induced by $B$ and $B'$ are $\Pi$-homotopic. Remark that every arch of $\overline{\mathcal{A}}_i$ has one of its half-arch in $\mathcal{B}_1$ and the other in $\mathcal{B}_2$. Remark, moreover, that either $\mathcal{B}_1$ or $\mathcal{B}_2$ (potentially none and not both) contains half-arches that touch $p$. Suppose, therefore, without loss of generality, that $\mathcal{B}_1$ does not contain any half-arch that touches $p$.
            
            If a subset $\mathcal{B}$ of arches in $\mathcal{B}_1$ is so that every half-arch of $\mathcal{B}$ in $\text{int}(C_i \cup C_{i+1}, \Pi)$ intersect $A_i$ on the same vertex $v_i$, then $\mathcal{B}$ intersect at most $\tilde{g}$ columns, as otherwise these arches would create $\tilde{g}+1$ isolated paths on $p$ and $v_i$ and this would be in contradiction with Proposition \ref{isolated_paths}. Moreover, if a subset $\mathcal{B}$ of arches in  $\mathcal{B}_1$ is so that the half-arches of $\mathcal{B}$ in $\text{int}(C_i \cup C_{i+1}, \Pi)$ intersect $A_i$ on disjoint vertices, then $\mathcal{B}$ intersect at most $\tilde{g}$ columns, as otherwise these arches would create $\tilde{g}+1$ isolated paths on $p$ and $A_i$ and this would be in contradiction with Proposition \ref{isolated_paths}.
            
            Finally, in total, there exists a set of at most $\tilde{g}^2 \leq \tilde{g}^2 + \tilde{g}$ columns so that every arch in $\overline{\mathcal{A}}_i$ touches either $c$ or one of these columns.
            
            \item \underline{If $A_i$ has one intersection point:}

            Let $v_i$ be the intersection vertex of $A_i$ and let $A_i^1$ and $A_i^2$ be the partition of $A_i$ into two subpaths of $A_i$ with $v_i$ as an endpoint.
            
            Remark that, by construction, the arches in $\overline{\mathcal{A}}_i$ always intersect the interior of $A_i$ either in a vertex or in a subpath of $A_i^1$ or $A_i^2$. Remark moreover that, by construction, as the arches in $\overline{\mathcal{A}}_i$ are empty above, their paths in $\text{int}(C_i \cup C_{i+1}, \Pi)$ are pairwise internally disjoint.
            Let $\overline{\mathcal{A}}_i^0$, $\overline{\mathcal{A}}_i^1$ and $\overline{\mathcal{A}}_i^2$ be the sets of arches in $\overline{\mathcal{A}}_i$ that intersect the interior of $A_i$ in respectively a vertex, $A_i^1$ and $A_i^2$.
            
            Remark that two arches $A, A' \in \overline{\mathcal{A}}_i^0$ must intersect the interior of $A_i$ in the same vertex (to avoid crossings). Let $v_i'$ be the vertex of $A_i$ that arches in $\overline{\mathcal{A}}_i^0$ all intersect (if $\overline{\mathcal{A}}_i^0$ is not empty). 
            For each arch $A \in \overline{\mathcal{A}}_i^0$, let $v_A$ be its intersection point with $A_i$ (if $A$ shares a subpath with $A_i$, the intersection point is the vertex on which $A$ and $A_i$ diverge), we call the partition of $A$ in the two subpaths of $A$ with $v_A$ as an endpoint its \textit{half-arches}.
            Let $\mathcal{B}_1$ and $\mathcal{B}_2$ be the partition of the half-arches in $\overline{\mathcal{A}}_i^0$ so that, for two half-arches $B, B' \in \mathcal{B}_1$ (resp. $\mathcal{B}_2$), the paths from $p$ to $A_i$ induced by $B$ and $B'$ are $\Pi$-homotopic. Remark that every arch of $\overline{\mathcal{A}}_i^0$ has one of its half-arch in $\mathcal{B}_1$ and the other in $\mathcal{B}_2$. Remark, moreover, that either $\mathcal{B}_1$ or $\mathcal{B}_2$ (potentially none and not both) contains half-arches that touch $p$. Suppose, therefore, without loss of generality, that $\mathcal{B}_1$ does not contain any half-arch that touches $p$.
            Remark that $\mathcal{B}_1$ intersects at most $\tilde{g}$ columns, as otherwise these arches would create $\tilde{g}+1$ isolated paths on $p$ and $v_i'$ and this would be in contradiction with Proposition \ref{isolated_paths}. Therefore, there exists a set of at most $\tilde{g}$ columns so that every arch in $\overline{\mathcal{A}}_i^0$ touches one of these columns.

            Suppose, without loss of generality, that $A_i^2$ intersects $c$. The, the arches in $\overline{\mathcal{A}}_i^2$ all touch $c$. Therefore, it only remains to bound the number of columns that the arches in $\overline{\mathcal{A}}_i^1$ intersect. 

            If a subset $\mathcal{A}^1$ of arches in $\overline{\mathcal{A}}_i^1$ is so that the arches of $\mathcal{A}^1$ in $\text{int}(C_i \cup C_{i+1}, \Pi)$ intersect $A_i$ on the same vertex $v_i^1$, then $\mathcal{A}^1$ intersect at most $\tilde{g}$ columns of $H$, otherwise these arches would create $\tilde{g}+1$ isolated paths on $p$ and $v_i^1$ and this would be in contradiction with Proposition \ref{isolated_paths}. Moreover, if a subset $\mathcal{A}^1$ of arches in $\overline{\mathcal{A}}_i^1$ is so that the arches of $\mathcal{A}^1$ in $\text{int}(C_i \cup C_{i+1}, \Pi)$ intersect $A_i$ on disjoint vertices, then $\mathcal{A}^1$ intersect at most $\tilde{g}$ columns of $H$, as otherwise these arches would create $\tilde{g}+1$ isolated paths on $p$ and $A_i^1$ and this would be in contradiction with Proposition \ref{isolated_paths}.
            
            Finally, there exists a set of at most $\tilde{g}^2$ columns so that every arch in $\overline{\mathcal{A}}_i^1$ touches one of these columns. And in total, there exists a set of at most $\tilde{g}^2 + \tilde{g}$ columns so that every arch in $\overline{\mathcal{A}}_i$ touches either $c$ or one of these columns.
        \end{itemize}
        \end{proof_claim}
        
        \vspace{0.3cm}

        By Claim \ref{not_covered_A_i}, there exists a set $\mathcal{C}$ of at most $4m$ columns and a set $\mathcal{I} \subseteq \llbracket1,p\rrbracket$ of size at most $4m$ so that all arches in $\overline{\mathcal{A}}_0$ and in $\cup_{i \notin \mathcal{I}} \overline{\mathcal{A}}_i$ touch at least one of these columns.
        Moreover, let $i \in \mathcal{I}$ and $c_i$ be a column of $\mathcal{C}$ that $A_i$ touches. By Claim \ref{columns_A_i}, there is a set of at most $\tilde{g}^2 + \tilde{g}$ columns so that the arches in $\overline{\mathcal{A}}_i$ all touch either $c_i$ or these columns. Finally, the arches in $\overline{\mathcal{A}}$ touch at most $4m(\tilde{g}^2 + \tilde{g} + 1)$ columns.

        In total, there exists at most $(3g+3) \times 2 \times 4m(\tilde{g}^2+\tilde{g}+1)$ columns so that the $\Pi$-noncontractible arches of $H$ all touch one of these columns.

        \vspace{0.5cm}

        Now, let's show that there are at least $\frac{M}{4m(\tilde{g}+1)^4}$ consecutive columns of $H$ so that, in the induced fan $H'$, $\mathcal{A}(H')$ contains no $\Pi$-noncontractible arch.

        Let $\mathcal{A}_p(H)$ and $\mathcal{A}_P(H)$ be respectively the $\Pi$-noncontractible arches of $\mathcal{A}(H)$ that touch $p$ and the arches that touch solely $P$.
        Let's show that there are at most $(3g+3) \times 2 \times [4m(\tilde{g}^2+\tilde{g}+1) \times 2\tilde{g} + 4 \times (3g+3) \times \tilde{g}] \leq 4m (\tilde{g}+1)^4$ columns so that:
        \begin{enumerate}
            \item the faces $\mathcal{F}_p(H)$ in $\mathcal{F}(H)$ that have an arch in $\mathcal{A}_p(H)$ touch only these columns, and
            \item the faces $\mathcal{F}_P(H)$ in $\mathcal{F}(H)$ that have arches only in $\mathcal{A}_P(H)$ are so that each of their $\Pi$-noncontractible arches touch one of these columns. 
        \end{enumerate}
        
        This will guarantee that there are at least $\frac{M}{4m(\tilde{g}+1)^4}$ consecutive columns of $H$ so that, in the induced fan $H'$, $\mathcal{A}(H')$ contains no $\Pi$-noncontractible arch.

        We start with the at most $(3g+3) \times 2 \times 4m(\tilde{g}^2+\tilde{g}+1)$ columns $\mathcal{C}_0$ found previously. Hence, the second property is already guaranteed.

        Let $\mathcal{A}$ be a maximal subset of independent $\Pi$-noncontractible homotopic empty above arches from $\mathcal{A}_p(H)$. Let $\mathcal{C} \subseteq \mathcal{C}_0$ of size at most $4m (\tilde{g}^2+\tilde{g}+1)$ be the columns so that every arch in $\mathcal{A}$ intersect a column of $\mathcal{C}$. Let $S_0$ be the smallest cylinder in $S$ that contains all the cycles associated with $\mathcal{A}$.

        For $c \in \mathcal{C}$, let $\mathcal{A}_c \subseteq \mathcal{A}$ be the arches in $\mathcal{A}$ that intersect the column $c$, $(\mathcal{A}_c)_{c \in \mathcal{C}}$ is a partition of $\mathcal{A}$. Remark that for each column $c \in \mathcal{C}$, every face associated with the arches in $\mathcal{A}_c$ except two (that we will call $f_c^1$ and $f_c^2$) intersects only $c$. Hence, adding columns that intersect only these two faces is enough. Moreover, let $\mathcal{F}_\mathcal{C} = \{ f_c^1, f_c^2 | c \in \mathcal{C}\}$, remark that every face in $\mathcal{F}_\mathcal{C}$ expect two ($f_1^*$ and $f_2^*$) is contained in $S_0$. By Lemma \ref{face_no_multiple_columns}, the faces in $\mathcal{F}_\mathcal{C} - \{f_1^*, f_2^*\}$ intersect at most $\tilde{g}$ columns. Moreover, the faces $f_1^*$ and $f_2^*$ contain at most $2 \times (3g+3)$ $\Pi$-noncontractible arches on $H$ whose associated cycles are pairwise $\Pi$-nonhomotopic and, by Lemma \ref{face_no_multiple_columns}, their $\Pi$-homotopic arches intersect at most $\tilde{g}$ columns. Hence, in total, the faces $f_1^*$ and $f_2^*$ intersect at most $2 \times (3g+3) \times \tilde{g}$ columns.

        Finally, there are at most $(3g+3) \times 2 \times [4m(\tilde{g}^2+\tilde{g}+1) \times 2\tilde{g} + 4 \times (3g+3) \times \tilde{g}]$ columns $\mathcal{C}_1$ so that the faces in $\mathcal{F}_p(H)$ touch only these columns and the arches in $\mathcal{A}_P(H)$ touch at least one of these columns.

        \vspace{0.3cm}

        To conclude, remark then that every set of consecutive columns of $H$ that does not contain $\mathcal{C}_1$ induces a fan $H'$ so that $\mathcal{A}(H')$ does not contain a $\Pi$-noncontractible arch. Therefore, there is a set of at least $\frac{M}{4m (\tilde{g}+1)^4}$ consecutive columns that induce a fan $H'$ from $v$ so that the arches in $\mathcal{A}(H')$ are all $\Pi$-contractible.
    \end{proof} 

    \begin{sublem}
        \label{new_fan}
        Let $1 \leq i \leq m$ and let $H$ be a fan (with an arch) from $p$ in $(G, \Pi)$ of size at least $\frac{f(g,i)-2}{4m(\tilde{g}+1)^4}$. Suppose that every arch in $\mathcal{A}(H)$ induces, together with $H$, a $\Pi$-contractible subgraph of $G$. Then, there exists $1 \leq k \leq i$ so that there are at least $f(g,i-k)$ maximal arches from $\mathcal{A}_{max}(H)$ of size at least $f(g,k-1)+2$.
    \end{sublem}

    \begin{proof}
        Suppose, by contradiction, that, for every $1 \leq k \leq i$, there are less than $f(g,i-k)$ maximal arches from $\mathcal{A}_{max}(H)$ of size at least $f(g,k-1)+2$.

        Remark that $\mathcal{A}_{max}$ can be partitioned into $i+1$ sets of arches $\mathcal{A}_0, ..., \mathcal{A}_i$ so that, for $0 \leq k \leq i$, $A \in \mathcal{A}_{max}$ is in $\mathcal{A}_k$ if its size is in the range $f(g,k-1)+2$ to $f(g,k)+1$ (with $f(g,-1) = -1$). Remark, moreover, that, by hypothesis, $\mathcal{A}_i$ is empty.

        If $i=1$, then $\mathcal{A}_{max}$ only contains maximal arches of size at most $2$. However, this contradicts the definition of $\mathcal{A}_{max}$.

        Hereafter, suppose $i \geq 2$.
        
        Therefore, 
        
        \begin{align*}
            \sum_{A \in \mathcal{A}_{max}} |A| & = \sum_{k=0}^{i-1} \sum_{A \in \mathcal{A}_k} |A| \\
            & = \sum_{A \in \mathcal{A}_0} |A| + \sum_{k=
        1}^{i-1} \sum_{A \in \mathcal{A}_k} |A| \\
            & \leq 4\sum_{k=1}^{i-1} f(g,k) +\sum_{k=1}^{i-1} (f(g,i-k)-1) \times (f(g,k)+1) \\
            & < 4\sum_{k=1}^{i-1} f(g,k) +\sum_{k=1}^{i-1} f(g,i-k) \times f(g,k)
        \end{align*}

        As, for $1 \leq k \leq i-1$, 
        
        \begin{align*}
             f(g,i-k) \times f(g,k) &=  \left(4\sqrt{2}m (\tilde{g}+1)^{2}\right)^{(i-k)^2+k^2} \\
            & = \left(4\sqrt{2}m (\tilde{g}+1)^{2}\right)^{i^2+2k(k-i)} \\
            & \leq \left(4\sqrt{2}m (\tilde{g}+1)^{2}\right)^{i^2-2} = \frac{f(g,i)}{32m^2(\tilde{g}+1)^4}
        \end{align*}

        and

        \begin{align*}
            4\sum_{k=1}^{i-1} f(g,k) &\leq 4 \times (i-1) \times f(g,i-1) \\
            & \leq (i-1)\times 4 \times \left(4\sqrt{2}m (\tilde{g}+1)^{2}\right)^{i^2-2} = (i-1)\times \frac{4f(g,i)}{32m^2(\tilde{g}+1)^4}
        \end{align*}
        
        Then 

        \begin{align*}
            \sum_{A \in \mathcal{A}_{max}} |A| &< (i-1) \times \left(\frac{4f(g,i)}{32m^2(\tilde{g}+1)^4} + \frac{f(g,i)}{32m^2(\tilde{g}+1)^4} \right) \\
            & \leq (i-1) \frac{f(g,i)-2}{4m^2(\tilde{g}+1)^4} \\
            & < \frac{f(g,i)-2}{4m(\tilde{g}+1)^4}
        \end{align*}

        However, as the maximal arches of $H$ cover all the horizontal path of $H$, $\sum_{A \in \mathcal{A}_{max}} |A| \geq \frac{f(g,i)-2}{4m(\tilde{g}+1)^4}$, a contradiction.
    \end{proof}

    \begin{sublem}
        \label{touch_arch}
        Let $H$ be a fan with an arch from $p$ in $(G,\Pi)$ of size at least $M$. Then, there exists a subset of consecutive faces from $\mathcal{F}(H)$ that spans at least $\frac{M}{\tilde{g}^2}$ columns.
    \end{sublem}

    \begin{proof}
        Let $P_{arch}$ be the arch path of $H$. By definition of $\mathcal{F}(H)$, no face in $\mathcal{F}(H)$ touches $P_{arch}$, but the faces in $\mathcal{F}(H)$ may not be consecutive.

        First, by Lemma \ref{face_no_multiple_columns}, each face of $\mathcal{F}_{arch}(H)$ intersect at most $\tilde{g}$ columns.

        If $P_{arch}$ has no intersection point, then $\mathcal{F}_{arch}(H)$ contains only one face  and $\mathcal{F}_{arch}(H)$ intersect at most $\tilde{g}$ columns in total.
        
        Suppose $P_{arch}$ has $1$ intersection point. Suppose that the faces in $\mathcal{F}_{arch}(H)$ touch at least $\tilde{g}^2+1$ columns, then it contains at least $\tilde{g}+1$ faces that touch pairwise disjoint columns. However, 
        then there would be at least $\tilde{g}+1$ isolated paths on $p$ and the intersection point of $P_{arch}$ and we would find a contradiction with Proposition \ref{isolated_paths}. 
        
        Hence, in any case, $\mathcal{F}_{arch}(H)$ touches at most $\tilde{g}^2$ disjoint columns.

        Finally, it is possible to find a subset of at least $\frac{M}{\tilde{g}^2}$ consecutive columns that are touched by no face from $\mathcal{F}_{arch}(H)$. Therefore, consecutive faces from $\mathcal{F}(H)$ touch these consecutive columns.
    \end{proof}
    
    \begin{sublem}
    \label{big_degree_nested_cycles}
        If there exists a piece $p$ of $(G, \Pi) $ such that $d(p) \geq \Delta(g) = 4m(\tilde{g}+1)^{4} \left(4\sqrt{2} m(\tilde{g}+1)^{2}\right)^{m^2}$ then $G$ contains $m+1$ cycles that are $\Pi$-contractible well nested pinched on $p$.
    \end{sublem}

    \begin{proof}
        Let $p$ be a piece of $(G, \Pi)$ of degree at least $\Delta(g)$.

        Let $f(g,i) = \left(4\sqrt{2} m (\tilde{g}+1)^{2}\right)^{i^2}$ for $i \in \mathbb{N}$. Let's prove simultaneously the following properties by induction on $i$ with $0 \leq i \leq m$:
        \begin{enumerate}
            \item If there is a fan $H$ from $p$ of size $f(g,i)$ in $G$, then $G$ contains $i$ $\Pi$-contractible well nested cycles with a column of $H$ strictly contained in the innermost nested cycle.
            \item If there is a fan with an arch $H$ from $p$ of size $f(g,i)$ in $G$, then there are $i$ $\Pi$-contractible well nested cycles in $\text{int}(H, \Pi)$ with a column of $H$ strictly contained in the innermost nested cycle. 
        \end{enumerate}

        First, for $i=0$, both properties are trivial.

        Then, let $1 \leq i \leq m$ and suppose that both properties are valid for every $k < i$. Let's prove each property for $i$.

    \begin{enumerate}
        \item \underline{Suppose that there is a fan $H$ from $p$ of size $f(g,i)$ in $(G,\Pi)$.} 
        
        \underline{Let's prove property (1) for $i$.} 
        
        Let $P$ be the horizontal path of $H$.  

        Remove the first and last column of the fan $H$; therefore, no face in $\mathcal{F}(H)$ touches the first or last vertical path of $H$ (except potentially in $p$).

        By Lemma \ref{layer_non_contractible_cycles}, there exists at least $\frac{f(g,i)-2}{4m (\tilde{g}+1)^4}$ consecutive columns of $H$ (let's call the associated fan $H'$, $P'$ its horizontal path) so that the arches in $\mathcal{A}(H')$ does not induce, together with $H'$, $\Pi$-noncontractible cycles.
        Therefore, every arch in $\mathcal{A}(H')$ induces, together with $H'$, a $\Pi$-contractible subgraph of $G$.

        By Lemma \ref{new_fan}, there exists $1 \leq k \leq i$ so that there are at least $f(g,i-k)$ maximal arches from $\mathcal{A}_{max}(H')$ of size at least $f(g,k-1)+2$. Remark then that we can construct a fan $H''$ of size at least $f(g,i-k)$ so that each maximal arch of size at least $f(g,k-1)+2$ is contained in one of its columns. Remark also that, even though each maximal arch $A$ has size at least $f(g,k-1)+42$, the column of a fan $H''$ containing $A$ may not contain all the columns of $H'$ that $A$ touches. It, however, contains at least $f(g,k-1)$ of them.

        By induction hypothesis on the property (1) for $H''$, $G$ contains $i-k$ $\Pi$-contractible well-nested cycles with a column of $H''$ strictly contained in the innermost nested cycle. Furthermore, by construction, this column contains an arch $A$ from $\mathcal{A}_{max}(H')$ (hence empty below) of size at least $f(g,k-1)+2$. Let $H_A$ be the fan with an arch induced by $H'$ whose arch path is $A$ (if $A$ touches $p$, then we artificially add columns on the sides so that the fan structure touches the arch appropriately). The fan $H_A$ has a size of at least $f(g,k-1)$. Therefore, by induction hypothesis on the property (2) for $H_A$, $G$ contains $k-1$ $\Pi$-contractible well-nested cycles in $\text{int}(H_A, \Pi)$ with a column of $H_A$ (which corresponds to a column of $H$) strictly contained in the innermost nested cycle.

        Finally, $G$ contains $i$ $\Pi$-contractible well-nested cycles with a column of $H$ strictly contained in the innermost nested cycle: the $i-k$ nested cycles found in $H''$, the cycle associated to $A$ and the $k-1$ nested cycles found in $H_A$.

        \item \underline{Suppose that there is a fan with an arch $H$ from $p$ of size $f(g,i)$ in $(G,\Pi)$.} 
        
        \underline{Let's prove property (2) for $i$.} 
        
        Let $P$ be the horizontal path of $H$ and $P_{Arch}$ be the arch path of $H$.
        
        Remove the first and last column of the fan $H$; therefore, no face in $\mathcal{F}(H)$ touches the first or last vertical path of $H$ (except potentially in $p$).
        
        By Lemma \ref{touch_arch}, we can restrict $\mathcal{F}(H)$ to a subset of consecutive faces from $\mathcal{F}(H)$ that spans at least $\frac{f(g,i)-2}{\tilde{g}^2} \geq \frac{f(g,i)-2}{4m(\tilde{g}+1)^4}$ columns so that this subset of faces does not touch $P_{arch}$ (let's call the associated fan $H'$, $P'$ its horizontal path and $P_{arch}$ is still its arch path). 

        By Lemma \ref{new_fan}, there exists $1 \leq k \leq i$ so that there are at least $f(g,i-k)$ maximal arches from $\mathcal{A}_{max}(H')$ of size at least $f(g,k-1)+2$. Remark then that we can construct a fan with an arch $H''$, with the arch path that is a subpath of $P_{Arch}$, of size at least $f(g,i-k)$ so that each maximal arch of size at least $f(g,k-1)+2$ is contained in one of its columns. Remark that, even though each maximal arch $A$ has size at least $f(g,k-1)+2$, the column of a fan $H''$ containing $A$ may not contain all the columns of $H'$ that $A$ touches. It, however, contains $f(g,k-1)$ of them.

        By induction hypothesis on the property (2) for $H''$, $G$ contains $i-k$ $\Pi$-contractible well-nested cycles in $\text{int}(H'', \Pi)$ with a column of $H''$ strictly contained in the innermost nested cycle. Furthermore, by construction, this column contains an arch $A$ from $\mathcal{A}_{max}(H')$ (hence empty below) of size at least $f(g,k-1)+2$. Let $H_A$ be the fan with an arch induced by $H'$ whose arch path is $A$. It has a size of at least $f(g,k-1)$. Therefore, by induction hypothesis on the property (2) for $H_A$, $G$ contains $k-1$ $\Pi$-contractible well-nested cycles in $\text{int}(H_A, \Pi)$ with a column of $H_A$ (which corresponds to a column of $H$) strictly contained in the innermost nested cycle.

        Finally, $G$ contains $i$ $\Pi$-contractible well-nested cycles in $\text{int}(H, \Pi)$ with a column of $H$ strictly contained in the innermost nested cycle: the $i-k$ nested cycles found in $H''$, the cycle associated to $A$ and the $k-1$ nested cycles found in $H_A$.
    \end{enumerate}
   
    This concludes the induction.

    By Lemma \ref{layer_non_contractible_cycles}, 
    there are at least $\frac{\Delta(g)}{4m (\tilde{g}+1)^4} = f(g,m)$ columns of $p$ that are consecutive in $\Pi$ so that the faces that touches these columns induce no $\Pi$-noncontractible cycles.

    Hence, these faces induce a $1$-layer structure $H$ from $v$ of size $f(g,m)$. 

    Let's apply the property (1) to $H$ with $i=m$. We obtain that $G$ contains $m$ $\Pi$-contractible well-nested cycles pinched on $p$ with a column of $H$ strictly contained in the innermost nested cycle in $(G,\Pi)$. Finally, by taking the cycle on the boundary of this column as the last nested cycle, $G$ contains $m+1$ $\Pi$-well-nested cycles pinched on $p$ in $(G, \Pi)$. 
\end{proof}
    
    Suppose either $\Delta(G) > \Delta(g)$ or $\Delta_F(G, \Pi) > \Delta(g)$. Then there exists a piece $p$ of $(G, \Pi)$ with $\Delta(p) > \Delta(g)$. 
    
    Then, by Lemma \ref{big_degree_nested_cycles}, $(G, \Pi)$ contains $m+1$ $\Pi$-contractible well nested cycles pinched on $p$. However, this contradicts Proposition \ref{good_square}. Finally, $d(p) \leq \Delta(g)$ and therefore $\Delta(G) \leq \Delta(g)$ and $\Delta_F(G, \Pi) \leq \Delta(g)$.
\end{proof}

\subsection{Bound on the height of the tree decomposition of \texorpdfstring{$G$}{G}} \label{height_tree_decomposition}

We show that the height of the tree $T$ in a minimal tree decomposition of $G$ of width $w$ is bounded by a function of order $O(w) \times g^{O(\log^3 g)}$. We proceed by first proving that if the longest path in $T$ is of length at least some bound $P(g) = O(w) \times g^{O(\log^3 g)}$, then the subgraph $G_0$ of $G$ induced by the bags along this path induces a nonplanar embedding $\Pi(G_0)$. Using this result, we quickly show that the height of $T$ is bounded by some other bound $P'(g) = O(w) \times g^{O(\log^3 g)}$.


\begin{lem}
    \label{property_minimal_tree_decomposition}
    Let $H$ be a graph and let $(T, (V_t)_{t \in T})$ be a minimal tree decomposition of $H$. Let $P = t_1 - ... - t_m$ be a path in $T$ ($m \geq 1$). Then, for every $1 \leq i \leq m$, $V_{t_i} \nsubseteq \cup_{j < i} V_{t_j}$.
\end{lem}

\begin{proof}
    The result for $i=1$ is trivial.
    
    Let $2 \leq i \leq m$. As $(T, (V_t)_{t \in T})$ is a minimal tree decomposition of $H$, $V_{t_i} \nsubseteq V_{t_{i-1}}$. Hence, there exists a vertex $v_i \in V_{t_i}$ that is not in $V_{t_{i-1}}$. Moreover, by property of the tree decomposition, $v_i \notin \cup_{j < i} V_{t_j}$.

    Finally, $V_{t_i} \nsubseteq \cup_{j < i} V_{t_j}$.
\end{proof}

\begin{defi}[Radius]
    Let $H$ be a graph with embedding $\Pi_H$ in a surface.
    Let $C$ be a $\Pi_H$-contractible cycle of $H$ and let $G_C = \text{Int}(C, \Pi_H)$. We define the \textit{radius} $rad(f)$ of a face $f$ in $\mathcal{F}(C, \Pi_H)$ inductively as follow:
    The faces that share a vertex with $V(C)$ have radius $1$. For $i \geq 2$, suppose we have already determined the faces of radius $1$ to $i$. Then, the faces of radius $i+1$ are those that are not of radius at most $i$ but share a vertex with the faces of radius $i$.
    We say that $G_C$ has \textit{radius $m$} if $\max\{ rad(f) | f \in \mathcal{F}(C, \Pi_H)\} = m$.

    An illustration can be found in Figure \ref{fig:radius_graph}.

    \vspace{0.3cm}

    Let $C, C'$ be $\Pi_H$-noncontractible homotopic cycles of $H$ and let $G_{C,C'} = \text{Int}(C \cup C', \Pi_H)$. We define the \textit{radius} $rad(f)$ of a face $f$ in $\mathcal{F}(C \cup C', \Pi_H)$ in $\Pi_H$ inductively as follow:
    The faces that share a vertex with $V(C \cup C')$ have radius $1$. For $i \geq 2$, suppose we have already determined the faces of radius $1$ to $i$. Then, the faces of radius $i+1$ are those that are not of radius at most $i$ but share a vertex with the faces of radius $i$.
    We say that $G_{C,C'}$ has \textit{radius $m$} if $\max\{ rad(f) | f \in \mathcal{F}(C \cup C', \Pi_H)\} = m$.
\end{defi}

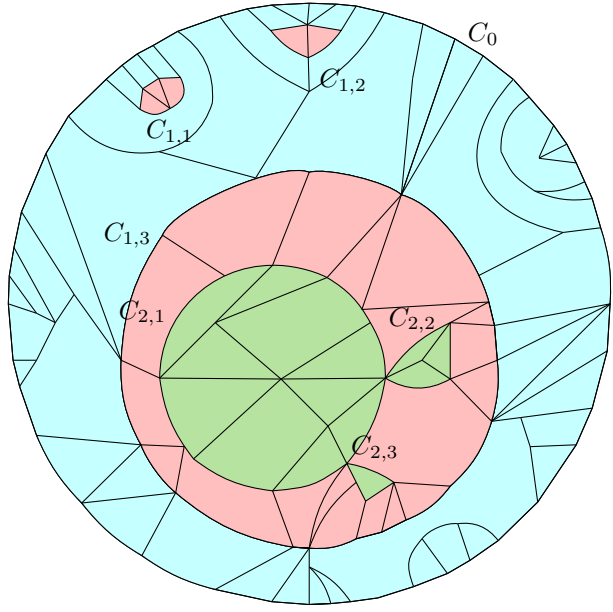
\begin{figure}[h!]
    \centering
    \input{images/radius_graph.tikz}
    \caption{Radius of a contractible graph. The three colors indicate the radius of the faces with respect to the outer cycle $C_0$: faces in blue, pink, and green are respectively of radius $1$, $2$, and $3$.}
    \label{fig:radius_graph}
\end{figure}

Let $C$ be a $\Pi$-contractible cycle. We define $\mathcal{B}(C)$ to be the faces in $\text{Int}(C,\Pi)$ that touch $C$. Let $(T, (V_t)_{t \in T})$ be a tree decomposition of a graph $H$. Let $P$ be a path in $T$ with endpoints $t_1$ and $t_2$. We define $\overline{P}$ to be the component of $T-\{t_1,t_2\}$ that contains the interior of $P$.

\begin{prop}
    \label{no_long_planar_path}
	Let $(T,(V_t)_{t \in T})$ be a minimal tree decomposition of $G$ of width $w$.
    Let $P$ be a path from $t_1$ to $t_2$ of length $P(g,w)$ in $T$ with $P(g,w) = \frac{\Delta(g)(\Delta(g)^{2m}-1)}{\Delta(g)-1} \times 2w + w+2$. 
    Let $G_0 = \bigcup_{t \in \overline{P}} V_t - (V_{t_1} \cup V_{t_2})$. Then $\Pi(G_0)$ is not an embedding in a disk on $S$.
\end{prop}

\begin{proof}
    Suppose by contradiction that $\Pi(G_0)$ is an embedding in a disk on $S$.

    First, let's show that $|G_0| \geq P(g,w)$. As $(T, (V_t)_{t \in T})$ is a minimal tree decomposition of $G$, by Lemma \ref{property_minimal_tree_decomposition}, for every vertex $t \in \overline{P}$, there exists a vertex $v_t \in V_t$ such that $v_t \notin \cup_{P_t} V_{t'}$ with $P_t$ being the path from $t_1$ (included) to $t$ (excluded). Therefore, the set $U = \{v_t, t \in \overline{P}\}$ has size $|V(\overline{P})| = P(g,w)-1$. Moreover, observe that $U$ does not contain, by construction, any vertex from $V_{t_1}$ and contains at most $w$ vertices from $V_{t_2}$. Finally, as $U - V_{t_2} \subseteq G_0$, we get $|G_0| \geq |U| - w = P(g,w) - (w+1)$.

    Then, let's show a contradiction by proving $|G_0| < P(g,w) - (w+1)$. Let $G' = \bigcup_{t \notin \overline{P}} V_t$. Let $A = \{ v \in V(G_0) \cap V(G')\}$ be the vertices that lie in both $G_0$ and $G'$. We call the vertices in $A$ "attach" vertices of $G_0$. Remark that $|A| \leq 2 w$.

    By hypothesis, $G_0$ is contained in a disk in $\Pi$. 
    Let $C_0$ be the cycle bounding $G_0$ in $\Pi$. Remark that $G_0 \subseteq \text{Int}(C_0, \Pi)$ and let $G_0' = \text{Int}(C_0, \Pi)$. We will show that $|G_0'| < P(g,w)-(w+1)$, therefore we will have $|G_0| \leq |G_0'| < P(g,w) -(w+1)$.

    For a cycle $C$ in $G_0'$, let $E(C)^+$ be the set of edges adjacent to $C$ in $\text{Int}(C,\Pi)$. By Theorem \ref{max_degree}, every vertex of $C$ is adjacent to at most $\Delta(g)$ edges, therefore $|E(C)^+| \leq \Delta(g) \times |V(C)|$. Let $C$ be a cycle in $G_0'$ and let $\mathcal{C}(C)$ be the cycles so that $\mathcal{B}(C)$ contains exactly the faces inside $\text{Int}(C, \Pi) - \bigcup_{C' \in \mathcal{C}(C)} \text{Int}(C', \Pi)$.

    \setcounter{claim}{0}

    \begin{claim}
        \label{size_cycles}
         Let $C$ be a cycle in $G_0'$. \[\sum_{C' \in \mathcal{C}(C)} |V(C')| \leq \Delta(g)^2 \times |V(C)|\]
    \end{claim}

    \begin{proof_claim}
        Let $H_C$ be the graph containing all the bridges inside $\text{int}(C, \Pi) - \cup_{C' \in \mathcal{C}(C)} \text{Int}(C', \Pi)$. Let $A(C)$ be the set of vertices in $H_C$ that have attaches on some cycle in $\mathcal{C}(C)$. 

        Let's first prove that $A(C)$ has order at most $\Delta(g) \times |V(C)|$. Remark that, by construction, there are fewer edges with an endpoint in $\cup_{C'\in \mathcal{C}(C)} C'$ in $H_C$ than to $C$ in $H_C$. Hence, we have $|A(C)| \leq |E(C)^+| \leq \Delta(g) \times |V(C)|$.
            
        Let $P$ be a path on a cycle $C'$ of $\mathcal{C}(C)$ between two consecutive attach vertices $a, a' \in A(C')$ in $H_C$. Then, a $\Pi$-facial walk contains $P$. Therefore, by Theorem \ref{max_degree}, $|P| \leq \Delta(g)$.

        Moreover, as $|A(C)| \leq \Delta(g) \times |V(C)|$, $\cup_{C' \in \mathcal{C}(C)} C'$ contains at most $\Delta(g) \times |V(C)|$ such paths. Finally, $\sum_{C' \in \mathcal{C}(C)} |V(C')| \leq \Delta(g)^2 \times |V(C)|$.
    \end{proof_claim}
	
	\vspace{0.3cm}

    By Proposition \ref{good_square}, the radius of $G_0'$ with embedding $\Pi(G_0')$ is at most $m$.
    Let's partition the faces of $(G'_0, \Pi(G_0'))$ into their radius class with the outer face boundary $C_0$: for $0 \leq i \leq m-1$, let $\mathcal{F}_i$ be the set of faces of $(G_0', \Pi(G_0'))$ of radius $i+1$. Let's define inductively a set of cycles $\mathcal{C}_i$ so that $\cup_{C \in \mathcal{C}_i} \mathcal{B}(C) = \mathcal{F}_i$ and $\sum_{C \in \mathcal{C}_i} |V(C)| \leq |V(C_0)| \times \Delta(g)^{2i}$.

    We define $\mathcal{C}_0 = \{C_0\}$. We have indeed that $\mathcal{B}(C_0) = \mathcal{F}_0$ and $|V(C_0)| = |V(C_0)| \times 1$.
    Now, suppose that for $0 \leq i \leq m-1$, $\mathcal{C}_i$ is defined and is so that $\cup_{C \in \mathcal{C}_i} \mathcal{B}(C) = \mathcal{F}_i$ and $\sum_{C \in \mathcal{C}_i} |V(C)| \leq |V(C_0)| \times \Delta(g)^{2i}$. Let's define $\mathcal{C}_{i+1}$. For $C \in \mathcal{C}_i$, let $\mathcal{C}(C)$ be the set of cycles so that $\mathcal{B}(C)$ contains exactly the faces inside $\text{Int}(C, \Pi) - \bigcup_{C' \in \mathcal{C}(C)} \text{Int}(C', \Pi)$. Then, by the Claim \ref{size_cycles}, $\sum_{C' \in \mathcal{C}(C)} |V(C')| \leq \Delta(g)^2 \times |V(C)|$. We define $\mathcal{C}_{i+1} = \cup_{C \in \mathcal{C}_i} \cup_{C' \in \mathcal{C}(C)} C'$. Then, $\sum_{C \in \mathcal{C}_{i+1}} |V(C)| = \sum_{C \in \mathcal{C}_i} \sum_{C' \in \mathcal{C}(C)} |V(C')| \leq \sum_{C \in \mathcal{C}_i} \Delta(g)^2 \times |V(C)| \leq |V(C_0)| \times \Delta(g)^{2i+2}$.
    Moreover, it is easy to verify that $\cup_{C \in C_{i+1}} \mathcal{B}(C)$ is exactly $\mathcal{F}_{i+1}$.

    In Figure \ref{fig:radius_graph}, the graph $G_0'$ represented has radius $3$, and the faces of radius $1$, $2$, and $3$ are respectively in blue, pink, and green. Moreover, we have $\mathcal{C}_0 = \{C_0\}$, $\mathcal{C}_1 = \{C_{1,1}, C_{1,2}, C_{1,3}\}$ and $\mathcal{C}_2 = \{C_{2,1}, C_{2,2}, C_{2,3}\}$.

    \vspace{0.3cm}

    Let $0 \leq i \leq m-2$.
    Remark that, by construction, the set $\mathcal{B}(C) \subseteq \mathcal{F}_{i}$ for $C \in \mathcal{C}_i$ induces a set of bridges $\mathcal{H}_C$ on $C$ and the cycles of $\mathcal{C}_{i+1}$. Each bridge in $\mathcal{H}_C$ is a tree whose leaves are on $C$ and whose roots are on some cycles in $\mathcal{C}_{i+1}$ (two distinct leaves or roots of these trees can lie on the same vertex of $C$).
    For $i = m-1$ and $C \in \mathcal{C}_{m-1}$, $\mathcal{B}(C)$ induces a set of bridges $\mathcal{H}_C$, and these bridges are trees whose leaves are on $C$ (two distinct leaves of these trees can lie on the same vertex of $C$).

    \begin{claim}
        \label{size_layers}
        For $0 \leq i \leq m-1$ and $C \in \mathcal{C}_i$, $C \cup \bigcup_{B \in \mathcal{H}_C} V(B)$ has size at most $(\Delta(g)+1)|V(C)|$.
    \end{claim}

    \begin{proof_claim}
        The trees in $\mathcal{H}_C$ have in total at most $|E(C)^+|$ leaves. 
        Moreover, by minimality of $G$, there is no induced path of length $2$ in any trees in $\mathcal{H}_C$. Therefore, these trees have more leaves than internal vertices.
        
        Finally, \[\sum_{B \in \mathcal{H}_C} |V(B)| \leq |E(C)^+| \text{\quad and \quad} |C| + \sum_{B \in  \mathcal{H}_C} |V(B)| \leq |E(C)^+| + |V(C)| \leq (\Delta(g)+1) \times |V(C)|\]
    \end{proof_claim}

	\vspace{0.3cm} 

    Remark that the set of vertices of $C_0$ on which edges of $\text{Ext}(C_0,\Pi)$ have an endpoint precisely in $A$. Let $P$ be a path on $C_0$ between two consecutive attach vertices $a, a' \in A$, then there is a $\Pi$-facial walk that contains $P$. Therefore, by Theorem \ref{max_degree}, $|P| \leq \Delta(g)$. Moreover, as $|A| \leq 2 w$, $C_0$ contains at most $2 w$ such paths. Finally, $C_0$ has size at most $2 w \times \Delta(g)$.

    By using Claim \ref{size_layers} and the inequality $\sum_{C \in \mathcal{C}_i} |V(C)| \leq |V(C_0)| \times \Delta(g)^{2i}$ for $0 \leq i \leq m-1$, we obtain that the order of $G'_0$ is bounded by

    \begin{align*}
        \sum_{i=0}^{m-1} \sum_{C \in \mathcal{C}_i} |V(C \cup \bigcup_{B \in \mathcal{H}_C} B)| &= \sum_{i=0}^{m-1} \sum_{C \in \mathcal{C}_i} (\Delta(g)+1)|V(C)| \\
        & = \sum_{i=0}^{m-1} (\Delta(g)+1)\Delta(g)^{2i} |V(C_0)| \\
        & \leq (\Delta(g)+1) \frac{\Delta(g)^{2m} - 1}{\Delta(g)^2-1}|V(C_0)| \\
        & \leq \frac{\Delta(g)^{2m}-1}{\Delta(g)-1}|V(C_0)| \\
        & \leq \frac{\Delta(g)(\Delta(g)^{2m}-1)}{\Delta(g)-1}\times 2w.
    \end{align*}

    Finally, $G_0'$ is of order at most $\frac{\Delta(g)(\Delta(g)^{2m}-1)}{\Delta(g)-1}\times 2w < P(g,w) - (w+1)$. 
\end{proof}

        


    

\begin{theo}
    \label{no_long_planar_path_T_extended}
    Let $(T,(V_t)_{t \in T})$ be a minimal tree decomposition of $G$ of width $w$.
    $T$ contains no path of length more than $P'(g,w) = (2m (3g+3)+1) \times P(g,w) -1 = O(w) \times g^{O(\log^3 g)}$. 
\end{theo}

\begin{proof}
    By contradiction, suppose a path $P$ in $T$ exists so that $|P| > P'(g,w)$. Let $P = p_0 - ... - p_{k-1}$ with $k > P'(g,w)$ and $q = \left\lfloor \frac{k}{P(g,w)} \right\rfloor \geq 2m(3g+3)+1$. For $0 \leq i \leq q-1$, let $P_i = p_{i \times P(g,w)} - ... - p_{i \times P(g,w) + P(g,w)-1}$. Then $(P_i)_{0 \leq i \leq q-1}$ are pairwise disjoint subpaths of $P$. For $0 \leq i \leq q-1$, let $G_i = \cup_{p \in \overline{P_i}} V_p - (V_{p_{i \times P(g,w)}} \cup V_{p_{i \times P(g,w) + P(g,w)-1}})$. Then, $(G_i)_{0 \leq i \leq q-1}$ are pairwise disjoint subgraphs of $G$ and, for $0 \leq i \leq q-2$, $G_i$ and $G_{i+1}$ are separated in $G$ by $V_{p_{i \times P(g,w) + P(g,w)-1}}$.

    As $P_i$ is of length $P(g,w)$ ($0 \leq i \leq q-1$), by Proposition \ref{no_long_planar_path}, $G_i$ is not contained in a disk in $\Pi$. Hence, $G_i$ contains a $\Pi$-noncontractible cycle $C_i$. As $(G_i)_{0 \leq i \leq q-1}$ are pairwise disjoint subgraphs of $G$, the $\Pi$-noncontractible cycles $(C_i)_{0 \leq i \leq m-1}$ are pairwise disjoint. However, recall that $q  = \left\lfloor \frac{k}{P(g,w)} \right\rfloor \geq 2m (3g+3) + 1 > 2m(3g+3)$ which is, by Corollary \ref{nb_non_contractible_cycles}, a contradiction. 
\end{proof}

\subsection{The order of  \texorpdfstring{$G$}{G} is quasi-polynomial} \label{size_quasi_polynomial}

The following results conclude our paper by showing that $G$ is of order $g^{O(\log^3 g)}$. It proceeds in two steps: first, it is easy to derive from Theorem \ref{no_long_planar_path_T_extended} that $G$ is of order bounded by a quasi single exponential $2^{g^{O(\log^3 g)}}$ in $g$; then, we deduce from this result a bound on the pathwidth of $G$, and use it to prove the quasi-polynomial bound on the order of $G$. Moreover, recall that we restricted ourselves to an excluded minor $G$ that is $2$-connected using Corollary \ref{G_2_connected}. We therefore conclude the bound on the order of any excluded minor for the surface $S$ using this result.

\begin{cor}
    \label{G_has_simple_exponential_size}

    \[ |V(G)| \leq 2^{Q(g)}\]

    with $Q(g) = g^{O(\log^3 g)}$ a quasi-polynomial in $g$ so that $Q(g) \geq \log((T(g)+1) \times \Delta_T(g)^{P'(g,T(g))})$.
\end{cor}

\begin{proof}
    Let $(T,(V_t)_{t \in T})$ be a minimal linked tree decomposition of $G$ of width $tw(G)$ (Thomas \cite{thomas} proved that, for any graph $H$, there exists a linked tree decomposition of $H$ of optimal width).

    First, by Theorem \ref{no_long_planar_path_T_extended}, the length of the longest path in $T$ is bounded by $P'(g,tw(G))$ and, as $tw(G) \leq T(g)$, it is bounded by $P'(g, T(g))$.

    Then, let's show that $|V(G)| \leq 2^{Q(g)}$. By Theorem \ref{treewidth}, we know that $tw(G) \leq T(g)$. Hence, every bag in the tree decomposition of $G$ is of order at most $T(g)+1$. Moreover, by Corollary \ref{seymour_tree_degree_corollary}, we know that the maximum degree of $T$ is at most $\Delta_T(g)$.

    Hence,

    \begin{align*}
        |V(G)| & \leq (tw(G)+1) \times |V(T)| \\
        & \leq (tw(G)+1) \times (\Delta(T)^{P'(g,T(g))}) \\
        & \leq (T(g)+1) \times \Delta_T(g)^{P'(g,T(g))} \\
        & \leq 2^{Q(g)} 
    \end{align*}

\end{proof}

\begin{cor}
    \label{pathwidth_polynomial_in_g}
     There exists a constant $A$ so that

    \[pw(G) \leq R(g)\]

    with $R(g) = A \times T(g) \times Q(g)$.
\end{cor}

\begin{proof}
    The following inequality gives a relation between pathwidth and treewidth \cite{bodlaender}:

    \[ pw(G) = O(tw(G) \log(|V(G)|))\]

    Hence, there exists a constant $A$ so that $pw(G) \leq A tw(G) \log(|V(G)|)$. Moreover, by Corollary \ref{G_has_simple_exponential_size}, we know that $|V(G)| \leq 2^{Q(g)}$. Finally, we deduce that

    \[ pw(G) \leq A tw(G) Q(g).\]

    As $tw(G) \leq T(g)$ by Theorem \ref{treewidth}, $A tw(G) Q(g) \leq A \times T(g) \times Q(g) = R(g)$. 
\end{proof}

\begin{cor}
    \label{G_has_polynomial_size}
    There exists a constant $A'$ so that 

    \[ |V(G)| \leq S(g) \]

    with $S(g) = A' \times P'(g, R(g)) \times R(g)$.
\end{cor}

\begin{proof}
    Let $(P, (V_p)_{p \in P})$ be a minimal path decomposition of $G$ of optimal width $pw(G)$.

    Then, by Theorem \ref{no_long_planar_path_T_extended}, the length of path $P$ is bounded by $P'(g,pw(G))$.

    Moreover, by Corollary \ref{pathwidth_polynomial_in_g}, $pw(G) \leq R(g)$. We have

    \[ |V(G)| \leq P'(g,pw(G)) \times (pw(G)+1) \leq P'(g, R(g)) \times (R(g)+1).\]

    Let $A'$ be a constant so that $P'(g, R(g)) \times (R(g)+1) \leq A' \times P'(g, R(g)) \times R(g) = S(g)$. Then, finally, $|V(G)| \leq S(g)$.
\end{proof}

To conclude the proof, let's recall Corollary \ref{G_2_connected}:

\begin{repcor}{G_2_connected}
    Suppose that, for any graph $H$ that is an excluded minor for some surface $S_H$ and that is $2$-connected, $|V(H)| \leq N(g(S_H))$ with $N$ an increasing function.

    Then, $|V(G)| \leq (g+2) \times N(g)$.
\end{repcor}

Therefore, by using Corollary \ref{G_2_connected}, we finally prove our main result:

\begin{reptheo}{main_result}
    Let $S$ be a given surface of the Euler genus $g$. Every excluded minor for $S$ is of order at most $U(g) = g^{O(\log^3 g)}$.
\end{reptheo}

\begin{proof}
    By combining Corollaries \ref{G_2_connected} and \ref{G_has_polynomial_size}, we get that an excluded minor $G$ for $S$ is of order at most 

    \[ |V(G)| \leq S(g) \times (g+2).\]

    Let $U(g) = S(g) \times (g+2)$ for $g \in \mathbb{N}$. We indeed have $U(g) = g^{O(\log^3 g)}$.
\end{proof}

\section{Conclusion} \label{sec:conclusion}

The question of knowing whether $G$ is bounded by a polynomial function or whether a quasi-polynomial bound is asymptotically the best possible bound remains open.

\begin{prob}
    Is there a polynomial bound on the order of $G$? Is a quasi-polynomial bound asymptotically the best possible one?
\end{prob}

Our proof contains two main bottlenecks:
\begin{itemize}
    \item the $O(\log g)$ bound for the number of $\Pi$-contractible well-nested cycles and $\Pi$-well-homotopic cycles, and
    \item the $g^{O(\log^2 g)}$ bound for the maximum degree of $G$ and for the maximum size of a face of $\Pi$.
\end{itemize} 

As said in Subsection \ref{subsec:nested_squares}, we believe that the $O(\log g)$ bound is asymptotically the best possible for the number of $\Pi$-contractible well-nested cycles and $\Pi$-well-homotopic cycles. An interesting question would be to know whether it is possible to bound the maximum degree of $G$ and the size of a face of $\Pi$ by a polynomial in $g$ instead of the quasi-polynomial bound $g^{O(\log^2 g)}$. 

\begin{prob}
    Is there a polynomial function $\Delta_P$ of $g$ such that 
    $\Delta(G) \leq \Delta_P(g)$ and $\Delta_F(G, \Pi) \leq \Delta_P(g)$?
\end{prob}

It is a necessary condition to get a polynomial bound on the order of $G$. However, this would not be sufficient, combined with our proof, to prove that the order of $G$ is bounded by a polynomial in $g$.


\section*{Appendix: Lower bound for the excluded minors for a surface}

A tight lower bound for the excluded minors for a surface can be quickly obtained by using the known results on the genus of cliques. The lower bound $L(g)$ for the excluded minors for a surface is therefore $L(g) = \left\lfloor \frac{7 + \sqrt{1+24g}}{2} \right \rfloor +1 = \Theta(\sqrt{g})$, as proved in Theorem \ref{lower_bound}. We can remark that this lower bound is closely related to Heawood number $H(g) = \left\lfloor \frac{7 + \sqrt{1+24g}}{2} \right \rfloor$, which is a tight upper bound on the chromatic number of graphs of genus at most $g$ \cite[Theorem 8.3.1]{graphs_on_surfaces}. Indeed, we have $L(g) = H(g) + 1$.

\begin{prop}[{\cite[Theorems 4.4.5 and 4.4.6]{graphs_on_surfaces}}]
    \label{genus_clique}
    For $k \geq 3$,
    \[g(K_k) = \left\lceil \frac{(k-3)(k-4)}{6}\right\rceil\]
\end{prop}

\begin{lem}
    \label{lower_bound_calculus}
    Let $g$ be an integer and let $L(g) = \min \{ k \geq 3, g(K_k) > g\}$. Then,

    \[ L(g) = \left\lfloor \frac{7 + \sqrt{1+24g}}{2} \right \rfloor +1\]
\end{lem}

\begin{proof}
    The result is obtained by solving the following quadratic equation in $k$:

    \[ g(K_k) = \left\lceil \frac{(k-3)(k-4)}{6}\right\rceil > g\]
\end{proof}

\begin{reptheo}{lower_bound}
    Let $S$ be a surface of genus $g$. Let $L(g)$ be the function in Lemma \ref{lower_bound_calculus}. There exists an excluded minor for $S$ of order 

    \[ |V(G)| = L(g) \]

    and no excluded minor for $S$ is of order $< L(g)$.
\end{reptheo}

\begin{proof}
    By definition of $L(g)$, the clique $K_{L(g)}$ cannot be embedded in $S$. Therefore, $K_{L(g)}$ contains an excluded minor $G$ for $S$ as a minor with $|V(G)| \leq L(g)$.
    
    Moreover, by definition of $L(g)$, every clique $K_\ell$ with $\ell < L(g)$ can be embedded in $S$. And, as every graph of order $\ell < L(g)$ is a minor of $K_\ell$, no excluded minor for $S$ is of order $< L(g)$.
    
    Finally, we proved that there exists an excluded minor $G$ for $S$ with $|V(G)| = L(g)$ and no excluded minor for $S$ is of order $< L(g)$.
\end{proof}

\bibliographystyle{plain}
\bibliography{bibliography}

\end{document}

%% file: images/samples.tikzstyles
\tikzstyle{basic}=[fill=black, draw=black, shape=circle, scale=0.5]
\tikzstyle{new style 0}=[fill={rgb,255: red,16; green,112; blue,255}, draw={rgb,255: red,16; green,112; blue,255}, shape=circle]

\tikzstyle{new edge style 0}=[-, draw={rgb,255: red,191; green,0; blue,64}]
\tikzstyle{new edge style 1}=[-, draw={rgb,255: red,17; green,0; blue,255}]
\tikzstyle{new edge style 2}=[-, draw={rgb,255: red,20; green,153; blue,0}]
\tikzstyle{new edge style 3}=[-, dashed]
\tikzstyle{new edge style 4}=[-, fill={rgb,255: red,198; green,255; blue,255}]
\tikzstyle{new edge style 5}=[-, fill=white]
\tikzstyle{new edge style 6}=[-, fill={rgb,255: red,255; green,191; blue,191}]
\tikzstyle{new edge style 7}=[-, fill={rgb,255: red,182; green,227; blue,160}, draw=black]
\tikzstyle{new edge style 8}=[-, draw=none]
\tikzstyle{new edge style 9}=[-, fill={rgb,255: red,255; green,191; blue,191}]
\tikzstyle{new edge style 10}=[-, draw={rgb,255: red,16; green,112; blue,255}, fill={rgb,255: red,198; green,255; blue,255}]
\tikzstyle{new edge style 11}=[-, fill={rgb,255: red,191; green,191; blue,191}]
\tikzstyle{new edge style 12}=[-, draw=white, fill=white]
\tikzstyle{new edge style 13}=[-, fill={rgb,255: red,255; green,191; blue,191}, draw=white]
\tikzstyle{new edge style 14}=[-, fill={rgb,255: red,198; green,255; blue,255}, draw=white]
\tikzstyle{new edge style 15}=[-, fill={rgb,255: red,182; green,227; blue,160}, draw={rgb,255: red,16; green,112; blue,255}]

%% file: images/almost_disjoint_cycles.tikz
\begin{tikzpicture}[scale=1.5]
	\begin{pgfonlayer}{nodelayer}
		\node [style=basic] (0) at (-4, 3) {};
		\node [style=basic] (1) at (0, 1) {};
		\node [style=basic] (2) at (4, 2) {};
		\node [style=none] (3) at (-6.5, 4.25) {};
		\node [style=none] (4) at (-5, 0.75) {};
		\node [style=none] (5) at (-2.5, 1.5) {};
		\node [style=none] (6) at (0, 3) {};
		\node [style=none] (7) at (0, -1) {};
		\node [style=none] (8) at (2.5, 0.5) {};
		\node [style=none] (9) at (4, 4) {};
		\node [style=none] (10) at (6, 1) {};
	\end{pgfonlayer}
	\begin{pgfonlayer}{edgelayer}
		\draw [in=105, out=60, looseness=1.50] (3.center) to (0);
		\draw [in=150, out=-135, looseness=1.25] (0) to (4.center);
		\draw [in=-150, out=-120, looseness=1.50] (3.center) to (0);
		\draw [in=-15, out=-90] (0) to (4.center);
		\draw [in=135, out=75, looseness=1.25] (5.center) to (1);
		\draw [in=0, out=45] (1) to (6.center);
		\draw [in=135, out=-180, looseness=1.25] (6.center) to (1);
		\draw [in=-150, out=-90] (5.center) to (1);
		\draw [in=180, out=-135, looseness=1.50] (1) to (7.center);
		\draw [in=0, out=-45, looseness=1.50] (1) to (7.center);
		\draw [in=135, out=180, looseness=1.75] (9.center) to (2);
		\draw [in=45, out=0, looseness=1.50] (9.center) to (2);
		\draw [in=120, out=-165, looseness=1.25] (2) to (8.center);
		\draw [in=-30, out=-90, looseness=1.25] (2) to (8.center);
		\draw [in=60, out=15, looseness=1.25] (2) to (10.center);
		\draw [in=-135, out=-60, looseness=1.50] (2) to (10.center);
	\end{pgfonlayer}
\end{tikzpicture}

%% file: images/2_separation.tikz
\begin{tikzpicture}[scale=1.25]
	\begin{pgfonlayer}{nodelayer}
		\node [style=basic, label={below left:$a$}] (0) at (-4, 3) {};
		\node [style=basic, label={below left:$b$}] (1) at (3, -1) {};
		\node [style=none] (2) at (-5.25, 3.25) {};
		\node [style=none] (3) at (-4.75, 3.75) {};
		\node [style=none] (4) at (-4, 4.25) {};
		\node [style=none] (5) at (4.25, -0.5) {};
		\node [style=none] (6) at (4.25, -1.25) {};
		\node [style=none] (7) at (3.5, -2) {};
		\node [style=none] (8) at (2.75, -2.25) {};
		\node [style=none] (9) at (-4, 5) {};
		\node [style=none] (10) at (-5.25, 4.25) {};
		\node [style=none] (11) at (-6.1, 3.475) {};
		\node [style=none] (12) at (2.55, -3.075) {};
		\node [style=none] (13) at (3.825, -2.625) {};
		\node [style=none] (14) at (5.025, -1.425) {};
		\node [style=none] (15) at (5.025, -0.275) {};
		\node [style=none] (16) at (-3.5, 1.25) {};
		\node [style=none] (17) at (-2.5, 2.25) {};
		\node [style=none] (18) at (-3.5, 2.25) {};
		\node [style=none] (19) at (-0.75, 2) {};
		\node [style=none] (20) at (-2.75, -0.25) {};
		\node [style=none] (21) at (-2, 0.5) {};
		\node [style=none] (22) at (-0.75, 0.75) {};
		\node [style=none] (23) at (-1, -0.5) {};
		\node [style=none] (24) at (1.5, 0.25) {};
		\node [style=none] (25) at (1, -1.25) {};
		\node [style=none] (26) at (2.25, 0) {};
		\node [style=none] (27) at (1.5, 1.25) {};
		\node [style=none] (28) at (2.25, 2) {};
		\node [style=none] (29) at (3, 0.25) {};
		\node [style=none, label={$B$}] (30) at (0.25, 3) {};
	\end{pgfonlayer}
	\begin{pgfonlayer}{edgelayer}
		\draw (1) to (8.center);
		\draw (1) to (7.center);
		\draw (1) to (6.center);
		\draw (1) to (5.center);
		\draw (0) to (2.center);
		\draw (0) to (3.center);
		\draw (0) to (4.center);
		\draw [style=new edge style 3] (5.center) to (15.center);
		\draw [style=new edge style 3] (6.center) to (14.center);
		\draw [style=new edge style 3] (7.center) to (13.center);
		\draw [style=new edge style 3] (8.center) to (12.center);
		\draw [style=new edge style 3] (4.center) to (9.center);
		\draw [style=new edge style 3] (3.center) to (10.center);
		\draw [style=new edge style 3] (2.center) to (11.center);
		\draw (0) to (18.center);
		\draw (0) to (16.center);
		\draw (0) to (17.center);
		\draw (18.center) to (17.center);
		\draw (16.center) to (17.center);
		\draw (18.center) to (16.center);
		\draw (16.center) to (20.center);
		\draw (16.center) to (21.center);
		\draw (17.center) to (19.center);
		\draw (19.center) to (22.center);
		\draw (22.center) to (21.center);
		\draw (20.center) to (23.center);
		\draw (23.center) to (25.center);
		\draw (25.center) to (1);
		\draw (22.center) to (23.center);
		\draw (22.center) to (25.center);
		\draw (21.center) to (23.center);
		\draw (19.center) to (28.center);
		\draw (19.center) to (27.center);
		\draw (19.center) to (24.center);
		\draw (24.center) to (26.center);
		\draw (26.center) to (25.center);
		\draw (24.center) to (27.center);
		\draw (27.center) to (28.center);
		\draw (28.center) to (26.center);
		\draw (28.center) to (29.center);
		\draw (29.center) to (1);
		\draw (29.center) to (26.center);
	\end{pgfonlayer}
\end{tikzpicture}

%% file: images/isolated_paths_disjoint.tikz
\begin{tikzpicture}[scale=0.9]
	\begin{pgfonlayer}{nodelayer}
		\node [style=basic] (0) at (4, 2) {};
		\node [style=basic] (1) at (3, 2) {};
		\node [style=basic] (2) at (2, 2) {};
		\node [style=basic] (3) at (1, 2) {};
		\node [style=basic] (4) at (0, 2) {};
		\node [style=basic] (5) at (-1, 2) {};
		\node [style=basic] (6) at (-2, 2) {};
		\node [style=basic] (7) at (-3, 2) {};
		\node [style=basic] (8) at (-4, 2) {};
		\node [style=basic] (9) at (4, -2) {};
		\node [style=basic] (10) at (3, -2) {};
		\node [style=basic] (11) at (2, -2) {};
		\node [style=basic] (12) at (1, -2) {};
		\node [style=basic] (13) at (0, -2) {};
		\node [style=basic] (14) at (-1, -2) {};
		\node [style=basic] (15) at (-2, -2) {};
		\node [style=basic] (16) at (-3, -2) {};
		\node [style=basic] (17) at (-4, -2) {};
		\node [style=none] (18) at (-5, 3.75) {};
		\node [style=none] (19) at (-4, 5.5) {};
		\node [style=none] (20) at (-2.5, 6) {};
		\node [style=none] (21) at (-1, 6.25) {};
		\node [style=none] (23) at (1, 6.25) {};
		\node [style=none] (24) at (2.5, 6) {};
		\node [style=none] (25) at (4, 5.5) {};
		\node [style=none] (26) at (5, 3.75) {};
		\node [style=none] (27) at (-5, -3.75) {};
		\node [style=none] (28) at (-4, -5.5) {};
		\node [style=none] (29) at (-2.5, -6) {};
		\node [style=none] (30) at (-1, -6.25) {};
		\node [style=none] (31) at (1, -6.25) {};
		\node [style=none] (32) at (2.5, -6) {};
		\node [style=none] (33) at (4, -5.5) {};
		\node [style=none] (34) at (5, -3.75) {};
	\end{pgfonlayer}
	\begin{pgfonlayer}{edgelayer}
		\draw (8) to (17);
		\draw (7) to (16);
		\draw (6) to (15);
		\draw (5) to (14);
		\draw (4) to (13);
		\draw (3) to (12);
		\draw (2) to (11);
		\draw (1) to (10);
		\draw (0) to (9);
		\draw (8)
			 to (7.center)
			 to (6.center)
			 to (5.center)
			 to (4.center)
			 to (3.center)
			 to (2.center)
			 to (1.center)
			 to (0);
		\draw (17)
			 to (16.center)
			 to (15.center)
			 to (14.center)
			 to (13.center)
			 to (12.center)
			 to (11.center)
			 to (10.center)
			 to (9);
		\draw [style=new edge style 3, bend left] (8) to (18.center);
		\draw [style=new edge style 3, bend left] (18.center) to (19.center);
		\draw [style=new edge style 3, bend left=5] (19.center) to (20.center);
		\draw [style=new edge style 3] (20.center) to (21.center);
		\draw [style=new edge style 3] (21.center) to (23.center);
		\draw [style=new edge style 3] (23.center) to (24.center);
		\draw [style=new edge style 3, bend left=5] (24.center) to (25.center);
		\draw [style=new edge style 3, bend left] (25.center) to (26.center);
		\draw [style=new edge style 3, bend left] (26.center) to (0);
		\draw [style=new edge style 3, bend left] (9) to (34.center);
		\draw [style=new edge style 3, bend left] (34.center) to (33.center);
		\draw [style=new edge style 3, bend left=5] (33.center) to (32.center);
		\draw [style=new edge style 3] (32.center) to (31.center);
		\draw [style=new edge style 3] (31.center) to (30.center);
		\draw [style=new edge style 3] (30.center) to (29.center);
		\draw [style=new edge style 3, bend left=5] (29.center) to (28.center);
		\draw [style=new edge style 3, bend left] (28.center) to (27.center);
		\draw [style=new edge style 3, bend left] (27.center) to (17);
	\end{pgfonlayer}
\end{tikzpicture}

%% file: images/isolated_paths_semi_disjoint.tikz
\begin{tikzpicture}[scale=1]
	\begin{pgfonlayer}{nodelayer}
		\node [style=basic] (0) at (0, 3) {};
		\node [style=basic] (1) at (-4, 0) {};
		\node [style=basic] (2) at (-3, 0) {};
		\node [style=basic] (3) at (-2, 0) {};
		\node [style=basic] (4) at (-1, 0) {};
		\node [style=basic] (5) at (0, 0) {};
		\node [style=basic] (6) at (1, 0) {};
		\node [style=basic] (7) at (2, 0) {};
		\node [style=basic] (8) at (3, 0) {};
		\node [style=basic] (9) at (4, 0) {};
		\node [style=none] (10) at (-5, -1.75) {};
		\node [style=none] (11) at (-4, -3.5) {};
		\node [style=none] (12) at (-2.5, -4) {};
		\node [style=none] (13) at (-1, -4.25) {};
		\node [style=none] (14) at (1, -4.25) {};
		\node [style=none] (15) at (2.5, -4) {};
		\node [style=none] (16) at (4, -3.5) {};
		\node [style=none] (17) at (5, -1.75) {};
	\end{pgfonlayer}
	\begin{pgfonlayer}{edgelayer}
		\draw (0) to (1);
		\draw (0) to (2);
		\draw (0) to (3);
		\draw (0) to (4);
		\draw (0) to (5);
		\draw (0) to (6);
		\draw (0) to (7);
		\draw (0) to (8);
		\draw (0) to (9);
		\draw (1)
			 to (2.center)
			 to (3.center)
			 to (4.center)
			 to (5.center)
			 to (6.center)
			 to (7.center)
			 to (8.center)
			 to (9);
		\draw [style=new edge style 3, bend right] (1) to (10.center);
		\draw [style=new edge style 3, bend right] (10.center) to (11.center);
		\draw [style=new edge style 3, bend right=5] (11.center) to (12.center);
		\draw [style=new edge style 3] (12.center) to (13.center);
		\draw [style=new edge style 3] (13.center) to (14.center);
		\draw [style=new edge style 3] (14.center) to (15.center);
		\draw [style=new edge style 3, bend right=5] (15.center) to (16.center);
		\draw [style=new edge style 3, bend right] (16.center) to (17.center);
		\draw [style=new edge style 3, bend right] (17.center) to (9);
	\end{pgfonlayer}
\end{tikzpicture}

%% file: images/isolated_paths_joint.tikz
\begin{tikzpicture}[scale=1.5]
	\begin{pgfonlayer}{nodelayer}
		\node [style=basic] (0) at (0, 2.5) {};
		\node [style=basic] (1) at (0, -2.5) {};
	\end{pgfonlayer}
	\begin{pgfonlayer}{edgelayer}
		\draw [bend left=75, looseness=1.75] (0) to (1);
		\draw [bend left=60, looseness=1.50] (0) to (1);
		\draw [bend left=45, looseness=1.25] (0) to (1);
		\draw [bend left] (0) to (1);
		\draw [bend right=60, looseness=1.50] (0) to (1);
		\draw [bend right=45, looseness=1.25] (0) to (1);
		\draw [bend right] (0) to (1);
		\draw [bend right=75, looseness=1.75] (0) to (1);
	\end{pgfonlayer}
\end{tikzpicture}

%% file: images/isolated_paths_3.tikz
\begin{tikzpicture}
	\begin{pgfonlayer}{nodelayer}
		\node [style=none] (0) at (-2, 1.5) {};
		\node [style=none] (1) at (-5, 1.5) {};
		\node [style=none] (2) at (-5, -1.5) {};
		\node [style=none] (3) at (-2, -1.5) {};
		\node [style=none] (4) at (2, 1.5) {};
		\node [style=none] (5) at (5, 1.5) {};
		\node [style=none] (6) at (2, -1.5) {};
		\node [style=none] (7) at (5, -1.5) {};
		\node [style=none] (8) at (2, 0.5) {};
		\node [style=none] (9) at (2, -0.25) {};
		\node [style=none] (10) at (3, -1.5) {};
		\node [style=none, label={below:$\tilde{C}_2$}] (11) at (3.5, -1.5) {};
		\node [style=none] (12) at (5, -0.5) {};
		\node [style=none] (13) at (3.5, -0.5) {};
		\node [style=none] (14) at (4, 0) {};
		\node [style=none] (15) at (3, 0.5) {};
		\node [style=none] (16) at (3.75, 1.5) {};
		\node [style=none] (17) at (3, 1.5) {};
		\node [style=none] (18) at (-2, 0.5) {};
		\node [style=none] (19) at (-2, -0.75) {};
		\node [style=none] (20) at (-2.75, -1.5) {};
		\node [style=none, label={below:$\tilde{C}_1$}] (21) at (-3.75, -1.5) {};
		\node [style=none] (22) at (-5, -0.25) {};
		\node [style=none] (23) at (-4, 1.5) {};
		\node [style=none] (24) at (-3.5, 1.5) {};
		\node [style=none] (25) at (-2.75, 1.5) {};
		\node [style=none] (26) at (-4, 0.5) {};
		\node [style=none] (27) at (-2.75, 0.5) {};
		\node [style=none] (28) at (-3.25, -0.25) {};
		\node [style=none] (29) at (-4, -0.25) {};
		\node [style=none, label={below:$C$}] (30) at (0, -1.5) {};
	\end{pgfonlayer}
	\begin{pgfonlayer}{edgelayer}
		\draw [style=new edge style 4] (3.center)
			 to (0.center)
			 to (4.center)
			 to (6.center)
			 to cycle;
		\draw (3.center) to (2.center);
		\draw (2.center) to (1.center);
		\draw (1.center) to (0.center);
		\draw (4.center) to (5.center);
		\draw (5.center) to (7.center);
		\draw (7.center) to (6.center);
		\draw (17.center) to (15.center);
		\draw (15.center) to (8.center);
		\draw (9.center) to (15.center);
		\draw (15.center) to (13.center);
		\draw (13.center) to (14.center);
		\draw (15.center) to (14.center);
		\draw (14.center) to (16.center);
		\draw (14.center) to (12.center);
		\draw (14.center) to (11.center);
		\draw (13.center) to (10.center);
		\draw [style=new edge style 6] (26.center) to (22.center);
		\draw [style=new edge style 6] (29.center) to (26.center);
		\draw [style=new edge style 6] (29.center) to (21.center);
		\draw [style=new edge style 6] (21.center) to (28.center);
		\draw [style=new edge style 6] (28.center) to (29.center);
		\draw [style=new edge style 6] (26.center) to (28.center);
		\draw [style=new edge style 6] (28.center) to (27.center);
		\draw [style=new edge style 6] (27.center) to (26.center);
		\draw [style=new edge style 6] (26.center) to (23.center);
		\draw [style=new edge style 6] (26.center) to (24.center);
		\draw [style=new edge style 6] (26.center) to (25.center);
		\draw [style=new edge style 6] (27.center) to (0.center);
		\draw [style=new edge style 6] (27.center) to (18.center);
		\draw [style=new edge style 6] (14.center) to (5.center);
		\draw [style=new edge style 6] (28.center) to (19.center);
		\draw [style=new edge style 6] (28.center) to (3.center);
		\draw [style=new edge style 6] (20.center) to (28.center);
		\draw [style=new edge style 6] (29.center) to (2.center);
	\end{pgfonlayer}
\end{tikzpicture}

%% file: images/isolated_paths_4.tikz
\begin{tikzpicture}
	\begin{pgfonlayer}{nodelayer}
		\node [style=none] (1) at (-7, -2) {};
		\node [style=none] (2) at (-3, -2) {};
		\node [style=none] (3) at (3, -2) {};
		\node [style=none] (4) at (7, -2) {};
		\node [style=none] (9) at (-5.25, -1) {};
		\node [style=none] (10) at (-2.625, 0.5) {};
		\node [style=none] (11) at (-4.25, -2) {};
		\node [style=none] (12) at (-2.25, -1) {};
		\node [style=none] (13) at (-3.5, -1) {};
		\node [style=none] (14) at (-2.75, -0.25) {};
		\node [style=none] (15) at (-3.75, -0.5) {};
		\node [style=none] (16) at (-5, -1.5) {};
		\node [style=none] (17) at (1.25, 1) {};
		\node [style=none] (18) at (2, 0) {};
		\node [style=none] (19) at (2.75, -0.75) {};
		\node [style=none] (20) at (5, -1.25) {};
		\node [style=none] (21) at (4.75, -1.75) {};
		\node [style=none] (22) at (2.75, 0.45) {};
		\node [style=none, label={below:$\tilde{C}_1$}] (25) at (-4.75, -2) {};
		\node [style=none, label={below:$\tilde{C}_2$}] (26) at (5, -2) {};
		\node [style=none, label={below:$C$}] (27) at (0, -2) {};
		\node [style=none] (0) at (0, 2) {};
		\node [style=none] (23) at (-1, -1) {};
		\node [style=none] (24) at (1, -1) {};
		\node [style=none] (5) at (0, 0.75) {};
		\node [style=none] (7) at (0, -0.25) {};
		\node [style=none] (8) at (0.575, 0.5) {};
		\node [style=none] (6) at (-0.5, 0.25) {};
	\end{pgfonlayer}
	\begin{pgfonlayer}{edgelayer}
		\draw (0.center) to (1.center);
		\draw (1.center) to (2.center);
		\draw (3.center) to (4.center);
		\draw (4.center) to (0.center);
		\draw (15.center) to (13.center);
		\draw (9.center) to (16.center);
		\draw (16.center) to (1.center);
		\draw (16.center) to (11.center);
		\draw (16.center) to (15.center);
		\draw (15.center) to (10.center);
		\draw (13.center) to (16.center);
		\draw (15.center) to (14.center);
		\draw (14.center) to (13.center);
		\draw (13.center) to (2.center);
		\draw (13.center) to (12.center);
		\draw (14.center) to (12.center);
		\draw (14.center) to (0.center);
		\draw (19.center) to (21.center);
		\draw (21.center) to (3.center);
		\draw (21.center) to (4.center);
		\draw (21.center) to (20.center);
		\draw (20.center) to (19.center);
		\draw (19.center) to (3.center);
		\draw (19.center) to (18.center);
		\draw (18.center) to (20.center);
		\draw (22.center) to (18.center);
		\draw (22.center) to (20.center);
		\draw (17.center) to (18.center);
		\draw (17.center) to (22.center);
		\draw (17.center) to (0.center);
		\draw [style=new edge style 4] (0.center)
			 to (2.center)
			 to (3.center)
			 to cycle;
		\draw [style=new edge style 12] (0.center)
			 to [in=165, out=-120, looseness=0.50] (23.center)
			 to [bend right=15] (24.center)
			 to [in=-75, out=15, looseness=0.75] cycle;
		\draw (0.center) to (5.center);
		\draw [in=135, out=-45] (6.center) to (7.center);
		\draw (7.center) to (8.center);
		\draw (8.center) to (5.center);
		\draw (7.center) to (5.center);
		\draw (0.center) to (8.center);
		\draw (5.center) to (6.center);
		\draw (0.center) to (6.center);
	\end{pgfonlayer}
\end{tikzpicture}

%% file: images/isolated_paths_5.tikz
\begin{tikzpicture}
	\begin{pgfonlayer}{nodelayer}
		\node [style=none] (0) at (0, 4) {};
		\node [style=none] (1) at (0, -4) {};
		\node [style=none, label={left:$\tilde{C}_1$}] (2) at (-5.375, 0) {};
		\node [style=none] (3) at (-4.75, 2) {};
		\node [style=none] (4) at (-2.75, 1) {};
		\node [style=none] (5) at (-2.075, -2.55) {};
		\node [style=none] (6) at (-3.25, -3.25) {};
		\node [style=none] (7) at (-4, -1) {};
		\node [style=none] (8) at (-3.5, -2.25) {};
		\node [style=none] (9) at (-4.75, 1) {};
		\node [style=none] (10) at (-3.75, 0.25) {};
		\node [style=none] (11) at (-2.5, 3) {};
		\node [style=none] (12) at (-1.75, 3.85) {};
		\node [style=none] (13) at (1.825, -2.925) {};
		\node [style=none] (14) at (2.85, -0.75) {};
		\node [style=none] (15) at (2.75, 1.25) {};
		\node [style=none, label={below left:$C$}] (16) at (2.125, 2.55) {};
		\node [style=none] (17) at (3.5, 3.15) {};
		\node [style=none] (18) at (2, 3.25) {};
		\node [style=none] (19) at (4.925, 1.75) {};
		\node [style=none] (20) at (4, -1) {};
		\node [style=none] (21) at (3.75, 1.75) {};
		\node [style=none] (22) at (4.5, 0.75) {};
		\node [style=none] (23) at (3, -2.5) {};
		\node [style=none] (24) at (3.25, -3.25) {};
		\node [style=none] (25) at (-0.25, 1.5) {};
		\node [style=none] (26) at (-0.5, 0.5) {};
		\node [style=none] (27) at (0.75, 0.25) {};
		\node [style=none] (28) at (-0.5, -1.25) {};
		\node [style=none] (29) at (0.25, -2) {};
		\node [style=none] (30) at (0.75, -1.5) {};
		\node [style=none, label={right:$\tilde{C}_2$}] (31) at (5, 0) {};
	\end{pgfonlayer}
	\begin{pgfonlayer}{edgelayer}
		\draw [style=new edge style 4] (1.center)
			 to [bend right=75, looseness=1.25] (0.center)
			 to [bend right=75, looseness=1.25] cycle;
		\draw [style=new edge style 12] (0.center)
			 to [bend left=45, looseness=0.75] (1.center)
			 to [bend left] cycle;
		\draw (6.center) to (5.center);
		\draw (6.center) to (8.center);
		\draw (8.center) to (5.center);
		\draw (8.center) to (2.center);
		\draw (7.center) to (8.center);
		\draw (7.center) to (10.center);
		\draw (7.center) to (2.center);
		\draw (2.center) to (10.center);
		\draw (10.center) to (4.center);
		\draw [bend right=15] (10.center) to (5.center);
		\draw (10.center) to (9.center);
		\draw (9.center) to (2.center);
		\draw (9.center) to (3.center);
		\draw (3.center) to (10.center);
		\draw (4.center) to (3.center);
		\draw (3.center) to (11.center);
		\draw (11.center) to (4.center);
		\draw (11.center) to (12.center);
		\draw [bend left=15, looseness=0.75] (0.center) to (18.center);
		\draw (18.center) to (16.center);
		\draw (18.center) to (17.center);
		\draw (17.center) to (16.center);
		\draw (16.center) to (21.center);
		\draw (21.center) to (17.center);
		\draw (15.center) to (21.center);
		\draw (21.center) to (19.center);
		\draw (21.center) to (22.center);
		\draw (22.center) to (14.center);
		\draw (22.center) to (20.center);
		\draw (22.center) to (19.center);
		\draw (20.center) to (23.center);
		\draw (23.center) to (14.center);
		\draw (20.center) to (14.center);
		\draw (20.center) to (24.center);
		\draw (23.center) to (13.center);
		\draw (13.center) to (24.center);
		\draw [bend right=15] (1.center) to (24.center);
		\draw (30.center) to (1.center);
		\draw (29.center) to (1.center);
		\draw (29.center) to (28.center);
		\draw (28.center) to (30.center);
		\draw (29.center) to (30.center);
		\draw (28.center) to (1.center);
		\draw (27.center) to (28.center);
		\draw (26.center) to (28.center);
		\draw (30.center) to (27.center);
		\draw (27.center) to (26.center);
		\draw (25.center) to (26.center);
		\draw (25.center) to (27.center);
		\draw (27.center) to (0.center);
		\draw (25.center) to (0.center);
		\draw [bend left=345, looseness=0.50] (0.center) to (12.center);
		\draw [bend left=15] (3.center) to (12.center);
		\draw [bend left=15] (2.center) to (3.center);
		\draw [bend right=15] (2.center) to (6.center);
		\draw [bend right=15] (6.center) to (1.center);
		\draw [bend left=15, looseness=0.75] (0.center) to (17.center);
		\draw [bend left=15] (17.center) to (19.center);
		\draw [bend left] (19.center) to (24.center);
	\end{pgfonlayer}
\end{tikzpicture}

%% file: images/isolated_paths_6.tikz
\begin{tikzpicture}
	\begin{pgfonlayer}{nodelayer}
		\node [style=none] (31) at (-2.825, -0.75) {};
		\node [style=none] (32) at (2.875, 0) {};
		\node [style=none] (1) at (0, -4) {};
		\node [style=none] (35) at (1.75, -0.75) {};
		\node [style=none] (36) at (-1.25, -1) {};
		\node [style=none] (0) at (0, 4) {};
		\node [style=none] (33) at (-1.5, 0.25) {};
		\node [style=none] (34) at (1.75, 0.5) {};
		\node [style=none, label={left:$\tilde{C}_1$}] (2) at (-5.375, 0) {};
		\node [style=none] (3) at (-4.75, 2) {};
		\node [style=none] (4) at (-2.75, 1) {};
		\node [style=none] (5) at (-2.075, -2.55) {};
		\node [style=none] (6) at (-3.25, -3.25) {};
		\node [style=none] (7) at (-4, -1) {};
		\node [style=none] (8) at (-3.5, -2.25) {};
		\node [style=none] (9) at (-4.75, 1) {};
		\node [style=none] (10) at (-3.5, 0) {};
		\node [style=none] (11) at (-2.5, 3) {};
		\node [style=none] (12) at (-1.75, 3.85) {};
		\node [style=none] (13) at (1.825, -2.925) {};
		\node [style=none] (14) at (2.85, -0.75) {};
		\node [style=none] (15) at (2.75, 1.25) {};
		\node [style=none, label={below left:$C$}] (16) at (2.075, 2.525) {};
		\node [style=none] (17) at (3.5, 3.15) {};
		\node [style=none] (18) at (2, 3.25) {};
		\node [style=none] (19) at (4.925, 1.75) {};
		\node [style=none] (20) at (4, -1) {};
		\node [style=none] (21) at (3.75, 1.75) {};
		\node [style=none] (22) at (4.5, 0.75) {};
		\node [style=none] (23) at (3, -2.5) {};
		\node [style=none] (24) at (3.25, -3.25) {};
		\node [style=none] (25) at (-0.25, 1.5) {};
		\node [style=none] (26) at (-0.5, 0.5) {};
		\node [style=none] (27) at (0.75, 0.25) {};
		\node [style=none] (28) at (-0.5, -1.25) {};
		\node [style=none] (29) at (0.25, -2) {};
		\node [style=none] (30) at (0.75, -1.5) {};
		\node [style=none, label={right:$\tilde{C}_2$}] (37) at (5.1, 0) {};
		\node [style=none] (38) at (0.775, 3.675) {};
		\node [style=none] (39) at (0.825, -3.675) {};
	\end{pgfonlayer}
	\begin{pgfonlayer}{edgelayer}
		\draw [style=new edge style 4] (31.center)
			 to (32.center)
			 to [in=15, out=-90] (1.center)
			 to [bend left] cycle;
		\draw [style=new edge style 4] (0.center)
			 to [bend left=45, looseness=0.75] (32.center)
			 to (31.center)
			 to [in=180, out=105, looseness=0.75] cycle;
		\draw [style=new edge style 12] (39.center)
			 to (1.center)
			 to [in=-165, out=135] (36.center)
			 to [bend left=15] (35.center)
			 to [in=60, out=-15] cycle;
		\draw (6.center) to (5.center);
		\draw (6.center) to (8.center);
		\draw (8.center) to (5.center);
		\draw (8.center) to (2.center);
		\draw (7.center) to (8.center);
		\draw (7.center) to (10.center);
		\draw (7.center) to (2.center);
		\draw (2.center) to (10.center);
		\draw (10.center) to (4.center);
		\draw [bend right=15] (10.center) to (5.center);
		\draw (10.center) to (9.center);
		\draw (9.center) to (2.center);
		\draw (9.center) to (3.center);
		\draw (3.center) to (10.center);
		\draw (4.center) to (3.center);
		\draw (3.center) to (11.center);
		\draw (11.center) to (4.center);
		\draw (11.center) to (12.center);
		\draw [bend left=15, looseness=0.75] (0.center) to (18.center);
		\draw (18.center) to (16.center);
		\draw (18.center) to (17.center);
		\draw (17.center) to (16.center);
		\draw (16.center) to (21.center);
		\draw (21.center) to (17.center);
		\draw (15.center) to (21.center);
		\draw (21.center) to (19.center);
		\draw (21.center) to (22.center);
		\draw (22.center) to (14.center);
		\draw (22.center) to (20.center);
		\draw (22.center) to (19.center);
		\draw (20.center) to (23.center);
		\draw (23.center) to (14.center);
		\draw (20.center) to (14.center);
		\draw (20.center) to (24.center);
		\draw [in=15, out=-165] (23.center) to (13.center);
		\draw (13.center) to (24.center);
		\draw [bend right=15] (1.center) to (24.center);
		\draw (30.center) to (1.center);
		\draw (29.center) to (1.center);
		\draw (29.center) to (28.center);
		\draw (28.center) to (30.center);
		\draw (29.center) to (30.center);
		\draw (28.center) to (1.center);
		\draw [style=new edge style 12] (33.center)
			 to [bend right, looseness=0.75] (34.center)
			 to [in=-60, out=30, looseness=0.75] (38.center)
			 to (0.center)
			 to [in=165, out=-135, looseness=0.50] cycle;
		\draw (27.center) to (28.center);
		\draw (26.center) to (28.center);
		\draw (30.center) to (27.center);
		\draw (27.center) to (26.center);
		\draw (25.center) to (26.center);
		\draw (25.center) to (27.center);
		\draw (27.center) to (0.center);
		\draw (25.center) to (0.center);
		\draw [bend right=45, looseness=0.75] (24.center) to (19.center);
		\draw [bend right, looseness=0.50] (19.center) to (17.center);
		\draw [bend right=15, looseness=0.75] (17.center) to (0.center);
		\draw [bend right=15, looseness=0.50] (0.center) to (12.center);
		\draw [bend right=15] (12.center) to (3.center);
		\draw [bend right=15] (3.center) to (2.center);
		\draw [bend right=15] (2.center) to (6.center);
		\draw [bend right=15] (6.center) to (1.center);
	\end{pgfonlayer}
\end{tikzpicture}

%% file: images/contractible_square.tikz
\begin{tikzpicture}[scale=0.8]
	\begin{pgfonlayer}{nodelayer}
		\node [style=none] (0) at (-3, 1) {};
		\node [style=none] (1) at (0, 4) {};
		\node [style=none] (2) at (3, 1) {};
		\node [style=none] (3) at (0, -2) {};
		\node [style=none] (4) at (-5, 1) {};
		\node [style=none] (5) at (0, 6) {};
		\node [style=none] (6) at (5, 1) {};
		\node [style=none] (7) at (0, -4) {};
		\node [style=none] (8) at (0, -7) {};
		\node [style=none] (9) at (-8, 1) {};
		\node [style=none] (10) at (0, 9) {};
		\node [style=none] (11) at (8, 1) {};
		\node [style=none] (12) at (-1, -6.925) {};
		\node [style=none] (13) at (-2.5, -6.6) {};
		\node [style=none] (14) at (-0.425, -3.975) {};
		\node [style=none] (15) at (0.575, -6.975) {};
		\node [style=none] (16) at (1.1, -6.925) {};
		\node [style=none] (17) at (0, -6) {};
		\node [style=none] (18) at (0.5, -4) {};
		\node [style=none] (19) at (1.825, -6.825) {};
		\node [style=none] (20) at (2.775, -6.525) {};
		\node [style=none] (21) at (5.025, -5.3) {};
		\node [style=none] (22) at (3.675, -6.175) {};
		\node [style=none] (23) at (4.275, -5.825) {};
		\node [style=none] (24) at (2.75, -5) {};
		\node [style=none] (25) at (3.75, -4.5) {};
		\node [style=none] (26) at (6.075, -4.275) {};
		\node [style=none] (27) at (0.975, -5) {};
		\node [style=none] (28) at (3.7, -3.675) {};
		\node [style=none] (29) at (4.275, -1.725) {};
		\node [style=none] (30) at (4.85, -0.625) {};
		\node [style=none] (31) at (5.875, -2.8) {};
		\node [style=none] (32) at (5.625, -2.025) {};
		\node [style=none] (33) at (7.1, -2.75) {};
		\node [style=none] (34) at (7.525, -1.775) {};
		\node [style=none] (35) at (7.925, -0.25) {};
		\node [style=none] (36) at (4.925, 1.925) {};
		\node [style=none] (37) at (4.775, 2.55) {};
		\node [style=none] (38) at (7.75, 3.025) {};
		\node [style=none] (39) at (7.525, 3.8) {};
		\node [style=none] (40) at (7.325, 4.25) {};
		\node [style=none] (41) at (7.1, 4.75) {};
		\node [style=none] (42) at (6.85, 5.25) {};
		\node [style=none] (43) at (6.5, 5.75) {};
		\node [style=none] (44) at (6.1, 4.875) {};
		\node [style=none] (45) at (6, 4) {};
		\node [style=none] (46) at (5.1, 5.05) {};
		\node [style=none] (47) at (5.85, 6.5) {};
		\node [style=none] (48) at (5.425, 6.975) {};
		\node [style=none] (49) at (6.75, 2.9) {};
		\node [style=none, label={below:$C'$}] (50) at (2.45, 5.4) {};
		\node [style=none, label={above:$C$}] (51) at (3.025, 8.45) {};
		\node [style=none] (52) at (3.875, 8.05) {};
		\node [style=none] (53) at (4.625, 7.6) {};
		\node [style=none] (54) at (2, 8.75) {};
		\node [style=none] (55) at (1.25, 8.925) {};
		\node [style=none] (56) at (0.725, 8.975) {};
		\node [style=none] (57) at (-1, 8.95) {};
		\node [style=none] (58) at (-1.65, 8.85) {};
		\node [style=none] (59) at (0, 7.75) {};
		\node [style=none] (60) at (0, 7.25) {};
		\node [style=none] (61) at (-0.05, 8.425) {};
		\node [style=none] (62) at (-2.375, 8.65) {};
		\node [style=none] (63) at (0.975, 8.525) {};
		\node [style=none] (64) at (-1.275, 8.425) {};
		\node [style=none] (65) at (0.775, 6.975) {};
		\node [style=none] (66) at (1.225, 7.75) {};
		\node [style=none] (67) at (2.75, 7) {};
		\node [style=none] (68) at (-3.5, 8.25) {};
		\node [style=none] (69) at (-4.15, 7.875) {};
		\node [style=none] (70) at (-4.575, 7.65) {};
		\node [style=none] (71) at (-4.975, 7.35) {};
		\node [style=none] (72) at (-5.4, 6.975) {};
		\node [style=none] (73) at (-5.9, 6.5) {};
		\node [style=none] (74) at (-2.775, 6.35) {};
		\node [style=none] (75) at (-4.2, 5.475) {};
		\node [style=none] (76) at (-4.8, 6.4) {};
		\node [style=none] (77) at (-4.425, 6.725) {};
		\node [style=none] (78) at (-4, 7) {};
		\node [style=none] (79) at (-3.6, 7.275) {};
		\node [style=none] (80) at (-3.7, 6.35) {};
		\node [style=none] (81) at (-7, 5) {};
		\node [style=none] (82) at (-4.15, 3.825) {};
		\node [style=none] (83) at (-5, 4.175) {};
		\node [style=none] (84) at (-3.5, 5.4) {};
		\node [style=none] (85) at (-7.65, 3.45) {};
		\node [style=none] (86) at (-7.85, 2.5) {};
		\node [style=none] (87) at (-7.975, 1.75) {};
		\node [style=none] (88) at (-7.875, -0.5) {};
		\node [style=none] (89) at (-7.7, -1.2) {};
		\node [style=none] (90) at (-6.25, 1.25) {};
		\node [style=none] (91) at (-6.75, 0.5) {};
		\node [style=none] (92) at (-7.25, -0.5) {};
		\node [style=none] (93) at (-7.25, 0.75) {};
		\node [style=none] (94) at (-4.475, -1.25) {};
		\node [style=none] (95) at (-3.55, -2.525) {};
		\node [style=none] (96) at (-7.25, -2.5) {};
		\node [style=none] (97) at (-6.05, -4.3) {};
		\node [style=none] (98) at (-4.45, -5.7) {};
		\node [style=none, label={below:$C"$}] (99) at (1.575, 3.55) {};
	\end{pgfonlayer}
	\begin{pgfonlayer}{edgelayer}
		\draw [style=new edge style 6] (0.center)
			 to [bend left=45] (1.center)
			 to [in=90, out=0] (2.center)
			 to [bend left=45] (3.center)
			 to [bend left=45] cycle;
		\draw [style=new edge style 4] (96.center)
			 to (89.center)
			 to (88.center)
			 to (9.center)
			 to (87.center)
			 to (86.center)
			 to (85.center)
			 to (81.center)
			 to (73.center)
			 to (72.center)
			 to (71.center)
			 to (70.center)
			 to (69.center)
			 to (68.center)
			 to (62.center)
			 to (58.center)
			 to (57.center)
			 to (10.center)
			 to (56.center)
			 to (55.center)
			 to (54.center)
			 to (51.center)
			 to (52.center)
			 to (53.center)
			 to (48.center)
			 to (47.center)
			 to (43.center)
			 to (42.center)
			 to (41.center)
			 to (40.center)
			 to (39.center)
			 to (38.center)
			 to [bend left=10] (11.center)
			 to (35.center)
			 to (34.center)
			 to (33.center)
			 to (26.center)
			 to (21.center)
			 to (23.center)
			 to (22.center)
			 to (20.center)
			 to (19.center)
			 to (16.center)
			 to (15.center)
			 to (8.center)
			 to (12.center)
			 to (13.center)
			 to [bend left=10] (98.center)
			 to [bend left=10] (97.center)
			 to [bend left=10] cycle
			 to (94.center)
			 to [bend right=15, looseness=0.75] (95.center)
			 to [bend right=15] (14.center)
			 to (7.center)
			 to (18.center)
			 to [bend right, looseness=0.75] (29.center)
			 to (30.center)
			 to [bend right=10] (6.center)
			 to (36.center)
			 to (37.center)
			 to [bend right, looseness=0.75] (50.center)
			 to [bend right, looseness=0.50] (5.center)
			 to [bend right, looseness=0.75] (82.center)
			 to [bend right=15] (4.center)
			 to [bend right=15] (94.center);
		\draw [style=new edge style 4] (95.center) to (98.center);
		\draw [style=new edge style 4] (13.center) to (14.center);
		\draw [style=new edge style 4] (14.center) to (7.center);
		\draw [style=new edge style 4] (7.center) to (12.center);
		\draw [style=new edge style 4] (8.center) to (15.center);
		\draw [style=new edge style 4] (8.center) to (17.center);
		\draw [style=new edge style 4, bend left=15] (17.center) to (16.center);
		\draw [style=new edge style 4, bend right=15] (15.center) to (17.center);
		\draw [style=new edge style 4] (17.center) to (7.center);
		\draw [style=new edge style 4] (7.center) to (14.center);
		\draw [style=new edge style 5] (18.center) to (27.center);
		\draw [style=new edge style 5] (28.center) to (29.center);
		\draw [style=new edge style 5, bend left, looseness=0.75] (29.center) to (18.center);
		\draw [style=new edge style 4, bend left, looseness=1.25] (20.center) to (24.center);
		\draw [style=new edge style 5] (25.center)
			 to [bend right] (24.center)
			 to (28.center)
			 to cycle;
		\draw [style=new edge style 4, bend left] (25.center) to (21.center);
		\draw [style=new edge style 4] (22.center) to (24.center);
		\draw [style=new edge style 4] (23.center) to (25.center);
		\draw [style=new edge style 4] (31.center) to (33.center);
		\draw [style=new edge style 4] (32.center) to (34.center);
		\draw [style=new edge style 4] (30.center) to (35.center);
		\draw [style=new edge style 4] (30.center) to (11.center);
		\draw [style=new edge style 4] (6.center) to (11.center);
		\draw [style=new edge style 4] (36.center) to (11.center);
		\draw [style=new edge style 4, bend right=60, looseness=1.50] (48.center) to (49.center);
		\draw [style=new edge style 4] (52.center) to (50.center);
		\draw [style=new edge style 4] (53.center) to (50.center);
		\draw [style=new edge style 4] (10.center) to (61.center);
		\draw [style=new edge style 4] (56.center) to (61.center);
		\draw [style=new edge style 4] (57.center) to (61.center);
		\draw [style=new edge style 4, bend right=15] (62.center) to (60.center);
		\draw [style=new edge style 4] (59.center) to (60.center);
		\draw [style=new edge style 5] (61.center)
			 to (64.center)
			 to [bend right=15] (59.center)
			 to [bend right=15] (63.center)
			 to cycle;
		\draw [style=new edge style 4, bend left=15] (66.center) to (65.center);
		\draw [style=new edge style 4] (69.center) to (79.center);
		\draw [style=new edge style 4] (70.center) to (78.center);
		\draw [style=new edge style 4] (71.center) to (77.center);
		\draw [style=new edge style 4] (72.center) to (76.center);
		\draw [style=new edge style 5] (78.center)
			 to (77.center)
			 to (76.center)
			 to [bend right, looseness=1.25] (80.center)
			 to [bend right=45, looseness=1.25] (79.center)
			 to cycle;
		\draw [style=new edge style 4, bend right] (73.center) to (75.center);
		\draw [style=new edge style 4, bend right] (75.center) to (74.center);
		\draw [style=new edge style 4, bend right] (74.center) to (68.center);
		\draw [style=new edge style 4] (4.center) to (90.center);
		\draw [style=new edge style 4] (90.center) to (85.center);
		\draw [style=new edge style 4] (90.center) to (91.center);
		\draw [style=new edge style 4] (91.center) to (92.center);
		\draw [style=new edge style 4] (92.center) to (89.center);
		\draw [style=new edge style 4] (92.center) to (88.center);
		\draw [style=new edge style 4] (91.center) to (86.center);
		\draw [style=new edge style 4] (91.center) to (93.center);
		\draw [style=new edge style 4] (93.center) to (9.center);
		\draw [style=new edge style 4] (93.center) to (87.center);
		\draw [style=new edge style 4] (95.center) to (97.center);
		\draw [style=new edge style 4] (94.center) to (97.center);
		\draw [style=new edge style 4] (83.center) to (81.center);
		\draw [style=new edge style 4] (75.center) to (84.center);
		\draw [style=new edge style 4, bend right=15] (60.center) to (66.center);
		\draw [style=new edge style 4, bend right=15] (66.center) to (54.center);
		\draw [style=new edge style 4] (67.center) to (51.center);
		\draw [style=new edge style 4] (37.center) to (49.center);
		\draw [style=new edge style 4] (49.center) to (38.center);
		\draw [style=new edge style 4] (30.center) to (32.center);
		\draw [style=new edge style 4] (32.center) to (31.center);
		\draw [style=new edge style 4] (31.center) to (26.center);
		\draw [style=new edge style 4, bend left=15] (28.center) to (26.center);
		\draw [style=new edge style 4] (27.center) to (19.center);
		\draw [style=new edge style 5] (58.center) to (64.center);
		\draw [style=new edge style 5] (55.center) to (63.center);
		\draw [style=new edge style 4] (43.center) to (44.center);
		\draw [style=new edge style 4] (42.center) to (44.center);
		\draw [style=new edge style 4] (41.center) to (44.center);
		\draw [style=new edge style 4, bend right] (47.center) to (46.center);
		\draw [style=new edge style 4, bend right=15, looseness=1.25] (46.center) to (45.center);
		\draw [style=new edge style 4, bend right=15] (45.center) to (39.center);
		\draw [style=new edge style 4, bend right=15] (45.center) to (40.center);
		\draw [style=new edge style 5, bend left, looseness=0.75] (82.center) to (5.center);
		\draw [style=new edge style 5, bend left, looseness=0.50] (5.center) to (50.center);
		\draw [style=new edge style 5] (50.center) to (67.center);
		\draw [style=new edge style 5] (65.center) to (84.center);
		\draw [style=new edge style 5] (84.center) to (83.center);
		\draw [style=new edge style 5] (83.center) to (82.center);
	\end{pgfonlayer}
\end{tikzpicture}

%% file: images/contractible_square_pinched_on_vertex.tikz
\begin{tikzpicture}[scale=0.85]
	\begin{pgfonlayer}{nodelayer}
		\node [style=none] (0) at (0, 8) {};
		\node [style=none] (1) at (3, 1.75) {};
		\node [style=none] (2) at (0, -1) {};
		\node [style=none] (3) at (-3, 1.75) {};
		\node [style=none] (4) at (0, 8) {};
		\node [style=none] (5) at (4.85, 0.75) {};
		\node [style=none] (6) at (0, -3.75) {};
		\node [style=none] (7) at (-4.875, 0.75) {};
		\node [style=none] (8) at (-7.875, 0.75) {};
		\node [style=none] (9) at (0, 8) {};
		\node [style=none] (10) at (7.8, 0.75) {};
		\node [style=none] (11) at (0, -6.75) {};
		\node [style=none] (12) at (-7.8, 1.5) {};
		\node [style=none] (13) at (-7.475, 3) {};
		\node [style=none] (14) at (-4.85, 0.925) {};
		\node [style=none] (15) at (-7.85, 0.175) {};
		\node [style=none] (16) at (-7.8, -0.35) {};
		\node [style=none] (17) at (-6.875, 0.75) {};
		\node [style=none] (18) at (-4.875, 0.25) {};
		\node [style=none] (19) at (-7.7, -1.075) {};
		\node [style=none] (20) at (-7.4, -1.775) {};
		\node [style=none] (21) at (-6.175, -3.775) {};
		\node [style=none] (22) at (-7.05, -2.425) {};
		\node [style=none] (23) at (-6.7, -3.025) {};
		\node [style=none] (24) at (-5.875, -1.75) {};
		\node [style=none] (25) at (-5.375, -2.5) {};
		\node [style=none] (26) at (-5.15, -4.825) {};
		\node [style=none] (27) at (-5.875, -0.225) {};
		\node [style=none] (28) at (-4.55, -2.45) {};
		\node [style=none] (29) at (-2.675, -3.075) {};
		\node [style=none] (30) at (-1.625, -3.6) {};
		\node [style=none] (31) at (-3.675, -4.625) {};
		\node [style=none] (32) at (-2.9, -4.375) {};
		\node [style=none] (33) at (-3.625, -5.85) {};
		\node [style=none] (34) at (-2.65, -6.275) {};
		\node [style=none] (35) at (-1.25, -6.675) {};
		\node [style=none] (36) at (0.925, -3.675) {};
		\node [style=none] (37) at (1.55, -3.525) {};
		\node [style=none] (38) at (1.825, -6.5) {};
		\node [style=none] (39) at (2.6, -6.275) {};
		\node [style=none] (40) at (3.05, -6.075) {};
		\node [style=none] (41) at (3.55, -5.85) {};
		\node [style=none] (42) at (4.05, -5.6) {};
		\node [style=none] (43) at (4.55, -5.25) {};
		\node [style=none] (44) at (3.675, -4.85) {};
		\node [style=none] (45) at (2.8, -4.75) {};
		\node [style=none] (46) at (3.85, -3.85) {};
		\node [style=none] (47) at (5.3, -4.6) {};
		\node [style=none] (48) at (5.775, -4.175) {};
		\node [style=none] (49) at (1.7, -5.5) {};
		\node [style=none, label={left:$C'$}] (50) at (4.2, -1.45) {};
		\node [style=none] (51) at (7.25, -2.025) {};
		\node [style=none] (52) at (6.85, -2.625) {};
		\node [style=none, label={right:$C$}] (53) at (6.4, -3.375) {};
		\node [style=none] (54) at (7.55, -1.25) {};
		\node [style=none] (55) at (7.725, -0.5) {};
		\node [style=none] (56) at (7.775, 0.025) {};
		\node [style=none] (57) at (7.75, 1.5) {};
		\node [style=none] (58) at (7.65, 2.15) {};
		\node [style=none] (59) at (6.55, 0.75) {};
		\node [style=none] (60) at (6.05, 0.75) {};
		\node [style=none] (61) at (7.225, 0.55) {};
		\node [style=none] (62) at (7.45, 2.875) {};
		\node [style=none] (63) at (7.325, -0.225) {};
		\node [style=none] (64) at (7.225, 1.775) {};
		\node [style=none] (65) at (5.9, -0.025) {};
		\node [style=none] (66) at (6.55, -0.475) {};
		\node [style=none] (67) at (5.8, -1.75) {};
		\node [style=none] (68) at (7.25, 3.5) {};
		\node [style=none] (69) at (6.875, 4.15) {};
		\node [style=none] (70) at (6.65, 4.575) {};
		\node [style=none] (71) at (6.35, 4.975) {};
		\node [style=none] (72) at (5.975, 5.4) {};
		\node [style=none] (73) at (5.5, 5.9) {};
		\node [style=none] (74) at (5.85, 3.025) {};
		\node [style=none] (75) at (4.725, 4.45) {};
		\node [style=none] (76) at (5.4, 4.8) {};
		\node [style=none] (77) at (5.725, 4.425) {};
		\node [style=none] (78) at (6, 4) {};
		\node [style=none] (79) at (6.275, 3.6) {};
		\node [style=none] (80) at (5.35, 3.7) {};
		\node [style=none] (81) at (4, 7) {};
		\node [style=none] (82) at (3.3, 4.9) {};
		\node [style=none] (83) at (3.425, 5.5) {};
		\node [style=none] (84) at (4.4, 4) {};
		\node [style=none] (85) at (2.45, 7.65) {};
		\node [style=none] (86) at (1.5, 7.85) {};
		\node [style=none] (87) at (0.75, 7.975) {};
		\node [style=none] (88) at (-1.5, 7.875) {};
		\node [style=none] (89) at (-2.2, 7.7) {};
		\node [style=none] (94) at (-3.25, 4.975) {};
		\node [style=none] (95) at (-4.025, 3.8) {};
		\node [style=none] (96) at (-3.5, 7.25) {};
		\node [style=none] (97) at (-5.3, 6.05) {};
		\node [style=none] (98) at (-6.7, 4.45) {};
		\node [style=none, label={left:$C"$}] (99) at (2.3, -0.025) {};
	\end{pgfonlayer}
	\begin{pgfonlayer}{edgelayer}
		\draw [style=new edge style 6] (0.center)
			 to [bend left=15, looseness=0.75] (1.center)
			 to [bend left=45] (2.center)
			 to [bend left=45] (3.center)
			 to [bend left=15, looseness=0.75] cycle;
		\draw [style=new edge style 4] (96.center)
			 to [bend left, looseness=0.25] (89.center)
			 to (88.center)
			 to (9.center)
			 to (87.center)
			 to (86.center)
			 to (85.center)
			 to [bend left=15, looseness=0.50] (81.center)
			 to (73.center)
			 to (72.center)
			 to (71.center)
			 to (70.center)
			 to (69.center)
			 to (68.center)
			 to (62.center)
			 to (58.center)
			 to (57.center)
			 to (10.center)
			 to (56.center)
			 to (55.center)
			 to (54.center)
			 to (51.center)
			 to (52.center)
			 to (53.center)
			 to (48.center)
			 to (47.center)
			 to (43.center)
			 to (42.center)
			 to (41.center)
			 to (40.center)
			 to (39.center)
			 to (38.center)
			 to [bend left=10] (11.center)
			 to (35.center)
			 to (34.center)
			 to (33.center)
			 to (26.center)
			 to (21.center)
			 to (23.center)
			 to (22.center)
			 to (20.center)
			 to (19.center)
			 to (16.center)
			 to (15.center)
			 to (8.center)
			 to (12.center)
			 to (13.center)
			 to [bend left=10] (98.center)
			 to [bend left=10] (97.center)
			 to [bend left=10] cycle
			 to (94.center)
			 to [bend right=45, looseness=0.25] (95.center)
			 to [bend right=15, looseness=0.75] (14.center)
			 to (7.center)
			 to (18.center)
			 to [bend right, looseness=0.75] (29.center)
			 to [bend right=15, looseness=0.75] (30.center)
			 to [bend right=10] (6.center)
			 to (36.center)
			 to (37.center)
			 to [bend right, looseness=0.75] (50.center)
			 to [bend right, looseness=0.50] (5.center)
			 to [bend right, looseness=0.50] (82.center)
			 to [bend right=15] (4.center)
			 to [bend right=15] (94.center);
		\draw [style=new edge style 4] (95.center) to (98.center);
		\draw [style=new edge style 4] (13.center) to (14.center);
		\draw [style=new edge style 4] (14.center) to (7.center);
		\draw [style=new edge style 4] (7.center) to (12.center);
		\draw [style=new edge style 4] (8.center) to (15.center);
		\draw [style=new edge style 4] (8.center) to (17.center);
		\draw [style=new edge style 4, bend left=15] (17.center) to (16.center);
		\draw [style=new edge style 4, bend right=15] (15.center) to (17.center);
		\draw [style=new edge style 4] (17.center) to (7.center);
		\draw [style=new edge style 4] (7.center) to (14.center);
		\draw [style=new edge style 5] (18.center) to (27.center);
		\draw [style=new edge style 5] (28.center) to (29.center);
		\draw [style=new edge style 5, bend left, looseness=0.75] (29.center) to (18.center);
		\draw [style=new edge style 4, bend left, looseness=1.25] (20.center) to (24.center);
		\draw [style=new edge style 5] (28.center)
			 to (24.center)
			 to [bend left] (25.center)
			 to cycle;
		\draw [style=new edge style 4, bend left] (25.center) to (21.center);
		\draw [style=new edge style 4] (22.center) to (24.center);
		\draw [style=new edge style 4] (23.center) to (25.center);
		\draw [style=new edge style 4] (31.center) to (33.center);
		\draw [style=new edge style 4] (32.center) to (34.center);
		\draw [style=new edge style 4] (30.center) to (35.center);
		\draw [style=new edge style 4] (30.center) to (11.center);
		\draw [style=new edge style 4] (6.center) to (11.center);
		\draw [style=new edge style 4] (36.center) to (11.center);
		\draw [style=new edge style 4, bend right=60, looseness=1.50] (48.center) to (49.center);
		\draw [style=new edge style 4] (52.center) to (50.center);
		\draw [style=new edge style 4] (53.center) to (50.center);
		\draw [style=new edge style 4] (10.center) to (61.center);
		\draw [style=new edge style 4] (56.center) to (61.center);
		\draw [style=new edge style 4] (57.center) to (61.center);
		\draw [style=new edge style 4, bend right=15] (62.center) to (60.center);
		\draw [style=new edge style 4] (59.center) to (60.center);
		\draw [style=new edge style 5] (61.center)
			 to (64.center)
			 to [bend right=15] (59.center)
			 to [bend right=15] (63.center)
			 to cycle;
		\draw [style=new edge style 4, bend left=15] (66.center) to (65.center);
		\draw [style=new edge style 5] (50.center) to (67.center);
		\draw [style=new edge style 5] (65.center) to (84.center);
		\draw [style=new edge style 5] (84.center) to (83.center);
		\draw [style=new edge style 5] (83.center) to (82.center);
		\draw [style=new edge style 4] (69.center) to (79.center);
		\draw [style=new edge style 4] (70.center) to (78.center);
		\draw [style=new edge style 4] (71.center) to (77.center);
		\draw [style=new edge style 4] (72.center) to (76.center);
		\draw [style=new edge style 5] (78.center)
			 to (77.center)
			 to (76.center)
			 to [bend right, looseness=1.25] (80.center)
			 to [bend right=45, looseness=1.25] (79.center)
			 to cycle;
		\draw [style=new edge style 4, bend right] (73.center) to (75.center);
		\draw [style=new edge style 4, bend right] (75.center) to (74.center);
		\draw [style=new edge style 4, bend right] (74.center) to (68.center);
		\draw [style=new edge style 4] (95.center) to (97.center);
		\draw [style=new edge style 4] (94.center) to (97.center);
		\draw [style=new edge style 4] (83.center) to (81.center);
		\draw [style=new edge style 4] (75.center) to (84.center);
		\draw [style=new edge style 4, bend right=15] (60.center) to (66.center);
		\draw [style=new edge style 4, bend right=15] (66.center) to (54.center);
		\draw [style=new edge style 4] (67.center) to (51.center);
		\draw [style=new edge style 4] (37.center) to (49.center);
		\draw [style=new edge style 4] (49.center) to (38.center);
		\draw [style=new edge style 4] (30.center) to (32.center);
		\draw [style=new edge style 4] (32.center) to (31.center);
		\draw [style=new edge style 4] (31.center) to (26.center);
		\draw [style=new edge style 4, bend left=15] (28.center) to (26.center);
		\draw [style=new edge style 4] (27.center) to (19.center);
		\draw [style=new edge style 5] (58.center) to (64.center);
		\draw [style=new edge style 5] (55.center) to (63.center);
		\draw [style=new edge style 4] (43.center) to (44.center);
		\draw [style=new edge style 4] (42.center) to (44.center);
		\draw [style=new edge style 4] (41.center) to (44.center);
		\draw [style=new edge style 4, bend right] (47.center) to (46.center);
		\draw [style=new edge style 4, bend right=15, looseness=1.25] (46.center) to (45.center);
		\draw [style=new edge style 4, bend right=15] (45.center) to (39.center);
		\draw [style=new edge style 4, bend right=15] (45.center) to (40.center);
	\end{pgfonlayer}
\end{tikzpicture}

%% file: images/contractible_square_pinched_on_face.tikz
\begin{tikzpicture}[scale=0.87]
	\begin{pgfonlayer}{nodelayer}
		\node [style=none] (6) at (0, -5.5) {};
		\node [style=none] (7) at (-4.875, -1) {};
		\node [style=none] (8) at (-7.875, -1) {};
		\node [style=none] (10) at (7.85, -1.575) {};
		\node [style=none] (11) at (0, -8.5) {};
		\node [style=none] (12) at (-7.9, 0) {};
		\node [style=none] (14) at (-4.85, 0) {};
		\node [style=none] (15) at (-7.85, -1.575) {};
		\node [style=none] (16) at (-7.8, -2.1) {};
		\node [style=none] (18) at (-4.875, -1.5) {};
		\node [style=none] (19) at (-7.65, -2.825) {};
		\node [style=none] (20) at (-7.4, -3.525) {};
		\node [style=none] (21) at (-6.175, -5.525) {};
		\node [style=none] (22) at (-7.05, -4.175) {};
		\node [style=none] (23) at (-6.7, -4.775) {};
		\node [style=none] (26) at (-5.15, -6.575) {};
		\node [style=none] (29) at (-2.675, -4.825) {};
		\node [style=none] (30) at (-1.625, -5.35) {};
		\node [style=none] (33) at (-3.625, -7.6) {};
		\node [style=none] (34) at (-2.65, -8.025) {};
		\node [style=none] (35) at (-1.25, -8.425) {};
		\node [style=none] (36) at (0.925, -5.425) {};
		\node [style=none] (37) at (1.55, -5.275) {};
		\node [style=none] (38) at (1.825, -8.25) {};
		\node [style=none] (39) at (2.6, -8.025) {};
		\node [style=none] (40) at (3.05, -7.825) {};
		\node [style=none] (41) at (3.55, -7.6) {};
		\node [style=none] (42) at (4.05, -7.35) {};
		\node [style=none] (43) at (4.55, -7) {};
		\node [style=none] (47) at (5.3, -6.35) {};
		\node [style=none] (48) at (5.775, -5.925) {};
		\node [style=none] (51) at (7.25, -3.775) {};
		\node [style=none] (52) at (6.9, -4.45) {};
		\node [style=none, label={right:$C$}] (53) at (6.45, -5.175) {};
		\node [style=none] (54) at (7.5, -3.25) {};
		\node [style=none] (55) at (7.675, -2.725) {};
		\node [style=none] (56) at (7.775, -2.125) {};
		\node [style=none] (57) at (7.875, -1.025) {};
		\node [style=none] (58) at (7.9, -0.45) {};
		\node [style=none] (62) at (7.925, 0.025) {};
		\node [style=none] (1) at (3, 0) {};
		\node [style=none] (2) at (0, -2.75) {};
		\node [style=none] (3) at (-3, 0) {};
		\node [style=none] (17) at (-6.875, -1) {};
		\node [style=none] (24) at (-5.875, -3.5) {};
		\node [style=none] (25) at (-5.375, -4.25) {};
		\node [style=none] (27) at (-5.875, -1.975) {};
		\node [style=none] (28) at (-4.55, -4.2) {};
		\node [style=none] (31) at (-3.675, -6.375) {};
		\node [style=none] (32) at (-2.9, -6.125) {};
		\node [style=none] (44) at (3.675, -6.6) {};
		\node [style=none] (45) at (2.8, -6.5) {};
		\node [style=none] (46) at (3.85, -5.6) {};
		\node [style=none] (49) at (1.7, -7.25) {};
		\node [style=none] (59) at (6.55, -1.25) {};
		\node [style=none] (60) at (6.05, -1.25) {};
		\node [style=none] (61) at (7.225, -1.45) {};
		\node [style=none] (63) at (7.325, -2.225) {};
		\node [style=none] (64) at (7.225, -0.725) {};
		\node [style=none] (66) at (6.55, -2.475) {};
		\node [style=none] (5) at (4.925, -1.125) {};
		\node [style=none, label={left:$C'$}] (50) at (4.3, -3.2) {};
		\node [style=none] (82) at (4.9, 0) {};
		\node [style=none] (83) at (5.425, 0) {};
		\node [style=none] (84) at (5.4, 0) {};
		\node [style=none] (65) at (5.65, -2.025) {};
		\node [style=none] (67) at (5.8, -3.5) {};
		\node [style=none] (86) at (8.375, 1.25) {};
		\node [style=none] (87) at (8, 2.75) {};
		\node [style=none] (90) at (6.75, 4) {};
		\node [style=none] (92) at (5, 5) {};
		\node [style=none] (94) at (3, 5.6) {};
		\node [style=none] (96) at (1, 5.95) {};
		\node [style=none] (97) at (0, 6.025) {};
		\node [style=none] (98) at (-8.375, 1.225) {};
		\node [style=none] (99) at (-8, 2.725) {};
		\node [style=none] (100) at (-6.75, 3.975) {};
		\node [style=none] (101) at (-5, 4.975) {};
		\node [style=none] (102) at (-3, 5.6) {};
		\node [style=none] (103) at (-1, 5.925) {};
		\node [style=none, label={left:$C"$}] (104) at (2.4, -1.7) {};
	\end{pgfonlayer}
	\begin{pgfonlayer}{edgelayer}
		\draw [style=new edge style 4] (41.center)
			 to (42.center)
			 to (43.center)
			 to (47.center)
			 to (48.center)
			 to (53.center)
			 to (52.center)
			 to (51.center)
			 to (54.center)
			 to (55.center)
			 to (56.center)
			 to (10.center)
			 to (57.center)
			 to (58.center)
			 to (62.center)
			 to (84.center)
			 to [bend left=360] (83.center)
			 to (82.center)
			 to [bend left=15, looseness=0.75] (5.center)
			 to [bend left=15, looseness=0.75] (50.center)
			 to [bend left, looseness=0.75] (37.center)
			 to (36.center)
			 to (6.center)
			 to [bend left=10] (30.center)
			 to [bend left=15, looseness=0.75] (29.center)
			 to [bend left, looseness=0.75] (18.center)
			 to (7.center)
			 to (14.center)
			 to (12.center)
			 to (8.center)
			 to (15.center)
			 to (16.center)
			 to (19.center)
			 to (20.center)
			 to (22.center)
			 to (23.center)
			 to (21.center)
			 to (26.center)
			 to (33.center)
			 to (34.center)
			 to (35.center)
			 to (11.center)
			 to [bend right=10] (38.center)
			 to (39.center)
			 to (40.center)
			 to cycle;
		\draw [style=new edge style 6] (2.center)
			 to [bend left=45] (3.center)
			 to (1.center)
			 to [bend left=45] cycle;
		\draw [style=new edge style 4] (15.center) to (8.center);
		\draw [style=new edge style 5] (18.center) to (27.center);
		\draw [style=new edge style 5] (28.center) to (29.center);
		\draw [style=new edge style 4] (7.center) to (12.center);
		\draw [style=new edge style 4] (8.center) to (17.center);
		\draw [style=new edge style 4, bend left=15] (17.center) to (16.center);
		\draw [style=new edge style 4, bend right=15] (15.center) to (17.center);
		\draw [style=new edge style 4] (17.center) to (7.center);
		\draw [style=new edge style 4, bend left, looseness=1.25] (20.center) to (24.center);
		\draw [style=new edge style 5] (28.center)
			 to (24.center)
			 to [bend left] (25.center)
			 to cycle;
		\draw [style=new edge style 4, bend left] (25.center) to (21.center);
		\draw [style=new edge style 4] (22.center) to (24.center);
		\draw [style=new edge style 4] (23.center) to (25.center);
		\draw [style=new edge style 4] (31.center) to (33.center);
		\draw [style=new edge style 4] (32.center) to (34.center);
		\draw [style=new edge style 4] (30.center) to (35.center);
		\draw [style=new edge style 4] (30.center) to (11.center);
		\draw [style=new edge style 4] (6.center) to (11.center);
		\draw [style=new edge style 4] (36.center) to (11.center);
		\draw [style=new edge style 4, bend right=60, looseness=1.50] (48.center) to (49.center);
		\draw [style=new edge style 4] (52.center) to (50.center);
		\draw [style=new edge style 4] (53.center) to (50.center);
		\draw [style=new edge style 4] (10.center) to (61.center);
		\draw [style=new edge style 4] (56.center) to (61.center);
		\draw [style=new edge style 4] (57.center) to (61.center);
		\draw [style=new edge style 4, bend right=15] (62.center) to (60.center);
		\draw [style=new edge style 4] (59.center) to (60.center);
		\draw [style=new edge style 5] (61.center)
			 to (64.center)
			 to [bend right=15] (59.center)
			 to [bend right=15] (63.center)
			 to cycle;
		\draw [style=new edge style 4, bend left=15] (66.center) to (65.center);
		\draw [style=new edge style 5] (50.center) to (67.center);
		\draw [style=new edge style 4, bend left=15] (65.center) to (84.center);
		\draw [style=new edge style 4, bend right=15] (60.center) to (66.center);
		\draw [style=new edge style 4, bend right=15] (66.center) to (54.center);
		\draw [style=new edge style 4] (67.center) to (51.center);
		\draw [style=new edge style 4] (37.center) to (49.center);
		\draw [style=new edge style 4] (49.center) to (38.center);
		\draw [style=new edge style 4] (30.center) to (32.center);
		\draw [style=new edge style 4] (32.center) to (31.center);
		\draw [style=new edge style 4] (31.center) to (26.center);
		\draw [style=new edge style 4, bend left=15] (28.center) to (26.center);
		\draw [style=new edge style 4] (27.center) to (19.center);
		\draw [style=new edge style 5] (58.center) to (64.center);
		\draw [style=new edge style 5] (55.center) to (63.center);
		\draw [style=new edge style 4] (43.center) to (44.center);
		\draw [style=new edge style 4] (42.center) to (44.center);
		\draw [style=new edge style 4] (41.center) to (44.center);
		\draw [style=new edge style 4, bend right] (47.center) to (46.center);
		\draw [style=new edge style 4, bend right=15, looseness=1.25] (46.center) to (45.center);
		\draw [style=new edge style 4, bend right=15] (45.center) to (39.center);
		\draw [style=new edge style 4, bend right=15] (45.center) to (40.center);
		\draw (14.center) to (3.center);
		\draw (1.center) to (82.center);
		\draw [style=new edge style 5, bend left=360] (84.center) to (83.center);
		\draw [style=new edge style 5] (83.center) to (82.center);
		\draw [style=new edge style 5, bend left=15, looseness=0.75] (82.center) to (5.center);
		\draw [style=new edge style 5, bend left=15, looseness=0.75] (5.center) to (50.center);
		\draw [style=new edge style 5] (50.center) to (67.center);
		\draw [style=new edge style 3] (12.center)
			 to [bend left, looseness=0.75] (98.center)
			 to [bend left=15, looseness=0.75] (99.center)
			 to [bend left, looseness=0.25] (100.center)
			 to [bend left=15, looseness=0.50] (101.center)
			 to [bend left=15, looseness=0.50] (102.center)
			 to [bend left=15, looseness=0.50] (103.center)
			 to [bend left=15, looseness=0.50] (97.center)
			 to [bend left=15, looseness=0.50] (96.center)
			 to [bend left=15, looseness=0.50] (94.center)
			 to [bend left=15, looseness=0.50] (92.center)
			 to [bend left=15, looseness=0.50] (90.center)
			 to [bend left, looseness=0.25] (87.center)
			 to [bend left=15, looseness=0.75] (86.center)
			 to [bend left, looseness=0.75] (62.center);
	\end{pgfonlayer}
\end{tikzpicture}

%% file: images/contractible_nested_square_claim1_1.tikz
\begin{tikzpicture}[scale=1.35]
	\begin{pgfonlayer}{nodelayer}
		\node [style=none] (0) at (0, 4) {};
		\node [style=none] (1) at (-4, 0) {};
		\node [style=none] (2) at (0, -4) {};
		\node [style=none] (3) at (4, 0) {};
		\node [style=none] (4) at (-2, 0.5) {};
		\node [style=none] (5) at (-1, 0.75) {};
		\node [style=none] (6) at (-2, -0.75) {};
		\node [style=none] (7) at (-1, -1.25) {};
		\node [style=none] (8) at (-0.5, -0.5) {};
		\node [style=none] (9) at (0.5, 0) {};
		\node [style=none] (10) at (1.25, 0.25) {};
		\node [style=none] (11) at (-0.25, 1.25) {};
		\node [style=none] (12) at (0.75, 1.75) {};
		\node [style=none] (13) at (2, -0.25) {};
		\node [style=none] (14) at (2.75, -1) {};
		\node [style=none] (15) at (2.25, 1.25) {};
		\node [style=none] (16) at (3.25, 0.5) {};
		\node [style=none] (17) at (-0.25, 1.25) {};
		\node [style=none] (18) at (-2.7, 1) {};
		\node [style=none] (19) at (-2.975, -1.125) {};
		\node [style=none] (20) at (-1, -2.1) {};
		\node [style=none] (21) at (-2.175, -1.75) {};
		\node [style=none] (22) at (0.025, -1.275) {};
		\node [style=none] (23) at (0.5, -0.9) {};
		\node [style=none] (24) at (1, -0.55) {};
		\node [style=none] (25) at (-1.6, 1.275) {};
		\node [style=none] (26) at (-1, 1.575) {};
		\node [style=none] (27) at (-0.675, 1.875) {};
		\node [style=none] (28) at (-0.15, 2.025) {};
		\node [style=none] (29) at (1.725, 2.075) {};
		\node [style=none] (30) at (2.25, 2.125) {};
		\node [style=none] (31) at (2.825, 1.9) {};
		\node [style=none] (32) at (0.9, 2.45) {};
		\node [style=none] (33) at (3.675, 1.075) {};
		\node [style=none] (34) at (2.175, -1.55) {};
		\node [style=none] (35) at (2.75, -1.85) {};
		\node [style=none] (36) at (3.25, -1.675) {};
		\node [style=none] (37) at (3.575, -1) {};
		\node [style=none] (38) at (1.725, -1.15) {};
		\node [style=none] (39) at (1.225, -0.8) {};
		\node [style=none, label={$C_1$}] (40) at (-1.35, -0.75) {};
		\node [style=none, label={$C_2$}] (41) at (0.55, 0.25) {};
		\node [style=none, label={$C_3$}] (42) at (2.575, -0.4) {};
		\node [style=none, label={above:$C$}] (43) at (2.725, 2.975) {};
	\end{pgfonlayer}
	\begin{pgfonlayer}{edgelayer}
		\draw [bend left=45] (1.center) to (0.center);
		\draw [bend left=45] (0.center) to (3.center);
		\draw [bend left=45] (3.center) to (2.center);
		\draw [bend left=45] (2.center) to (1.center);
		\draw [style=new edge style 4] (8.center)
			 to (5.center)
			 to (4.center)
			 to (6.center)
			 to (7.center)
			 to cycle;
		\draw (5.center) to (11.center);
		\draw [style=new edge style 4] (11.center)
			 to (9.center)
			 to (10.center)
			 to (12.center)
			 to (17.center);
		\draw (9.center) to (8.center);
		\draw (12.center) to (15.center);
		\draw [style=new edge style 4] (13.center)
			 to (15.center)
			 to (16.center)
			 to (14.center)
			 to cycle;
		\draw (13.center) to (10.center);
		\draw [style=new edge style 3] (4.center) to (18.center);
		\draw [style=new edge style 3] (6.center) to (19.center);
		\draw [style=new edge style 3] (6.center) to (21.center);
		\draw [style=new edge style 3] (7.center) to (20.center);
		\draw [style=new edge style 3] (8.center) to (22.center);
		\draw [style=new edge style 3] (5.center) to (25.center);
		\draw [style=new edge style 3] (5.center) to (26.center);
		\draw [style=new edge style 3] (17.center) to (27.center);
		\draw [style=new edge style 3] (17.center) to (28.center);
		\draw [style=new edge style 3] (12.center) to (32.center);
		\draw [style=new edge style 3] (15.center) to (29.center);
		\draw [style=new edge style 3] (15.center) to (30.center);
		\draw [style=new edge style 3] (15.center) to (31.center);
		\draw [style=new edge style 3] (16.center) to (33.center);
		\draw [style=new edge style 3] (14.center) to (34.center);
		\draw [style=new edge style 3] (14.center) to (35.center);
		\draw [style=new edge style 3] (14.center) to (36.center);
		\draw [style=new edge style 3] (14.center) to (37.center);
		\draw [style=new edge style 3] (13.center) to (38.center);
		\draw [style=new edge style 3] (13.center) to (39.center);
		\draw [style=new edge style 3] (10.center) to (24.center);
		\draw [style=new edge style 3] (9.center) to (23.center);
	\end{pgfonlayer}
\end{tikzpicture}

%% file: images/contractible_nested_square_claim1_2.tikz
\begin{tikzpicture}
	\begin{pgfonlayer}{nodelayer}
		\node [style=none] (1) at (4.55, 4) {};
		\node [style=none] (3) at (4.575, -4) {};
		\node [style=none] (4) at (0.35, 0) {};
		\node [style=none] (5) at (1.6, 4) {};
		\node [style=none] (6) at (2.85, 0) {};
		\node [style=none] (7) at (1.6, -4) {};
		\node [style=none] (8) at (-3.15, 0) {};
		\node [style=none] (9) at (-1.9, 4) {};
		\node [style=none] (10) at (-0.65, 0) {};
		\node [style=none] (11) at (-1.9, -4) {};
		\node [style=none] (13) at (-5.15, 4) {};
		\node [style=none] (14) at (-3.9, 0) {};
		\node [style=none] (15) at (-5.15, -4) {};
		\node [style=none] (17) at (-8, 4) {};
		\node [style=none] (19) at (-7.975, -4) {};
		\node [style=none] (28) at (-4.4, 3) {};
		\node [style=none] (30) at (-4.15, 1.75) {};
		\node [style=none] (40) at (-4.35, -2.75) {};
		\node [style=none, label={right:$C_2$}] (52) at (-1.15, 2.75) {};
		\node [style=none, label={right:$C_3$}] (53) at (2.35, 2.75) {};
		\node [style=none, label={right:$C_1$}] (55) at (-4.4, 2.75) {};
		\node [style=none] (56) at (-0.725, -1) {};
		\node [style=none] (57) at (-0.725, 1.075) {};
		\node [style=none] (58) at (-2.875, 2.15) {};
		\node [style=none] (59) at (-2.725, -2.625) {};
		\node [style=none] (62) at (-0.25, 4) {};
		\node [style=none] (63) at (2.6, 1.75) {};
		\node [style=none] (64) at (2.675, -1.775) {};
		\node [style=none] (65) at (-0.25, -4) {};
	\end{pgfonlayer}
	\begin{pgfonlayer}{edgelayer}
		\draw [style=new edge style 6] (9.center) to (5.center);
		\draw [in=90, out=-15, looseness=0.50] (5.center) to (6.center);
		\draw [in=0, out=-90, looseness=0.50] (6.center) to (7.center);
		\draw [style=new edge style 6] (7.center) to (11.center);
		\draw [in=-90, out=0, looseness=0.50] (11.center) to (10.center);
		\draw [in=-15, out=90, looseness=0.50] (10.center) to (9.center);
		\draw [style=new edge style 3, in=165, out=-90, looseness=0.50] (4.center) to (7.center);
		\draw [style=new edge style 3, in=180, out=90, looseness=0.50] (4.center) to (5.center);
		\draw [style=new edge style 3, in=165, out=-90, looseness=0.50] (8.center) to (11.center);
		\draw [style=new edge style 3, in=180, out=90, looseness=0.50] (8.center) to (9.center);
		\draw [style=new edge style 4] (17.center) to (13.center);
		\draw [in=90, out=-15, looseness=0.50] (13.center) to (14.center);
		\draw [in=0, out=-90, looseness=0.50] (14.center) to (15.center);
		\draw [style=new edge style 4] (15.center) to (19.center);
		\draw (13.center) to (9.center);
		\draw (5.center) to (1.center);
		\draw (7.center) to (3.center);
		\draw (15.center) to (11.center);
		\draw [style=new edge style 3, bend left=75, looseness=0.50] (15.center) to (13.center);
		\draw (40.center) to (56.center);
		\draw (30.center) to (57.center);
		\draw [style=new edge style 3, in=165, out=45, looseness=0.50] (58.center) to (62.center);
		\draw [in=135, out=0, looseness=0.50] (62.center) to (63.center);
		\draw [style=new edge style 3, in=180, out=-45] (59.center) to (65.center);
		\draw [in=-135, out=0, looseness=0.75] (65.center) to (64.center);
	\end{pgfonlayer}
\end{tikzpicture}

%% file: images/contractible_nested_square_claim1_3.tikz
\begin{tikzpicture}
	\begin{pgfonlayer}{nodelayer}
		\node [style=none] (0) at (6.55, 4) {};
		\node [style=none] (1) at (6.575, -4) {};
		\node [style=none] (2) at (2.35, 0) {};
		\node [style=none] (3) at (3.6, 4) {};
		\node [style=none] (4) at (4.85, 0) {};
		\node [style=none] (5) at (3.6, -4) {};
		\node [style=none] (7) at (0.1, 4) {};
		\node [style=none] (9) at (0.1, -4) {};
		\node [style=none] (10) at (-3.15, 4) {};
		\node [style=none] (11) at (-1.9, 0) {};
		\node [style=none] (12) at (-3.15, -4) {};
		\node [style=none] (13) at (-6, 4) {};
		\node [style=none] (14) at (-5.975, -4) {};
		\node [style=none] (15) at (-2.4, 3) {};
		\node [style=none] (16) at (-2.15, 1.75) {};
		\node [style=none] (17) at (-2.35, -2.75) {};
		\node [style=none, label={left:$C_2$}] (18) at (1.55, -0.025) {};
		\node [style=none, label={right:$C_3$}] (19) at (4.35, 2.75) {};
		\node [style=none, label={right:$C_1$}] (20) at (-2.4, 2.75) {};
		\node [style=none] (21) at (-0.325, -2) {};
		\node [style=none] (22) at (-0.5, 2) {};
		\node [style=none] (23) at (1.975, 2.075) {};
		\node [style=none] (24) at (2.175, -2) {};
		\node [style=none] (25) at (4.675, -1.6) {};
		\node [style=none] (26) at (4.625, 1.75) {};
	\end{pgfonlayer}
	\begin{pgfonlayer}{edgelayer}
		\draw [style=new edge style 6] (7.center) to (3.center);
		\draw [in=90, out=-15, looseness=0.50] (3.center) to (4.center);
		\draw [in=0, out=-90, looseness=0.50] (4.center) to (5.center);
		\draw [style=new edge style 6] (5.center) to (9.center);
		\draw [style=new edge style 3, in=165, out=-90, looseness=0.50] (2.center) to (5.center);
		\draw [style=new edge style 3, in=180, out=90, looseness=0.50] (2.center) to (3.center);
		\draw [style=new edge style 4] (13.center) to (10.center);
		\draw [in=90, out=-15, looseness=0.50] (10.center) to (11.center);
		\draw [in=0, out=-90, looseness=0.50] (11.center) to (12.center);
		\draw [style=new edge style 4] (12.center) to (14.center);
		\draw (10.center) to (7.center);
		\draw (3.center) to (0.center);
		\draw (5.center) to (1.center);
		\draw (12.center) to (9.center);
		\draw [style=new edge style 3, bend left=75, looseness=0.50] (12.center) to (10.center);
		\draw (17.center) to (21.center);
		\draw (16.center) to (22.center);
		\draw (22.center) to (21.center);
		\draw (22.center) to (23.center);
		\draw (23.center) to (24.center);
		\draw (24.center) to (21.center);
		\draw (24.center) to (25.center);
		\draw (23.center) to (26.center);
	\end{pgfonlayer}
\end{tikzpicture}

%% file: images/contractible_nested_square_end.tikz
\begin{tikzpicture}[scale=1.25]
	\begin{pgfonlayer}{nodelayer}
		\node [style=none, label={above:$C$}] (51) at (3.025, 7.45) {};
		\node [style=none] (67) at (2.75, 6) {};
		\node [style=none, label={below:$C'$}] (50) at (2.45, 4.4) {};
		\node [style=none] (52) at (3.875, 7.05) {};
		\node [style=none] (53) at (4.625, 6.6) {};
		\node [style=none] (4) at (-5, 0) {};
		\node [style=none] (85) at (-7.65, 2.45) {};
		\node [style=none] (90) at (-6.25, 0.25) {};
		\node [style=none] (9) at (-8, 0) {};
		\node [style=none] (86) at (-7.85, 1.5) {};
		\node [style=none] (88) at (-7.875, -1.5) {};
		\node [style=none] (89) at (-7.7, -2.2) {};
		\node [style=none] (91) at (-6.75, -0.5) {};
		\node [style=none] (93) at (-7.25, -0.25) {};
		\node [style=none] (94) at (-4.475, -2.25) {};
		\node [style=none] (96) at (-7.25, -3.5) {};
		\node [style=none] (10) at (0, 8) {};
		\node [style=none] (54) at (2, 7.75) {};
		\node [style=none] (55) at (1.25, 7.925) {};
		\node [style=none] (56) at (0.725, 7.975) {};
		\node [style=none] (57) at (-1, 7.95) {};
		\node [style=none] (58) at (-1.65, 7.85) {};
		\node [style=none] (62) at (-2.375, 7.65) {};
		\node [style=none] (65) at (0.775, 5.975) {};
		\node [style=none] (68) at (-3.5, 7.25) {};
		\node [style=none] (69) at (-4.15, 6.875) {};
		\node [style=none] (70) at (-4.575, 6.65) {};
		\node [style=none] (71) at (-4.975, 6.35) {};
		\node [style=none] (72) at (-5.4, 5.975) {};
		\node [style=none] (73) at (-5.9, 5.5) {};
		\node [style=none] (81) at (-7, 4) {};
		\node [style=none] (82) at (-4.15, 2.825) {};
		\node [style=none] (83) at (-5, 3.175) {};
		\node [style=none] (84) at (-3.5, 4.4) {};
		\node [style=none] (6) at (5, 0) {};
		\node [style=none] (11) at (8, 0) {};
		\node [style=none] (33) at (7.1, -3.75) {};
		\node [style=none] (34) at (7.525, -2.775) {};
		\node [style=none] (35) at (7.925, -1.25) {};
		\node [style=none] (26) at (6.075, -5.275) {};
		\node [style=none] (30) at (4.775, -1.95) {};
		\node [style=none] (31) at (5.875, -3.8) {};
		\node [style=none] (32) at (5.625, -3.025) {};
		\node [style=none] (5) at (0, 5) {};
		\node [style=none] (8) at (0, -8) {};
		\node [style=none] (12) at (-1, -7.925) {};
		\node [style=none] (13) at (-2.5, -7.6) {};
		\node [style=none] (14) at (-0.425, -4.975) {};
		\node [style=none] (15) at (0.575, -7.975) {};
		\node [style=none] (24) at (2.75, -6) {};
		\node [style=none] (25) at (3.75, -5.5) {};
		\node [style=none] (36) at (4.925, 0.925) {};
		\node [style=none] (37) at (4.775, 1.55) {};
		\node [style=none] (38) at (7.75, 2.025) {};
		\node [style=none] (39) at (7.525, 2.8) {};
		\node [style=none] (40) at (7.325, 3.25) {};
		\node [style=none] (41) at (7.1, 3.75) {};
		\node [style=none] (42) at (6.85, 4.25) {};
		\node [style=none] (43) at (6.5, 4.75) {};
		\node [style=none] (44) at (6.1, 3.875) {};
		\node [style=none] (45) at (6, 3) {};
		\node [style=none] (46) at (5.1, 4.05) {};
		\node [style=none] (47) at (5.85, 5.5) {};
		\node [style=none] (48) at (5.425, 5.975) {};
		\node [style=none] (49) at (6.75, 1.9) {};
		\node [style=none] (60) at (0, 6.25) {};
		\node [style=none] (66) at (1.225, 6.75) {};
		\node [style=none] (74) at (-2.775, 5.35) {};
		\node [style=none] (75) at (-4.2, 4.475) {};
		\node [style=none] (87) at (-7.975, 0.75) {};
		\node [style=none] (92) at (-7.25, -1.5) {};
		\node [style=none] (95) at (-3.55, -3.525) {};
		\node [style=none] (97) at (-6.05, -5.3) {};
		\node [style=none] (98) at (-4.45, -6.7) {};
		\node [style=none] (7) at (0, -5) {};
		\node [style=none] (16) at (1.1, -7.925) {};
		\node [style=none] (17) at (0, -7) {};
		\node [style=none] (18) at (0.5, -5) {};
		\node [style=none] (19) at (1.825, -7.825) {};
		\node [style=none] (20) at (2.775, -7.525) {};
		\node [style=none] (21) at (5.025, -6.3) {};
		\node [style=none] (22) at (3.675, -7.175) {};
		\node [style=none] (23) at (4.275, -6.825) {};
		\node [style=none] (27) at (0.975, -6) {};
		\node [style=none] (28) at (3.7, -4.675) {};
		\node [style=none] (29) at (4.275, -2.725) {};
		\node [style=none] (59) at (0, 6.75) {};
		\node [style=none] (61) at (-0.05, 7.425) {};
		\node [style=none] (63) at (0.975, 7.525) {};
		\node [style=none] (64) at (-1.275, 7.425) {};
		\node [style=none] (76) at (-4.8, 5.4) {};
		\node [style=none] (77) at (-4.425, 5.725) {};
		\node [style=none] (78) at (-4, 6) {};
		\node [style=none] (79) at (-3.6, 6.275) {};
		\node [style=none] (80) at (-3.7, 5.35) {};
		\node [style=new style 0] (100) at (5, -1) {};
		\node [style=new style 0] (101) at (-4.9, -1.2) {};
		\node [style=none] (102) at (-2.1, 4.225) {};
		\node [style=none] (103) at (-4.65, -1.5) {};
		\node [style=none] (104) at (-4.5, -1.175) {};
		\node [style=none] (105) at (-1.625, 4.45) {};
		\node [style=none] (106) at (3.75, 3) {};
		\node [style=none] (107) at (4.05, 2.625) {};
		\node [style=none] (108) at (-0.75, 4.75) {};
		\node [style=none] (109) at (-1.3, 4.575) {};
		\node [style=none] (110) at (4.6, -0.825) {};
		\node [style=none] (111) at (4.925, -0.6) {};
		\node [style=none] (112) at (3.025, -3.675) {};
		\node [style=none] (113) at (3.45, -3.4) {};
		\node [style=none] (114) at (2.7, -3.975) {};
		\node [style=none] (115) at (2.225, -4.25) {};
		\node [style=none] (116) at (-2.05, -4.125) {};
		\node [style=none] (117) at (-1.575, -4.375) {};
		\node [style=none] (118) at (-3, -3.475) {};
		\node [style=none] (119) at (-2.575, -3.725) {};
		\node [style=none] (120) at (-4.825, -0.775) {};
		\node [style=none] (121) at (-4.525, -1.025) {};
		\node [style=none] (122) at (4.275, 2.325) {};
		\node [style=none] (123) at (4.5, 1.9) {};
		\node [style=none] (124) at (4.575, -1) {};
		\node [style=none] (125) at (4.775, -1.325) {};
	\end{pgfonlayer}
	\begin{pgfonlayer}{edgelayer}
		\draw [style=new edge style 9] (51.center)
			 to (52.center)
			 to (53.center)
			 to (50.center)
			 to (67.center)
			 to cycle;
		\draw [style=new edge style 9] (4.center)
			 to (90.center)
			 to (85.center)
			 to (86.center)
			 to (91.center)
			 to (93.center)
			 to (9.center)
			 to (88.center)
			 to (89.center)
			 to (96.center)
			 to (94.center)
			 to [bend left=15] cycle;
		\draw [style=new edge style 10] (5.center)
			 to [bend left, looseness=0.50] (50.center)
			 to (67.center)
			 to (51.center)
			 to (54.center)
			 to (55.center)
			 to (56.center)
			 to (10.center)
			 to (57.center)
			 to (58.center)
			 to (62.center)
			 to (68.center)
			 to (69.center)
			 to (70.center)
			 to (71.center)
			 to (72.center)
			 to (73.center)
			 to (81.center)
			 to (85.center)
			 to (90.center)
			 to (4.center)
			 to [bend left=15] (82.center)
			 to [bend left, looseness=0.75] cycle;
		\draw (82.center) to (83.center);
		\draw [style=new edge style 9] (30.center)
			 to (32.center)
			 to (31.center)
			 to (26.center)
			 to (33.center)
			 to (34.center)
			 to (35.center)
			 to (11.center)
			 to [bend left=360] (6.center)
			 to [bend left=10] cycle;
		\draw [style=new edge style 10] (87.center)
			 to (9.center)
			 to (93.center)
			 to (91.center)
			 to (86.center)
			 to cycle;
		\draw [style=new edge style 10] (41.center)
			 to (40.center)
			 to (39.center)
			 to (38.center)
			 to [bend left=10] (11.center)
			 to (6.center)
			 to (36.center)
			 to (37.center)
			 to [bend right, looseness=0.75] (50.center)
			 to (53.center)
			 to (48.center)
			 to (47.center)
			 to (43.center)
			 to (42.center)
			 to cycle;
		\draw [style=new edge style 9] (8.center) to (15.center);
		\draw [style=new edge style 9] (8.center)
			 to (15.center)
			 to (16.center)
			 to [bend right=15] (17.center)
			 to (7.center)
			 to (12.center)
			 to cycle;
		\draw [style=new edge style 10] (96.center)
			 to (94.center)
			 to [bend right=15, looseness=0.75] (95.center)
			 to [bend right=15] (14.center)
			 to (7.center)
			 to (12.center)
			 to (13.center)
			 to [bend left=10] (98.center)
			 to [bend left=10] (97.center)
			 to [bend left=10] cycle;
		\draw [style=new edge style 4] (14.center) to (7.center);
		\draw [style=new edge style 4] (95.center) to (98.center);
		\draw [style=new edge style 4] (13.center) to (14.center);
		\draw [style=new edge style 4] (14.center) to (7.center);
		\draw [bend right=15] (15.center) to (17.center);
		\draw [style=new edge style 10] (17.center)
			 to (7.center)
			 to (18.center)
			 to [bend right, looseness=0.75] (29.center)
			 to (30.center)
			 to (32.center)
			 to (31.center)
			 to (26.center)
			 to (21.center)
			 to (23.center)
			 to (22.center)
			 to (20.center)
			 to (19.center)
			 to (16.center)
			 to [bend right=15, looseness=0.75] cycle;
		\draw [style=new edge style 4] (31.center) to (33.center);
		\draw [style=new edge style 4] (32.center) to (34.center);
		\draw [style=new edge style 4] (30.center) to (35.center);
		\draw [style=new edge style 4] (30.center) to (11.center);
		\draw [style=new edge style 4] (36.center) to (11.center);
		\draw [bend right=60, looseness=1.50] (48.center) to (49.center);
		\draw [style=new edge style 4] (52.center) to (50.center);
		\draw [style=new edge style 4] (10.center) to (61.center);
		\draw [style=new edge style 4] (56.center) to (61.center);
		\draw [style=new edge style 4] (57.center) to (61.center);
		\draw [bend right=15] (62.center) to (60.center);
		\draw [style=new edge style 4] (59.center) to (60.center);
		\draw [bend left=15] (66.center) to (65.center);
		\draw [style=new edge style 4] (69.center) to (79.center);
		\draw [style=new edge style 4] (70.center) to (78.center);
		\draw [style=new edge style 4] (71.center) to (77.center);
		\draw [style=new edge style 4] (72.center) to (76.center);
		\draw [bend right] (73.center) to (75.center);
		\draw [bend right] (75.center) to (74.center);
		\draw [bend right] (74.center) to (68.center);
		\draw [style=new edge style 4] (90.center) to (91.center);
		\draw [style=new edge style 4] (91.center) to (92.center);
		\draw [style=new edge style 4] (92.center) to (89.center);
		\draw [style=new edge style 4] (92.center) to (88.center);
		\draw [style=new edge style 4] (93.center) to (87.center);
		\draw [style=new edge style 4] (95.center) to (97.center);
		\draw [style=new edge style 4] (94.center) to (97.center);
		\draw [style=new edge style 4] (83.center) to (81.center);
		\draw [style=new edge style 4] (75.center) to (84.center);
		\draw [bend right=15] (60.center) to (66.center);
		\draw [bend right=15] (66.center) to (54.center);
		\draw [style=new edge style 4] (37.center) to (49.center);
		\draw [style=new edge style 4] (49.center) to (38.center);
		\draw [style=new edge style 5] (58.center) to (64.center);
		\draw [style=new edge style 5] (55.center) to (63.center);
		\draw [style=new edge style 4] (43.center) to (44.center);
		\draw [style=new edge style 4] (42.center) to (44.center);
		\draw [style=new edge style 4] (41.center) to (44.center);
		\draw [bend right] (47.center) to (46.center);
		\draw [bend right=15, looseness=1.25] (46.center) to (45.center);
		\draw [bend right=15] (45.center) to (39.center);
		\draw [bend right=15] (45.center) to (40.center);
		\draw (18.center) to (27.center);
		\draw (28.center) to (29.center);
		\draw [bend left, looseness=1.25] (20.center) to (24.center);
		\draw [style=new edge style 5] (24.center)
			 to [bend left] (25.center)
			 to (28.center)
			 to cycle;
		\draw [bend left] (25.center) to (21.center);
		\draw [style=new edge style 4] (22.center) to (24.center);
		\draw [style=new edge style 4] (23.center) to (25.center);
		\draw [bend left=15] (28.center) to (26.center);
		\draw [style=new edge style 4] (27.center) to (19.center);
		\draw [style=new edge style 5] (59.center)
			 to [bend right=15] (63.center)
			 to (61.center)
			 to (64.center)
			 to [bend right=15] cycle;
		\draw [style=new edge style 5] (78.center)
			 to (77.center)
			 to (76.center)
			 to [bend right, looseness=1.25] (80.center)
			 to [bend right=45, looseness=1.25] (79.center)
			 to cycle;
		\draw [style=new edge style 9] (17.center) to (8.center);
		\draw (83.center) to (84.center);
		\draw (84.center) to (65.center);
		\draw [style=new edge style 11] (105.center)
			 to [bend left=15] (121.center)
			 to (120.center)
			 to [bend right=15] (102.center)
			 to cycle;
		\draw [style=new edge style 11] (107.center)
			 to [bend left=15] (109.center)
			 to (108.center)
			 to [bend right=15] (106.center)
			 to cycle;
		\draw [style=new edge style 11] (122.center)
			 to (123.center)
			 to [bend right=15] (111.center)
			 to (110.center)
			 to [bend left=15] cycle;
		\draw [style=new edge style 11] (124.center)
			 to [bend right=15] (112.center)
			 to (113.center)
			 to [bend left=15] (125.center)
			 to cycle;
		\draw [style=new edge style 11] (116.center)
			 to [bend left=15] (114.center)
			 to (115.center)
			 to [bend left=345] (117.center)
			 to cycle;
		\draw [style=new edge style 11] (118.center)
			 to [bend right, looseness=0.75] (103.center)
			 to (104.center)
			 to [bend left, looseness=0.75] (119.center)
			 to cycle;
	\end{pgfonlayer}
\end{tikzpicture}

%% file: images/square_lemma_between.tikz
\begin{tikzpicture}[scale=1.3]
	\begin{pgfonlayer}{nodelayer}
		\node [style=none] (0) at (0, 4) {};
		\node [style=none] (23) at (2, 3.5) {};
		\node [style=none] (30) at (2.75, 6.475) {};
		\node [style=none] (2) at (4, 0) {};
		\node [style=none] (5) at (7, 0) {};
		\node [style=none] (32) at (6.85, -1.5) {};
		\node [style=none] (9) at (-1.5, -6.875) {};
		\node [style=none] (4) at (0, -7) {};
		\node [style=none] (8) at (-1, -3.9) {};
		\node [style=none] (1) at (-4, 0) {};
		\node [style=none] (6) at (0, 7) {};
		\node [style=none] (7) at (-7, 0) {};
		\node [style=none] (10) at (-4.475, -5.45) {};
		\node [style=none] (11) at (-2.625, -3.075) {};
		\node [style=none] (12) at (-6, -3.75) {};
		\node [style=none] (13) at (-3.5, -2) {};
		\node [style=none] (14) at (-6.925, -1) {};
		\node [style=none] (15) at (-3.95, 0.75) {};
		\node [style=none] (16) at (-3.75, 1.5) {};
		\node [style=none] (17) at (-1.75, 3.6) {};
		\node [style=none] (18) at (-5.875, 3.875) {};
		\node [style=none] (19) at (-6.8, 1.75) {};
		\node [style=none] (20) at (-4.85, 5.1) {};
		\node [style=none] (21) at (-3.425, 6.15) {};
		\node [style=none] (22) at (-1.75, 6.775) {};
		\node [style=none] (24) at (3.5, 2) {};
		\node [style=none] (25) at (3.8, 1.275) {};
		\node [style=none] (26) at (7, 0.75) {};
		\node [style=none] (27) at (6.25, 3.2) {};
		\node [style=none] (28) at (4.9, 5.075) {};
		\node [style=none] (29) at (5.45, 4.475) {};
		\node [style=none] (31) at (6.35, -3) {};
		\node [style=none] (33) at (5.075, -4.875) {};
		\node [style=none] (34) at (4, -5.8) {};
		\node [style=none] (35) at (2, -3.5) {};
		\node [style=none] (36) at (3, -2.725) {};
		\node [style=none] (37) at (3.75, -1.45) {};
		\node [style=none] (38) at (3.525, -1.95) {};
		\node [style=none] (3) at (0, -4) {};
		\node [style=none, label={$f_0$}] (39) at (-4.075, -4) {};
		\node [style=none, label={$f_0$}] (40) at (-0.3, -5.425) {};
		\node [style=none, label={$f_0$}] (41) at (5.5, -0.075) {};
		\node [style=none, label={$f_0$}] (42) at (1, 5.25) {};
	\end{pgfonlayer}
	\begin{pgfonlayer}{edgelayer}
		\draw [style=new edge style 6] (0.center)
			 to (30.center)
			 to (23.center)
			 to [bend right, looseness=0.50] cycle;
		\draw [style=new edge style 6] (5.center)
			 to (2.center)
			 to (32.center)
			 to [bend right=15, looseness=0.25] cycle;
		\draw [style=new edge style 6] (8.center)
			 to (9.center)
			 to (4.center)
			 to cycle;
		\draw [style=new edge style 15] (30.center)
			 to (0.center)
			 to [bend left=345, looseness=0.50] (17.center)
			 to (22.center)
			 to [bend left=15, looseness=0.50] (6.center)
			 to [bend left=15, looseness=0.50] cycle;
		\draw (22.center) to (0.center);
		\draw (0.center) to (6.center);
		\draw [style=new edge style 15] (27.center)
			 to (24.center)
			 to [bend right, looseness=0.50] (23.center)
			 to (29.center)
			 to [bend left=15, looseness=0.50] cycle;
		\draw (23.center) to (27.center);
		\draw [style=new edge style 10] (24.center)
			 to (25.center)
			 to (26.center)
			 to [bend right=15, looseness=0.75] (27.center)
			 to cycle;
		\draw [style=new edge style 15] (5.center)
			 to (2.center)
			 to [bend right=15, looseness=0.75] (25.center)
			 to (26.center)
			 to cycle;
		\draw [style=new edge style 15] (33.center)
			 to (36.center)
			 to [bend right, looseness=0.50] (38.center)
			 to (31.center)
			 to [bend left=15, looseness=0.75] cycle;
		\draw [style=new edge style 10] (34.center)
			 to (35.center)
			 to [bend right=15, looseness=0.75] (36.center)
			 to (33.center)
			 to [bend left=15, looseness=0.50] cycle;
		\draw [style=new edge style 15] (35.center)
			 to (34.center)
			 to [bend left=15] (4.center)
			 to (8.center)
			 to [bend right, looseness=0.25] (3.center)
			 to [bend right=15] cycle;
		\draw [style=new edge style 10] (18.center)
			 to [bend left=15, looseness=0.50] (20.center)
			 to [bend left=15, looseness=0.50] (21.center)
			 to [bend left=15, looseness=0.50] (22.center)
			 to (17.center)
			 to [bend right=15] (16.center)
			 to (19.center)
			 to [bend left=15, looseness=0.75] cycle;
		\draw (19.center) to (17.center);
		\draw (18.center) to (17.center);
		\draw (20.center) to (17.center);
		\draw [style=new edge style 15] (16.center)
			 to (19.center)
			 to [bend right=15, looseness=0.50] (7.center)
			 to (13.center)
			 to [bend left=15] (1.center)
			 to (15.center)
			 to cycle;
		\draw [style=new edge style 6] (12.center)
			 to (13.center)
			 to (14.center)
			 to [bend right=15, looseness=0.75] cycle;
		\draw [style=new edge style 15] (13.center)
			 to (12.center)
			 to [bend right=15, looseness=0.75] (10.center)
			 to (11.center)
			 to [bend left, looseness=0.50] cycle;
		\draw [style=new edge style 10] (7.center)
			 to (13.center)
			 to (14.center)
			 to cycle;
		\draw [style=new edge style 10] (10.center)
			 to [bend right=15] (9.center)
			 to (8.center)
			 to [bend left, looseness=0.50] (11.center)
			 to cycle;
		\draw [style=new edge style 10] (28.center)
			 to [bend left=15, looseness=0.50] (29.center)
			 to (23.center)
			 to (30.center)
			 to [bend left=15, looseness=0.50] cycle;
		\draw [style=new edge style 10] (32.center)
			 to (2.center)
			 to [bend left=15, looseness=0.50] (37.center)
			 to (38.center)
			 to (31.center)
			 to [bend right=15, looseness=0.50] cycle;
		\draw (11.center) to (10.center);
		\draw (3.center) to (4.center);
		\draw (35.center) to (34.center);
		\draw (36.center) to (33.center);
		\draw (37.center) to (31.center);
		\draw (38.center) to (31.center);
		\draw (25.center) to (26.center);
		\draw (24.center) to (27.center);
		\draw (23.center) to (28.center);
		\draw (23.center) to (29.center);
		\draw (21.center) to (17.center);
		\draw (22.center) to (17.center);
		\draw (19.center) to (1.center);
		\draw (19.center) to (15.center);
		\draw (19.center) to (13.center);
		\draw (19.center) to (16.center);
		\draw (7.center) to (13.center);
	\end{pgfonlayer}
\end{tikzpicture}

%% file: images/good_square_the_F_i.tikz
\begin{tikzpicture}
	\begin{pgfonlayer}{nodelayer}
		\node [style=none] (4) at (-5, 0) {};
		\node [style=none] (5) at (0, 5) {};
		\node [style=none] (6) at (5, 0) {};
		\node [style=none] (7) at (0, -5) {};
		\node [style=none] (8) at (0, -8) {};
		\node [style=none] (9) at (-8, 0) {};
		\node [style=none] (10) at (0, 8) {};
		\node [style=none] (11) at (8, 0) {};
		\node [style=none] (12) at (-1, -7.925) {};
		\node [style=none] (13) at (-2.5, -7.6) {};
		\node [style=none] (14) at (-0.425, -4.975) {};
		\node [style=none] (15) at (0.575, -7.975) {};
		\node [style=none] (16) at (1.1, -7.925) {};
		\node [style=none] (17) at (0, -7) {};
		\node [style=none] (18) at (0.5, -5) {};
		\node [style=none] (19) at (1.825, -7.825) {};
		\node [style=none] (20) at (2.775, -7.525) {};
		\node [style=none] (21) at (5.025, -6.3) {};
		\node [style=none] (22) at (3.675, -7.175) {};
		\node [style=none] (23) at (4.275, -6.825) {};
		\node [style=none] (24) at (2.75, -6) {};
		\node [style=none] (25) at (3.75, -5.5) {};
		\node [style=none] (26) at (6.075, -5.275) {};
		\node [style=none] (27) at (0.975, -6) {};
		\node [style=none] (28) at (3.7, -4.675) {};
		\node [style=none] (29) at (4.275, -2.725) {};
		\node [style=none] (30) at (4.85, -1.625) {};
		\node [style=none] (31) at (5.875, -3.8) {};
		\node [style=none] (32) at (5.625, -3.025) {};
		\node [style=none] (33) at (7.1, -3.75) {};
		\node [style=none] (34) at (7.525, -2.775) {};
		\node [style=none] (35) at (7.925, -1.25) {};
		\node [style=none] (36) at (4.925, 0.925) {};
		\node [style=none] (37) at (4.775, 1.55) {};
		\node [style=none] (38) at (7.75, 2.025) {};
		\node [style=none] (39) at (7.525, 2.8) {};
		\node [style=none] (40) at (7.325, 3.25) {};
		\node [style=none] (41) at (7.1, 3.75) {};
		\node [style=none] (42) at (6.85, 4.25) {};
		\node [style=none] (43) at (6.5, 4.75) {};
		\node [style=none] (44) at (6.1, 3.875) {};
		\node [style=none] (45) at (6, 3) {};
		\node [style=none] (46) at (5.1, 4.05) {};
		\node [style=none] (47) at (5.85, 5.5) {};
		\node [style=none] (48) at (5.425, 5.975) {};
		\node [style=none] (49) at (6.75, 1.9) {};
		\node [style=none, label={below:$C'$}] (50) at (2.45, 4.4) {};
		\node [style=none, label={above:$C$}] (51) at (3.025, 7.45) {};
		\node [style=none] (52) at (3.875, 7.05) {};
		\node [style=none] (53) at (4.625, 6.6) {};
		\node [style=none] (54) at (2, 7.75) {};
		\node [style=none] (55) at (1.25, 7.925) {};
		\node [style=none] (56) at (0.725, 7.975) {};
		\node [style=none] (57) at (-1, 7.95) {};
		\node [style=none] (58) at (-1.65, 7.85) {};
		\node [style=none] (60) at (0, 6.25) {};
		\node [style=none] (62) at (-2.375, 7.65) {};
		\node [style=none] (65) at (0.775, 5.975) {};
		\node [style=none] (66) at (1.225, 6.75) {};
		\node [style=none] (67) at (2.75, 6) {};
		\node [style=none] (68) at (-3.5, 7.25) {};
		\node [style=none] (69) at (-4.15, 6.875) {};
		\node [style=none] (70) at (-4.575, 6.65) {};
		\node [style=none] (71) at (-4.975, 6.35) {};
		\node [style=none] (72) at (-5.4, 5.975) {};
		\node [style=none] (73) at (-5.9, 5.5) {};
		\node [style=none] (74) at (-2.775, 5.35) {};
		\node [style=none] (75) at (-4.2, 4.475) {};
		\node [style=none] (81) at (-7, 4) {};
		\node [style=none] (82) at (-4.15, 2.825) {};
		\node [style=none] (83) at (-5, 3.175) {};
		\node [style=none] (84) at (-3.5, 4.4) {};
		\node [style=none] (85) at (-7.65, 2.45) {};
		\node [style=none] (86) at (-7.85, 1.5) {};
		\node [style=none] (87) at (-7.975, 0.75) {};
		\node [style=none] (88) at (-7.875, -1.5) {};
		\node [style=none] (89) at (-7.7, -2.2) {};
		\node [style=none] (90) at (-6.25, 0.25) {};
		\node [style=none] (91) at (-6.75, -0.5) {};
		\node [style=none] (92) at (-7.25, -1.5) {};
		\node [style=none] (93) at (-7.25, -0.25) {};
		\node [style=none] (94) at (-4.475, -2.25) {};
		\node [style=none] (95) at (-3.55, -3.525) {};
		\node [style=none] (96) at (-7.25, -3.5) {};
		\node [style=none] (97) at (-6.05, -5.3) {};
		\node [style=none] (98) at (-4.45, -6.7) {};
		\node [style=none] (76) at (-4.8, 5.4) {};
		\node [style=none] (77) at (-4.425, 5.725) {};
		\node [style=none] (78) at (-4, 6) {};
		\node [style=none] (79) at (-3.6, 6.275) {};
		\node [style=none] (80) at (-3.7, 5.35) {};
		\node [style=none] (59) at (0, 6.75) {};
		\node [style=none] (61) at (-0.05, 7.425) {};
		\node [style=none] (63) at (0.975, 7.525) {};
		\node [style=none] (64) at (-1.275, 7.425) {};
	\end{pgfonlayer}
	\begin{pgfonlayer}{edgelayer}
		\draw [style=new edge style 4] (96.center)
			 to (89.center)
			 to (88.center)
			 to (9.center)
			 to (87.center)
			 to (86.center)
			 to (85.center)
			 to (81.center)
			 to (73.center)
			 to (72.center)
			 to (71.center)
			 to (70.center)
			 to (69.center)
			 to (68.center)
			 to (62.center)
			 to (58.center)
			 to (57.center)
			 to (10.center)
			 to (56.center)
			 to (55.center)
			 to (54.center)
			 to (51.center)
			 to (52.center)
			 to (53.center)
			 to (48.center)
			 to (47.center)
			 to (43.center)
			 to (42.center)
			 to (41.center)
			 to (40.center)
			 to (39.center)
			 to (38.center)
			 to [bend left=10] (11.center)
			 to (35.center)
			 to (34.center)
			 to (33.center)
			 to (26.center)
			 to (21.center)
			 to (23.center)
			 to (22.center)
			 to (20.center)
			 to (19.center)
			 to (16.center)
			 to (15.center)
			 to (8.center)
			 to (12.center)
			 to (13.center)
			 to [bend left=10] (98.center)
			 to [bend left=10] (97.center)
			 to [bend left=10] cycle
			 to (94.center)
			 to [bend right=15, looseness=0.75] (95.center)
			 to [bend right=15] (14.center)
			 to (7.center)
			 to (18.center)
			 to [bend right, looseness=0.75] (29.center)
			 to (30.center)
			 to [bend right=10] (6.center)
			 to (36.center)
			 to (37.center)
			 to [bend right, looseness=0.75] (50.center)
			 to [bend right, looseness=0.50] (5.center)
			 to [bend right, looseness=0.75] (82.center)
			 to [bend right=15] (4.center)
			 to [bend right=15] (94.center);
		\draw [style=new edge style 4] (95.center) to (98.center);
		\draw [style=new edge style 4] (13.center) to (14.center);
		\draw [style=new edge style 4] (14.center) to (7.center);
		\draw [style=new edge style 4] (7.center) to (12.center);
		\draw [style=new edge style 4] (8.center) to (15.center);
		\draw [style=new edge style 7] (8.center)
			 to (15.center)
			 to (16.center)
			 to [bend right=15] (17.center)
			 to cycle;
		\draw [style=new edge style 4, bend right=15] (15.center) to (17.center);
		\draw [style=new edge style 4] (17.center) to (7.center);
		\draw [style=new edge style 4] (7.center) to (14.center);
		\draw [style=new edge style 5] (18.center) to (27.center);
		\draw [style=new edge style 5] (28.center) to (29.center);
		\draw [style=new edge style 5, bend left, looseness=0.75] (29.center) to (18.center);
		\draw [style=new edge style 7] (20.center)
			 to [bend left, looseness=1.25] (24.center)
			 to [bend left] (25.center)
			 to (28.center)
			 to [bend left=15] (26.center)
			 to (21.center)
			 to (23.center)
			 to (22.center)
			 to cycle;
		\draw [style=new edge style 5] (25.center)
			 to [bend right] (24.center)
			 to (28.center)
			 to cycle;
		\draw [style=new edge style 4] (25.center)
			 to [bend left] (21.center)
			 to (23.center)
			 to cycle;
		\draw [style=new edge style 4] (22.center) to (24.center);
		\draw [style=new edge style 7] (34.center)
			 to (32.center)
			 to (31.center)
			 to (26.center)
			 to (33.center)
			 to cycle;
		\draw [style=new edge style 4] (30.center) to (35.center);
		\draw [style=new edge style 4] (30.center) to (11.center);
		\draw [style=new edge style 4] (6.center) to (11.center);
		\draw [style=new edge style 4] (36.center) to (11.center);
		\draw [style=new edge style 7] (42.center)
			 to (41.center)
			 to (40.center)
			 to (39.center)
			 to (38.center)
			 to (49.center)
			 to [bend left=60, looseness=1.50] (48.center)
			 to (47.center)
			 to (43.center)
			 to cycle;
		\draw [style=new edge style 4] (52.center) to (50.center);
		\draw [style=new edge style 4] (53.center) to (50.center);
		\draw [style=new edge style 7] (10.center)
			 to (56.center)
			 to (55.center)
			 to (54.center)
			 to [bend left=15] (66.center)
			 to [bend left=15] (65.center)
			 to (84.center)
			 to (83.center)
			 to (81.center)
			 to (73.center)
			 to (72.center)
			 to (71.center)
			 to (70.center)
			 to (69.center)
			 to (68.center)
			 to (62.center)
			 to (58.center)
			 to (57.center)
			 to cycle;
		\draw [style=new edge style 4] (70.center)
			 to (71.center)
			 to (72.center)
			 to (73.center)
			 to [bend right] (75.center)
			 to [bend right] (74.center)
			 to [bend right] (68.center)
			 to (69.center)
			 to cycle;
		\draw [style=new edge style 7] (75.center) to (84.center);
		\draw [style=new edge style 4] (4.center) to (90.center);
		\draw [style=new edge style 7] (89.center)
			 to (88.center)
			 to (9.center)
			 to (87.center)
			 to (86.center)
			 to (85.center)
			 to (90.center)
			 to (91.center)
			 to (92.center)
			 to cycle;
		\draw [style=new edge style 4] (92.center) to (88.center);
		\draw [style=new edge style 4] (91.center) to (86.center);
		\draw [style=new edge style 4] (91.center) to (93.center);
		\draw [style=new edge style 4] (87.center)
			 to (9.center)
			 to (93.center)
			 to cycle;
		\draw [style=new edge style 4] (95.center) to (97.center);
		\draw [style=new edge style 4] (94.center) to (97.center);
		\draw [bend right=15] (60.center) to (66.center);
		\draw [style=new edge style 4] (67.center) to (51.center);
		\draw [style=new edge style 4] (37.center) to (49.center);
		\draw [style=new edge style 4] (30.center) to (32.center);
		\draw [style=new edge style 4] (27.center) to (19.center);
		\draw [style=new edge style 4] (60.center)
			 to (59.center)
			 to [bend right=15] (63.center)
			 to (55.center)
			 to (56.center)
			 to (61.center)
			 to (64.center)
			 to (58.center)
			 to (62.center)
			 to [bend right=15] cycle;
		\draw [style=new edge style 4] (43.center)
			 to (42.center)
			 to (41.center)
			 to (40.center)
			 to (39.center)
			 to [bend left=15] (45.center)
			 to [bend left=15, looseness=1.25] (46.center)
			 to [bend left] (47.center)
			 to cycle;
		\draw [style=new edge style 4, bend right=15] (45.center) to (40.center);
		\draw [style=new edge style 5, bend left, looseness=0.75] (82.center) to (5.center);
		\draw [style=new edge style 5, in=135, out=15, looseness=0.50] (5.center) to (50.center);
		\draw [style=new edge style 5] (50.center) to (67.center);
		\draw [style=new edge style 5] (83.center) to (82.center);
		\draw [style=new edge style 7] (70.center)
			 to (69.center)
			 to (79.center)
			 to (78.center)
			 to cycle;
		\draw [style=new edge style 7] (76.center)
			 to (72.center)
			 to (71.center)
			 to (77.center)
			 to cycle;
		\draw [style=new edge style 5] (78.center)
			 to (77.center)
			 to (76.center)
			 to [bend right, looseness=1.25] (80.center)
			 to [bend right=45, looseness=1.25] (79.center)
			 to cycle;
		\draw [style=new edge style 4] (57.center) to (61.center);
		\draw [style=new edge style 4] (61.center) to (10.center);
		\draw [style=new edge style 5] (61.center)
			 to (64.center)
			 to [bend right=15] (59.center)
			 to [bend right=15] (63.center)
			 to cycle;
		\draw [style=new edge style 4] (31.center) to (33.center);
		\draw [style=new edge style 7] (43.center)
			 to (42.center)
			 to (41.center)
			 to (44.center)
			 to cycle;
		\draw [style=new edge style 4] (42.center) to (44.center);
		\draw [style=new edge style 4, bend right=15] (64.center) to (59.center);
	\end{pgfonlayer}
\end{tikzpicture}

%% file: images/hexagonal_grid.tikz
\begin{tikzpicture}[scale=1.35]
	\begin{pgfonlayer}{nodelayer}
		\node [style=none] (6) at (0, 2.75) {};
		\node [style=none] (7) at (0, -1.75) {};
		\node [style=none] (8) at (1, 3.5) {};
		\node [style=none] (9) at (2, 2.75) {};
		\node [style=none] (10) at (2, 1.75) {};
		\node [style=none] (11) at (3, 1) {};
		\node [style=none] (12) at (3, 0) {};
		\node [style=none] (13) at (2, -0.75) {};
		\node [style=none] (14) at (2, -1.75) {};
		\node [style=none] (15) at (1, -2.5) {};
		\node [style=none] (16) at (-1, -2.5) {};
		\node [style=none] (17) at (-2, -1.75) {};
		\node [style=none] (18) at (-2, -0.75) {};
		\node [style=none] (19) at (-3, 0) {};
		\node [style=none] (20) at (-3, 1) {};
		\node [style=none] (21) at (-2, 1.75) {};
		\node [style=none] (22) at (-2, 2.75) {};
		\node [style=none] (23) at (-1, 3.5) {};
		\node [style=none] (24) at (-1, -3.5) {};
		\node [style=none] (25) at (1, -3.5) {};
		\node [style=none] (26) at (0, -4.25) {};
		\node [style=none] (27) at (-3, -2.5) {};
		\node [style=none] (28) at (-3, -3.5) {};
		\node [style=none] (29) at (-2, -4.25) {};
		\node [style=none] (30) at (-4, -1.75) {};
		\node [style=none] (31) at (-4, -0.75) {};
		\node [style=none] (32) at (-5, 0) {};
		\node [style=none] (33) at (-5, 1) {};
		\node [style=none] (34) at (-4, 1.75) {};
		\node [style=none] (35) at (-4, 2.75) {};
		\node [style=none] (36) at (-3, 3.5) {};
		\node [style=none] (37) at (-3, 4.5) {};
		\node [style=none] (38) at (-2, 5.25) {};
		\node [style=none] (39) at (-1, 4.5) {};
		\node [style=none] (40) at (0, 5.25) {};
		\node [style=none] (41) at (1, 4.5) {};
		\node [style=none] (42) at (2, 5.25) {};
		\node [style=none] (43) at (3, 4.5) {};
		\node [style=none] (44) at (3, 3.5) {};
		\node [style=none] (45) at (4, 2.75) {};
		\node [style=none] (46) at (4, 1.75) {};
		\node [style=none] (47) at (5, 1) {};
		\node [style=none] (48) at (5, 0) {};
		\node [style=none] (49) at (4, -0.75) {};
		\node [style=none] (50) at (4, -1.75) {};
		\node [style=none] (51) at (3, -2.5) {};
		\node [style=none] (52) at (3, -3.5) {};
		\node [style=none] (53) at (2, -4.25) {};
		\node [style=none] (0) at (0, 1.75) {};
		\node [style=none] (1) at (-1, 1) {};
		\node [style=none] (2) at (-1, 0) {};
		\node [style=none] (3) at (1, 0) {};
		\node [style=none] (4) at (0, -0.75) {};
		\node [style=none] (5) at (1, 1) {};
	\end{pgfonlayer}
	\begin{pgfonlayer}{edgelayer}
		\draw [style=new edge style 4] (45.center)
			 to (44.center)
			 to (43.center)
			 to (42.center)
			 to (41.center)
			 to (40.center)
			 to (39.center)
			 to (38.center)
			 to (37.center)
			 to (36.center)
			 to (35.center)
			 to (34.center)
			 to (33.center)
			 to (32.center)
			 to (31.center)
			 to (30.center)
			 to (27.center)
			 to (28.center)
			 to (29.center)
			 to (24.center)
			 to (26.center)
			 to (25.center)
			 to (53.center)
			 to (52.center)
			 to (51.center)
			 to (50.center)
			 to (49.center)
			 to (48.center)
			 to (47.center)
			 to (46.center)
			 to cycle;
		\draw [style=new edge style 6] (10.center)
			 to (11.center)
			 to (12.center)
			 to (13.center)
			 to (14.center)
			 to (15.center)
			 to (7.center)
			 to (16.center)
			 to (17.center)
			 to (18.center)
			 to (19.center)
			 to (20.center)
			 to (21.center)
			 to (22.center)
			 to (23.center)
			 to (6.center)
			 to (8.center)
			 to (9.center)
			 to cycle;
		\draw (10.center) to (5.center);
		\draw (13.center) to (3.center);
		\draw (7.center) to (4.center);
		\draw (21.center) to (1.center);
		\draw (18.center) to (2.center);
		\draw (39.center) to (23.center);
		\draw (36.center) to (22.center);
		\draw (34.center) to (20.center);
		\draw (31.center) to (19.center);
		\draw (27.center) to (17.center);
		\draw (24.center) to (16.center);
		\draw (25.center) to (15.center);
		\draw (51.center) to (14.center);
		\draw (49.center) to (12.center);
		\draw (46.center) to (11.center);
		\draw (44.center) to (9.center);
		\draw (41.center) to (8.center);
		\draw [style=new edge style 7] (3.center)
			 to (5.center)
			 to (0.center)
			 to (1.center)
			 to (2.center)
			 to (4.center)
			 to cycle;
		\draw (6.center) to (0.center);
	\end{pgfonlayer}
\end{tikzpicture}

%% file: images/homotopic_square.tikz
\begin{tikzpicture}
	\begin{pgfonlayer}{nodelayer}
		\node [style=none] (0) at (3.75, 0) {};
		\node [style=none] (1) at (5, 4) {};
		\node [style=none] (2) at (6.25, 0) {};
		\node [style=none] (3) at (5, -4) {};
		\node [style=none] (4) at (0.5, 0) {};
		\node [style=none] (5) at (1.75, 4) {};
		\node [style=none] (6) at (3, 0) {};
		\node [style=none] (7) at (1.75, -4) {};
		\node [style=none] (8) at (-3, 0) {};
		\node [style=none] (9) at (-1.75, 4) {};
		\node [style=none] (10) at (-0.5, 0) {};
		\node [style=none] (11) at (-1.75, -4) {};
		\node [style=none] (12) at (-6.25, 0) {};
		\node [style=none] (13) at (-5, 4) {};
		\node [style=none] (14) at (-3.75, 0) {};
		\node [style=none] (15) at (-5, -4) {};
		\node [style=none] (16) at (-9.5, 0) {};
		\node [style=none] (17) at (-8.25, 4) {};
		\node [style=none] (18) at (-7, 0) {};
		\node [style=none] (19) at (-8.25, -4) {};
		\node [style=none] (20) at (7, 0) {};
		\node [style=none] (21) at (8.25, 4) {};
		\node [style=none] (22) at (9.5, 0) {};
		\node [style=none] (23) at (8.25, -4) {};
		\node [style=none] (24) at (-10.25, 4) {};
		\node [style=none] (25) at (-10.25, -4) {};
		\node [style=none] (26) at (10, 4) {};
		\node [style=none] (27) at (10, -4) {};
		\node [style=none] (29) at (-4.25, 3) {};
		\node [style=none] (30) at (-7.175, 1.575) {};
		\node [style=none] (31) at (-4, 1.75) {};
		\node [style=none] (32) at (-6.5, 1.25) {};
		\node [style=none] (33) at (-7.075, 1.025) {};
		\node [style=none] (34) at (-7, 0.5) {};
		\node [style=none, label={left:$C^1$}] (35) at (-7.475, 2.75) {};
		\node [style=none] (36) at (-7, -1) {};
		\node [style=none] (37) at (-7.25, -2) {};
		\node [style=none] (38) at (-7.5, -2.75) {};
		\node [style=none] (39) at (-6.5, -2.5) {};
		\node [style=none] (40) at (-6.25, -1.75) {};
		\node [style=none] (41) at (-4.25, -2.75) {};
		\node [style=none, label={right:$C^2$}] (42) at (9.025, 2.75) {};
		\node [style=none] (43) at (9.25, 1.75) {};
		\node [style=none] (44) at (8.325, 2.75) {};
		\node [style=none] (45) at (8.55, 1.7) {};
		\node [style=none] (46) at (7.9, 2.325) {};
		\node [style=none] (47) at (6, 2.025) {};
		\node [style=none] (48) at (6.175, -1) {};
		\node [style=none] (49) at (9.3, -1.75) {};
		\node [style=none] (50) at (5.975, -2.25) {};
		\node [style=none] (51) at (9, -3) {};
		\node [style=none] (52) at (7.75, 1) {};
		\node [style=none, label={left:$C^{1"}$}] (53) at (-0.75, 2.75) {};
		\node [style=none, label={right:$C^{2"}$}] (54) at (2.5, 2.75) {};
		\node [style=none, label={left:$C^{2'}$}] (55) at (6, 2.75) {};
		\node [style=none, label={right:$C^{1'}$}] (56) at (-4.25, 2.75) {};
	\end{pgfonlayer}
	\begin{pgfonlayer}{edgelayer}
		\draw [style=new edge style 4] (23.center)
			 to (3.center)
			 to [in=-90, out=0, looseness=0.50] (2.center)
			 to [in=-15, out=90, looseness=0.50] (1.center)
			 to (21.center)
			 to [in=90, out=-15, looseness=0.50] (22.center)
			 to [in=0, out=-90, looseness=0.50] cycle;
		\draw [style=new edge style 3, in=165, out=-90, looseness=0.50] (0.center) to (3.center);
		\draw [style=new edge style 3, in=180, out=90, looseness=0.50] (0.center) to (1.center);
		\draw [style=new edge style 6] (9.center)
			 to (5.center)
			 to [in=90, out=-15, looseness=0.50] (6.center)
			 to [in=0, out=-90, looseness=0.50] (7.center)
			 to (11.center)
			 to [in=-90, out=0, looseness=0.50] (10.center)
			 to [in=-15, out=90, looseness=0.50] cycle;
		\draw [style=new edge style 3, in=165, out=-90, looseness=0.50] (4.center) to (7.center);
		\draw [style=new edge style 3, in=180, out=90, looseness=0.50] (4.center) to (5.center);
		\draw [style=new edge style 3, in=165, out=-90, looseness=0.50] (8.center) to (11.center);
		\draw [style=new edge style 3, in=180, out=90, looseness=0.50] (8.center) to (9.center);
		\draw [style=new edge style 4] (18.center)
			 to [in=-15, out=90, looseness=0.50] (17.center)
			 to (13.center)
			 to [in=90, out=-15, looseness=0.50] (14.center)
			 to [in=0, out=-90, looseness=0.50] (15.center)
			 to (19.center)
			 to [in=-90, out=0, looseness=0.50] cycle;
		\draw [style=new edge style 3, in=165, out=-90, looseness=0.50] (12.center) to (15.center);
		\draw [style=new edge style 3, in=180, out=90, looseness=0.50] (12.center) to (13.center);
		\draw [style=new edge style 3, in=165, out=-90, looseness=0.50] (16.center) to (19.center);
		\draw [style=new edge style 3, in=180, out=90, looseness=0.50] (16.center) to (17.center);
		\draw [style=new edge style 3, in=165, out=-90, looseness=0.50] (20.center) to (23.center);
		\draw [style=new edge style 3, in=180, out=90, looseness=0.50] (20.center) to (21.center);
		\draw (13.center) to (9.center);
		\draw (5.center) to (1.center);
		\draw (7.center) to (3.center);
		\draw (15.center) to (11.center);
		\draw (24.center) to (17.center);
		\draw (25.center) to (19.center);
		\draw (21.center) to (26.center);
		\draw (23.center) to (27.center);
		\draw (35.center) to (29.center);
		\draw (30.center) to (29.center);
		\draw (30.center) to (32.center);
		\draw (33.center) to (32.center);
		\draw (34.center) to (32.center);
		\draw (32.center) to (31.center);
		\draw (18.center) to (31.center);
		\draw (18.center) to (14.center);
		\draw [bend left=15] (36.center) to (40.center);
		\draw [bend right, looseness=1.25] (39.center) to (40.center);
		\draw [bend right=15] (38.center) to (39.center);
		\draw (37.center) to (39.center);
		\draw (37.center) to (40.center);
		\draw [bend left=15] (36.center) to (41.center);
		\draw [bend right=15] (38.center) to (41.center);
		\draw (42.center) to (44.center);
		\draw [style=new edge style 5] (45.center)
			 to (44.center)
			 to [bend right=60] (46.center)
			 to [bend right, looseness=1.25] cycle;
		\draw (43.center) to (45.center);
		\draw (46.center) to (47.center);
		\draw (48.center) to (49.center);
		\draw (49.center) to (50.center);
		\draw (50.center) to (51.center);
		\draw [style=new edge style 5] (48.center)
			 to [bend right=15, looseness=0.50] (47.center)
			 to (52.center)
			 to cycle;
		\draw (52.center) to (22.center);
	\end{pgfonlayer}
\end{tikzpicture}

%% file: images/fan_vertex.tikz
\begin{tikzpicture}
	\begin{pgfonlayer}{nodelayer}
		\node [style=none] (0) at (0, 3.5) {};
		\node [style=none] (1) at (-4, -0.5) {};
		\node [style=none] (2) at (-3, -0.5) {};
		\node [style=none] (3) at (-2, -0.5) {};
		\node [style=none] (4) at (-1, -0.5) {};
		\node [style=none] (5) at (0, -0.5) {};
		\node [style=none] (6) at (1, -0.5) {};
		\node [style=none] (7) at (2, -0.5) {};
		\node [style=none] (8) at (3, -0.5) {};
		\node [style=none] (9) at (4, -0.5) {};
		\node [style=none] (10) at (-5, -0.5) {};
		\node [style=none] (11) at (5, -0.5) {};
		\node [style=none] (12) at (6, -0.5) {};
		\node [style=none] (13) at (-6, -0.5) {};
	\end{pgfonlayer}
	\begin{pgfonlayer}{edgelayer}
		\draw (0.center) to (1.center);
		\draw (0.center) to (2.center);
		\draw (0.center) to (3.center);
		\draw (0.center) to (4.center);
		\draw (0.center) to (5.center);
		\draw (0.center) to (6.center);
		\draw (0.center) to (7.center);
		\draw (0.center) to (8.center);
		\draw (0.center) to (9.center);
		\draw (1.center)
			 to (2.center)
			 to (3.center)
			 to (4.center)
			 to (5.center)
			 to (6.center)
			 to (7.center)
			 to (8.center)
			 to (9.center);
		\draw (10.center) to (0.center);
		\draw (13.center) to (0.center);
		\draw (13.center) to (10.center);
		\draw (10.center) to (1.center);
		\draw (9.center) to (11.center);
		\draw (11.center) to (12.center);
		\draw (0.center) to (11.center);
		\draw (0.center) to (12.center);
	\end{pgfonlayer}
\end{tikzpicture}

%% file: images/fan_face.tikz
\begin{tikzpicture}
	\begin{pgfonlayer}{nodelayer}
		\node [style=none] (0) at (2, 3.5) {};
		\node [style=none] (1) at (1.5, 3.5) {};
		\node [style=none] (2) at (1, 3.5) {};
		\node [style=none] (3) at (0.5, 3.5) {};
		\node [style=none] (4) at (0, 3.5) {};
		\node [style=none] (5) at (-0.5, 3.5) {};
		\node [style=none] (6) at (-1, 3.5) {};
		\node [style=none] (7) at (-1.5, 3.5) {};
		\node [style=none] (8) at (-2, 3.5) {};
		\node [style=none] (9) at (5, -0.5) {};
		\node [style=none] (10) at (3.75, -0.5) {};
		\node [style=none] (11) at (2.5, -0.5) {};
		\node [style=none] (12) at (1.25, -0.5) {};
		\node [style=none] (13) at (0, -0.5) {};
		\node [style=none] (14) at (-1.25, -0.5) {};
		\node [style=none] (15) at (-2.5, -0.5) {};
		\node [style=none] (16) at (-3.75, -0.5) {};
		\node [style=none] (17) at (-5, -0.5) {};
		\node [style=none] (18) at (-3.75, 4.5) {};
		\node [style=none] (19) at (-3.5, 5.5) {};
		\node [style=none] (20) at (-2.775, 6.025) {};
		\node [style=none] (21) at (-1.85, 6.325) {};
		\node [style=none] (26) at (-2.5, 3.5) {};
		\node [style=none] (27) at (-3, 3.5) {};
		\node [style=none] (28) at (-7.5, -0.5) {};
		\node [style=none] (29) at (-6.25, -0.5) {};
		\node [style=none] (30) at (6.25, -0.5) {};
		\node [style=none] (31) at (7.5, -0.5) {};
		\node [style=none] (32) at (3, 3.5) {};
		\node [style=none] (33) at (2.5, 3.5) {};
		\node [style=none] (34) at (-1, 6.45) {};
		\node [style=none] (35) at (3.75, 4.55) {};
		\node [style=none] (36) at (0, 6.55) {};
		\node [style=none] (37) at (3.5, 5.55) {};
		\node [style=none] (38) at (2.8, 6.025) {};
		\node [style=none] (39) at (1, 6.475) {};
		\node [style=none] (40) at (1.925, 6.325) {};
	\end{pgfonlayer}
	\begin{pgfonlayer}{edgelayer}
		\draw (8.center) to (17.center);
		\draw (7.center) to (16.center);
		\draw (6.center) to (15.center);
		\draw (5.center) to (14.center);
		\draw (4.center) to (13.center);
		\draw (3.center) to (12.center);
		\draw (2.center) to (11.center);
		\draw (1.center) to (10.center);
		\draw (0.center) to (9.center);
		\draw (8.center)
			 to (7.center)
			 to (6.center)
			 to (5.center)
			 to (4.center)
			 to (3.center)
			 to (2.center)
			 to (1.center)
			 to (0.center);
		\draw (17.center)
			 to (16.center)
			 to (15.center)
			 to (14.center)
			 to (13.center)
			 to (12.center)
			 to (11.center)
			 to (10.center)
			 to (9.center);
		\draw [style=new edge style 3, bend left=15] (18.center) to (19.center);
		\draw [style=new edge style 3, bend left=5] (19.center) to (20.center);
		\draw [style=new edge style 3] (20.center) to (21.center);
		\draw (0.center) to (33.center);
		\draw (33.center) to (32.center);
		\draw (9.center) to (30.center);
		\draw (30.center) to (31.center);
		\draw (29.center) to (17.center);
		\draw (28.center) to (29.center);
		\draw (27.center) to (26.center);
		\draw (26.center) to (8.center);
		\draw (27.center) to (28.center);
		\draw (26.center) to (29.center);
		\draw (33.center) to (30.center);
		\draw (32.center) to (31.center);
		\draw [style=new edge style 3, bend right=15] (18.center) to (27.center);
		\draw [style=new edge style 3] (21.center) to (34.center);
		\draw [style=new edge style 3] (34.center) to (36.center);
		\draw [style=new edge style 3] (36.center) to (39.center);
		\draw [style=new edge style 3] (39.center) to (40.center);
		\draw [style=new edge style 3] (40.center) to (38.center);
		\draw [style=new edge style 3, bend left=5] (38.center) to (37.center);
		\draw [style=new edge style 3, bend left=15, looseness=0.75] (37.center) to (35.center);
		\draw [style=new edge style 3, bend left=15] (35.center) to (32.center);
	\end{pgfonlayer}
\end{tikzpicture}

%% file: images/fan_with_an_arch_vertex.tikz
\begin{tikzpicture}[scale=0.85]
	\begin{pgfonlayer}{nodelayer}
		\node [style=none] (0) at (0, 9) {};
		\node [style=none] (1) at (-4, 5) {};
		\node [style=none] (2) at (-3, 5) {};
		\node [style=none] (3) at (-2, 5) {};
		\node [style=none] (4) at (-1, 5) {};
		\node [style=none] (5) at (0, 5) {};
		\node [style=none] (6) at (1, 5) {};
		\node [style=none] (7) at (2, 5) {};
		\node [style=none] (8) at (3, 5) {};
		\node [style=none] (9) at (4, 5) {};
		\node [style=none] (10) at (-5, 5) {};
		\node [style=none] (11) at (5, 5) {};
		\node [style=none] (12) at (6, 5) {};
		\node [style=none] (13) at (-6, 5) {};
		\node [style=none] (14) at (-7.5, 3.25) {};
		\node [style=none] (15) at (-6.75, 2) {};
		\node [style=none] (16) at (-5.5, 1) {};
		\node [style=none] (17) at (-4.25, 0.425) {};
		\node [style=none] (18) at (-3, 0) {};
		\node [style=none] (19) at (-2, -0.25) {};
		\node [style=none] (20) at (-1, -0.35) {};
		\node [style=none] (21) at (0, -0.4) {};
		\node [style=none] (22) at (1, -0.35) {};
		\node [style=none] (23) at (2, -0.25) {};
		\node [style=none] (24) at (3, 0) {};
		\node [style=none] (25) at (4.25, 0.425) {};
		\node [style=none] (26) at (5.5, 1) {};
		\node [style=none] (27) at (6.75, 2) {};
		\node [style=none] (28) at (7.5, 3.25) {};
	\end{pgfonlayer}
	\begin{pgfonlayer}{edgelayer}
		\draw (0.center) to (1.center);
		\draw (0.center) to (2.center);
		\draw (0.center) to (3.center);
		\draw (0.center) to (4.center);
		\draw (0.center) to (5.center);
		\draw (0.center) to (6.center);
		\draw (0.center) to (7.center);
		\draw (0.center) to (8.center);
		\draw (0.center) to (9.center);
		\draw [style=new edge style 5] (1.center) to (10.center);
		\draw [style=new edge style 5] (10.center) to (13.center);
		\draw [style=new edge style 5, in=135, out=-90, looseness=0.75] (14.center) to (15.center);
		\draw [style=new edge style 5, bend right=15, looseness=0.50] (15.center) to (16.center);
		\draw [style=new edge style 5, bend right, looseness=0.25] (16.center) to (17.center);
		\draw [style=new edge style 5] (17.center) to (18.center);
		\draw [style=new edge style 5] (18.center) to (19.center);
		\draw [style=new edge style 5] (19.center) to (20.center);
		\draw [style=new edge style 5] (20.center) to (21.center);
		\draw [style=new edge style 5] (21.center) to (22.center);
		\draw [style=new edge style 5] (22.center) to (23.center);
		\draw [style=new edge style 5] (23.center) to (24.center);
		\draw [style=new edge style 5] (24.center) to (25.center);
		\draw [style=new edge style 5, bend right=15, looseness=0.50] (25.center) to (26.center);
		\draw [style=new edge style 5, bend right=15, looseness=0.50] (26.center) to (27.center);
		\draw [style=new edge style 5, in=-90, out=45, looseness=0.75] (27.center) to (28.center);
		\draw [style=new edge style 5] (12.center) to (11.center);
		\draw [style=new edge style 5] (11.center) to (9.center);
		\draw [style=new edge style 5] (9.center) to (8.center);
		\draw [style=new edge style 5] (8.center) to (7.center);
		\draw [style=new edge style 5] (7.center) to (6.center);
		\draw [style=new edge style 5] (6.center) to (5.center);
		\draw [style=new edge style 5] (5.center) to (4.center);
		\draw [style=new edge style 5] (4.center) to (3.center);
		\draw [style=new edge style 5] (3.center) to (2.center);
		\draw [style=new edge style 5] (2.center) to (1.center);
		\draw (10.center) to (0.center);
		\draw (13.center) to (0.center);
		\draw (0.center) to (11.center);
		\draw (0.center) to (12.center);
		\draw [in=90, out=-150, looseness=0.75] (0.center) to (14.center);
		\draw [in=90, out=-30, looseness=0.75] (0.center) to (28.center);
	\end{pgfonlayer}
\end{tikzpicture}

%% file: images/fan_with_an_arch_face.tikz
\begin{tikzpicture}[scale=0.8]
	\begin{pgfonlayer}{nodelayer}
		\node [style=none] (0) at (2, 5) {};
		\node [style=none] (1) at (1.5, 5) {};
		\node [style=none] (2) at (1, 5) {};
		\node [style=none] (3) at (0.5, 5) {};
		\node [style=none] (4) at (0, 5) {};
		\node [style=none] (5) at (-0.5, 5) {};
		\node [style=none] (6) at (-1, 5) {};
		\node [style=none] (7) at (-1.5, 5) {};
		\node [style=none] (8) at (-2, 5) {};
		\node [style=none] (9) at (5, 1) {};
		\node [style=none] (10) at (3.75, 1) {};
		\node [style=none] (11) at (2.5, 1) {};
		\node [style=none] (12) at (1.25, 1) {};
		\node [style=none] (13) at (0, 1) {};
		\node [style=none] (14) at (-1.25, 1) {};
		\node [style=none] (15) at (-2.5, 1) {};
		\node [style=none] (16) at (-3.75, 1) {};
		\node [style=none] (17) at (-5, 1) {};
		\node [style=none] (18) at (-3.75, 6) {};
		\node [style=none] (19) at (-3.5, 7) {};
		\node [style=none] (20) at (-2.775, 7.525) {};
		\node [style=none] (21) at (-1.85, 7.825) {};
		\node [style=none] (22) at (-2.5, 5) {};
		\node [style=none] (23) at (-3, 5) {};
		\node [style=none] (24) at (-7.5, 1) {};
		\node [style=none] (25) at (-6.25, 1) {};
		\node [style=none] (26) at (6.25, 1) {};
		\node [style=none] (27) at (7.5, 1) {};
		\node [style=none] (28) at (3, 5) {};
		\node [style=none] (29) at (2.5, 5) {};
		\node [style=none] (30) at (-1, 7.95) {};
		\node [style=none] (31) at (3.75, 6.05) {};
		\node [style=none] (32) at (0, 8.05) {};
		\node [style=none] (33) at (3.5, 7.05) {};
		\node [style=none] (34) at (2.8, 7.525) {};
		\node [style=none] (35) at (1, 7.975) {};
		\node [style=none] (36) at (1.925, 7.825) {};
		\node [style=none] (37) at (-9, -0.525) {};
		\node [style=none] (38) at (-8.55, -1.975) {};
		\node [style=none] (39) at (-7.1, -3.475) {};
		\node [style=none] (40) at (-5.775, -4.3) {};
		\node [style=none] (41) at (-4.3, -4.975) {};
		\node [style=none] (42) at (-2.975, -5.375) {};
		\node [style=none] (43) at (-1.6, -5.675) {};
		\node [style=none] (45) at (1.75, -5.675) {};
		\node [style=none] (46) at (3, -5.325) {};
		\node [style=none] (47) at (4.25, -4.9) {};
		\node [style=none] (48) at (5.75, -4.2) {};
		\node [style=none] (49) at (7.025, -3.375) {};
		\node [style=none] (50) at (8.25, -2.05) {};
		\node [style=none] (51) at (8.75, -0.55) {};
		\node [style=none] (52) at (0, -5.8) {};
		\node [style=none] (53) at (-3.5, 5.5) {};
		\node [style=none] (54) at (3.5, 5.5) {};
	\end{pgfonlayer}
	\begin{pgfonlayer}{edgelayer}
		\draw (8.center) to (17.center);
		\draw (7.center) to (16.center);
		\draw (6.center) to (15.center);
		\draw (5.center) to (14.center);
		\draw (4.center) to (13.center);
		\draw (3.center) to (12.center);
		\draw (2.center) to (11.center);
		\draw (1.center) to (10.center);
		\draw (0.center) to (9.center);
		\draw (8.center)
			 to (7.center)
			 to (6.center)
			 to (5.center)
			 to (4.center)
			 to (3.center)
			 to (2.center)
			 to (1.center)
			 to (0.center);
		\draw [style=new edge style 5] (46.center) to (47.center);
		\draw [style=new edge style 5, bend right=15, looseness=0.50] (47.center) to (48.center);
		\draw [style=new edge style 5, bend right=15, looseness=0.25] (48.center) to (49.center);
		\draw [style=new edge style 5, bend right=15, looseness=0.75] (49.center) to (50.center);
		\draw [style=new edge style 5, bend right=15, looseness=0.75] (50.center) to (51.center);
		\draw [style=new edge style 5] (27.center) to (26.center);
		\draw [style=new edge style 5] (26.center) to (9.center);
		\draw [style=new edge style 5] (9.center) to (10.center);
		\draw [style=new edge style 5] (10.center) to (11.center);
		\draw [style=new edge style 5] (11.center) to (12.center);
		\draw [style=new edge style 5] (12.center) to (13.center);
		\draw [style=new edge style 5] (13.center) to (14.center);
		\draw [style=new edge style 5] (14.center) to (15.center);
		\draw [style=new edge style 5] (15.center) to (16.center);
		\draw [style=new edge style 5] (16.center) to (17.center);
		\draw [style=new edge style 5] (17.center) to (25.center);
		\draw [style=new edge style 5] (25.center) to (24.center);
		\draw [style=new edge style 5, bend right=15, looseness=0.75] (37.center) to (38.center);
		\draw [style=new edge style 5, bend right=15, looseness=0.75] (38.center) to (39.center);
		\draw [style=new edge style 5, bend right, looseness=0.25] (39.center) to (40.center);
		\draw [style=new edge style 5, bend right=15, looseness=0.50] (40.center) to (41.center);
		\draw [style=new edge style 5] (41.center) to (42.center);
		\draw [style=new edge style 5] (42.center) to (43.center);
		\draw [style=new edge style 5] (43.center) to (52.center);
		\draw [style=new edge style 5] (52.center) to (45.center);
		\draw [style=new edge style 5] (45.center) to (46.center);
		\draw [style=new edge style 3, bend left=15] (18.center) to (19.center);
		\draw [style=new edge style 3, bend left=5] (19.center) to (20.center);
		\draw [style=new edge style 3] (20.center) to (21.center);
		\draw (0.center) to (29.center);
		\draw (29.center) to (28.center);
		\draw (23.center) to (22.center);
		\draw (22.center) to (8.center);
		\draw (23.center) to (24.center);
		\draw (22.center) to (25.center);
		\draw (29.center) to (26.center);
		\draw (28.center) to (27.center);
		\draw [style=new edge style 3, bend right=15] (18.center) to (23.center);
		\draw [style=new edge style 3] (21.center) to (30.center);
		\draw [style=new edge style 3] (30.center) to (32.center);
		\draw [style=new edge style 3] (32.center) to (35.center);
		\draw [style=new edge style 3] (35.center) to (36.center);
		\draw [style=new edge style 3] (36.center) to (34.center);
		\draw [style=new edge style 3, bend left=5] (34.center) to (33.center);
		\draw [style=new edge style 3, bend left=15, looseness=0.75] (33.center) to (31.center);
		\draw [style=new edge style 3, bend left=15] (31.center) to (28.center);
		\draw [in=90, out=-45, looseness=0.75] (54.center) to (51.center);
		\draw [in=90, out=-135, looseness=0.75] (53.center) to (37.center);
	\end{pgfonlayer}
\end{tikzpicture}

%% file: images/fan_faces.tikz
\begin{tikzpicture}
	\begin{pgfonlayer}{nodelayer}
		\node [style=none] (0) at (0, 3.5) {};
		\node [style=none] (1) at (-4, -0.5) {};
		\node [style=none] (2) at (-3, -0.5) {};
		\node [style=none] (3) at (-2, -0.5) {};
		\node [style=none] (4) at (-1, -0.5) {};
		\node [style=none] (5) at (0, -0.5) {};
		\node [style=none] (6) at (1, -0.5) {};
		\node [style=none] (7) at (2, -0.5) {};
		\node [style=none] (8) at (3, -0.5) {};
		\node [style=none] (9) at (4, -0.5) {};
		\node [style=none] (10) at (-5, -0.5) {};
		\node [style=none] (11) at (5, -0.5) {};
		\node [style=none] (12) at (6, -0.5) {};
		\node [style=none] (13) at (-6, -0.5) {};
		\node [style=none] (14) at (-7, -2) {};
		\node [style=none] (15) at (-5.75, -2) {};
		\node [style=none] (16) at (-4, -2) {};
		\node [style=none] (17) at (-5.25, -2) {};
		\node [style=none] (18) at (-5.25, -2) {};
		\node [style=none] (19) at (-3.5, -1.25) {};
		\node [style=none] (20) at (-3.5, -0.5) {};
		\node [style=none] (21) at (-2.75, -1.25) {};
		\node [style=none] (22) at (-2, -2) {};
		\node [style=none] (23) at (-1, -2) {};
		\node [style=none] (24) at (2.5, -2) {};
		\node [style=none] (25) at (0.5, -0.5) {};
		\node [style=none] (26) at (0.5, -1.25) {};
		\node [style=none] (27) at (1, -0.5) {};
		\node [style=none] (28) at (1.5, -0.5) {};
		\node [style=none] (29) at (3.75, -2) {};
		\node [style=none] (30) at (4.75, -2) {};
		\node [style=none] (31) at (5.75, -2) {};
	\end{pgfonlayer}
	\begin{pgfonlayer}{edgelayer}
		\draw (0.center) to (1.center);
		\draw (0.center) to (2.center);
		\draw (0.center) to (3.center);
		\draw (0.center) to (4.center);
		\draw (0.center) to (5.center);
		\draw (0.center) to (6.center);
		\draw (0.center) to (7.center);
		\draw (0.center) to (8.center);
		\draw (0.center) to (9.center);
		\draw [style=new edge style 4] (13.center)
			 to (14.center)
			 to (15.center)
			 to (18.center)
			 to (16.center)
			 to (22.center)
			 to (23.center)
			 to [in=180, out=0] (24.center)
			 to (29.center)
			 to (30.center)
			 to (31.center)
			 to (12.center)
			 to (11.center)
			 to (9.center)
			 to (8.center)
			 to (7.center)
			 to (6.center)
			 to (5.center)
			 to (4.center)
			 to (3.center)
			 to (2.center)
			 to (1.center)
			 to (10.center)
			 to cycle;
		\draw (10.center) to (0.center);
		\draw (13.center) to (0.center);
		\draw (0.center) to (11.center);
		\draw (0.center) to (12.center);
		\draw (15.center) to (13.center);
		\draw (18.center) to (13.center);
		\draw (18.center) to (10.center);
		\draw (18.center) to (1.center);
		\draw (1.center) to (19.center);
		\draw (19.center) to (20.center);
		\draw (19.center) to (2.center);
		\draw (19.center) to (21.center);
		\draw (21.center) to (3.center);
		\draw [bend right=15] (16.center) to (4.center);
		\draw (22.center) to (5.center);
		\draw (5.center) to (23.center);
		\draw [bend right, looseness=1.50] (5.center) to (26.center);
		\draw (26.center) to (25.center);
		\draw [bend left=45] (27.center) to (26.center);
		\draw [bend right=15] (28.center) to (24.center);
		\draw (7.center) to (24.center);
		\draw (8.center) to (24.center);
		\draw (9.center) to (29.center);
		\draw (30.center) to (11.center);
		\draw (11.center) to (31.center);
	\end{pgfonlayer}
\end{tikzpicture}

%% file: images/fan_with_an_arch_faces.tikz
\begin{tikzpicture}[scale=0.8]
	\begin{pgfonlayer}{nodelayer}
		\node [style=none] (0) at (0, 2.5) {};
		\node [style=none] (1) at (-4, -1.5) {};
		\node [style=none] (2) at (-3, -1.5) {};
		\node [style=none] (3) at (-2, -1.5) {};
		\node [style=none] (4) at (-1, -1.5) {};
		\node [style=none] (5) at (0, -1.5) {};
		\node [style=none] (6) at (1, -1.5) {};
		\node [style=none] (7) at (2, -1.5) {};
		\node [style=none] (8) at (3, -1.5) {};
		\node [style=none] (9) at (4, -1.5) {};
		\node [style=none] (10) at (-5, -1.5) {};
		\node [style=none] (11) at (5, -1.5) {};
		\node [style=none] (12) at (6, -1.5) {};
		\node [style=none] (13) at (-6, -1.5) {};
		\node [style=none] (29) at (5.5, -1.5) {};
		\node [style=none] (30) at (-4.5, -1.5) {};
		\node [style=none] (31) at (-4.5, -2.25) {};
		\node [style=none] (32) at (-4.5, -3.5) {};
		\node [style=none] (33) at (-1.725, -3) {};
		\node [style=none] (34) at (-0.75, -3) {};
		\node [style=none] (35) at (0.5, -1.5) {};
		\node [style=none] (36) at (1.25, -2.25) {};
		\node [style=none] (37) at (3.75, -2.5) {};
		\node [style=none] (38) at (4.5, -2.5) {};
		\node [style=none] (39) at (5, -1.5) {};
		\node [style=none] (40) at (5.675, -2.5) {};
		\node [style=none] (41) at (2.275, -2.925) {};
		\node [style=none] (42) at (-6, -2.675) {};
		\node [style=none] (43) at (-7.5, -3.75) {};
		\node [style=none] (44) at (-6.75, -5) {};
		\node [style=none] (45) at (-5.5, -6) {};
		\node [style=none] (46) at (-4.25, -6.575) {};
		\node [style=none, label={below:$a$}] (47) at (-3, -7) {};
		\node [style=none] (48) at (-2, -7.25) {};
		\node [style=none] (49) at (-1, -7.35) {};
		\node [style=none] (50) at (0, -7.4) {};
		\node [style=none] (51) at (1, -7.35) {};
		\node [style=none] (52) at (2, -7.25) {};
		\node [style=none] (53) at (3, -7) {};
		\node [style=none] (54) at (4.25, -6.575) {};
		\node [style=none] (55) at (5.5, -6) {};
		\node [style=none] (56) at (6.75, -5) {};
		\node [style=none] (57) at (7.5, -3.75) {};
		\node [style=none] (58) at (3.5, -4.75) {};
		\node [style=none] (59) at (6.25, -3.75) {};
		\node [style=none] (60) at (-0.25, -6) {};
		\node [style=none] (61) at (1.75, -6) {};
		\node [style=none] (62) at (-6.75, -2.5) {};
	\end{pgfonlayer}
	\begin{pgfonlayer}{edgelayer}
		\draw (0.center) to (1.center);
		\draw (0.center) to (2.center);
		\draw (0.center) to (3.center);
		\draw (0.center) to (4.center);
		\draw (0.center) to (5.center);
		\draw (0.center) to (6.center);
		\draw (0.center) to (7.center);
		\draw (0.center) to (8.center);
		\draw (0.center) to (9.center);
		\draw [style=new edge style 4] (10.center)
			 to (1.center)
			 to (2.center)
			 to [bend left=45, looseness=0.75] (32.center)
			 to (42.center)
			 to (62.center)
			 to (13.center)
			 to cycle;
		\draw [style=new edge style 4] (10.center) to (42.center);
		\draw [style=new edge style 6] (32.center)
			 to [bend right=45, looseness=0.75] (2.center)
			 to (47.center)
			 to cycle;
		\draw [style=new edge style 6] (51.center)
			 to (52.center)
			 to (53.center)
			 to (54.center)
			 to [bend right=15, looseness=0.50] (55.center)
			 to [bend right=15, looseness=0.50] (56.center)
			 to [in=-90, out=45, looseness=0.75] (57.center)
			 to [in=-30, out=90] (0.center)
			 to [in=90, out=-150] (43.center)
			 to [in=135, out=-90, looseness=0.75] (44.center)
			 to [bend right=15, looseness=0.50] (45.center)
			 to [bend right, looseness=0.25] (46.center)
			 to (47.center)
			 to (32.center)
			 to (42.center)
			 to (62.center)
			 to (13.center)
			 to (0.center)
			 to (12.center)
			 to (40.center)
			 to [bend left, looseness=0.50] (59.center)
			 to [bend left=90, looseness=0.50] (58.center)
			 to [bend left=15] (37.center)
			 to (8.center)
			 to (7.center)
			 to (41.center)
			 to [bend left=15] (61.center)
			 to [bend left=45, looseness=0.50] (60.center)
			 to [bend left=15] (33.center)
			 to (3.center)
			 to (2.center)
			 to (47.center)
			 to (48.center)
			 to (49.center)
			 to (50.center)
			 to cycle;
		\draw [style=new edge style 6] (29.center) to (40.center);
		\draw [style=new edge style 4] (11.center)
			 to (9.center)
			 to (8.center)
			 to (37.center)
			 to (38.center)
			 to (40.center)
			 to (12.center)
			 to (29.center)
			 to (39.center);
		\draw [style=new edge style 4] (29.center) to (40.center);
		\draw [style=new edge style 4] (37.center) to (8.center);
		\draw [style=new edge style 6] (7.center) to (41.center);
		\draw [style=new edge style 4] (41.center)
			 to (7.center)
			 to (6.center)
			 to (5.center)
			 to (4.center)
			 to (3.center)
			 to (33.center)
			 to (34.center)
			 to cycle;
		\draw (10.center) to (0.center);
		\draw (0.center) to (11.center);
		\draw [bend right] (10.center) to (31.center);
		\draw (30.center) to (31.center);
		\draw [bend left=45, looseness=1.25] (1.center) to (31.center);
		\draw [in=165, out=-105, looseness=0.75] (10.center) to (32.center);
		\draw (3.center) to (33.center);
		\draw (33.center) to (4.center);
		\draw (4.center) to (34.center);
		\draw (34.center) to (5.center);
		\draw [bend right=45] (35.center) to (36.center);
		\draw [bend right, looseness=0.75] (36.center) to (7.center);
		\draw (6.center) to (36.center);
		\draw [in=-90, out=0] (34.center) to (7.center);
		\draw (39.center) to (38.center);
		\draw (38.center) to (29.center);
		\draw (38.center) to (9.center);
		\draw (10.center) to (62.center);
	\end{pgfonlayer}
\end{tikzpicture}

%% file: images/arches.tikz
\begin{tikzpicture}
	\begin{pgfonlayer}{nodelayer}
		\node [style=none, label={above:$v$}] (0) at (0, 3.5) {};
		\node [style=none] (1) at (-7, -0.5) {};
		\node [style=none] (2) at (-5, -0.5) {};
		\node [style=none] (3) at (-3, -0.5) {};
		\node [style=none] (4) at (-1, -0.5) {};
		\node [style=none] (5) at (1.25, -0.5) {};
		\node [style=none] (6) at (3, -0.5) {};
		\node [style=none] (7) at (5, -0.5) {};
		\node [style=none, label={below:$f$}] (8) at (7, -0.5) {};
		\node [style=none] (9) at (9, -0.5) {};
		\node [style=none, label={below left:$a$}] (10) at (-9, -0.5) {};
		\node [style=none] (11) at (11, -0.5) {};
		\node [style=none] (13) at (-11, -0.5) {};
		\node [style=none] (14) at (-8, -0.5) {};
		\node [style=none] (15) at (-2, -0.5) {};
		\node [style=none, label={below:$g$}] (16) at (0, -3) {};
		\node [style=none, label={right:$b$}] (17) at (-7.675, -0.875) {};
		\node [style=none, label={left:$c$}] (18) at (-2.3, -0.875) {};
		\node [style=none, label={right:$d$}] (19) at (1.2, -0.8) {};
		\node [style=none, label={left:$e$}] (20) at (2.8, -0.85) {};
	\end{pgfonlayer}
	\begin{pgfonlayer}{edgelayer}
		\draw (0.center) to (1.center);
		\draw (0.center) to (2.center);
		\draw (0.center) to (3.center);
		\draw (0.center) to (4.center);
		\draw (0.center) to (5.center);
		\draw (0.center) to (6.center);
		\draw (0.center) to (7.center);
		\draw (0.center) to (8.center);
		\draw (0.center) to (9.center);
		\draw (1.center) to (2.center);
		\draw (2.center) to (3.center);
		\draw [style=new edge style 4] (15.center)
			 to (4.center)
			 to (5.center)
			 to [bend left, looseness=0.75] (16.center)
			 to [in=-105, out=-150] (10.center)
			 to (14.center)
			 to [bend right=60, looseness=1.25] cycle;
		\draw (5.center) to (6.center);
		\draw [style=new edge style 4] (8.center)
			 to [bend left=15] (16.center)
			 to [bend right=15, looseness=1.25] (6.center)
			 to (7.center)
			 to cycle;
		\draw (8.center) to (9.center);
		\draw (10.center) to (0.center);
		\draw (13.center) to (0.center);
		\draw (13.center) to (10.center);
		\draw (9.center) to (11.center);
		\draw (0.center) to (11.center);
		\draw (14.center) to (1.center);
		\draw (3.center) to (15.center);
	\end{pgfonlayer}
\end{tikzpicture}

%% file: images/maximal_arches.tikz
\begin{tikzpicture}
	\begin{pgfonlayer}{nodelayer}
		\node [style=none] (0) at (0, 3.5) {};
		\node [style=none] (1) at (-4, -0.5) {};
		\node [style=none] (2) at (-3, -0.5) {};
		\node [style=none] (3) at (-2, -0.5) {};
		\node [style=none] (4) at (-1, -0.5) {};
		\node [style=none] (5) at (0, -0.5) {};
		\node [style=none] (6) at (1, -0.5) {};
		\node [style=none] (7) at (2, -0.5) {};
		\node [style=none] (8) at (3, -0.5) {};
		\node [style=none] (9) at (4, -0.5) {};
		\node [style=none] (10) at (-5, -0.5) {};
		\node [style=none] (11) at (5, -0.5) {};
		\node [style=none] (12) at (6, -0.5) {};
		\node [style=none] (13) at (-6, -0.5) {};
		\node [style=none] (14) at (-7, -2) {};
		\node [style=none] (15) at (-5.75, -2) {};
		\node [style=none] (16) at (-3.75, -2) {};
		\node [style=none] (17) at (-5.25, -2) {};
		\node [style=none] (18) at (-5.25, -2) {};
		\node [style=none] (19) at (-3.5, -1.25) {};
		\node [style=none] (20) at (-3.5, -0.5) {};
		\node [style=none] (21) at (-2.75, -1) {};
		\node [style=none] (22) at (-2, -2) {};
		\node [style=none] (23) at (-1, -2) {};
		\node [style=none] (24) at (2.5, -2) {};
		\node [style=none] (25) at (0.5, -0.5) {};
		\node [style=none] (26) at (0.5, -1.25) {};
		\node [style=none] (27) at (1, -0.5) {};
		\node [style=none] (28) at (1.5, -0.5) {};
		\node [style=none] (29) at (3.75, -2) {};
		\node [style=none] (30) at (4.75, -2) {};
		\node [style=none] (31) at (5.75, -2) {};
		\node [style=none] (32) at (-7, -0.5) {};
		\node [style=none] (33) at (-8, -0.5) {};
		\node [style=none] (34) at (-9, -0.5) {};
		\node [style=none] (35) at (-10, -0.5) {};
		\node [style=none] (36) at (7, -0.5) {};
		\node [style=none] (37) at (8, -0.5) {};
		\node [style=none] (38) at (9, -0.5) {};
		\node [style=none] (39) at (10, -0.5) {};
		\node [style=none] (40) at (-10, -2) {};
		\node [style=none] (41) at (9, -2) {};
		\node [style=none] (42) at (10, -2) {};
		\node [style=none] (43) at (-7.5, -1) {};
	\end{pgfonlayer}
	\begin{pgfonlayer}{edgelayer}
		\draw (0.center) to (1.center);
		\draw (0.center) to (2.center);
		\draw (0.center) to (3.center);
		\draw (0.center) to (4.center);
		\draw (0.center) to (5.center);
		\draw (0.center) to (6.center);
		\draw (0.center) to (7.center);
		\draw (0.center) to (8.center);
		\draw (0.center) to (9.center);
		\draw [style=new edge style 3] (15.center) to (18.center);
		\draw [style=new edge style 3] (18.center) to (10.center);
		\draw [style=new edge style 4] (4.center)
			 to (3.center)
			 to (2.center)
			 to (1.center)
			 to (10.center)
			 to (13.center)
			 to (32.center)
			 to (33.center)
			 to (34.center)
			 to (35.center)
			 to (40.center)
			 to [bend right] (16.center)
			 to cycle;
		\draw [style=new edge style 3] (14.center) to (15.center);
		\draw [style=new edge style 6] (4.center)
			 to (16.center)
			 to (22.center)
			 to (23.center)
			 to (24.center)
			 to (29.center)
			 to (9.center)
			 to (8.center)
			 to (7.center)
			 to (6.center)
			 to (5.center)
			 to cycle;
		\draw [style=new edge style 3] (22.center) to (5.center);
		\draw [style=new edge style 4] (38.center)
			 to (39.center)
			 to (42.center)
			 to (41.center)
			 to (31.center)
			 to (30.center)
			 to (29.center)
			 to (9.center)
			 to (11.center)
			 to (12.center)
			 to (36.center)
			 to (37.center)
			 to cycle;
		\draw [style=new edge style 3] (30.center) to (11.center);
		\draw [style=new edge style 3] (31.center) to (12.center);
		\draw (10.center) to (0.center);
		\draw (13.center) to (0.center);
		\draw (0.center) to (11.center);
		\draw (0.center) to (12.center);
		\draw [style=new edge style 3] (15.center) to (13.center);
		\draw [style=new edge style 3] (18.center) to (13.center);
		\draw [style=new edge style 3] (18.center) to (1.center);
		\draw [style=new edge style 3] (1.center) to (19.center);
		\draw [style=new edge style 3] (19.center) to (20.center);
		\draw [style=new edge style 3] (19.center) to (2.center);
		\draw [style=new edge style 3] (19.center)
			 to (21.center)
			 to (3.center);
		\draw [style=new edge style 3] (5.center) to (23.center);
		\draw [style=new edge style 3, bend right, looseness=1.50] (5.center) to (26.center);
		\draw [style=new edge style 3] (26.center) to (25.center);
		\draw [style=new edge style 3, bend left=45] (27.center) to (26.center);
		\draw [style=new edge style 3, bend right=15] (28.center) to (24.center);
		\draw [style=new edge style 3] (7.center) to (24.center);
		\draw [style=new edge style 3] (8.center) to (24.center);
		\draw [style=new edge style 3] (11.center) to (31.center);
		\draw (0.center) to (32.center);
		\draw (0.center) to (33.center);
		\draw (0.center) to (34.center);
		\draw (0.center) to (35.center);
		\draw (0.center) to (36.center);
		\draw (0.center) to (37.center);
		\draw (0.center) to (38.center);
		\draw (0.center) to (39.center);
		\draw [style=new edge style 3] (41.center) to (38.center);
		\draw [style=new edge style 3] (41.center) to (39.center);
		\draw [style=new edge style 3] (13.center) to (14.center);
		\draw [style=new edge style 3] (33.center) to (43.center);
		\draw [style=new edge style 3] (43.center) to (32.center);
		\draw [style=new edge style 3] (43.center) to (14.center);
	\end{pgfonlayer}
\end{tikzpicture}

%% file: images/construction_A_i.tikz
\begin{tikzpicture}[scale=1.6]
	\begin{pgfonlayer}{nodelayer}
		\node [style=none] (3) at (-2, -0.5) {};
		\node [style=none] (4) at (-1, -0.5) {};
		\node [style=none] (7) at (2, -0.5) {};
		\node [style=none] (15) at (1, -0.5) {};
		\node [style=none] (2) at (-3, -0.5) {};
		\node [style=none] (8) at (3, -0.5) {};
		\node [style=none] (9) at (4, -0.5) {};
		\node [style=none] (1) at (-4, -0.5) {};
		\node [style=none] (11) at (5, -0.5) {};
		\node [style=none] (0) at (0, 4) {};
		\node [style=none] (6) at (1, -0.5) {};
		\node [style=none] (10) at (-5, -0.5) {};
		\node [style=none] (12) at (6, -0.5) {};
		\node [style=none] (31) at (8, -2) {};
		\node [style=none] (32) at (5, -5.75) {};
		\node [style=none] (33) at (-0.25, -6) {};
		\node [style=none] (34) at (-4.5, -3.75) {};
		\node [style=none] (22) at (2.75, -4.05) {};
		\node [style=none] (23) at (3.525, -3.6) {};
		\node [style=none, label={below:$A_2$}] (25) at (0.25, -4.525) {};
		\node [style=none] (14) at (-3.5, -0.5) {};
		\node [style=none] (20) at (0, -3.275) {};
		\node [style=none] (21) at (-2.75, -2.25) {};
		\node [style=none] (19) at (0.75, -2.475) {};
		\node [style=none] (5) at (0, -0.5) {};
		\node [style=none] (13) at (-6, -0.5) {};
		\node [style=none] (16) at (1.5, -0.5) {};
		\node [style=none] (17) at (-1.5, -0.5) {};
		\node [style=none] (18) at (-0.75, -1.925) {};
		\node [style=none] (24) at (4.75, -0.5) {};
		\node [style=none] (26) at (-3.75, -0.5) {};
		\node [style=none, label={right:$A_3$}] (27) at (1.6, -3.25) {};
		\node [style=none, label={right:$A_4$}] (28) at (1.775, -1.15) {};
		\node [style=none, label={right:$A_5$}] (29) at (-0.275, -2.15) {};
		\node [style=none, label={above:$A_6$}] (30) at (0.25, -1.425) {};
		\node [style=none, label={right:$A_1$}] (35) at (5.5, -5.75) {};
	\end{pgfonlayer}
	\begin{pgfonlayer}{edgelayer}
		\draw [style=new edge style 4] (6.center)
			 to (7.center)
			 to (8.center)
			 to (9.center)
			 to (11.center)
			 to (12.center)
			 to (0.center)
			 to [in=90, out=-30, looseness=0.50] (31.center)
			 to [in=15, out=270, looseness=0.75] (32.center)
			 to [bend left, looseness=0.25] (33.center)
			 to [bend left, looseness=0.75] (34.center)
			 to [in=-105, out=135, looseness=0.75] (10.center)
			 to (1.center)
			 to [in=180, out=360] (2.center)
			 to (3.center)
			 to (4.center)
			 to [bend right=75, looseness=1.50] (15.center);
		\draw [style=new edge style 7] (15.center)
			 to [bend left=75, looseness=1.50] (4.center)
			 to (3.center)
			 to (2.center)
			 to (1.center)
			 to [in=-180, out=-90, looseness=1.25] (25.center)
			 to [bend right=15, looseness=0.75] (22.center)
			 to [bend right=15, looseness=0.50] (23.center)
			 to [in=-90, out=15, looseness=0.50] (11.center)
			 to (9.center)
			 to (8.center)
			 to (7.center)
			 to cycle;
		\draw [style=new edge style 4] (21.center)
			 to [in=-90, out=150, looseness=0.50] (14.center)
			 to (2.center)
			 to (3.center)
			 to (4.center)
			 to [bend right=75, looseness=1.50] (15.center)
			 to (7.center)
			 to (8.center)
			 to (9.center)
			 to [bend left, looseness=1.25] (20.center)
			 to [bend left=15, looseness=0.75] cycle;
		\draw [style=new edge style 7] (3.center)
			 to (4.center)
			 to [bend right=75, looseness=1.50] (15.center)
			 to (7.center)
			 to [bend left] (19.center)
			 to [bend right=300] cycle;
		\draw (0.center) to (1.center);
		\draw (0.center) to (2.center);
		\draw (0.center) to (3.center);
		\draw (0.center) to (4.center);
		\draw (0.center) to (5.center);
		\draw (0.center) to (6.center);
		\draw (0.center) to (7.center);
		\draw (0.center) to (8.center);
		\draw (0.center) to (9.center);
		\draw [style=new edge style 4] (6.center) to (5.center);
		\draw [style=new edge style 4] (5.center) to (4.center);
		\draw [style=new edge style 4] (10.center) to (13.center);
		\draw (10.center) to (0.center);
		\draw (13.center) to (0.center);
		\draw (0.center) to (11.center);
		\draw [bend left=15] (16.center) to (19.center);
		\draw [bend left=15] (7.center) to (19.center);
		\draw [style=new edge style 4] (3.center)
			 to (4.center)
			 to [bend right=75, looseness=1.50] (6.center)
			 to (16.center)
			 to [bend left=45] (18.center)
			 to [bend left, looseness=0.75] cycle;
		\draw [bend right] (10.center) to (33.center);
		\draw [bend right] (33.center) to (12.center);
		\draw [bend left=15, looseness=0.75] (24.center) to (23.center);
		\draw [in=45, out=-75, looseness=0.75] (9.center) to (22.center);
		\draw [bend left=45] (9.center) to (25.center);
		\draw [bend right, looseness=1.25] (26.center) to (25.center);
		\draw [bend right] (2.center) to (20.center);
		\draw [bend left, looseness=1.25] (8.center) to (20.center);
		\draw (2.center) to (21.center);
		\draw [bend right=15] (17.center) to (18.center);
		\draw [bend left] (16.center) to (18.center);
	\end{pgfonlayer}
\end{tikzpicture}

%% file: images/radius_graph.tikz
\begin{tikzpicture}
	\begin{pgfonlayer}{nodelayer}
		\node [style=none] (4) at (-5, -1.5) {};
		\node [style=none] (5) at (0, 3.5) {};
		\node [style=none] (6) at (5, -1.5) {};
		\node [style=none] (7) at (0, -6.5) {};
		\node [style=none] (8) at (0, -8) {};
		\node [style=none] (9) at (-8, 0) {};
		\node [style=none] (10) at (0, 8) {};
		\node [style=none] (11) at (8, 0) {};
		\node [style=none] (12) at (-1, -7.925) {};
		\node [style=none] (13) at (-2.5, -7.6) {};
		\node [style=none] (14) at (-0.425, -6.475) {};
		\node [style=none] (15) at (0.575, -7.975) {};
		\node [style=none] (16) at (1.1, -7.925) {};
		\node [style=none] (17) at (0, -7) {};
		\node [style=none] (19) at (1.825, -7.825) {};
		\node [style=none] (20) at (2.775, -7.525) {};
		\node [style=none] (21) at (5.025, -6.3) {};
		\node [style=none] (22) at (3.675, -7.175) {};
		\node [style=none] (23) at (4.275, -6.825) {};
		\node [style=none] (24) at (3, -6.25) {};
		\node [style=none] (25) at (3.95, -5.85) {};
		\node [style=none] (26) at (6.075, -5.275) {};
		\node [style=none] (30) at (4.85, -3.125) {};
		\node [style=none] (31) at (5.875, -3.8) {};
		\node [style=none] (32) at (5.625, -3.025) {};
		\node [style=none] (33) at (7.1, -3.75) {};
		\node [style=none] (34) at (7.525, -2.775) {};
		\node [style=none] (35) at (7.925, -1.25) {};
		\node [style=none] (36) at (4.925, -0.575) {};
		\node [style=none] (37) at (4.775, 0.05) {};
		\node [style=none] (38) at (7.75, 2.025) {};
		\node [style=none] (39) at (7.525, 2.8) {};
		\node [style=none] (40) at (7.325, 3.25) {};
		\node [style=none] (41) at (7.1, 3.75) {};
		\node [style=none] (42) at (6.85, 4.25) {};
		\node [style=none] (43) at (6.5, 4.75) {};
		\node [style=none] (44) at (6.1, 3.875) {};
		\node [style=none] (45) at (6, 3) {};
		\node [style=none] (46) at (5.1, 4.05) {};
		\node [style=none] (47) at (5.85, 5.5) {};
		\node [style=none] (48) at (5.425, 5.975) {};
		\node [style=none] (49) at (6.75, 1.9) {};
		\node [style=none] (50) at (2.45, 2.9) {};
		\node [style=none] (51) at (3.025, 7.45) {};
		\node [style=none] (52) at (3.875, 7.05) {};
		\node [style=none, label={above:$C_0$}] (53) at (4.625, 6.6) {};
		\node [style=none] (54) at (2, 7.75) {};
		\node [style=none] (55) at (1.25, 7.925) {};
		\node [style=none] (56) at (0.725, 7.975) {};
		\node [style=none] (57) at (-1, 7.95) {};
		\node [style=none] (58) at (-1.65, 7.85) {};
		\node [style=none, label={below right:$C_{1,2}$}] (59) at (-0.025, 6.55) {};
		\node [style=none] (60) at (0, 5.65) {};
		\node [style=none] (61) at (-0.05, 7.425) {};
		\node [style=none] (62) at (-2.375, 7.65) {};
		\node [style=none] (63) at (0.85, 7.275) {};
		\node [style=none] (64) at (-1.025, 7.275) {};
		\node [style=none] (67) at (2.75, 6) {};
		\node [style=none] (68) at (-3.5, 7.25) {};
		\node [style=none] (69) at (-4.15, 6.875) {};
		\node [style=none] (70) at (-4.575, 6.65) {};
		\node [style=none] (71) at (-4.975, 6.35) {};
		\node [style=none] (72) at (-5.4, 5.975) {};
		\node [style=none] (73) at (-6.4, 5) {};
		\node [style=none] (74) at (-2.575, 5.35) {};
		\node [style=none] (75) at (-4, 4.05) {};
		\node [style=none] (76) at (-4.5, 5.175) {};
		\node [style=none] (77) at (-4.425, 5.725) {};
		\node [style=none] (78) at (-4, 6) {};
		\node [style=none] (79) at (-3.425, 6.025) {};
		\node [style=none, label={below:$C_{1,1}$}] (80) at (-3.7, 5.2) {};
		\node [style=none] (81) at (-7, 4) {};
		\node [style=none, label={left:$C_{1,3}$}] (82) at (-3.9, 1.825) {};
		\node [style=none] (85) at (-7.65, 2.45) {};
		\node [style=none] (86) at (-7.85, 1.5) {};
		\node [style=none] (87) at (-7.975, 0.75) {};
		\node [style=none] (88) at (-7.875, -1.5) {};
		\node [style=none] (89) at (-7.7, -2.2) {};
		\node [style=none] (90) at (-6.25, 0.25) {};
		\node [style=none] (91) at (-6.75, -0.5) {};
		\node [style=none] (92) at (-7.25, -1.5) {};
		\node [style=none] (93) at (-7.25, -0.25) {};
		\node [style=none] (94) at (-4.475, -3.75) {};
		\node [style=none] (95) at (-3.55, -5.025) {};
		\node [style=none] (96) at (-7.25, -3.5) {};
		\node [style=none] (97) at (-6.05, -5.3) {};
		\node [style=none] (98) at (-4.45, -6.7) {};
		\node [style=none] (99) at (3.75, -0.5) {};
		\node [style=none] (100) at (3.75, -2) {};
		\node [style=none] (105) at (1.25, -4.75) {};
		\node [style=none] (106) at (1.25, -6.25) {};
		\node [style=none] (107) at (1.925, -6.075) {};
		\node [style=none] (108) at (2.575, -5.75) {};
		\node [style=none] (109) at (3.75, -4.85) {};
		\node [style=none] (110) at (1.5, -5.25) {};
		\node [style=none] (111) at (2.25, -4.75) {};
		\node [style=none] (0) at (-3.975, -1.975) {};
		\node [style=none] (1) at (-0.975, 1.025) {};
		\node [style=none] (2) at (2.025, -1.975) {};
		\node [style=none] (3) at (-0.975, -4.975) {};
		\node [style=none] (101) at (0.475, 0.7) {};
		\node [style=none] (102) at (1.625, -0.5) {};
		\node [style=none] (103) at (-3.075, -4.125) {};
		\node [style=none] (104) at (1, -4.25) {};
		\node [style=none] (112) at (3, -1.5) {};
		\node [style=none] (113) at (-2.5, -0.5) {};
		\node [style=none] (114) at (-0.75, -2) {};
		\node [style=none] (115) at (0.5, -3.25) {};
		\node [style=none] (116) at (-3.325, -3.8) {};
		\node [style=none] (117) at (-2.25, 0.75) {};
		\node [style=none] (118) at (1.4, -0.15) {};
		\node [style=none] (119) at (-1.425, 3.35) {};
		\node [style=none] (120) at (4.5, 0.75) {};
		\node [style=none, label={left:$C_{2,1}$}] (121) at (-3.5, -0.25) {};
		\node [style=none, label={above:$C_{2,2}$}] (122) at (2.75, -1.075) {};
		\node [style=none, label={above:$C_{2,3}$}] (123) at (1.75, -4.5) {};
	\end{pgfonlayer}
	\begin{pgfonlayer}{edgelayer}
		\draw [style=new edge style 4] (52.center)
			 to (51.center)
			 to (54.center)
			 to (55.center)
			 to (56.center)
			 to (10.center)
			 to (57.center)
			 to (58.center)
			 to (62.center)
			 to (68.center)
			 to (69.center)
			 to (70.center)
			 to (71.center)
			 to (72.center)
			 to (73.center)
			 to (81.center)
			 to (85.center)
			 to (86.center)
			 to (87.center)
			 to (9.center)
			 to (88.center)
			 to (89.center)
			 to (96.center)
			 to [bend right=10] (97.center)
			 to [bend right=10] (98.center)
			 to [bend right=10] (13.center)
			 to (12.center)
			 to (8.center)
			 to (15.center)
			 to (16.center)
			 to (19.center)
			 to (20.center)
			 to (22.center)
			 to (23.center)
			 to (21.center)
			 to (26.center)
			 to (33.center)
			 to (34.center)
			 to (35.center)
			 to (11.center)
			 to [bend right=10] (38.center)
			 to (39.center)
			 to (40.center)
			 to (41.center)
			 to (42.center)
			 to (43.center)
			 to (47.center)
			 to (48.center)
			 to (53.center)
			 to cycle
			 to (50.center)
			 to [bend right, looseness=0.50] (5.center)
			 to [bend right=45, looseness=0.50] (82.center)
			 to [bend right=15] (4.center)
			 to [bend right=15] (94.center)
			 to [bend right=15, looseness=0.75] (95.center)
			 to [bend right=15] (14.center)
			 to (7.center)
			 to [bend right=15] (106.center)
			 to (107.center)
			 to (108.center)
			 to [bend right=15] (109.center)
			 to [bend right=15] (30.center)
			 to [bend right=10] (6.center)
			 to (36.center)
			 to (37.center)
			 to [in=345, out=105, looseness=0.75] (50.center);
		\draw [style=new edge style 6] (36.center)
			 to (37.center)
			 to [in=345, out=105, looseness=0.75] (50.center)
			 to [bend right, looseness=0.50] (5.center)
			 to [bend right=45, looseness=0.50] (82.center)
			 to [bend right=15] (4.center)
			 to [bend right=15] (94.center)
			 to [bend right=15, looseness=0.75] (95.center)
			 to [bend right=15] (14.center)
			 to (7.center)
			 to [bend right=15] (106.center)
			 to (107.center)
			 to (108.center)
			 to [bend right=15] (109.center)
			 to [bend right=15] (30.center)
			 to [bend right=10] (6.center)
			 to cycle;
		\draw [style=new edge style 6, bend right, looseness=0.50] (50.center) to (5.center);
		\draw [style=new edge style 6, bend right=45, looseness=0.50] (5.center) to (82.center);
		\draw [style=new edge style 6, bend right=15] (82.center) to (4.center);
		\draw [style=new edge style 6, bend right=15] (4.center) to (94.center);
		\draw [style=new edge style 6, bend right=15, looseness=0.75] (94.center) to (95.center);
		\draw [style=new edge style 6, bend right=15] (95.center) to (14.center);
		\draw [style=new edge style 4] (14.center) to (7.center);
		\draw [style=new edge style 6, bend right=15] (7.center) to (106.center);
		\draw [style=new edge style 6] (106.center) to (107.center);
		\draw [style=new edge style 6] (107.center) to (108.center);
		\draw [style=new edge style 6, bend right=15] (108.center) to (109.center);
		\draw [style=new edge style 6, bend right=15] (109.center) to (30.center);
		\draw [style=new edge style 6, bend right=10] (30.center) to (6.center);
		\draw [style=new edge style 6] (6.center) to (36.center);
		\draw [style=new edge style 6] (36.center) to (37.center);
		\draw [style=new edge style 6, in=345, out=105, looseness=0.75] (37.center) to (50.center);
		\draw [style=new edge style 4] (50.center) to (52.center);
		\draw [style=new edge style 4] (52.center) to (53.center);
		\draw [style=new edge style 4] (53.center) to (48.center);
		\draw [style=new edge style 4] (48.center) to (47.center);
		\draw [style=new edge style 4] (47.center) to (43.center);
		\draw [style=new edge style 4] (43.center) to (42.center);
		\draw [style=new edge style 4] (42.center) to (41.center);
		\draw [style=new edge style 4] (41.center) to (40.center);
		\draw [style=new edge style 4] (40.center) to (39.center);
		\draw [style=new edge style 4] (39.center) to (38.center);
		\draw [style=new edge style 4, bend left=10] (38.center) to (11.center);
		\draw [style=new edge style 4] (11.center) to (35.center);
		\draw [style=new edge style 4] (35.center) to (34.center);
		\draw [style=new edge style 4] (34.center) to (33.center);
		\draw [style=new edge style 4] (33.center) to (26.center);
		\draw [style=new edge style 4] (26.center) to (21.center);
		\draw [style=new edge style 4] (21.center) to (23.center);
		\draw [style=new edge style 4] (23.center) to (22.center);
		\draw [style=new edge style 4] (22.center) to (20.center);
		\draw [style=new edge style 4] (20.center) to (19.center);
		\draw [style=new edge style 4] (19.center) to (16.center);
		\draw [style=new edge style 4] (16.center) to (15.center);
		\draw [style=new edge style 4] (15.center) to (8.center);
		\draw [style=new edge style 4] (8.center) to (12.center);
		\draw [style=new edge style 4] (12.center) to (13.center);
		\draw [style=new edge style 4, bend left=10] (13.center) to (98.center);
		\draw [style=new edge style 4, bend left=10] (98.center) to (97.center);
		\draw [style=new edge style 4, bend left=10] (97.center) to (96.center);
		\draw [style=new edge style 4] (96.center) to (89.center);
		\draw [style=new edge style 4] (89.center) to (88.center);
		\draw [style=new edge style 4] (88.center) to (9.center);
		\draw [style=new edge style 4] (9.center) to (87.center);
		\draw [style=new edge style 4] (87.center) to (86.center);
		\draw [style=new edge style 4] (86.center) to (85.center);
		\draw [style=new edge style 4] (85.center) to (81.center);
		\draw [style=new edge style 4] (81.center) to (73.center);
		\draw [style=new edge style 4] (73.center) to (72.center);
		\draw [style=new edge style 4] (72.center) to (71.center);
		\draw [style=new edge style 4] (71.center) to (70.center);
		\draw [style=new edge style 4] (70.center) to (69.center);
		\draw [style=new edge style 4] (69.center) to (68.center);
		\draw [style=new edge style 4] (68.center) to (62.center);
		\draw [style=new edge style 4] (62.center) to (58.center);
		\draw [style=new edge style 4] (58.center) to (57.center);
		\draw [style=new edge style 4] (57.center) to (10.center);
		\draw [style=new edge style 4] (10.center) to (56.center);
		\draw [style=new edge style 4] (56.center) to (55.center);
		\draw [style=new edge style 4] (55.center) to (54.center);
		\draw [style=new edge style 4] (54.center) to (51.center);
		\draw [style=new edge style 4] (51.center) to (52.center);
		\draw [style=new edge style 4] (15.center) to (8.center);
		\draw [style=new edge style 4] (96.center) to (94.center);
		\draw [style=new edge style 4] (14.center) to (7.center);
		\draw [style=new edge style 4] (95.center) to (98.center);
		\draw [style=new edge style 4] (13.center) to (14.center);
		\draw [style=new edge style 6] (14.center) to (7.center);
		\draw [style=new edge style 4] (7.center) to (12.center);
		\draw [style=new edge style 4] (8.center) to (17.center);
		\draw [style=new edge style 4, bend left=15] (17.center) to (16.center);
		\draw [style=new edge style 4, bend right=15] (15.center) to (17.center);
		\draw [style=new edge style 4] (17.center) to (7.center);
		\draw [style=new edge style 4, bend left, looseness=1.25] (20.center) to (24.center);
		\draw [style=new edge style 4, bend left, looseness=0.75] (24.center) to (25.center);
		\draw [style=new edge style 4, bend left] (25.center) to (21.center);
		\draw [style=new edge style 4] (22.center) to (24.center);
		\draw [style=new edge style 4] (23.center) to (25.center);
		\draw [style=new edge style 4] (31.center) to (33.center);
		\draw [style=new edge style 4] (32.center) to (34.center);
		\draw [style=new edge style 4] (30.center) to (35.center);
		\draw [style=new edge style 4] (30.center) to (11.center);
		\draw [style=new edge style 4] (6.center) to (11.center);
		\draw [style=new edge style 4] (36.center) to (11.center);
		\draw [style=new edge style 4] (43.center) to (44.center);
		\draw [style=new edge style 4] (42.center) to (44.center);
		\draw [style=new edge style 4] (41.center) to (44.center);
		\draw [style=new edge style 4, bend right=60, looseness=1.50] (48.center) to (49.center);
		\draw [style=new edge style 4] (53.center) to (50.center);
		\draw [style=new edge style 4] (10.center) to (61.center);
		\draw [style=new edge style 4] (56.center) to (61.center);
		\draw [style=new edge style 4] (57.center) to (61.center);
		\draw [style=new edge style 4, bend right=15] (62.center) to (60.center);
		\draw [style=new edge style 4] (59.center) to (60.center);
		\draw [style=new edge style 6] (61.center)
			 to (64.center)
			 to [bend right=15] (59.center)
			 to [bend right=15] (63.center)
			 to cycle;
		\draw [style=new edge style 5] (50.center) to (67.center);
		\draw [style=new edge style 4] (69.center) to (79.center);
		\draw [style=new edge style 4] (70.center) to (78.center);
		\draw [style=new edge style 4] (71.center) to (77.center);
		\draw [style=new edge style 4] (72.center) to (76.center);
		\draw [style=new edge style 6] (78.center)
			 to (77.center)
			 to (76.center)
			 to [bend right, looseness=1.25] (80.center)
			 to [bend right=45, looseness=1.25] (79.center)
			 to cycle;
		\draw [style=new edge style 4, bend right] (73.center) to (75.center);
		\draw [style=new edge style 4, bend right] (75.center) to (74.center);
		\draw [style=new edge style 4, bend right] (74.center) to (68.center);
		\draw [style=new edge style 4] (4.center) to (90.center);
		\draw [style=new edge style 4] (90.center) to (85.center);
		\draw [style=new edge style 4] (90.center) to (91.center);
		\draw [style=new edge style 4] (91.center) to (92.center);
		\draw [style=new edge style 4] (92.center) to (89.center);
		\draw [style=new edge style 4] (92.center) to (88.center);
		\draw [style=new edge style 4] (91.center) to (86.center);
		\draw [style=new edge style 4] (91.center) to (93.center);
		\draw [style=new edge style 4] (93.center) to (9.center);
		\draw [style=new edge style 4] (93.center) to (87.center);
		\draw [style=new edge style 4] (95.center) to (97.center);
		\draw [style=new edge style 4] (94.center) to (97.center);
		\draw [style=new edge style 4] (67.center) to (51.center);
		\draw [style=new edge style 4] (49.center) to (38.center);
		\draw [style=new edge style 4] (30.center) to (32.center);
		\draw [style=new edge style 4] (32.center) to (31.center);
		\draw [style=new edge style 4] (31.center) to (26.center);
		\draw [style=new edge style 5] (58.center) to (64.center);
		\draw [style=new edge style 5] (55.center) to (63.center);
		\draw (80.center) to (78.center);
		\draw (80.center) to (77.center);
		\draw (59.center) to (61.center);
		\draw (81.center) to (4.center);
		\draw (100.center) to (30.center);
		\draw [style=new edge style 7] (2.center)
			 to [bend left=330] (100.center)
			 to (99.center)
			 to [bend right=15] cycle;
		\draw (99.center) to (37.center);
		\draw (99.center) to (36.center);
		\draw (100.center) to (6.center);
		\draw (1.center) to (5.center);
		\draw (4.center) to (0.center);
		\draw (3.center) to (14.center);
		\draw [style=new edge style 7] (111.center)
			 to [bend right=15, looseness=0.75] (104.center)
			 to (105.center)
			 to (110.center)
			 to cycle;
		\draw [bend left=15] (7.center) to (104.center);
		\draw [bend right=15] (105.center) to (7.center);
		\draw (106.center) to (110.center);
		\draw (111.center) to (107.center);
		\draw (111.center) to (108.center);
		\draw (111.center) to (109.center);
		\draw (101.center) to (50.center);
		\draw [style=new edge style 7] (1.center)
			 to [bend right=45] (0.center)
			 to [bend right=15] (103.center)
			 to [bend right=15] (3.center)
			 to [bend right=15, looseness=1.25] (104.center)
			 to [bend right=15] (2.center)
			 to [bend right=15] (102.center)
			 to [bend right=15] (101.center)
			 to [bend right=15, looseness=0.75] cycle;
		\draw [style=new edge style 7] (112.center) to (2.center);
		\draw [style=new edge style 7] (112.center) to (100.center);
		\draw [style=new edge style 7] (112.center) to (99.center);
		\draw [style=new edge style 7] (114.center) to (103.center);
		\draw [style=new edge style 7] (114.center) to (0.center);
		\draw [style=new edge style 7] (0.center) to (113.center);
		\draw [style=new edge style 7] (113.center) to (101.center);
		\draw [style=new edge style 7] (113.center) to (114.center);
		\draw [style=new edge style 7] (1.center) to (113.center);
		\draw [style=new edge style 7] (114.center) to (102.center);
		\draw [style=new edge style 7] (115.center) to (114.center);
		\draw [style=new edge style 7] (114.center) to (2.center);
		\draw [style=new edge style 7] (115.center) to (104.center);
		\draw [style=new edge style 7] (115.center) to (3.center);
		\draw [style=new edge style 7] (115.center) to (2.center);
		\draw [style=new edge style 4, bend right] (47.center) to (46.center);
		\draw [style=new edge style 4, bend right=15, looseness=1.25] (46.center) to (45.center);
		\draw [style=new edge style 4, bend right=15] (45.center) to (39.center);
		\draw [style=new edge style 4, bend right=15] (45.center) to (40.center);
		\draw [bend right=15] (60.center) to (54.center);
		\draw (94.center) to (116.center);
		\draw (95.center) to (116.center);
		\draw (82.center) to (117.center);
		\draw (118.center) to (50.center);
		\draw (118.center) to (37.center);
		\draw (75.center) to (119.center);
		\draw (60.center) to (119.center);
		\draw (49.center) to (120.center);
	\end{pgfonlayer}
\end{tikzpicture}